\pdfoutput=1
\documentclass [11pt,reqno]{amsart}
\usepackage[english]{babel}
\usepackage[utf8]{inputenc}
\usepackage[T1]{fontenc}
\usepackage{lmodern}
\usepackage{amssymb,amsmath,amsthm,amsfonts}
\usepackage{thmtools,mathtools}
\usepackage{microtype}
\allowdisplaybreaks
\usepackage{tikz,tikz-cd}
\usepackage[breaklinks,pdfencoding=auto,psdextra]{hyperref}
\usepackage{geometry}
\usepackage{tikz-cd}
\usepackage{tikz}
\usetikzlibrary{tikzmark}
\usetikzlibrary{arrows.meta}
\usepackage{comment}
\usepackage{xfrac}
\usepackage{tabularx}

\usepackage{chngcntr}

\newcolumntype{M}{>{\centering\arraybackslash}m{5.2cm}}

\def\bC{{\mathbb{C}}}

\def\bN{{\mathbb{N}}}

\def\bP{{\mathbb{P}}}
\def\bQ{{\mathbb{Q}}}
\def\bR{{\mathbb{R}}}

\def\bZ{{\mathbb{Z}}}
\DeclareMathAlphabet{\mathmybb}{U}{bbold}{m}{n}

\def\cA{{\mathcal{A}}}
\def\cB{{\mathcal{B}}}

\def\cE{{\mathcal{E}}}
\def\cF{{\mathcal{F}}}
\def\cG{{\mathcal{G}}}

\def\cL{{\mathcal{L}}}
\def\cM{{\mathcal{M}}}
\def\cN{{\mathcal{N}}}
\def\cO{{\mathcal{O}}}
\def\cP{{\mathcal{P}}}

\def\cU{{\mathcal{U}}}

\def\cX{{\mathcal{X}}}
\def\cY{{\mathcal{Y}}}
\def\cZ{{\mathcal{Z}}}

\def\m{\mathfrak m}
\def\p{\mathfrak p}

\def\an{\operatorname{an}}

\def\Aut{\operatorname{Aut}}

\def\coker{\operatorname{coker}}

\def\cont{\operatorname{cont}}

\def\diff{\mathrm{d}}
\def\dc{\mathrm{d}^c}

\def\div{\operatorname{div}}

\def\exc{\operatorname{exc}}

\def\fin{\mathrm{fin}}

\def\Frac{\operatorname{Frac}}
\def\FS{\operatorname{FS}}
\DeclareMathOperator{\Gal}{Gal}

\def\GL{\operatorname{GL}}

\def\Hom{\operatorname{Hom}}
\def\hyb{\operatorname{hyb}}

\def\im{\operatorname{im}}
\def\ITAC{\operatorname{ITAC}}

\def\lsc{\operatorname{lsc}}

\def\MA{\operatorname{MA}}

\def\ord{\operatorname{ord}}

\def\Spec{\operatorname{Spec}}
\def\Spm{\operatorname{Spm}}

\def\Supp{\operatorname{Supp}}

\def\TAC{\operatorname{TAC}}

\def\triv{\operatorname{triv}}
\def\um{\operatorname{um}}
\def\usc{\operatorname{usc}}

\def\vol{\operatorname{vol}}
\def\ZR{\operatorname{ZR}}

\def\tilde{\widetilde}
\def\setminus{\smallsetminus}
\def\emptyset{\varnothing}

\def\upint{\mathchoice%
    {\mkern13mu\overline{\vphantom{\intop}\mkern7mu}\mkern-20mu}%
    {\mkern7mu\overline{\vphantom{\intop}\mkern7mu}\mkern-14mu}%
    {\mkern7mu\overline{\vphantom{\intop}\mkern7mu}\mkern-14mu}%
    {\mkern7mu\overline{\vphantom{\intop}\mkern7mu}\mkern-14mu}%
  \int}
\def\lowint{\mkern3mu\underline{\vphantom{\intop}\mkern7mu}\mkern-10mu\int}

\addto\extrasenglish{}
\addto\extrasenglish{}
\addto\extrasenglish{\def\equationautorefname~#1\null{\textnormal{(#1)}\null}}
\declaretheorem[name=Theorem,refname={Theorem},style=plain,numberwithin=subsection]{theorem}
\declaretheorem[name=Theorem-Definition,refname={Theorem-Definition},style=plain,sibling=theorem]{theorem-definition}
\declaretheorem[name=Theorem,refname={Theorem},style=plain,numbered=no]{theorem*}
\declaretheorem[name=Proposition-Definition,refname={Proposition-Definition},style=definition,sibling=theorem]{proposition-definition}
\declaretheorem[name=Proposition,refname={Proposition},style=plain,sibling=theorem]{proposition}
\declaretheorem[name=Proposition,refname={Proposition},style=plain,numbered=no]{proposition*}
\declaretheorem[name=Lemma,refname={Lemma},style=plain,sibling=theorem]{lemma}
\declaretheorem[name=Lemma,refname={Lemma},style=plain,numbered=no]{lemma*}
\declaretheorem[name=Definition,refname={Definition},style=definition,sibling=theorem]{definition}
\declaretheorem[name=Definition,refname={Definition},style=definition,numbered=no]{definition*}
\declaretheorem[name=Remark,refname={Remark},style=definition,sibling=theorem]{remark}
\declaretheorem[name=Remark,refname={Remark},style=remark,numbered=no]{remark*}
\declaretheorem[name=Corollary,refname={Corollary},style=plain,sibling=theorem]{corollary}
\declaretheorem[name=Corollary,refname={Corollary},style=plain,numbered=no]{corollary*}
\declaretheorem[name=Example,refname={Example},style=definition,sibling=theorem]{example}
\declaretheorem[name=Notation,refname={Notation},style=definition,sibling=theorem]{notation}
\declaretheorem[name=Conjecture,refname={Conjecture},style=plain,sibling=theorem]{conjecture}
\declaretheorem[name=Claim,refname={Claim},style=plain,sibling=theorem]{claim}
\declaretheorem[name=Claim,refname={Claim},style=plain,numbered=no]{claim*}
\declaretheorem[name=Théorème,refname={Théorème},style=plain,numbered=no]{theoremfr*}
\declaretheorem[name=Slogan,refname={Slogan},style=remark,numbered=no]{slogan*}

\newtheorem{theo}{Theorem}

\newcommand{\va}{|\cdot|}

\newcommand{\fonction}[5]{\begin{array}{l|rcl}
#1: & #2 & \longrightarrow & #3 \\
    & #4 & \longmapsto & #5 \end{array}}
\newcommand{\hlocale}[6]{#1_{#2,...,#3}^{#4,...,#5}(#6)}

 \thanks{The author was partly supported by the collaborative research 
	center SFB 1085 \emph{Higher Invariants - Interactions between Arithmetic Geometry and Global Analysis} funded by the Deutsche Forschungsgemeinschaft. The main part of this research was conducted at Université Paris Cité.}

\begin{document}

\title[Topological adelic curves]{Topological adelic curves: Zariski-Riemann spaces, \\ 
algebraic coverings, Harder-Narsimhan filtrations and heights}
\date{\today}
\author{Antoine Sédillot}
\address{A. Sédillot, Mathematik, Universit{\"a}t 
		Regensburg, 93040 Regensburg, Germany}
	\email{antoine.sedillot@mathematik.uni-regensburg.de}

\begin{abstract}
In this article, we introduce topological adelic curves. Roughly speaking, a topological adelic curve is a topological space of (generalised) absolute values on a given field satisfying a product formula. Topological adelic curves are the topological counterpart to adelic curves introduced by Chen and Moriwaki. They aim at handling Arakelov geometry over possibly uncountable fields and give further ideas in the formalisation of the analogy between Diophantine approximation and Nevanlinna theory. Using the notion of pseudo-absolute values developed in \cite{Sedillot_pav}, we prove several fundamental properties of topological adelic curves: algebraic coverings, existence of Harder-Narasimhan filtrations and of volume functions. We also define heights of cycles and give a generalisation of Nevanlinna's first main theorem in this framework. Another important feature of topological adelic curves is that they come equipped with Zariski-Riemann type spaces that admit a natural locally ringed space structure and usual Arakelov theoretic objects (e.g. adelic vector bundles) admit a natural interpretation in terms of metrised objects on these Zariski-Riemann spaces. 
\end{abstract}


\maketitle

\tableofcontents

\section*{Introduction}

\subsection*{Motivations and background}

\subsubsection*{Arakelov geometry over number fields and arithmetic function fields}

Arakelov theory stems from the analogy between number fields and function fields. Roughly speaking, one can formulate this analogy as follows: the geometry of schemes of finite type over $\Spec(\bZ)$ should be similar to the geometry of schemes of finite type over a smooth projective curve. Unfortunately, schemes over $\Spec(\bZ)$ are not "compact" and it is not quite clear how to "compactify" them within the world of schemes. To address this issue, Arakelov \cite{Arakelov74} added analytic data to algebro-geometric objects. Arakelov's ideas have been used by Faltings \cite{Faltings91} in his proof of Mordell's conjecture and in the proof of Bogomolov's conjecture \cite{Ullmo98,Zhang98}.



Studying the arithmetic of fields that are more general than global fields has also been developed to a great extent. Lang remarked that it was natural to study \emph{arithmetic function fields}, namely finite type field extensions of $\bQ$ \cite{Lang74,Lang86}. Indeed, Mordell-Weil's and Faltings' theorems both hold over such fields \cite{Lang91}. Later, Moriwaki constructed a height theory over arithmetic function fields \cite{Moriwakiheights} (see also \cite{BPS16}). Recently, Vojta proved a version of Roth's theorem over arithmetic function fields \cite{Vojta21}.

The study of infinite algebraic extensions of number fields has also been a great inspiration for developing analogues of the tools of Diophantine geometry. Let us mention for instance a version of Siegel's lemma \cite{RoyTHunder96}, the study of tensorial semistability \cite{BostChen13} and the introduction of Siegel fields \cite{GaudronRemond17}.

\subsubsection*{Arakelov geometry over adelic curves}

In \cite{ChenMori},Chen and Moriwaki introduced an Arakelov theory over arbitrary countable fields. The central object of the theory is called an \emph{adelic curve}. Namely, an adelic curve is the data $S=(K,(\Omega,\nu),(\va_{\omega})_{\omega\in\Omega})$, where $K$ is a field, $(\Omega,\nu)$ is a measure space and $(\va_{\omega})_{\omega\in\Omega}$ is a family of absolute values on $K$. Moreover, an adelic curve $S=(K,(\Omega,\nu),(\va_{\omega})_{\omega\in\Omega})$ is called \emph{proper} if the following \emph{product formula} holds:
\begin{align*}
\forall a\in K^{\times},\quad \int_{\Omega}\log|f|_{\omega}\nu(\diff\omega)=0.
\end{align*}

Adelic curves arise naturally in various number theoretic situations. In particular, any global field can be naturally equipped with an adelic structure. More generally, any countable field can be endowed with an adelic structure. Furthermore, adelic curves allow us to study global fields, trivially valued fields and arithmetic function fields uniformly.

Let us now introduce the counterpart of the usual tools of Arakelov geometry over adelic curves. Let $S=(K,(\Omega,\nu),(\va_{\omega})_{\omega\in\Omega})$ be an adelic curve. For any $K$-scheme $X$, for any $\omega\in\Omega$, we denote by $K_{\omega}$ the completion of $K$ with respect to the absolute value $|\cdot|_{\omega}$ and set $X_{\omega} := X \otimes_{K} K_{\omega}$. The avatar of a line bundle in algebraic geometry is called an \emph{adelic line bundle}. Let $X$ be a $K$-scheme, an adelic line bundle over $X$ is the data $\overline{L}=(L,\varphi)$, where $L$ is a line bundle over $X$ and $\varphi=(\varphi_{\omega})_{\omega\in\Omega}$ is a family of continuous metrics (in the sense of Berkovich analytic spaces) over each $L_{\omega} := L \otimes_{\cO_X} \cO_{X_\omega}$. In addition, the metric family is subject to dominance and measurability conditions (cf. \S 6.1 in \cite{ChenMori}).

Among the results of the theory, let us first mention the development of the geometry of numbers via the slope theory of adelic vector bundles (\cite{ChenMori}, Chapter IV). Note that even in the case of number fields, Chen-Moriwaki's approach yields a new interpretation of known results. An arithmetic intersection theory for adelic line bundles is also constructed in \cite{ChenMori21}. The study of the positivity of adelic line bundles and a Hilbert-Samuel formalism are introduced in \cite{ChenMori24}. As an application, a generalisation of Bogomolov's conjecture over a (countable) field of characteristic zero is proven. Note also that a version of Roth's theorem over a particular class of adelic curves is established by Dolce and Zucconi in \cite{DolceZucconi25}. This result generalises Vojta's aforementioned result \cite{Vojta21}.

\subsubsection*{Nevanlinna theory and \emph{M}-fields}

Another, but yet not disconnected, motivation for our work is to study the analogy between Diophantine geometry and Nevanlinna theory. This analogy was spotted first by Osgood \cite{Osgood81} and further explored by Vojta in \cite{Vojta87}. Roughly speaking, Nevanlinna theory is the study of equations of the form $f(z)=a$, where $f$ is meromorphic on $\bC$ and $a\in \mathbb{P}^{1}(\bC)$. It builds on two fundamental theorems. In the analogy, the first one corresponds to Weil's theorem for heights (e.g. \cite{BombieriGubler}, Theorem 2.3.8). The second one is seen as an analogue of Roth's theorem \cite{Roth55}. In Appendix \ref{sec:Nevanlinna_classical}, we recalled the basic notions of Nevanlinna theory as well as the main ideas of the analogy.

In \cite{Gubler97}, Gubler introduced the notion of \emph{M-fields}, with the idea of including Nevanlinna theory in an Arakelov theoretic framework. Roughly speaking, an $M$\emph{-field} $K$ is a field $K$ equipped with a measure space $M$ such that any element $a$ in $K$ defines an integrable real function $|a|_{\cdot}$ defined almost everywhere on $M$. Moreover, these functions are assumed to satisfy the axioms of absolute values almost everywhere. The following example of an $M$-field coming from Nevanlinna theory is fundamental for our purposes. Consider the field $\cM(\bC)$ of meromorphic complex functions. Fix a real number $R>0$ and set $M_R:=\{z\in \bC : |z|_{\infty}\leq R\}$ where the boundary $\{z\in \bC: |z|_{\infty}=R\}$ is equipped with the Haar probability measure and the open disc $\{z\in \bC: |z|_{\infty}<R\}$ is equipped with a counting measure. For any $f\in \cM(\bC)$, consider the map
\begin{align*}
(z\in M_R)\mapsto\left\{\begin{matrix}
|f(z)|_{\infty}\quad\text{if }|z|_{\infty}=R, \\
e^{-\ord(f,z)}\quad\text{if }|z|_{\infty}<R,
\end{matrix}\right.
\end{align*}
which is well-defined everywhere except poles of $f$ on the circle of radius $R$, hence almost everywhere. Then one can check that we have an $M_R$-field $\cM(\bC)$.

Using $M$-fields, Gubler obtains a generalisation of Nevanlinna's first main theorem which includes notably the construction of a height function for fields of arithmetic nature. Nonetheless, it does not seem clear how one could hope to obtain further results, e.g. geometry of numbers, in the framework of $M$-fields. 

\subsection*{Goal}

The goal of this article is to introduce objects of Arakelov geometric nature, allowing us to handle uncountable fields and to formalise the analogy with Nevanlinna theory. We focus on the foundation of the theory, but we list at the end of the introduction several future works and open questions that could be handled using the material we develop here.

\subsubsection*{Hints from the above discussion}
On the one hand, in the theory of adelic curves, the countability condition is imposed by the fact that the parameter space of absolute values is a measure space. This is due to the apparition of suprema and infima of measurable functions when considering operations on adelic vector bundles. Although this approach gives a lot of flexibility, we cannot expect the tools to transpose directly in the uncountable setting. A natural idea is to consider a topological space as parameter space and replace the measurability conditions with (semi-)continuity ones. 

On the other hand, the Nevanlinna theory example of $M$-field suggests that the space of possible arithmetic inputs should be larger than the space of usual absolute values. Note that even in the case of a classical adelic curve structure on $\bQ(T)$, namely coming from classical Arakelov geometry over $\mathbb{P}^{1}_{\bZ}$ (\cite{ChenMori}, \S 3.2.5), the natural topological space parametrising the Archimedean absolute values is the set of transcendental elements of the complex unit disc. From the point of view of measure theory, it is just the difference of the complex unit disc and the countable subset of algebraic numbers. However, this space, equipped with its usual topology, is very pathological.

\subsubsection*{Pseudo-absolute values, globally valued fields and Zariski-Riemann spaces}

The framework of globally valued fields, introduced by Ben Yaacov-Hrushovski \cite{Hrushovski16}, gives another approach to handling arithmetic over fields. Roughly speaking, a \emph{globally valued field} (GVF for short) is a field equipped with a family of heights satisfying the usual height compatibility axioms. This notion originates from model theory and there are several equivalent characterisations of GVFs (\cite{GVF24}, Theorem 7.7). The link with the above discussion is the following: a countable GVF is an equivalence class of proper adelic curves (\emph{loc. cit.}, Corollary 7.11). This link with model theory yields another motivation for developing Arakelov geometry over uncountable fields: indeed, the ultraproduct construction is fundamental in model theory and in general, ultraproducts are uncountable.   

Over a possibly uncountable field $K$, the GVF structures can be interpreted as a suitable measure on the  \emph{space of pseudo-absolute values} (or of \emph{pseudo-valuations}) on $K$. More precisely, a pseudo-absolute value on a field $K$ is a map $|\cdot| : K \to [0,+\infty]$ satisfying 
\begin{itemize}
\item[(i)] $|1| = 1$ and $|0|=0$;
	\item[(ii)] for all $a,b\in K$, $|a+b| \leq |a|+|b|$;
	\item[(iii)] for all $a,b\in K$ such that $\{|a|,|b|\}\neq \{0,+\infty\}$, $|ab| = |a||b|$. 
\end{itemize}
Moreover, $A_{|\cdot|}:=\{a\in K:|a|<+\infty\}$ is a valuation ring of $K$ with maximal ideal $m_{|\cdot|}:=\{a\in K:|a|=0\}$ and $|\cdot|$ induces an absolute value on the residue field $\kappa_{|\cdot|}:=A_{|\cdot|}/\m_{|\cdot|}$. In other terms, a pseudo-absolute value is an absolute value allowing "singularities". This notion was first introduced by Weil in \cite{Weil51} and was developed independently by Ben Yaacov-Destic-Hrushovski-Szachniewicz \cite{GVF24} and the author in \cite{Sedillot_pav}. 

Considering the development of Berkovich's non-Archimedean analytic geometry \cite{Berko90,ChambertLoir06,CLD12,GK17,GK19,LemanissierPoineau24}, it is natural to expect a "global analytic" approach to Arakelov geometry \cite{Paugam09,YuanZhang21,CaiGubler24,Song25}. In this article, we propose  a framework to do so, starting from the observation that the space $M_{K}$ of all pseudo-absolute values on a field $K$ equipped with the topology of pointwise convergence is compact Hausdorff and can be interpreted as a Berkovich analytic Zariski-Riemann space (\cite{GVF24}, Proposition 2.7 and \cite{Sedillot_pav}, Theorems A-C): namely $M_{K}$ can be described as a projective limit of Berkovich analytic spaces (over the prime subring). This observation allows us to equip $M_{K}$ with the structure of a locally ringed space and opens the door for studying coherent sheaves on it. Pseudo-absolute values are thus the natural candidate to encode the local aspects of the theoretical constructions that will follow later on. The work \cite{Sedillot_pav} was done with this goal in mind. 

\subsubsection*{Heuristic guideline}
The general philosophy of this article is the following. To an "algebraic object" $X$ (a locally ringed space of algebro-geometric nature), we associate an analytic one: namely
\begin{align*}
X^{\an}:=\{x=(p,|\cdot|_{x}) : p\in X \text{ and }|\cdot|_{x} \text{ is an absolute value on }\kappa(p)\},
\end{align*}
or rather, a subset of the latter. This space is equipped with a topology (in a similar way to the one of the Berkovich analytification of a variety over a completely valued field) and should enjoy sufficiently nice topological properties (e.g. Hausdorff, locally compact, paracompact...). We now attach an "adelic object" to our analytic object by essentially specifying a Borel measure. Arakelov geometry is performed on this object. However, this Arakelov adelic geometry should be governed by the geometry of the adelic space. We list the instances of this philosophy that we study in this article. 

\begin{center}
\begin{tabular}{|*{3}{M|}}
    \hline
    \tikzmark{la11} Algebraic object \tikzmark{ra11}   &   \tikzmark{la21} Analytic object \tikzmark{ra21} & \tikzmark{la31} Adelic object \tikzmark{ra31}\\\hline
    \tikzmark{la12} $\Spec(K)$, where $K$ is a field\tikzmark{ra12}   & \tikzmark{la22}$\{\text{Absolute values on } K\}$\tikzmark{ra22} & \tikzmark{la32} Adelic curve over $K$\tikzmark{ra32}\\\hline
    \tikzmark{la13} $\ZR(K)$, where $K$ is a field\tikzmark{ra13}   &   \tikzmark{la23}$M_{K}=\{\text{Pseudo-absolute values on }K\}$\tikzmark{ra23} & \tikzmark{la33}Topological adelic curve with adelic field $K$\tikzmark{ra33}\\\hline
    \tikzmark{la14} $\ZR(K/A)\cong\Spec(A)$, where $A$ is a Prüfer domain with fraction field $K$\tikzmark{ra14}   &   \tikzmark{la24}$\cM(A,\|\cdot\|)$, where $\|\cdot\|$ is a Banach norm on $A$\tikzmark{ra24} & \tikzmark{la34}Integral topological adelic curve with adelic field $K$and integral structure $(A,\|\cdot\|)$\tikzmark{ra34}\\\hline
  \end{tabular}
  \begin{tikzpicture}
    [
      remember picture,
      overlay,
      -latex,
      color=black,
      yshift=1ex,
      shorten >=1pt,
      shorten <=1pt,
    ]
    \draw ({pic cs:la12}) to [bend right] ({pic cs:la13});
    \draw ({pic cs:la14}) to [bend left] ({pic cs:la13});
  \end{tikzpicture}
  \end{center}
  The arrow going down is the generic point and the one going up is the inclusion of a quasi-compact subset. In particular, we think of adelic curves as the generic fibre of the topological adelic curves we introduce here.
 
\subsection*{Content of the article}

\subsubsection*{First part}
This article consists of three parts. In the first one, we start by some reminders and by introducing our Zariski-Riemann spaces (\S \ref{sub:ZR_analytic_spaces}). Then we introduce topological adelic curves. A \emph{topological adelic curve} is the data $S=(K,\phi : \Omega \to M_{K}, \nu)$, where $K$ is a field, $\Omega$ is a locally compact Hausdorff topological space, $\phi:(\omega\in\Omega)\mapsto |\cdot|_{\omega}\in M_K$ is a continuous map between $\Omega$ and the set $M_K$ of all pseudo-absolute values on $K$, and $\nu$ is a Borel measure on $\Omega$ such that, for any $f\in K^{\times}$, the function $(\omega\in\Omega)\mapsto \log|f|_{\omega} \in [-\infty,+\infty]$ is $\nu$-integrable (\S \ref{sec:topological_adelic_curves}). We say that the topological adelic curve $S$ is \emph{proper} if the following \emph{product formula}
\begin{align*}
\forall f\in K^{\times}, \quad \int_{\Omega}\log|f|_{\omega}\nu(\diff\omega) = 0,
\end{align*} 
holds. Moreover, to a topological adelic curve $S$ as above, we attach two Zariski-Riemann spaces $\ZR(K)_{S}$ and $\ZR(K)_{S}^{\an}$ (\S \ref{sub:ZR_spaces_adelic_curves}). 

Let us give an example of a topological adelic curve arising in Nevanlinna theory. We fix $R>0$. We define a topological adelic curve $S_{R}=(K_{R},\phi_{R}:\Omega_{R}\to M_{K_{R}},\nu_{R}),\phi_{R})$, where
\begin{itemize}
	\item $K_R$ is the field of (germs of) meromorphic functions over the closed disc 
	$$\overline{D(R)}:= \{z\in\bC : |z|_{\infty}\leq R\} \subset \bC;$$
	\item $\Omega_R = \{z\in\bC:|z|_{\infty}<R\}\coprod \{z\in\bC :|z|_{\infty}=R\}$, where $\{z\in\bC:|z|_{\infty}<R\}$ is equipped with the discrete topology and  $\{z\in\bC :|z|_{\infty}=R\}$ is equipped with the usual topology;
	\item the map $\phi_R : \Omega_R \to M_{K_R}$ is defined by
	\begin{align*}
	\forall z\in \Omega_R, \quad \phi_R(z) := \left\{\begin{matrix}
(f\in K_{R}) \mapsto |f(z)|_{\infty}\in [0,+\infty] &\text{ if } |z|_{\infty}=R,\\
(f\in K_{R}) \mapsto e^{-\ord(f,z)} \in \bR_{\geq 0}&\text{ if } |z|_{\infty}<R;\end{matrix}\right.
	\end{align*}
	\item the measure $\nu_R$ is defined by
	\begin{align*}
	\forall z\in \{z'\in\bC:|z'|_{\infty}<R\}, \quad  \nu_R(\{z\}) := \left\{\begin{matrix}
\log\frac{R}{|z|_{\infty}} \text{ if } &0<|z|_{\infty}<R,\\
\log R \text{ if } &z= 0,\end{matrix}\right.
\end{align*}		
	and $\nu_R$ is the Lebesgue probability measure on $\{z\in\bC :|z|_{\infty}=R\}$.
\end{itemize}

From the point of view of Nevanlinna theory, we are not only interested in the study of heights over one $S_{R}$ as above, but rather on the collection $(S_{R})_{R>0}$. More precisely, we study characteristic functions, which are functions $\bR_{>0}\to\bR$ and especially their asymptotic growth when $R\to+\infty$ (\S \ref{sec:Nevanlinna_classical}). This means that our height should take values coming from functions $\bR_{>0}\to\bR$ (instead of $\bR$). We are also interested in comparing these heights, especially in view of slope theory and Harder-Narasimhan filtrations, as we will see later. A natural framework to do so is non-standard analysis. We fix a free ultrafilter $\cU$ on $\bR_{>0}$. Our height functions will be elements of $\prod_{\cU}\bR$, the ultraproduct of the reals w.r.t. the ultrafilter $\cU$ (cf. \S \ref{sub:ultraproducts}). Then $\prod_{\cU}\bR$ is an ordered field. Note that if $\cU$ avoids finite Lebesgue measure subsets of $\bR_{>0}$, then inequalities of the form $f\leq g$ in $\prod_{\cU}$ imply that $f(R)\leq g(R)$ for all $R>0$ except on a subset of $\bR_{>0}$ that does not belong to $\cU$. Even though the latter is weaker than the assertion "$f(R)\leq g(R)$ for all $R>0$ except on a subset of finite Lebesgue measure" appearing in Nevanlinna theory (cf. Theorem \ref{th:Nevanlinna_second_main_theorem}), it gives a reasonable starting point. 

The major obstacle to defining heights lies in the fact that the product formula is not satisfied on the $S_{R}$'s. Indeed, Jensen's formula yields
\begin{align*}
\forall f\in K_{R}^{\times}, \quad \int_{\Omega_{R}} \log|f|_{\omega}\nu_{R}(\diff\omega) = \log|c(f,0)|,
\end{align*}
where $c(f,0)$ denotes the first non-zero coefficient in the Laurent series expansion of $f$ in $0$. In the context of Diophantine geometry, the product formula is the ingredient that allows us to define height functions relative to a metrised line bundle. The analogue of this construction in Nevanlinna theory is the so-called First Main Theorem (Theorem \ref{th:Nevanlinna_first_main_theorem}). The latter results from the fact that the defect in the product formula for the collection of topological adelic curves $(S_{R})_{R>0}$ is a bounded function of $R$ (in fact, constant in the present case). Since the height functions that we consider are typically divergent as $R\to+\infty$, we can quotient $\prod_{\cU}\bR$ by an equivalence relation $\sim$ to force the product formula. Another advantage is that the total order on $\prod_{\cU}\bR$ restricts to $\prod_{\cU}\bR/\sim$.   

Using the two above paragraphs, we introduce the notion of \emph{families of topological adelic curves} (\S \ref{sub:family_tac}). This includes the above discussion and also gives a natural framework to formulate Diophantine approximation over topological adelic curves.  

We conclude the first part by studying algebraic coverings of topological adelic curves (\S \ref{sec:algebraic_covering_tac}), namely the extension of topological adelic structures w.r.t. algebraic extensions of the base field. More precisely, we have the following result (which is a variant of Proposition 1.9 in \cite{GVF24}).  

\begin{theo}[Propositions \ref{prop:algebraic_extension_tac}, \ref{prop:action_of_Galois_group_algebraic_extension} and Remark \ref{rem:uniqueness}]
Let $S=(K,\phi : \Omega \to V, \nu)$ be a topological adelic curve. Let $K'/K$ be an algebraic field extension. 
\begin{itemize}
	\item[(1)] There exist a topological adelic curve $S' := S \otimes_K K' := (K',\phi':\Omega' \to M_{K'},\nu')$ and a morphism of topological adelic curves $S' \to S$. Moreover, $S'$ is proper if so is $S$. 
	\item[(2)] Assume that  $K'/K$ is Galois and that, for any $\omega\in\Omega$, the residue field of the pseudo-absolute value $|\cdot|_{\omega}$ is perfect. Then we have a homeomorphism
	\begin{align*}
	\Omega'/\Aut(L/K) \cong \Omega.
	\end{align*}
	Moreover, if $\nu$ is Radon, then $S'$ is the only topological adelic curve with base field $K'$ extending $S'$ with Galois-invariant measure.  
\end{itemize}
\end{theo}
We also have an analogue of this result for families of topological adelic curves (\S \ref{sub:algebraic_covering_family_tac}). In Nevanlinna theory, it means that our framework allows us to include Nevanlinna theory of meromorphic functions of finite coverings of $\bC$ (and even families of such).

\subsubsection*{Second part}

In the second part, we build the foundation of the geometry of numbers over topological adelic curves. As initiated by Bost \cite{Bost96,Bost01}, this is done by studying slopes of (adelic) vector bundles. Let $S=(K,\phi : \Omega \to M_{K}, \nu)$ be a topological adelic curve. An \emph{adelic vector bundle} on $S$ is then defined as a pair $\overline{E}=(E,\xi)$, where $E$ is a $K$-vector space of finite rank and $\xi=(\|\cdot\|_{\omega\in \Omega})$ is a family of pseudo-norm (the suitable generalisation of a norm over a pseudo-valued field) on $E$ satisfying suitable regularity and integrability conditions (cf. \S \ref{sec:pseudo-norm_families}). Adelic vector bundles can also be interpreted as metrised vector bundles on the Zariski-Riemann space $\ZR(K)_{S}$ (cf. \S \ref{sub:ZR_interpretation_pseudo-norm_families}).

In the case where the base topological adelic curve is proper, the \emph{Arakelov degree} of an adelic vector bundle $\overline{E}=(E,\xi)$ is defined as
\begin{align*}
\widehat{\deg}(\overline{E}):= - \int_{\Omega} \log\|\eta\|_{\omega,\det}\nu(\diff\omega),
\end{align*}
where $\eta\in\det(E)\setminus\{0\}$, is independent of the choice of $\eta$. Denote also by $\mu(\overline{E}):=\widehat{\deg}(\overline{E}=)/\dim_{K}(E)$ the \emph{slope} of $\overline{E}$. Define as well the \emph{maximal slope} $\mu_{\max}(\overline{E}):=\sup_{0 \neq F\subseteq E}\mu(\overline{F})$ and the \emph{minimal slope} $\widehat{\mu}_{\min}(\overline{E}) := \inf_{E \twoheadrightarrow G \neq \{0\} } \widehat{\mu}(\overline{G})$.	

In \S \ref{sec:slopes_proper_case}, we adapt the Harder-Narasimhan formalism for adelic vector bundles over a proper topological adelic curve. Our result is the following. 
\begin{theo}[Theorem \ref{th:HN_general_proper}]
\label{th:HN_intro}
Let $\overline{E}=(E,\xi)$ be an adelic vector bundle on $S$. We assume that the pseudo-norm family $\xi$ is ultrametric on $\Omega_{\um}$. Then there exists a unique flag 
\begin{align*}
0=E_0 \subsetneq E_1 \subsetneq \cdots \subsetneq E_n = E,
\end{align*}
of $\overline{E}$, such that
\begin{itemize}
	\item[(1)] for any $i=1,...,n$, $\overline{E_i/E_{i-1}}$ is semistable, i.e. for any non-zero vector subspace $F\subset E$, we have $\mu_{\min}(\overline{F})\leq \mu_{\min}(\overline{E})$;
	\item[(2)] we have the inequalities 
	\begin{align*}
	\widehat{\mu}(\overline{E_1/E_0}) > \cdots > \widehat{\mu}(\overline{E_n/E_{n-1}}).
	\end{align*}
\end{itemize}
\end{theo}

We also obtain the existence of such Harder-Narasimhan filtrations in our example of families of topological adelic curves $\mathbf{S}=(S_{R})_{R>0}$ arising in Nevanlinna theory. In this case, inequalities between slopes are understood in $\bR_{\mathbf{S}}$, the Dedekind-MacNeille completion of the ultraproduct $\prod_{\cU}\bR/\sim$ (cf. Theorem \ref{th:HN_Nevanlinna}). This opens the door for studying adelic vector bundles in Nevanlinna theory, with a possible interpretation of transcendental methods of Diophantine geometry in an Arakelov geometric framework.

\subsubsection*{Third part}

In the third and last part, we introduce height functions attached to adelic line bundles on a projective scheme. Classically in Arakelov geometry, an adelic line bundle is the data of a line bundle together with a family of metrics at each place (where the metrics are understood in the sense of Berkovich over non-Archimedean places). Since we are working over pseudo-absolute values, we need some adjustments.

We fix a topological adelic curve $S=(K,\phi : \Omega \to V, \nu)$ and a projective $K$-scheme $\pi : X \to \Spec(K)$. There are two Zariski-Riemann spaces $\ZR(X)_{S},\ZR(X)^{\an}_{S}$ attached to $X$. A \emph{pseudo-metric} on a line bundle $L$ over $X$ is a family $\varphi=(|\cdot|_{\varphi}(\mathbf{x})_{\mathbf{x}\in\ZR(X)_{S}^{\an}}$ of pseudo-norms on the fibres of $L$ satisfying a glueing condition (\S \ref{sub:pseudo-metric_families_definitions}). This glueing condition is equivalent to saying that a line bundle equipped with a pseudo-metric is the same thing as a metrised line bundle on $\ZR(X)_{S}$ (\S \ref{sub:ZR_interpretation_pseudo-metric}). \emph{Adelic line bundle} are then such objects satisfying certain regularity and integrability conditions (\S \ref{sub:dominated_pseudo-metric_family}-\ref{sub:adelic_line_bundles}). 

This notation admits a natural generalisation for families of topological adelic curves (\S \ref{sub:adelic_line_bundles_families_of_tac}). The most natural example from the point of view of Nevanlinna theory is the adelic line bundle determined by a metrised line bundle on a projective complex variety. In this case, the height functions we obtain coincide with the classical characteristic functions of Nevanlinna theory.

We also discuss the pushforward of adelic line bundles (\S \ref{sub:pushforward_pseudo-metric_families}). Assume that $X$ is geometrically integral and that $\Omega=M_{K}$. Let $\overline{L}=(L,\varphi)$ be an adelic line bundle on $X$. Define $\pi_{\ast}\varphi=(\|\cdot\|_{\omega})_{\omega\in\Omega}$, where 
\begin{align*}
\forall s\in H^{0}(X,L), \quad \|s\|_{\omega}:=\displaystyle\sup_{\mathbf{x}\in (f_{S}^{\an})^{-1}(\omega)} |s|_{\varphi}(\mathbf{x}),
\end{align*}
where $f_{S}^{\an} : \ZR(X)_{S}^{\an}\to \Omega$ is the structural map. 

\begin{theo}[Theorem \ref{th:pushforward_pseudo-metric_family_dominated_general} and Proposition \ref{prop:existence_pushforward_pseudo-norm_family}]
\label{th_intro:pushforward_adelic_line_bundle}
$\pi_{\ast}\overline{L}:=(\pi_{\ast}L,\pi_{\ast}\varphi)$ is a upper-semicontinuous adelic line bundle on $S$.
\end{theo}

Theorem \ref{th_intro:pushforward_adelic_line_bundle} allows us to define arithmetic ($\chi$-)volume functions on proper topological adelic curves (\S \ref{sub:volume_functions_proper_tac}). 

In the final two sections, we introduce height functions attached to adelic line bundles on $X$. Our results are the following.

\begin{theo}[Theorem \ref{th:height_closed_points_proper} and Theorem-Definition \ref{th-def:arithmetic_intersection_product_tac}]
Assume that the topological adelic curve $S$ is proper.
\begin{itemize}
	\item[(1)] Let $\overline{L}$ be an adelic line bundle on $X$. We have a \emph{height function}
	\begin{align*}
	\fonction{h_{L}}{X(\overline{K})}{\bR}{P}{h_{\overline{L}}(P).}
	\end{align*}
	\item[(2)] Let $\overline{L^{(1)}}=(L^{(1)},\varphi^{(1)}),\overline{L^{(2)}}=(L^{(2)},\varphi^{(2)})$ be adelic line bundles on $X$.
	\begin{itemize}
		\item[(i)] For any closed point $P$ of $X$, we have
	\begin{align*}
	h_{\overline{L^{(1)}}+\overline{L^{(2)}}}(P) = h_{\overline{L^{(1)}}}(P) + h_{\overline{L^{(2)}}}(P).
	\end{align*}
		\item[(ii)] Let $P$ be a closed point of $X$. Assume that $L^{(1)}=L^{(2)}$. Then we have 
	\begin{align*}
	h_{\overline{L^{(1)}}}(P)= h_{\overline{L^{(1)}}}(P) + O(1),
	\end{align*}
	where the bound does not depend on $P$ (but depends on $\varphi^{(1)},\varphi^{(2)}$). 
	\end{itemize}
	\item[(3)] For any $l$-cycle $Z$ on $X$ and any integrable adelic line bundles $\overline{L^{(0)}}=(L^{(0)},\varphi^{(0)}),...,\overline{L^{(l)}}=(L^{(l)},\varphi^{(l)})$ on $X$, we have the \emph{multi-height} of $Z$ w.r.t. $\overline{L^{(0)}},...,\overline{L^{(l)}}$
	\begin{align*}
	h_{\overline{L^{(0)}},...,\overline{L^{(l)}}}(Z)\in\bR.
	\end{align*}
	Moreover, this multi-height is symmetric and multi-linear w.r.t. the $\overline{L^{(i)}}$'s and satisfies a projection formula. If furthermore the $L^{(i)}$'s are semi-ample and the $\varphi^{(i)}$'s are semi-positive, the multi-height $h_{\overline{L^{(0)}},...,\overline{L^{(l)}}}(Z)$ depends on the $\varphi^{(i)}$'s up to a bounded quantity independent of $Z$.
\end{itemize}
\end{theo}

We also have an analogue result for asymptotically proper families of topological adelic curves. Let us describe it in the Nevanlinna theory example. We consider the family $\mathbf{S}=(S_{R})_{R>0}$ of topological adelic curves as above. Let $X_{0}$ be a projective $\bC$-scheme. Recall that any continuously metrised line bundle on $X_{0}$ determines an adelic line bundle $\overline{L}=(L,\varphi)$ on $X:=X_{0}\otimes_{\bC}\cM(\bC)$ over $\mathbf{S}$. Note that closed points of $X$ are in one-to-one correspondence with holomorphic curves $f: A \to X_{0}$, where $A$ runs over finite holomorphic coverings of $\bC$. We have the following generalisation of Nevanlinna's first main theorem of (\cite{Gubler97}, Theorem 3.18). 

\begin{theo}[Example \ref{example:arithmetic_intersection_product_Nevanlinna} and cf. Theorem \ref{th:height_of_closed_points_families} for a more general statement]
Let $X$ be a projective $\cM(\bC)$-scheme.
\begin{itemize}
	\item[(1)] Let $\overline{L}$ be an adelic line bundle on $X$ over $\mathbf{S}$. We have a \emph{height function}
	\begin{align*}
	\fonction{h_{L}}{X(\overline{\cM(\bC)})}{\prod_{\cU}\bR/\sim}{P}{h_{\overline{L}}(P).}
	\end{align*}
	Moreover, if $\overline{L}$ is determined by a continuously metrised line on $X_{0}$, the height of a closed point $P\in X$ w.r.t. $\overline{L}$ coincides with the (class in $\prod_{\cU}\bR/\sim$) of the characteristic function attached to the holomorphic curve corresponding to $P$.
	\item[(2)] Let $\overline{L^{(1)}}=(L^{(1)},\varphi^{(1)}),\overline{L^{(2)}}=(L^{(2)},\varphi^{(2)})$ be adelic line bundles on $X$.
	\begin{itemize}
		\item[(i)] For any closed point $P$ of $X$, we have
	\begin{align*}
	h_{\overline{L^{(1)}}+\overline{L^{(2)}}}(P) = h_{\overline{L^{(1)}}}(P) + h_{\overline{L^{(2)}}}(P).
	\end{align*}
		\item[(ii)] Let $P$ be a closed point of $X$. Assume that $X\cong X_{0}\otimes_{\bC}\cM(\bC)$, where $X_{0}$ is a projective $\bC$-scheme and that $\overline{L^{(1)}},\overline{L^{(2)}}$ are both determined by continuously metrised line bundles $(L^{(1)}_{0},\varphi_{0}^{(2)}),(L_{0}^{(2)},\varphi_{0}^{(2)})$ on $X_{0}$, where $L_{0}^{(1)}=L_{0}^{(2)}$. Then we have 
	\begin{align*}
	h_{\overline{L_1}}(P)= h_{\overline{L_2}}(P).
	\end{align*}
\end{itemize}
	\item[(3)] For any $l$-cycle $Z$ on $X$ and any integrable adelic line bundles $\overline{L^{(0)}}=(L^{(0)},\varphi^{(0)}),...,\overline{L^{(l)}}=(L^{(l)},\varphi^{(l)})$ on $X$, we have the \emph{multi-height} of $Z$ w.r.t. $\overline{L^{(0)}},...,\overline{L^{(l)}}$
	\begin{align*}
	h_{\overline{L^{(0)}},...,\overline{L^{(l)}}}(Z)\in\displaystyle\prod_{\cU}\bR/\sim.
	\end{align*}
	Moreover, this multi-height is symmetric and multi-linear w.r.t. the $\overline{L^{(i)}}$'s and satisfies a projection formula. If furthermore there exist a projective $\bC$-scheme $X_{0}$ such that $X\cong X_{0}\otimes_{\bC}\cM(\bC)$ and integrable continuously metrised line bundles $\overline{L_{0}^{(0)}},...,\overline{L_{0}^{(l)}}$ on $X_{0}$ determining respectively $\overline{L^{(0)}},...,\overline{L^{(l)}}$, then the multi-height $h_{\overline{L^{(0)}},...,\overline{L^{(l)}}}(Z)$ is independent of the $\varphi^{(i)}$'s.
\end{itemize}
\end{theo}

\subsection*{Upcoming work}

Let us mention open questions of interest for future developments.

\begin{itemize}
	\item In this text, the notion of family of topological adelic curves is developed with the example of Nevanlinna theory in mind. However, the flexibility of the framework allows us to include what should correspond to Diophantine approximation over proper topological adelic curves (cf. Example \ref{example:families_of_tac} (1)). A natural question is thus to try to formulate classical results in Diophantine approximation in our context, where the use of topological arguments should be helpful (cf. \cite{Vojta21,DolceZucconi25} where the authors implicitly use a hypothesis of topological nature on the adelic curve involved).
	\item As mentioned above, slope theory in the context of Nevanlinna theory should be useful to reformulate transcendental arguments in Diophantine geometry in a purely Arakelov geometric framework (cf. e.g. \cite{Gasbarri25}). 
	\item In this article, we focus our attention on vector bundles on Zariski-Riemann spaces. In view of Propositions \ref{prop:fp_sheaves_ZR_algebraic} and \ref{prop:coherent_sheaves_locally_closed_integral}, it is natural to hope for a satisfactory study of metrised coherent sheaves on Zariski-Riemann spaces. For instance, we expect coherence to be preserved via direct image. This property, combined with the fact that coherent sheaves on Zariski-Riemann spaces have Tor-dimension $\leq 1$, and thus that Zariski-Riemann spaces behave like a smooth projective curve, should be of crucial help for tackling the next questions below.
	\item Harder-Narasimhan filtrations are a fundamental tool in algebraic geometry but also in a wider range of context (e.g. in $p$-adic Hodge theory over the Fargues-Fontaine curve). Our Harder-Narasimhan filtrations relate adelic vector bundles to semi-stable ones and the interpretation of adelic vector bundles, or more generally "adelic" coherent sheaves, in terms of the corresponding metrised algebraic object on Zariski-Riemann spaces should allow us to study moduli spaces of adelic vector bundles.
	\item The definition of arithmetic volumes invites questions concerning the regularity of these volumes. In light of \cite{Sedillot23_differentiability}, we expect Siu inequality and differentiability properties to hold for these volumes. Note that this differentiability is a key ingredient for proving existential closedness for globally valued fields (cf. \cite{Michal23}) and would achieve a major step towards the existence of a model companion for the language of GVF (\cite{GVF24}, Conjecture 12.7). 
	\item Our formulation of heights over topological adelic curves and families of the latter allows us to define essential minima and other Arakelov geometric tools that encode equidistribution phenomena (cf. e.g. \cite{BallaySombra24}). In Nevanlinna theory, this should lead to equidistribution results for Nevanlinna and Ahlfors currents attached to holomorphic curves. Moreover, assuming any strong enough Siu-type inequality for arithmetic volumes, logarithmic equidistribution results would follow using the methods from \emph{loc. cit.}.
	\item Finally, our formulation of Arakelov geometry as coming from analysis on a Berkovich-type space of global nature naturally invites to formulate global pluripotential theory on such spaces (cf. e.g. \cite{Morrow25}). 
\end{itemize}

\subsection*{Acknowledgements}

We would like to thank Huayi Chen for his support and discussions during the elaboration of this paper. We also thank Sébastien Boucksom, José Ignacio Burgos Gil, Keita Goto, Walter Gubler, Klaus Künnemann and Jérôme Poineau for numerous remarks and suggestions. Finally, we thank Debam Biswas and Micha\l{} Szachniewicz for a lot of fruitful discussions. 

\section*{Conventions and notation}

\begin{itemize}
	\item All rings considered in this article are commutative with unit.
	\item Let $A$ be a ring. We denote by $\Spm(A)$ the set of maximal ideals of $A$.
	\item By a local ring $(A,\m)$, we mean that $A$ is a local ring and $\m$ is its maximal ideal. In general, if $A$ is a local ring, the maximal ideal of $A$ is denoted by $\m_A$.  
	\item Let $A$ be a ring and let $X \to \Spec(A)$ be a scheme over $A$. Let $A\to B$ be an $A$-algebra. Then we denote $X \otimes_{A} B := X \times_{\Spec(A)} \Spec(B)$.
	\item Let $K$ be a field and $X\to \Spec(K)$ be a $K$-scheme. For any domain $A$ with fraction field $K$, we call \emph{model} of $X/A$ any $A$-scheme $\cX \to \Spec(A)$ whose generic fibre is isomorphic to $X$. A model of $\pi : \cX\to \Spec(A)$ $X/A$ is respectively called \emph{projective,flat,coherent} if $\pi$ is projective, flat, finitely presented.
	\item Let $k$ be a field. We denote by $|\cdot|_{\triv}$ the trivial absolute value on $k$. If we have an embedding $k \hookrightarrow \bC$, we denote by $|\cdot|_{\infty}$ the restriction of the usual Archimedean absolute value on $\bC$.
	\item Let $(k,|\cdot|)$ be a valued field. Unless mentioned otherwise and when no confusion may arise, we will denote by $\widehat{k}$ the completion of $k$ w.r.t. $|\cdot|$.
	\item Let $(X,\cO_X)$ be a locally ringed space. We respectively denote by $\mathrm{Fp}(X),\mathrm{Coh}(X)$ the categories of finitely presented, coherent sheaves on $X$. Moreover, we call \emph{vector bundles} on $X$ any locally free sheaves of $\cO_{X}$-modules of finite rank and we denote by $\mathrm{Vb}(X)$ the corresponding subcategory of $\mathrm{Coh}(X)$ and we use the additive notation for tensor products of line bundles. 
	\item Let $(\Omega,\cA,\nu)$ be a measure space. Denote by $\cL^{1}(\Omega,\cA,\nu)$ be the set of all $\nu$-integrable functions $f:\Omega \to [-\infty,+\infty]$. Let $f : \Omega \to [-\infty,+\infty]$, we define
\begin{align*}
\upint_{\Omega} f(\omega)\nu(\diff\omega) := \inf\left\{\int_{\Omega}g(\omega)\nu(\diff\omega) : g\in \cL^{1}(\Omega,\cA,\nu) \text{ and } f \leq g \quad\nu\text{-a.e.}\right\},
\end{align*}
and 
\begin{align*}
\lowint_{\Omega} f(\omega)\nu(\diff\omega) := \sup\left\{\int_{\Omega}g(\omega)\nu(\diff\omega) : g\in \cL^{1}(\Omega,\cA,\nu) \text{ and } g \leq f \quad\nu\text{-a.e.}\right\}.
\end{align*}
We say that $f$ is $\nu$\emph{-dominated} if
\begin{align*}
\upint_{\Omega} f(\omega)\nu(\diff\omega) <+\infty \quad \text{and} \quad \lowint_{\Omega} f(\omega)\nu(\diff\omega)>-\infty.
\end{align*}
Equivalently, $f$ is $\nu$-dominated iff there exists $g\in\cL^1(\Omega,\cA,\nu)$ such that $|f|\leq g$ $\nu$-a.e.
\end{itemize}

\part{Topological adelic curves: definition and algebraic coverings}
\label{part:tac}

\section{Pseudo-absolute values, Zariski-Riemann spaces, adelic curves and globally valued fields}
\label{sec:prelim}

\subsection{Pseudo-absolute values}

\subsubsection{Definition}
Let $K$ be a field. A \emph{pseudo-absolute value} on $K$ is a map $|\cdot|:K\to[0,+\infty]$ satisfying 
\begin{itemize}
	\item[(i)] $|0|=0$ and $|1|=1$;
	\item[(ii)] for all $a,b\in K$, $|a+b| \leq |a|+|b|$;
	\item[(iii)] for all $a,b\in K$ such that $\{|a|,|b|\}\neq \{0,+\infty\}$, $|ab| = |a||b|$. 
\end{itemize}
Recall that any pseudo-absolute value $|\cdot|$ on $K$ determines a \emph{finiteness ring} $A_{|\cdot|}=\{|\cdot|<+\infty\}$, which is a valuation ring of $K$ with maximal ideal $\m_{|\cdot|}=\{|\cdot|=0\}$, called its \emph{kernel}. Moreover, $|\cdot|$ induces an absolute value on the \emph{residue field} $\kappa_{|\cdot|}:=A_{|\cdot|}/\m_{|\cdot|}$ called the \emph{residue absolute value}. A pseudo-absolute value is called \emph{Archimedean}, \emph{non-Archimedean}, \emph{residually trivial} if the associated residue absolute value is Archimedean, non-Archimedean, trivial. 

We use the same notation as in \cite{Sedillot_pav}. Namely, by "let $(\va,A,\m,\kappa)$ be a pseudo-absolute value", we mean that $\va$ is a pseudo-absolute value on $K$ with finiteness ring $A$, kernel $\m$, residue field $\kappa$. By default, when we write "let $v$ be a pseudo absolute value", we mean $v=(|\cdot\|_{v},A_v,\m_v,\kappa_v)$. Moreover, if $v$ is a pseudo-absolute value on $K$, we denote by $\widehat{\kappa_v}$ the completion of the residue field $\kappa_v$ w.r.t. the residue absolute value induced by $v$.

\subsubsection{Extension of pseudo-absolute values}

Let $L/K$ be a finite separable field extension. Let $v$ be a pseudo-absolute value on $K$. Denote by $A'$ the integral closure of $A_v$ in $L$. This is a Prüfer domain, namely its prime localisations are valuation rings. Moreover, the extensions of $A_v$ to $L$ are in bijection with $\Spm(A')$. For any $\m_{w}\in\Spm(A')$, we denote by $\kappa_{w}$ the corresponding residue field, this is a finite field extension of $\kappa_{v}$.

\begin{proposition}[\cite{Sedillot_pav}, Proposition 3.1.2]
\label{prop:formula_extension_sav}
\begin{itemize}
	\item[(1)] There is a bijective correspondence between the set of pseudo-absolute values on $L$ above $v$ and the set of extensions of the residue absolute value of $v$ with respect to extensions of the form $\kappa_v\to\kappa_w$, where $w$ runs over the set of maximal ideals of $A'$.
	\item[(2)] Furthermore, we have the equality
\begin{align}
\label{eq:formula_extension_sav}
\displaystyle\sum_{\m_w\in \Spm(A')} \frac{1}{|\Spm(A')|} \sum_{i | v} \frac{[\widehat{\kappa_{w,i}}:\widehat{\kappa_v}]_s}{[\kappa_w:\kappa_v]_s} = 1,
\end{align}
where, for all $\m_w\in \Spm(A')$, $i$ runs over the set of extensions of the residue absolute value of $\va_v$ to $\kappa_w$ and $\widehat{\kappa_{w,i}}$ denotes the completion of $\kappa_w$ for any such absolute value. 
\end{itemize}

\end{proposition}

Now consider an arbitrary finite field extension $L/K$. Denote by $K'$ the separable closure of $K$ in $L$. 

\begin{proposition}[\cite{Sedillot_pav}, Corollary 3.2.2]
\label{cor:purely_inseparable_extension_sav}
Let $v$ be a pseudo-absolute value on $K$. Then the set of extensions of $v$ on $L$ is in bijection with the set of extensions of $v$ on $K'$.
\end{proposition}

The following proposition characterises the action of the group of automorphisms on pseudo-absolute values.

\begin{proposition}[\cite{Sedillot_pav}, Propositions 3.3.1 and 3.3.4]
\label{prop:action_Galois}
Let $L/K$ be an algebraic field extension. 
\begin{itemize}
	\item[(1)] $\Aut(L/K)$ induces a right action on $M_L$ as follows. For all $x \in M_L$, for all $\tau \in \Aut(L/K)$, the map
\begin{align*}
\fonction{\va_{\tau(x)}}{L}{[0,+\infty]}{a}{|\tau(a)|_{x}}
\end{align*}
defines a pseudo-absolute value on $L$ denoted by $x\circ \tau$. 
	\item[(2)] Assume that $L/K$ is Galois with Galois group $G$. Let $v\in M_K$ such that the residue field $\kappa_v$ is perfect. Then $G$ acts transitively on the set $M_{L,v}$ of extensions of $v$ to $L$.  
\end{itemize}
\end{proposition}

\subsubsection{Space of pseudo-absolute values and integral structures}
\label{subsub:prelim_integral_structures}
Let $K$ be a field. Recall that the set $M_{K}$ of all pseudo-absolute values on $K$ equipped with the topology of point-wise convergence is a (non-empty) compact Hausdorff topological space (\emph{loc. cit.}, Theorem 7.1.2). We denote by $M_{K,\infty}$ and $M_{K,\um}$ respectively the set of Archimedean and non-Archimedean pseudo-absolute values on $K$. Define a map $\epsilon : M_{K,\infty} \to ]0,1]$ by sending any $|\cdot|\in M_{K,\infty}$ to the unique $\epsilon(|\cdot|)\in]0,1]$ such that the restriction of the residue absolute value of $|\cdot|$ to $\bQ$ is $|\cdot|_{\infty}^{\epsilon(|\cdot|)}$.

An \emph{integral structure} for $K$ is a Banach ring $(A,\|\cdot\|)$ such that $A$ is a Prüfer domain with fraction field $K$. If $(A,\|\cdot\|)$ is an integral structure for $K$, then the Berkovich analytic spectrum $\cM(A,\|\cdot\|)$ identifies as a closed subspace of $M_K$ (\emph{loc. cit.}, Proposition 9.1.5). Moreover, an integral structure $(A,\|\cdot\|)$ for $K$ is called \emph{tame} if 
\begin{itemize}
	\item[(i)] $\cM(A,\|\cdot\|)$ contains the trivial absolute value on $K$;
	\item[(ii)] for any ultrametric element $|\cdot|_x\in \cM(A,\|\cdot\|)$ and any $f\in A$, the inequality
	\begin{align*}
	|f|_x \leq 1
	\end{align*}
	is satisfied;
	\item[(iii)] $(A, \|\cdot\|)$ is a uniform Banach ring.
\end{itemize}

\begin{proposition}[\cite{Sedillot_pav}, Proposition 9.3.3]
\label{prop:algebraic_extension_of_tame_spaces}
Let $(A,\|\cdot\|)$ be a tame integral structure for $K$. Let $L/K$ be an algebraic extension. Then the integral closure $B$ of $A$ in $L$ can be equipped with a norm $\|\cdot\|_B$ such that $(B,\|\cdot\|_{B})$ is a tame integral structure for $L$. Moreover, $\cM(B,\|\cdot\|_{B})$ can be identified with the preimage of $\cM(A,\|\cdot\|)$ via the restriction $M_L \to M_K$.
\end{proposition}

\subsubsection{Examples of integral structure in Nevanlinna theory}
\label{subsub:example_integral_structures}

Let $R>0$ and let $\overline{D(R)}$ denote the complex closed disc of radius $R$. We denote respectively by $A_R=\cO(\overline{D(R)})$ and $K_R=\cM(\overline{D(R)})$ the ring of germs of holomorphic functions and the field of germs of meromorphic functions on $\overline{D(R)}$. Let $\|\cdot\|_R$ denote the supremum norm on $\overline{D(R)}$  and define $\|\cdot\|_{R,\hyb}:=\max\{\|\cdot\|_R,|\cdot|_{\triv}\}$, where $|\|\cdot\|_{\triv}$ denotes the trivial norm on $A_R$. Then $(A_R,\|\cdot\|_{R,\hyb})$ is a Banach ring and $(A_R,\|\cdot\|_{R,\hyb})$ defines an integral structure for $K$ (\cite{Sedillot_pav}, Example 9.2.1 (4)). Moreover, results from (\emph{loc. cit.}, \S 9.4.3) imply that the integral structure $(A_R,\|\cdot\|_{R,\hyb})$ is tame and the space $V_{R}:=\cM(A_R,\|\cdot\|_{R,\hyb})$ has the following description. 

\begin{proposition}[\emph{loc. cit.}, Proposition 9.4.7]
\label{prop:global_space_of_sav_compact_disc}
\begin{itemize}
	\item[(i)] We have homeomorphisms 
\begin{align*}
V_{R,\infty} \cong ]0,1]\times \overline{D(R)}, \quad V_{R,\um} \cong \bigsqcup_{z \in \overline{D(R)}} [0,+\infty]/\sim,
\end{align*}	
where $\sim$ denotes the equivalence relation which identifies the extremity $0$ of each branch.
	\item[(ii)] $V_{R,\infty}$ is dense in $V_{R}$.
	\item[(iii)] The Banach ring $(A_R,\|\cdot\|_{R,\hyb})$ is a geometric base ring. 
\end{itemize}
\end{proposition} 

\begin{remark}
\label{remark:geoemtric_base_ring}
The notion of \emph{geometric base ring} was introduced in (\cite{LemanissierPoineau24}, Définition 3.3.8). It is a technical definition that includes the usual examples of Banach rings that we will be concerned with. The reader who is unfamiliar with this notion should think of it as the general condition to ensure a reasonable notion of Berkovich analytification. 
\end{remark}

\subsubsection{Pseudo-norms}
\label{subsub:pseudo-norms}

Let $K$ be a field and $v\in M_K$ be a pseudo-absolute value on $K$. Let $E$ be a finite-dimensional vector space over $K$ of dimension $d$. A \emph{pseudo-norm} on $E$ in $v$ is a map $\|\cdot\|_{v} : E \to [0,+\infty]$ satisfying the following conditions:
\begin{itemize}
	\item[(i)] $\|0\|_v = 0$ and there exists a basis $(e_1,...,e_d)$ of $E$ such that $\|e_1\|_v,\cdots,\|e_d\|_v \in \bR_{>0}$, such a basis is called \emph{adapted} to $\|\cdot\|_{v}$;
	\item[(ii)] for any $(\lambda,x)\in K\times E$ such that $\{|\lambda|_v,\|x\|_v\}\neq\{0,+\infty\}$, we have $\|\lambda x\|_v=|\lambda|_v\|x\|_v$;
	\item[(iii)] for any $x,y\in E$, $\|x+y\|_v\leq \|x\|_v+\|y\|_v$.
\end{itemize}
Under these assumptions, $(E,\|\cdot\|_v)$ is called a \emph{pseudo-normed} vector space in $v$. 

\begin{proposition}[\cite{Sedillot_pav}, Proposition 6.1.3]
\label{prop:local_semi-norms}
Let $(E,\|\cdot\|_v)$ be a pseudo-normed vector space in $v$. The \emph{finiteness module} $\cE_{\|\cdot\|_{v}}=\{\|\cdot\|_{v}<+\infty\}$ is a free $A_v$-module of rank $d$ generated by any basis of $E$ satisfying condition (i) above, the \emph{kernel} $N_{\|\cdot\|_{v}}:=\{\|\cdot\|_{v}=0\}$ is equal to the $A_v$-submodule $\m_v\cE_{\|\cdot\|_{v}}$ of the finiteness module. Moreover, $\|\cdot\|_{v}$ induces a norm on the \emph{residue vector space} $\widehat{E}_{\|\cdot\|_{v}} := \cE_{\|\cdot\|_{v}} \otimes_{A_v} \widehat{\kappa_v}$ called the \emph{residue norm}.
\end{proposition}

\begin{remark}
\label{rem:pseudo-norm_metrised_vector_bundle}
The construction from Proposition \ref{prop:local_semi-norms} can be reversed (\cite{Sedillot_pav}, Proposition 1.1.7). Moreover, we will see in Example \ref{example:metrised_vector_bundle_ZR} that pseudo-normed vector spaces can be interpreted as metrised vector bundles on some Zariski-Riemann space
\end{remark}

We say that a pseudo-norm $\|\cdot\|_{v}$ is \emph{ultrametric}, resp. \emph{Hermitian} if so is the residue norm. We also use the notation from (\cite{Sedillot_pav}, \S 6). Namely, by "let $(\|\cdot\|,\cE,N,\widehat{E})$ be a pseudo-norm on the $K$-vector space $E$ in $v$", we mean that $\|\cdot\|$ is a pseudo-norm on $K$ in $v$ with finiteness module $\cE$, kernel $N$ and residue vector space $\widehat{E}$. Moreover, without additional specification, by "let $\|\cdot\|_v$ be a pseudo-norm on $E$ in $v$", we mean the pseudo-norm $(\|\cdot\|_v,\cE_v,N_v,\widehat{E}_v)$. Moreover, if no confusion may arise, we omit "in $v$".

Recall that in (\emph{loc. cit.}, \S 6.2), we have introduced the usual algebraic constructions for pseudo-normed vector spaces. More precisely, let $(E,\|\cdot\|_{v})$ be a pseudo-normed vector space in $v\in M_{K}$. 
\begin{itemize}
	\item[(1)] Let $F$ be a non-zero vector subspace of $E$. Then $\|\cdot\|_{v}$ induces a pseudo-norm on $F$ denoted again by $\|\cdot\|_{v}$.
	\item[(2)] Let $G$ be non-zero quotient of $E$. Then $\|\cdot\|_{v}$ induces a \emph{quotient} pseudo-norm on $G$ denoted by $\|\cdot\|_{v,G}$.
	\item[(3)] $\|\cdot\|_{v}$ induces a \emph{dual} pseudo-norm on $E^{\vee}$ denoted by $\|\cdot\|_{v,\ast}$.
	\item[(4)] Let $(E',\|\cdot\|'_{v})$ be another pseudo-normed vector space. Then this data induces an $\epsilon$\emph{-tensor product} pseudo-norm on $E$ in. Likewise, we have a $\pi$\emph{-tensor product} pseudo-norm on $E$.
	\item[(5)] Let $i\geq 1$ be an integer. Then $\|\cdot\|_{v}$ induces the $i^{th}\epsilon$\emph{-exterior power} pseudo-norm and $i^{th}\pi$\emph{-exterior power} pseudo-norm on $\Lambda^{i}E$ denoted respectively by $\|\cdot\|_{v,\Lambda_{\epsilon}^{i}E}$ and $\|\cdot\|_{v,\Lambda_{\pi}^{i}E}$. In the particular, when $i=d$, the pseudo-norm $\|\cdot\|_{v,\Lambda_{\pi}^{i}E}$ is called the \emph{determinant} pseudo norm on $\det(E)$ and is denoted by $\|\cdot\|_{v,\det}$.
\end{itemize}

We now list generalisations of useful properties of norms in the context of pseudo-norms. 

\begin{proposition}[\cite{Sedillot_pav}, Propositions 6.2.2-6.2.4]
Let $(E,(\|\cdot\|_v,\cE_v,N_v,\widehat{E}_v))$ be a pseudo-normed finite-dimensional $K$-vector space in $v\in M_{K}$.
\begin{itemize}
	\item[(1)]\label{prop:correspondence_dual_quotient} Let $G$ be a quotient of $E$. Then the dual pseudo-norm $\|\cdot\|_{v,G,\ast}$ on $G^{\vee}$ identifies with the restriction of the pseudo-norm $\|\cdot\|_{v,\ast}$ on $E^{\vee}$ to $G^{\vee}$.
	\item[(2)]\label{prop:comparaison_double_dual_semi-norm} The inequality 
\begin{align*}
\|\cdot\|_{v,\ast\ast} \leq \|\cdot\|_v
\end{align*}
holds, where $\|\cdot\|_{v,\ast\ast}$ denotes the dual pseudo-norm of $\|\cdot\|_{v,\ast}$ on $E^{\vee\vee}\cong E$. Moreover, if either $v$ is Archimedean, or if $v$ is non-Archimedean and the pseudo-norm $\|\cdot\|_v$ is ultrametric, then we have 
\begin{align*}
\|\cdot\|_{v,\ast\ast} = \|\cdot\|_v.
\end{align*}
	\item[(3)]\label{prop:Hadamrd_inequality_local_semi-norm} Let $(e_1,..,e_r)$ be a basis of $E$ which is adapted to $\|\cdot\|_v$. Then, for any $\eta\in \det(E)$, we have the equality
\begin{align*}
\|\eta\|_{v,\det} = \inf\left\{\|x_1\|_v\cdots\|x_r\|_v : x_1,...,x_r \in \cE_v \text{ and } \eta = x_1\wedge\cdots\wedge x_r\right\}.
\end{align*}
\end{itemize}
\end{proposition}

\begin{definition}
\label{def:diagonalisable_local_semi-norm}
Let $E$ be a finite-dimensional vector space over $K$ and let $\mathbf{e}=(e_1,...,e_r)$ be any basis of $E$. A pseudo-norm $\|\cdot\|$ on $E$ in $v=(|\cdot|, A,\m,\kappa)$ for which $\mathbf{e}$ is an adapted basis is called \emph{diagonalisable} if there exists a basis $\mathbf{e'}=(e'_1,...,e'_r)$ of $E$ such that the transition matrix from $\mathbf{e}$ to $\mathbf{e'}$ belongs to $\GL_{r}(A)$ and we have
\begin{align*}
\forall a_1,...,a_r\in K^{r}, \quad \|a_1e'_{1}+...+a_{r}e'_{r}\|=\left\{\begin{matrix}
\displaystyle\sqrt{\sum_{i=1}^{r}|a_i|^{2}\|e'_{i}\|^2 }&\text{ if } v \text{ is Archimedean},\\
\displaystyle\max_{i=1,...,r} |a_i|\|e'_{i}\|&\text{ if } v \text{ is ultrametric}.\end{matrix}\right.
\end{align*}
In that case, we say that $\mathbf{e'}$ is \emph{orthogonal} for $\|\cdot\|$, it is moreover called \emph{orthonormal} if it satisfies the additional condition: $\|e'_i\|=1$ for all $i=1,...,r$. 
\end{definition}

\subsection{Algebraic and analytic Zariski-Riemann spaces}
\label{sub:ZR_analytic_spaces}

Zariski-Riemann spaces will play a fundamental role in this article. They have already appeared in (\cite{Sedillot_pav}, \S 8 and 10) for fields and we extend their definition for general projective schemes. Our algebraic Zariski-Riemann spaces are special cases of Temkin's relative Zariski-Riemann spaces \cite{Temkin11} (\S \ref{subsub:Zariski-Riemann_space_algebraic}). Our analytic Zariski-Riemann spaces are "analytifications" of the algebraic ones (\S \ref{subsub:ZR_space_analytic}). We then investigate finitely presented and coherent sheaves on such spaces (\S \ref{subsub:coherence_structure_sheaf_field_algebraic}-\ref{subsub:coherence_struture_sheaf_GAGA_analytic}). Finally, we introduce metrised vector bundles on algebraic Zariski-Riemann spaces (\S \ref{subsub:metrised_vector_bundle_ZR_space}). They will serve as the building blocks for the adelic vector bundles we will study in our Arakelov geometric applications.

\subsubsection{Projective (sub)-models}
\label{subsub:projective_sub-models}

Let $K$ be a field and $k$ be a subring of $K$ (we do not necessarily assume that $k$ is a domain with quotient field $K$). Let $X$ be a projective $K$-scheme. By a \emph{projective sub-model} $\cX$ of $X$ over $k$, we mean a projective $k$-scheme $\cX$ together with a factorisation $X\to \cX \to \Spec(k)$, where the first arrow is schematically dominant and the second is projective, i.e. a $X$-modification of $\Spec(k)$ in the terminology of \cite{Temkin11}. Let $\cX_{1},\cX_{2}$ be two projective sub-models of $X$ over $k$. We say that $\cX_{2}$ \emph{dominates} $\cX_{1}$ if there exists a $k$-morphism of schemes $\cX_{2} \to \cX_{1}$ compatible with the schematically dominant maps $X\to \cX_{1}$ and $X\to \cX_{2}$. In this case, such a morphism is unique. The category of projective sub-models is cofiltered: indeed given two projective sub-models $X \to X_{1} \to \Spec(k)$, $X \to \cX_{2} \to \Spec(k)$, let $X$ be the Zariski closure of $X$ in $\cX_{1}\times_{\Spec(k)}\cX_{2}$, this is a projective sub-model of $X$ over $k$ dominating both $\cX_{1}$ and $\cX_{2}$.  

If $k$ is a domain with fraction field $K$, a projective sub-model $\cX$ of $X$ over $k$ is called a \emph{projective model} of $X$ over $A$ if the generic fibre of $\cX$ is isomorphic to $K$. Note that any closed immersion $X\to \mathbb{P}^{n}_{K}$ yields a projective model of $X$ over $k$ by considering the scheme-theoretic image of $X\to \mathbb{P}^{n}_{k}$. 

\subsubsection{Algebraic Zariski-Riemann space}
\label{subsub:Zariski-Riemann_space_algebraic}

Let $K$ be a field, $k$ be a subring of $K$ and $X$ be a projective $K$-scheme. Denote by $\cM_{X/k}$ the category of projective sub-models of $X$ over $k$ introduced above. Define the \emph{(algebraic) Zariski-Riemann space} of $X$ over $k$ as  
\begin{align*}
\ZR(X/k) := \displaystyle\varprojlim_{\cX \in \cM_{X/k}} \cX.
\end{align*}
This is a locally ringed space and we denote its structure sheaf by $\cO_{\ZR(X/k)}$. Recall that 
\begin{align*}
\cO_{\ZR(X/k)} = \displaystyle\varinjlim_{\cX \in \cM_{X/k}} p_{\cX}^{-1}\cO_{\cX},
\end{align*}
where, for all $\cX\in\cM_{X/k}$, $p_{\cX} : \ZR(X/k)\to \cX$ denotes the projection. Moreover, assume that $k$ is an integral domain with fraction field $K$, in the projective limit defining $\ZR(X/k)$, one can only consider projective models of $X$ over $k$.

\begin{example}
\label{example:ZR_algebraic}
When $X=\Spec(K)$, the space $\ZR(X/k):=\ZR(\Spec(K)/k)$ is the classical Zariski-Riemann space consisting of the valuation rings of $K$ containing $k$ equipped with the Zariski topology (cf. e.g. \cite{Sedillot_pav}, Theorem 1.4.3). For any $A\in \ZR(K/k)$, the local ring at $A$ is simply $A$ itself. If we further assume that $k$ is a Prüfer domain with fraction field $K$, then $\ZR(K/k)$ is isomorphic to $\Spec(k)$.
\end{example}

In the general case, $\ZR(X/k)$ admits a similar description. Let $\mathrm{Val}(X/k)$ denote the set of triples $\mathbf{x}=(x,A,\phi)$, where $x\in X$ is a scheme point, $A$ is a valuation ring of $\kappa(x)$ and $\phi : \Spec(A) \to \Spec(k)$ is a morphism that is compatible with $\Spec(\kappa(x))\to X$ such that the induced morphism $\Spec(\kappa(x)) \to \Spec(A) \times_{\Spec(k)} X$ is a closed immersion. Denote by $\cO_{\mathbf{x}}$ the preimage of the valuation ring $A$ in $\cO_{X,x}$. To give the algebraic properties of the ring $\cO_{\mathbf{x}}$, let us recall the following notions. Let $R$ be a ring we call \emph{semi-valuation} on $R$ a valuation $v : R \to \Gamma\cup\{0\}$ such that its support $\p_{v}$ contains the set of zero-divisors of $R$ and, for any $a,b\in R$ such that $v(a)\leq v(b)\neq 0$, then $b|a$. A ring equipped with a semi-valuation is called a \emph{semi-valuation ring}. Let $v$ be a semi-valuation on a ring $A$. We say that the local ring $A_{\p_{v}}$ is the \emph{semi-fraction ring} of $A$.

\begin{theorem}[\cite{Temkin11}, Corollary 3.4.7 and Theorem 3.5.2]
\label{th:ZR_Temkin}
We use the same notation as above. 
\begin{itemize}
	\item[(1)] For any $\mathbf{x}=(x,A,\phi)\in\mathrm{Val}(X/k)$, $\cO_{\mathbf{x}}$ is a semi-valuation ring.
	\item[(2)] There exists a bijective map $\psi: \mathrm{Val}(X/k) \to \ZR(X/k)$ inducing, for any $\mathbf{x}=(x,A,\phi)\in\mathrm{Val}(X/k)$, an isomorphism of local rings $\cO_{\ZR(X/k),\psi(\mathbf{x})} \cong \cO_{\mathbf{x}}$.
	\item[(3)] The natural map $\eta : X \to \ZR(X/k)$ is an injective morphism of locally ringed spaces and induces a sheaf $\cM_{\ZR(X/k)} :=  \eta_{\ast}\cO_{X}$. Any point $y\in \ZR(X/k)$ possesses a unique minimal generalisation $x\in \eta(X)$, we have an isomorphism $\cM_{\ZR(X/k),y}\cong \cO_{X,x}$ and $\cM_{\ZR(X/k),y}$ is the semi-fraction field of $\cO_{\ZR(X/k),y}$. Moreover, $\cO_{\ZR(X/k)}$ is a subsheaf of  $\cM_{\ZR(X/k)}$. 
\end{itemize}
\end{theorem}

Let us now briefly discuss some elementary functorial properties of algebraic Zariski-Riemann spaces. First, assume that $f:Y\to X$ is a dominant $K$-morphism between projective $K$-schemes. Then any projective sub-model of $X$ over $k$ is a projective sub-model of $Y$ over $k$. Therefore, we get a topologically surjective morphism of locally ringed spaces $\ZR(Y/k)\to \ZR(X/k)$. 

Now assume that $f:Y \to X$ is a closed immersion of projective $K$-schemes. Let $\mathbf{y}=(y,A,\phi)\in \ZR(Y/k)$, where $y\in Y$ is a scheme point, $A$ is a valuation ring of $\kappa(y)$ and $\phi : \Spec(A) \to \Spec(k)$ is a morphism that is compatible with $\Spec(\kappa(y))\to Y$ such that the induced morphism $\Spec(\kappa(y)) \to \Spec(A) \times_{\Spec(k)} Y$ is a closed immersion. Then the morphism $\Spec(\kappa(x)) \to \Spec(A) \times_{\Spec(k)} X$ induced by $\varphi$ is again a closed immersion, as it is the composition of the closed immersion $\Spec(\kappa(y)) \to \Spec(A) \times_{\Spec(k)} Y$ and pullback of the closed immersion $Y\to X$ along the projection $\Spec(A) \times_{\Spec(k)} X \to X$. Therefore we get an injective map $\ZR(Y/k)\to \ZR(X/k)$. Moreover, this map is continuous by (\cite{Temkin11}, Lemma 3.1.1). Alternatively, one can see the map $\ZR(Y/k)\to\ZR(X/k)$ by looking at projective sub-models as follows. Let $\cX$ be a projective sub-model of $X$ over $k$. Denote by $\cY$ the scheme-theoretic image of $Y$ in $\cX$, this is a projective sub-model of $Y$ over $k$. Moreover, using Segre embeddings, one can see that any projective sub-model of $Y$ over $k$ is dominated by a projective sub-model as above. Thus, the map $\ZR(Y/k)\to\ZR(X/k)$ is the projective limit of the closed immersions $\cY\to\cX$, where $\cX$ runs over the projective sub-models of $X$ over $k$ and $\cY$ denotes the scheme-theoretic closure of $Y$ in $\cY$. In particular, we see that the map $\ZR(Y/k)\to\ZR(X/k)$ is a morphism of locally ringed spaces. 

Consider an arbitrary $K$-morphism $f:Y\to X$ between projective $K$-schemes. Let $Y'$ denote the scheme-theoretic image of $Y$ in $X$. Then $f$ is the composition of a dominant morphism and a closed immersion. Using the two cases above, we get a morphism of locally ringed spaces $f_{\ZR}:\ZR(Y/k)\to\ZR(X/k)$.
 
Finally, assume that $A$ is a Prüfer domain with fraction field $K$ and that $k$ is the prime subring of $K$. Then we have seen that $\ZR(K/A)$ is a subspace of $\ZR(K/k)$ which identifies with $\Spec(A)$. More generally, using Temkin's interpretation of $\ZR(X/A)$ and $\ZR(X/k)$ as spaces of semi-valuations on $X$, we get a morphism of locally ringed spaces $\ZR(X/A)\to \ZR(X/k)$ that is topologically injective.

\subsubsection{Analytic Zariski-Riemann space}
\label{subsub:ZR_space_analytic}

Let us now give the analytic counterpart of the spaces introduced in \S \ref{subsub:Zariski-Riemann_space_algebraic}. Let $K$ be a field, $k$ be a subring of $K$ and $X$ be a projective $K$-scheme. Assume that $k$ is equipped with a norm $\|\cdot\|$ such that $(k,\|\cdot\|)$ is a geometric base ring. This implies that for any projective sub-model $\cX\in\cM_{X/k}$, we have an associated Berkovich analytic space $\cX^{\an}$. As a set, $\cX^{\an}$ is described as
\begin{align*}
\cX^{\an} = \{x=(p,|\cdot|_{x}) : p \in \cX \text{ and }|\cdot|_{x} \text{ is an absolute value on }\kappa(x) \text{ s.t. }|\cdot|_{x|k}\in \cM(k,\|\cdot\|)\}.
\end{align*}
Moreover, for any $x=(p,|\cdot|_{x})\in \cX^{\an}$, we denote by $\widehat{\kappa}(x)$ the completion of $\kappa(p)$ w.r.t. $|\cdot|_{x}$. It is proved in \cite{LemanissierPoineau24} that $\cX^{\an}$ is a compact Hausdorff topological space that comes equipped with a structure sheaf $\cO^{\an}_{\cX}$.

Define the \emph{analytic Zariski-Riemann space} of $X$ over $k$ as
\begin{align*}
\ZR(X/k)^{\an} := \displaystyle\varprojlim_{\cX \in \cM_{X/k}} \cX^{\an}.
\end{align*}
This is a locally ringed space whose structure sheaf is denoted by $\cO^{\an}_{\ZR(X/k)}$. Recall that 
\begin{align*}
\cO_{\ZR(X/k)^{\an}} = \displaystyle\varinjlim_{\cX \in \cM_{X/k}} (p^{\an}_{\cX})^{-1}\cO_{\cX^{\an}},
\end{align*}
where, for all $\cX\in\cM_{X/k}$, $p^{\an}_{\cX} : \ZR(X/k)^{\an}\to \cX^{\an}$ denotes the projection. Let $\mathbf{x}=(x_{\cX})_{\cX\in\cM_{X/k}}\in \ZR(X/k)^{\an}$. Then its residue field $\kappa(\mathbf{x})$ is described as the union of the residue fields $\kappa(x_{\cX})$, where $\cX$ runs over the projective submodels in $\cM_{X/k}$. It is thus a Henselian valued field (\cite{Poineau13a}, Théorème 5.1 and Corollaire 5.3 and \cite{stacks}, \href{https://stacks.math.columbia.edu/tag/04GI}{Lemma 04GI}). The completion of $\kappa(\mathbf{x})$ w.r.t. to its absolute value is called the \emph{completed residue field} of $\mathbf{x}$ and is denoted by $\widehat{\kappa}(\mathbf{x})$. Note that we have a commutative diagram of locally ringed spaces
\begin{center}
\begin{tikzcd}
\ZR^{}(X/k)^{\an} \arrow[d] \arrow[r] & \ZR(X/k) \arrow[d] \\
{\cM(k,\|\cdot\|)} \arrow[r]          & \Spec(k)          
\end{tikzcd}.
\end{center}

The analytic Zariski-Riemann space also admits a description in terms of pseudo absolute values. Let $\mathbf{x}\in\ZR(X/k)^{\an}$. Then the description given in \S \ref{subsub:Zariski-Riemann_space_algebraic} implies that $j_{X/k}(\mathbf{x})$ can be seen as a triple $(x,A,\varphi)$, where $x\in X$ is a scheme point, $A$ is a valuation ring of $\kappa(x)$ and $\varphi:\Spec(A)\to\Spec(k)$ is a morphism satisfying several properties. we call $x$ the \emph{underlying scheme point} of $\mathbf{x}$. Note that the residue field $\kappa(j_{X/k}(\mathbf{x}))$ is the residue field of $A$ and thus, we obtain the following description.

\begin{proposition}
\label{prop:ZR_analytic_description_Berkovich-like}
We have a bijection between $\ZR(X/k)^{\an}$ and 
\begin{align*}
\{\mathbf{x}=(\mathbf{p},|\cdot|_{\mathbf{x}}): \mathbf{p}\in\mathrm{Val}(X/k) \text{ and } |\cdot|_{\mathbf{x}}\in M_{\kappa(p)} \text{ with finiteness ring }A_{\mathbf{p}}\}.
\end{align*}
\end{proposition}

\begin{example}
\label{example:ZR_space_analytic}
\begin{itemize}
	\item[(1)] Assume that $X=\Spec(K)$. Let $k$ denote the prime subring of $K$ equipped with the norm $\|\cdot\|$ defined to be $|\cdot|_{\infty}$ if $\mathrm{char}(K)=0$ and the trivial norm if $\mathrm{char}(K)>0$. Then $\ZR(K)^{\an}:=\ZR(X/k)$ is homeomorphic to the space $M_{K}$ of pseudo-absolute values on $K$ (\cite{Sedillot_pav}, Corollary 10.2.4). 
	\item[(2)] Assume that $X=\Spec(K)$ and that $k$ is a field equipped with the trivial norm. Then $\ZR(K/k)^{\an}:=\ZR(K/k)^{\an}$ is homeomorphic to the space of pseudo-absolute values on $K$ that restrict to the trivial absolute value on $k$.
\end{itemize}
\end{example}

As in the algebraic case, we discuss two cases of functoriality. First, let $f:Y\to X$ be a $K$-morphism, where $Y$ is a projective $K$-scheme. Then the morphism $\ZR(Y/k) \to \ZR(X/k)$ described earlier induces a morphism of locally ringed spaces $f^{\an}:\ZR(Y/k)^{\an}\to\ZR(X/k)^{\an}$. This morphism is topologically surjective, resp. topologically closed and injective, if $f$ is dominant, resp. a closed immersion. Second, assume that $A$ is a Prüfer domain with fraction field $K$ and that $k$ is the prime subring of $K$. Assume that $A$ is equipped with a norm $\|\cdot\|$ such that $(A,\|\cdot\|)$ is an integral structure for $K$. Recall that $\ZR(K/A)^{\an}\cong\cM(A,\|\cdot\|)$ is a subspace of $\ZR(K)^{\an}\cong M_{K}$. More generally, $\ZR(X/A)^{\an}$ is a (closed) subspace of $\ZR(X)^{\an}$.

A variant of bullet (2) above is as follows. Assume that $k$ is a valuation ring such that $\Frac(k)\subset K$ with residue field $\kappa$. Let $|\cdot|$ be an absolute value on $\kappa$. This data determines a pseudo-absolute value $v$ on $\Frac(k)$. Denote by $\widehat{\kappa}$ the completion of $\kappa$ w.r.t. $|\cdot|$. Then 
	\begin{align*}
	\ZR(X/k)^{\an} := \displaystyle\varprojlim_{\cX \in \cM_{X/k}} (\cX\otimes_{k}\widehat{\kappa})^{\an}
	\end{align*}
	is the \emph{local Zariski-Riemann analytic space} associated with $X$ above $v$ defined in (\cite{Sedillot_pav}, \S 8.3). As above, we have a commutative diagram of locally ringed spaces
	\begin{center}
\begin{tikzcd}
\ZR^{}(X/k)^{\an} \arrow[d] \arrow[r] & \ZR(X/k) \arrow[d] \\
\{v\} \arrow[r]                       & \Spec(k)          
\end{tikzcd}.
	\end{center}
For any $x=(x_{\cX})_{\cX\in\cM_{X/k}}\in \ZR(X/k)^{\an}$, we define the \emph{completed residue} field as the completion of the union of the residue fields $\kappa(x_{\cX})$, where $\cX$ runs over $\cM_{X/k}$.

In the particular case where $X=\Spec(K)$, then $\ZR(X/k)^{\an}$ is homeomorphic to the set of pseudo-absolute values on $K$ extending $v$ on $\Frac(k)$ (\cite{Sedillot_pav}, Theorem 8.2.4). If we further assume that $K=\Frac(k)$, then $\ZR(X/k)^{\an}$ is simply the point $\{v\}\subset M_{K}$ whose residue field is $\widehat{\kappa}$. 

\begin{remark}
\label{rem:variant_ZR_analytic}
Although we use similar notation for global and local Zariski-Riemann spaces, the definitions do not coincide. The reason is the following. Assume that $k$ is the prime subring of a field $K$ and let $f:X\to \Spec(K)$ be a projective $K$-scheme. As it is customary in algebraic geometry, we want to consider the morphism $f^{\an}:\ZR(X)^{\an} \to M_{K}\cong \ZR(K/k)^{\an}$ as a family of fibres. Let $v\in M_{K}$. The fibre $(f^{\an})^{-1}(v) := \ZR(X)^{\an} \times_{M_{K}} \{v\}$ is isomorphic to the space $\ZR(X/A_{v})^{\an}$ (cf. \cite{LemanissierPoineau24}, Proposition 4.5.3).
\end{remark}

\subsubsection{Coherence of the structure sheaf and coherent modules: algebraic case}
\label{subsub:coherence_structure_sheaf_field_algebraic}

Let $K$ be a field and $k\subset K$ be a subring that is assumed to be stably coherent (e.g. $k$ is the prime subring or a Prüfer domain). Let $X$ be a projective integral $K$-scheme. Let us recall a well-known fact about finitely presented sheaves of $\cO_{\ZR(X/k)}$-modules on $\ZR(X/k)$.

\begin{proposition}
\label{prop:fp_sheaves_ZR_algebraic}
Let $U$ be a quasi-compact open subset of $\ZR(X/k)$, written as $U=\varprojlim_{\cX\in \cM_{U}}U_{\cX}$, where $\cM_{U}$ is cofinal in $\cM_{X/k}$ and the $U_{\cX}$'s are open subsets of the $\cX$'s. Then we have an equivalence of categories between $\varinjlim_{\cX\in\cM_{U}}\mathrm{Fp}(U_{\cX})$ and $\mathrm{Fp}(U)$. Moreover, this equivalence restricts to an equivalence of categories between $\varinjlim_{\cX\in\cM_{U}}\mathrm{Vb}(U_{\cX})$ and $\mathrm{Vb}(U)$.
\end{proposition}

\begin{proof}
Note that $\ZR(X/k)$ is coherent (i.e. quasi-compact, quasi-separated and admits a basis of quasi-compact open subsets) and sober by (\cite{Fujiwara-Kato}, Chapter 0, Theorem 2.2.10). Thus $U$ is coherent and sober as well as the $U_{\cX}$'s (\emph{loc. cit.}, Chapter 0, Proposition 2.2.9). Moreover, the transition maps between the $U_{\cX}$'s are quasi-compact. Therefore, the result follows from (\emph{loc. cit.}, Chapter 0, Theorem 4.2.2).
\end{proof}

We now prove that the structure sheaf of $\ZR(X/k)$ is coherent. 

\begin{proposition}
\label{prop:coherence_algeabraic_ZR_sheaf}
Let $U$ be a quasi-compact open subset of $\ZR(X/k)$, written as $U=\varprojlim_{\cX\in \cM_{U}}U_{\cX}$, where $\cM_{U}$ is cofinal in $\cM_{X/k}$ and the $U_{\cX}$'s are open subsets of the $\cX$'s. The sheaf $\cO_{U}$ is coherent. In particular, a sheaf of $\cO_{U}$-modules is coherent iff it is finitely presented. Moreover, any coherent sheaf on $U$ has Tor-dimension $\leq 1$.
\end{proposition}

\begin{proof}
It suffices to prove the result for $U=\ZR(X/k)$. Throughout the proof, we denote by $\cM$ the category of projective sub-models of $X$ over $k$ that are integral. Note that since $X$ is integral, $\cM$ is cofinal in $\cM_{X/k}$ and thus $\ZR(X/k)\cong\varprojlim_{\cX\in\cM}\cX$. We adapt a part of the proof of (\cite{Kerz17}, Proposition 6.4). Let us start by proving two lemmas. 

\begin{lemma}
\label{lemma:Tor-dimension_pullback}
\begin{itemize}
		\item[(i)] Let $\cX\in\cM$. Then for any $\cX' \in \cM$ with arrow $q:\cX'\to \cX$, the pullback functor $q^{\ast}$ preserves coherent modules of Tor-dimension $\leq 1$. Moreover, the restricted functor is exact.
		
		\item[(ii)] Let $\cX\in \cM$. Let $\varphi : \cF \to \cG$ be a morphism of coherent sheaves on $\cX$ such that $\cF,\cG$ and $\coker(\varphi)$ are of Tor-dimension $\leq 1$. Then for any morphism $q:\cX'\to \cX$, the natural maps
		\begin{align*}
		q^{\ast}(\ker(\varphi)) \to \ker(q^{\ast}\varphi), \quad q^{\ast}(\im(\varphi)) \to \im(q^{\ast}\varphi)
		\end{align*}
		are isomorphisms.
\end{itemize}
\end{lemma}

\begin{proof}
\textbf{(i):} Let $\cF$ be a coherent $\cO_{\cX}$-module of Tor-dimension $\leq 1$. Recall that $\cO_{\cX}$ is coherent since $k$ is stably coherent. Thus, there exists an exact sequence $\displaystyle 0\to\cE_{1}\xrightarrow{\varphi}\cE_{0}\to\cF\to 0$ with $\cE_{0},\cE_{1}$ vector bundles, of rank $n,m$ respectively.  Let $q:\cX'\to \cX$ be an arrow in $\cM$. By right-exactness of $q^{\ast}$, the sequence $0\to q^{\ast}\cE_{1}\xrightarrow{q^{\ast}\varphi} q^{\ast}\cE_{0} \to q^{\ast}\cF\to 0$ is exact at $q^{\ast}\cE_{0},q^{\ast}\cF$. Let us prove that $q^{\ast}\varphi$ is injective. Let $\eta',\eta$ be the generic point of $\cX',\cX$ respectively. Then $(q^{\ast}\varphi)_{\eta'}$ identifies with $\varphi_{\eta}$, and is hence injective. Therefore, for any $x'\in \cX'$, the map $(q^{\ast}\varphi)_{x'}:\cO_{\cX',x'}^{m}\to\cO_{\cX',x'}^{n}$ is injective since $\cO_{\cX',x'}$ embeds into $K(\cX')$. Hence $q^{\ast}\varphi$ is injective and $q^{\ast}\cF$ is of Tor-dimension $\leq 1$. The exactness of the restriction $q^{\ast}$ follows from the nine lemma. 

\textbf{(ii):} It is a direct consequence of (i). 
\end{proof}

\begin{lemma}
\label{lemma:Tor-dimension_flattening}
Let $\cX\in \cM$. 
\begin{itemize}
	\item[(i)] Let $\cF$ be a coherent $\cO_{\cX}$-module. Then there exists a map $q:\cX'\to \cX$ in $\cM$ with $\kappa(\cX')=\kappa(\cX)$ and $q^{\ast}\cF$ is of Tor-dimension $\leq 1$.
	\item[(ii)] Let $\varphi : \cF \to \cG$ be a morphism of coherent $\cO_{\cX}$-modules. Then there exists a map $q:\cX'\to \cX$ in $\cM$ with $\kappa(\cX')=\kappa(\cX)$ such that $q^{\ast}\cF,q^{\ast}\cG,\ker(q^{\ast}\varphi),\im(q^{\ast}\varphi)$ and $\coker(q^{\ast}\varphi)$ are of Tor-dimension $\leq 1$.
\end{itemize}
\end{lemma}

\begin{proof}
\textbf{(i):} By (\cite{stacks}, \href{https://stacks.math.columbia.edu/tag/0ESR}{Lemma 0ESR}), there exists a non-empty open subset $U\subset \cX$ and a $U$-modification $q : \cX'\to \cX$ such that $q^{\ast}\cF$ is of Tor-dimension $\leq 1$.

\textbf{(ii):} Note that for any morphism $q:\cX'\to \cX$, we have $q^{\ast}\coker(\varphi)\cong \coker(q^{\ast}\varphi)$ by right-exactness. By (i), there exists $q:\cX'\to \cX$ such that $q^{\ast}\cF,q^{\ast}\cG$ and $\coker(q^{\ast}\varphi)$ have Tor-dimension $\leq 1$. It follows that $\ker(q^{\ast}\varphi)$ and $\im(q^{\ast}\varphi)$ also have Tor-dimension $\leq 1$. 
\end{proof}

To prove that $\cO_{\ZR(X/k)}$ is coherent, it suffices to prove that for any open subset $U\subset \ZR(X/k)$ and for any morphism $f: \cO_{U}^{n}\to \cO_{U}$, the kernel $\ker(f)$ is of finite type. First assume that $U=\ZR(X/k)$. Let $f:\cO_{\ZR(X/k)}^{n}\to \cO_{\ZR(X/k)}$ be such a morphism. By Proposition \ref{prop:fp_sheaves_ZR_algebraic}, there exists $\cX\in\cM$ such that $f$ is induced by a morphism $\varphi:\cO_{\cX}^{n}\to \cO_{\cX}$. By Lemma \ref{lemma:Tor-dimension_flattening} (ii) there exists a morphism $q:\cX'\to \cX$ in $\cM$ such that $\ker(q^{\ast}\varphi),\im(q^{\ast}\varphi)$ and $\coker(q^{\ast}\varphi)$ have Tor-dimension $\leq 1$. Moreover, since $\cX'$ is coherent, $\ker(q^{\ast}\varphi)$ is of finite type, namely there exists a surjection $\cO_{\cX'}^{m}\to \ker(q^{\ast}\varphi)$. Note that $q^{\ast}\varphi$ also defines $f$ and by Lemma \ref{lemma:Tor-dimension_pullback} (ii), $\ker(q^{\ast}\varphi)$ defines $\ker(f)$, which is therefore of finite type. The case of a general $U$ is treated exactly the same way using Proposition \ref{prop:fp_sheaves_ZR_algebraic}.
\end{proof}

\begin{remark}
\label{rem:coherent_sheaves_Tor-dimension}
The fact that coherent sheaves on $\ZR(X/k)$ have Tor-dimension $\leq 1$ indicates that cohomologically speaking, algebraic Zariski-Riemann spaces behave like a smooth projective curve over a field. 
\end{remark}

\subsubsection{Finitely presented modules on analytic Zariski-Riemann spaces}
\label{subsub:coherence_struture_sheaf_GAGA_analytic}

We now want to give an analytic counterpart to the results of \ref{subsub:coherence_structure_sheaf_field_algebraic} in the analytic case. Let $K$ be a field and $k$ be a subring of $K$. Let $X$ be a projective $K$-scheme. Assume that $k$ is equipped with a norm $\|\cdot\|$ such that $(k,\|\cdot\|)$ is a geometric base ring. We first want to link the category of finitely presented sheaves on $\ZR(X/k)^{\an}=\varprojlim_{\cX\in\cM_{K/k}}\cX^{\an}$ to the inductive limit of the categories of finitely presented sheaves on the $\cX^{\an}$'s. Unfortunately, the results of (\cite{Fujiwara-Kato}, Chapter 0) do not apply and, since we did not find any suitable reference in the literature, we have to prove the analogue of Proposition \ref{prop:fp_sheaves_ZR_algebraic} directly.  

\begin{proposition}
\label{prop:fp_sheaves_compact_OX_mdoules}
Let $(X,\cO_{X})$ be a locally ringed space and let $\cF$ be an $\cO_{X}$-module of finite presentation. Let $(\cG_{i})_{i\in I}$ a filtered inductive system of $\cO_{X}$-modules. Then the canonical map of sheaves of $\cO_{X}$-modules 
\begin{align*}
\displaystyle\varinjlim_{i\in I}\Hom_{\cO_{X}}(\cF,\cG_{i}) \to \Hom_{\cO_{X}}(\cF,\varinjlim_{i\in I}\cG_{i})
\end{align*}
is an isomorphism.
\end{proposition}

\begin{proof}
It suffices to check that for any $x\in X$, the induced map on stalks
\begin{align*}
\left(\displaystyle\varinjlim_{i\in I}\Hom_{\cO_{X}}(\cF,\cG_{i})\right)_{x} \to \Hom_{\cO_{X}}(\cF,\varinjlim_{i\in I}\cG_{i})_{x}
\end{align*}
is an isomorphism. Let $x\in X$, we have isomorphisms
\begin{align*}
\left(\displaystyle\varinjlim_{i\in I}\Hom_{\cO_{X}}(\cF,\cG_{i})\right)_{x}&\cong\varinjlim_{i\in I}\Hom_{\cO_{X}}(\cF,\cG_{i})_{x} \cong \varinjlim_{i\in I}\Hom_{\cO_{X,x}}(\cO_{X,x},\cG_{i,x})\\ 
&\cong \Hom_{\cO_{X,x}}(\cO_{X,x},\varinjlim_{i\in I}\cG_{i,x}) \cong \Hom_{\cO_{X}}(\cF,\varinjlim_{i\in I}\cG_{i})_{x},
\end{align*}
where the second and the fourth isomorphisms are given by (\cite{Gortz10}, Proposition 7.27) and the third is given by the fact that compact objects in the category of modules over a ring are precisely the finitely presented ones. 
\end{proof}

Let us now consider the following setup. We consider a cofiltered projective system $(X_{i})_{i\in I}$ of locally compact Hausdorff topological spaces whose transition maps $p_{ij}:X_{j}\to X_{i}$ are surjective and proper. Consider the projective limit $X:=\varprojlim_{i\in I}X_{i}$. This is a locally compact Hausdorff space and the projections $p_{i}:X\to X_{i}$ are surjective and proper. 

\begin{proposition}
\label{prop:abstract_nonsense_sheaves_of_modules_projective_limit}
For any $i\in I$, let $\cF_{i}$ be a sheaf of $\cO_{X_{i}}$-modules. Assume that for any $i,j\in I$ such that $i\leq j$, we have a morphism $\varphi_{ij}:p_{ij}^{-1}\cF_{i}\to\cF_{j}$ such that $\varphi_{ik}=\varphi_{jk}\circ p_{jk}^{-1}\varphi_{ij}$ if $i\leq j \leq k$, so that $(p_{i}^{-1}\cF_{i})$ is an inductive system of sheaves on $X$. Denote $\cF:=\varinjlim_{i\in I}p_{i}^{-1}\cF_{i}$. 
\begin{itemize}
	\item[(1)] Let $i_{0}\in I$ and $K_{i_{0}}\subset X_{i_{0}}$ be a compact subset. For any $i\geq i_{0}$, denote $K_{i}:=p_{i_{0},i}^{-1}(K_{i_{0}})$, this is a compact subset of $K_{i}$, and $K:=\varprojlim_{i\geq i_{0}}K_{i}$, this is a compact subset of $X$. Then the canonical map
	\begin{align*}
	\displaystyle\varinjlim_{i\geq i_{0}}H^{0}(K_{i},\cF_{i}) \to H^{0}(K,\cF)
	\end{align*}
	is an isomorphism.
	\item[(2)] Let $i_{0}\in I$. Then the canonical morphism 
	\begin{align*}
	\displaystyle\varinjlim_{i\geq i_{0}}p_{i_{0}i,\ast}p_{i_{0}i}^{\ast}\cF_{i_{0}} \to p_{i_{0},\ast}p_{i_{0}}^{\ast}\cF_{i_{0}}
	\end{align*}
	is an isomorphism.
\end{itemize}
\end{proposition}

\begin{proof}
\textbf{(1):} We may assume that $I=\{j\geq i\}$. Since $K$ is compact, (\cite{Schapira23}, Proposition 11.1.2) yields an isomorphism
\begin{align*}
H^{0}(K,\cF) \cong \displaystyle\varinjlim_{i\in I} H^{0}(K,p_{i}^{-1}).
\end{align*}
Therefore it suffices to show the bijectivity of the map $\varinjlim_{i\geq i_{0}}H^{0}(K_{i},\cF_{i})\to\varinjlim_{i\in I} H^{0}(K,p_{i}^{-1})$. By (\cite{KS90}, Proposition 2.5.1), the canonical map
\begin{align*}
\displaystyle\varinjlim_{i\in I}\varinjlim_{U} H^{0}(U,\cF_{i}) \to \varinjlim_{i\in I}H^{0}(K_{i},\cF_{i}),
\end{align*}
where, for any $i\in I$, $U$ runs over the open neighbourhood of $K_{i}$ in $X_{i}$, is an isomorphism. 

Now let $s\in \varinjlim_{i\in I} H^{0}(K,p_{i}^{-1})$. By compactness of the $K_{i}$'s and $K$, there exist a finite compact covering $K=\bigcup_{k=1}^{n}V_{k}$ of $K$ and $i\in I$ such that: for any $k=1,...,n$, there exists a compact subset $V_{k,i}\subset K_{i}$ such that $V_{k}=p_{i}^{-1}(V_{k,i})$ and $s_{|V_{k}}$ lies in $\varinjlim_{U}H^{0}(U,\cF_{i})$, where $U$ runs over the open neighbourhood of $V_{k,i}=p_{i}(V_{k})$ in $X_{i}$. 

Thus, for any $k=1,...,n$, $s_{|V_{k}}$ define an element in $\varinjlim_{j\geq i}\varinjlim_{V}H^{0}(V,\cF_{j})$, where for any $j\geq i$, $V$ runs over the open neighbourhoods of $p_{j}(V_{k})=:V_{k,j}$ in $X_{j}$. By the previous isomorphism, this defines an element in $\varinjlim_{j\geq i}H^{0}(V_{k,j},\cF_{j})$.  Since the $V_{k,j}\cap V_{k',j}$'s are compact, the first isomorphism above ensures that we can glue these elements to obtain a unique element in $\varinjlim_{i\geq i_{0}}H^{0}(K_{i},\cF_{i})$ that is mapped to $s$ by construction. 

\textbf{(2):} Let $x\in X_{i_{0}}$. By (\cite{KS90}, Remark 2.5.3), we have isomorphisms
\begin{align*}
&\left(\displaystyle\varinjlim_{i\geq i_{0}}p_{i_{0}i,\ast}p_{i_{0}i}^{\ast}\cF_{i_{0}}\right)_{x}\cong \varinjlim_{i\geq i_{0}} \left(p_{i_{0}i,\ast}p_{i_{0}i}^{\ast}\cF_{i_{0}}\right)_{x} \cong \varinjlim_{i\geq i_{0}}H^{0}(p_{i_{0}i}^{-1}(x),p_{i_{0}i}^{\ast}\cF_{i_{0}}),\\
&\left(p_{i_{0},\ast}p_{i_{0}}^{\ast}\cF_{i_{0}}\right)_{x}\cong H^{0}(p_{i_{0}}^{-1}(x),p_{i_{0}}^{\ast}\cF_{i_{0}}).
\end{align*}
Now since $p_{i_{0}}^{-1}(x)\cong \varprojlim_{i\leq i_{0}}p_{i_{0}i}^{-1}(x)$, (1) yields an isomorphism
\begin{align*}
\displaystyle\varinjlim_{i\geq i_{0}}H^{0}(p_{i_{0}i}^{-1}(x),p_{i_{0}i}^{\ast}\cF_{i_{0}})\cong H^{0}(p_{i_{0}}^{-1}(x),p_{i_{0}}^{\ast}\cF_{i_{0}})
\end{align*}
which finishes the proof.
\end{proof}

\begin{proposition}
\label{prop:equivalence_categories_fp_sheaves}
Let $i_{0}\in I$ and $\cF_{i_{0}},\cG_{i_{0}}$ be two $\cO_{X_{i_{0}}}$-modules. If $\cF_{i_{0}}$ is finitely presented, then the canonical map
\begin{align*}
\displaystyle\varinjlim_{i\leq i_{0}}\Hom_{\cO_{X_{i}}}(p_{i_{0}j}^{\ast}\cF_{i_{0}},p_{i_{0}i}^{\ast}\cG_{i_{0}}) \to \Hom_{\cO_{X}}(p_{i_{0}}^{\ast}\cF_{i_{0}},p_{i_{0}}^{\ast}\cG_{i_{0}})
\end{align*}
is an isomorphism.
\end{proposition}

\begin{proof}
We have isomorphisms 
\begin{align*}
\Hom_{\cO_{X}}(p_{i_{0}}^{\ast}\cF_{i_{0}},p_{i_{0}}^{\ast}\cG_{i_{0}}) &\cong \Hom_{\cO_{X_{i_{0}}}}(\cF_{i_{0}},p_{i_{0},\ast}p_{i_{0}}^{\ast}\cG_{i_{0}}) \cong \Hom_{\cO_{X_{i_{0}}}}(\cF_{i_{0}},\displaystyle\varinjlim_{i\geq i_{0}}p_{i_{0}i,\ast}p_{i_{0}i}^{\ast}\cF_{i_{0}})\\
& \cong \varinjlim_{i\geq i_{0}}\Hom_{\cO_{X_{i_{0}}}}(\cF_{i_{0}},p_{i_{0}i,\ast}p_{i_{0}i}^{\ast}\cF_{i_{0}})\cong \varinjlim_{i\geq i_{0}}\Hom_{\cO_{X_{i}}}(p_{i_{0}i}^{\ast}\cF_{i_{0}},p_{i_{0}i}^{\ast}\cF_{i_{0}}),
\end{align*}
where the first and fourth ones are given by adjunction, the second one by Proposition \ref{prop:abstract_nonsense_sheaves_of_modules_projective_limit} (2) and the third one by Proposition \ref{prop:fp_sheaves_compact_OX_mdoules}.
\end{proof}

We can now prove the analogue of Proposition \ref{prop:fp_sheaves_ZR_algebraic}.

\begin{proposition}
\label{prop:fp_sheaves_ZR_analytic}
We have an equivalence of categories between $\varinjlim_{\cX\in\cM_{X/k}} \mathrm{Fp}(\cX^{\an})$ and the category $\mathrm{Fp(\ZR(X/k)^{\an})}$. Moreover, this equivalence of categories restricts to an equivalence of categories between $\varinjlim_{\cX\in\cM_{X/k}}\mathrm{Vb}(\cX^{\an})$ and $\mathrm{Vb}(\ZR(X/k)^{\an})$.
\end{proposition}

\begin{proof}
Concerning the first part of the statement, it suffices to prove the essential surjectivity, since the full faithfulness will follow from Proposition \ref{prop:equivalence_categories_fp_sheaves}. Denote $\cM:=\cM_{X/k}$ and $X^{\an}:=\ZR(X/k)^{\an}$. Note that the transition maps between the $\cX^{\an}$'s are proper and surjective. Let $\cF$ be a finitely presented $\cO_{X^{\an}}$-module. There exists a finite open covering $X^{\an}=\bigcup_{k=1}^{n}U_{k}$ such that, for any $k=1,...,n$, there exists a presentation
\begin{align*}
\cO_{U_{k}}^{p_{k}} \to \cO_{U_{k}}^{q_{k}}\to \cF_{|U_{k}}\to 0.
\end{align*}
Write each $U_{k}$ as $\varprojlim_{\cX\in\cM_{k}}U_{k,\cX}$, where $\cM_{k}$ is cofinal in $\cM$ and the $U_{k,\cX}$'s are open subsets of the $\cX^{\an}$'s. The transition maps between the $U_{k,\cX}$'s (for a fixed $k$) are proper and can be assumed to be surjective. By Proposition \ref{prop:equivalence_categories_fp_sheaves}, for any $k=1,...,n$, there exist $\cX_{k}\in \cM_{k}$ and a morphism $\cO_{U_{\cX_{k}}}^{p} \to \cO_{U_{\cX_{k}}}^{q}$ inducing the morphism $\cO_{U_{k}}^{p} \to \cO_{U_{k}}^{q}$. Then $\cF_{k}:=\coker(\cO_{U_{\cX_{k}}}^{p} \to \cO_{U_{\cX_{k}}}^{q})$ is a finitely presented $\cO_{U_{\cX_{k}}}$-module whose pullback to $U_{k}$ is isomorphic to $\cF_{|U_{k}}$. Since we have only finitely many $U_{k}$'s, up to pulling back on a bigger sub-model, we may assume that $\cX_{k}=\cX_{k'}=:\cX$ for all $k,k'=1,...,n$. Since the $(p_{\cX}^{\an})^{\ast}\cF_{k}$'s glue on $X^{\an}$, Proposition \ref{prop:equivalence_categories_fp_sheaves} implies that there exists a morphism $q:\cY\to \cX$ in $\cM$ such that the $q^{\ast}\cF_{k}$'s glue on $\cY$, yielding an $\cO_{\cY^{\an}}$-module $\cF_{\cY}$. By construction, $\cF_{\cY}$ is finitely presented and $(p_{\cY}^{\an})^{\ast}\cF_{\cY}\cong\cF$.

Finally, let $\cE$ be a vector bundle on $X^{\an}$. Choose a finite open covering $X^{\an}=\bigcup_{k=1}^{n}U_{k}$ such that $\cE_{|U_{k}}\cong \cO_{U_{k}}^{r}$ for all $k=1,...,n$, where $r$ is an integer. Let $\cX\in\cM$ and $\cE_{\cX}$ be a finitely presented $\cO_{\cX^{\an}}$-module such that $\cE\cong (p_{\cX}^{\an})^{\ast}\cE_{\cX}$. By Proposition \ref{prop:equivalence_categories_fp_sheaves}, there exist an arrow $q:\cY\to\cX$ in $\cM$ and an open covering $\cY=\bigcup_{k=1}^{n}V_{k}$ such that for any $k=1,...,r$, the isomorphism $\cE_{|U_{k}}\cong \cO_{U_{k}}^{r}$ is the pullback of a morphism $((q^{\an})^{\ast}\cE_{\cX})_{|V_{k}}\to\cO_{V_{k}}^{r}$. Since the morphism induced by $\cE_{|U_{k}}\cong \cO_{U_{k}}^{r}$ on the stalks is an isomorphism and using the surjectivity of $p_{\cY}$, we deduce that the morphism $((q^{\an})^{\ast}\cE_{\cX})_{|V_{k}}\to\cO_{V_{k}}^{r}$ is an isomorphism. Then the $((q^{\an})^{\ast}\cE_{\cX})_{|V_{k}}$'s glue to get a vector bundle on $\cY$ whose pullback to $X^{\an}$ is isomorphic to $\cE$.
\end{proof}

Under certain conditions, we are able to prove that finitely presented sheaves on $\ZR(X/k)^{\an}$ are in fact algebraic. This will rely on the following "GAGA type" result, whose proof reproduces the argument of Poineau in (\cite{Poineau25}, Theorem 2.43). 

\begin{proposition}
\label{prop:GAGA}
Assume that $k$ is a Prüfer domain equipped with a norm $\|\cdot\|$ such that $|\cdot|_{\triv}\leq \|\cdot\|$ and $(k,\|\cdot\|)$ is a geometric base ring and the specification morphism $j_{\Spec(k)} : \cM(k,\|\cdot\|)\to\Spec(k)$ is flat and surjective. Let $\cX$ be a scheme locally of finite type over $k$ (i.e. locally of finite presentation since $k$ is stably coherent). For any coherent sheaf $\cF$ on $\cX$, denote by $\cF^{\an}:=j_{\cX}^{\ast}\cF$, this is a coherent sheaf on $\cX^{\an}$ by (\cite{LemanissierPoineau24}, Lemme 6.5.6). The following assertions hold.
\begin{itemize}
	\item[(1)] The specification morphism $j_{\cX}:\cX^{\an}\to\cX$ is flat.
	\item[(2)] $\cF\mapsto \cF^{\an}$ defines an equivalence of categories between the categories of coherent sheaves on $\cX$ and $\cX^{\an}$.
	\item[(3)] The analytification functor above restricts to an equivalence of categories between the categories of locally free sheaves on $\cX$ and $\cX^{\an}$.
\end{itemize}
\end{proposition}

\begin{proof}
\textbf{(1):} Flatness of $j_{\Spec(k)}$ allows us to adapt the proof of (\cite{LemanissierPoineau24}, Proposition 6.6.4) to obtain, for any integer $n$, the flatness of $j_{\mathbb{A}^{n}_{k}} : \mathbb{A}^{n,\an}_{k} \to \mathbb{A}^{n}_{k}$. One can then argue as in the proof of (\emph{loc. cit.}, Théorème 6.6.5).

\textbf{(2):}  By (\cite{LemanissierPoineau24}, Corollaire 6.6.7), the analytification functor is exact and faithful. By definition of the structure sheaf on $\cM(k,\|\cdot\|)$ and since $|\cdot|_{\triv}\leq \|\cdot\|$, we have inclusions $k\subset \cO(\cM(k,\|\cdot\|))\subset \Frac(k)$. Since $j_{\Spec(k)} : \cM(k,\|\cdot\|)\to\Spec(k)$ is surjective, we have $\cO(\cM(k,\|\cdot\|))\subset \bigcap_{\m\in\Spec(k)}k_{\m}=k$, thus $\cO(\cM(k,\|\cdot\|))=k$. Now using (\cite{Poineau13a}, Corollaire 2.8), we get that, for any integer $n$, $\cO(\mathbb{A}^{n,\an}_{k})=k[T_{1},...,T_{n}]$. Then one can reproduce the proof of (\cite{Poineau25}, Theorem 2.43) to get (2).

\textbf{(3):} This follows from (\cite{LemanissierPoineau24}, Proposition 6.6.9), whose proof remains valid by flatness of $j_{\cX}:\cX^{\an}\to\cX$.
\end{proof}

\begin{proposition}
\label{prop:GAGA_ZR_sheaf}
Assume that $k$ is a Prüfer domain equipped with a norm $\|\cdot\|$ such that $|\cdot|_{\triv}\leq \|\cdot\|$ and $(k,\|\cdot\|)$ is a geometric base ring and the specification morphism $j_{\Spec(k)} : \cM(k,\|\cdot\|)\to\Spec(k)$ is flat and surjective. For any $\cO_{\ZR(K/k)}$-module $\cF$, denote $\cF^{\an}:=j_{X/k}^{\ast}\cF$, where $j_{X/k}:\ZR(X/k)^{\an}\to\ZR(X/k)$ is the specification morphism. The following assertions hold.
\begin{itemize}
	\item[(1)] $\cF\mapsto \cF^{\an}$ defines an equivalence of categories between the categories of finitely presented sheaves on $\ZR(X/k)$ and $\ZR(X/k)^{\an}$.
	\item[(2)] The analytification functor above restricts to an equivalence of categories between the categories of locally free sheaves on $\ZR(X/k)$ and $\ZR(X/k)^{\an}$.
\end{itemize}
\end{proposition}

\begin{proof}
Throughout the proof, we denote $\cM:=\cM_{X/k}$ and $X^{\an}:=\ZR(K/k)^{\an}$. 

\textbf{(1):} For any finitely presented sheaf $\cF$ on $\ZR(X/k)$, the sheaf $\cF^{\an}$ is finitely presented since finitely presented sheaves are preserved by pullback. Now by Propositions \ref{prop:fp_sheaves_ZR_algebraic} and \ref{prop:fp_sheaves_ZR_analytic}: one between $\varinjlim_{\cX\in\cM}\mathrm{Coh}(\cX)=\varinjlim_{\cX\in\cM}\mathrm{Fp}(\cX)$ and $\mathrm{Coh}(\ZR(X/k))=\mathrm{Fp}(\ZR(X/k))$ and one between $\varinjlim_{\cX\in\cM}\mathrm{Fp}(\cX^{\an})$ and $\mathrm{Fp}(\ZR(X/k)^{\an})$. Now the analytification functor $\mathrm{Fp}(\ZR(X/k))\to \mathrm{Fp}(\ZR(X/k)^{\an})$ is the direct limit of the analytification functors from Proposition \ref{prop:GAGA}. Since the latter are equivalences of categories, we deduce the desired equivalence of categories.

\textbf{Proof of (3).} By Propositions \ref{prop:fp_sheaves_ZR_algebraic} and \ref{prop:fp_sheaves_ZR_analytic} again, the categories of locally free sheaves on $\ZR(X/k)$ and $\ZR(X/k)^{\an}$ are respectively equivalent to the direct limit of the corresponding categories of locally free sheaves on the $\cX$'s and $\cX^{\an}$'s. Now using Proposition \ref{prop:GAGA} (3), we get the desired equivalence of categories.
\end{proof}

\begin{remark}
\label{rem:coherence_strcture_sheaf_ZR_analytic}
We do not know if the structure sheaf $\cO_{\ZR(K/k)^{\an}}$ is coherent. It would be a key step to study coherent sheaves on analytic Zariski-Riemann spaces. However, for our applications in Arakelov geometry presented in this article, we will mainly care about vector bundles for which we have a complete description.
\end{remark}

\begin{example}
\label{example:GAGA}
We made several hypotheses that condition the use of Propositions \ref{prop:GAGA}-\ref{prop:GAGA_ZR_sheaf}. Let us give the two cases of applications we will be concerned with in this article. 
\begin{itemize}
	\item[(1)] First assume that $k$ is the prime subring of $K$ equipped with the norm $\|\cdot\|$ defined to be $|\cdot|_{\infty}$ if $\mathrm{char}(K)=0$ and the trivial norm if $\mathrm{char}(K)>0$. Then Propositions \ref{prop:GAGA}-\ref{prop:GAGA_ZR_sheaf} hold. 
	\item[(2)] Let $R>0$ and $(k,\|\cdot\|)=(A_{R},\|\cdot\|_{R,\hyb})$ from Proposition \ref{prop:global_space_of_sav_compact_disc}. Then Propositions \ref{prop:GAGA}-\ref{prop:GAGA_ZR_sheaf} hold. Indeed the morphism $\cM(k,\|\cdot\|)\to\Spec(A_{R})$ is flat and surjective, this follows from the explicit description of the stalks of $\cO_{\cM(k,\|\cdot\|)}$ given in (\cite{Sedillot_pav}, Proposition 9.4.10).
\end{itemize}
\end{example}

\subsubsection{(Metrised) coherent sheaves on Zariski-Riemann spaces}
\label{subsub:metrised_vector_bundle_ZR_space}

Let $K$ be a field, $k$ be an integral subdomain of $K$ and $X$ be an integral projective $K$-scheme. Assume that either $k$ is the finiteness ring of a pseudo-absolute value $v\in M_{K}$ or $k$ is equipped with a norm $\|\cdot\|$ making it a geometric base ring. In both case, we have a morphism of locally ringed spaces $j_{X/k} : \ZR(X/k)^{\an}\to \ZR(X/k)$. Recall that the structure sheaf of $\ZR(X/k)$ is coherent and thus coherent sheaves on $\ZR(X/k)$ coincide with finitely presented sheaves.

Let $\cE$ be a coherent sheaf on $\ZR(X/k)$. By a \emph{metric} on $\cE$, we mean a family $\varphi=(|\cdot|_{\varphi}(\mathbf{x}))_{\mathbf{x}\in \ZR(X/k)^{\an}}$ where, for any $\mathbf{x}\in \ZR(X/k)^{\an}$, $|\cdot|_{\varphi}(\mathbf{x})$ is a norm on the $\widehat{\kappa}(\mathbf{x})$-vector space $\cE(\mathbf{x}):=\cE\otimes_{\cO_{\ZR(X/k)}} \widehat{\kappa}(x)$. We further assume that the norms appearing in the metrics are ultrametric over non-Archimedean absolute values.

Let $\varphi=(|\cdot|_{\varphi}(x))_{x\in \ZR(X/k)^{\an}}$ be a metric on $\cE$. We say that $\varphi$ is \emph{lsc/usc/continuous} if, for any open subset $U\subset \ZR(X/k)$ and for any $s\in H^{0}(U,\cE)$, the map
\begin{align*}
\fonction{|s|_{\varphi}(\cdot)}{U^{\an}}{\bR_{\geq 0}}{x}{|s(x)|_{\varphi}(x)}
\end{align*}
is lsc/usc/continuous, where $U^{\an}:=j_{X/k}^{-1}(U)$.

\begin{definition}
\label{def:metrised_vector_bundle_Zariski-Riemann}
A \emph{metrised vector bundle} on $\ZR(X/k)$ is the data $\overline{\cE}=(\cE,\varphi)$ where $\cE$ is a vector bundle on $\ZR(X/k)$ and $\varphi$ is a metric on $\cE$. In that case, $\overline{\cE}$ is called an \emph{lsc/usc/continuous} metrised vector bundle if $\varphi$ is lsc/usc/continuous.
\end{definition} 

\begin{proposition-definition}
\label{prop-def:_metrised_vector_bundles_constructions}
Let $\overline{\cE}=(\cE,\varphi=(|\cdot|_{\varphi}(\mathbf{x}))_{\mathbf{x}\in\ZR(X/k)^{\an}})$ be a metrised vector bundle on $\ZR(X/k)$.
\begin{itemize}
	\item[(1)] Let $\cF$ be a vector subbundle of $\cE$. Let $\mathbf{x}\in\ZR(X/k)^{\an}$. Then the map $\cF(\mathbf{x})\to\cE(\mathbf{x})$ is injective and the norm $|\cdot|_{\varphi}(\mathbf{x})$ induces a norm $|\cdot|_{\varphi_{\cF}}(\mathbf{x})$ on $\cF(\mathbf{x})$. Then $\varphi_{\cF}:=(|\cdot|_{\varphi_{\cF}}(\mathbf{x}))_{\mathbf{x\in\ZR(X/k)^{\an}}}$ is a metric on $\cF$ called the \emph{restriction} of $\varphi$ to $\cF$. Moreover, $\varphi_{\cF}$ is usc/lsc/continuous if so is $\varphi$.
	\item[(2)] Let $q:\cE\to\cG$ be a surjective homomorphism between vector bundles on $\ZR(X/k)$. Let $\mathbf{x}\in\ZR(X/k)^{\an}$. Then the map $\cE(\mathbf{x})\to\cG(\mathbf{x})$ is surjective and the norm $|\cdot|_{\varphi}(\mathbf{x})$ induces by quotient a norm $|\cdot|_{\varphi_{\cG}}(\mathbf{x})$ on $\cG(\mathbf{x})$. Then $\varphi_{\cG}:=(|\cdot|_{\varphi_{\cG}}(\mathbf{x}))_{\mathbf{x\in\ZR(X/k)^{\an}}}$ is a metric on $\cG$ called the \emph{quotient metric} induced by $\varphi$ on $\cG$. Moreover, $\varphi_{\cG}$ is usc if so is $\varphi$.
	\item[(3)] Let $\mathbf{x}\in\ZR(X/k)^{\an}$. Denote by $|\cdot|_{-\varphi}(\mathbf{x})$ the dual norm induced by $|\cdot|_{\varphi}(\mathbf{x})$ of $|\cdot|_{\varphi}(\mathbf{x})$ on $(-\cE)(\mathbf{x})$. Then $-\varphi=(|\cdot|_{-\varphi}(\mathbf{x}))_{\mathbf{x}\in\ZR(X/k)^{\an}}$ is a metric on $-\cE$ called the \emph{dual metric} of $\varphi$. Moreover, if $\varphi$ is usc, then $-\varphi$ is lsc. 
	\item[(4)] Let $q:\cE\to\cG$ be a surjective homomorphism between vector bundles on $\ZR(X/k)$. Then the dual morphism $-q : -\cG \to -\cE$ is injective and we have
	\begin{align*}
	-\varphi_{\cG} = (-\varphi)_{-\cG}.
	\end{align*}
	\item[(5)] Let $f:Y\to X$ be a projective morphism between projective integral $K$-schemes inducing a morphism $f^{\an}:\ZR(Y/k)^{\an}\to\ZR(X/k)^{\an}$. Let $\mathbf{y}\in\ZR(Y/k)^{\an}$ with $\mathbf{x}:=f^{\an}(\mathbf{x})$. Define $|\cdot|_{f^{\ast}\varphi}(\mathbf{y})$ to be the norm on $(f^{\ast}\cE)(\mathbf{y})\cong\cE(\mathbf{x})\otimes_{\widehat{\kappa}(\mathbf{x})}\widehat{\kappa}(\mathbf{y})$ defined as the $\pi$-extension of scalars of $|\cdot|_{\varphi}(\mathbf{x})$ if the absolute value on $\widehat{\kappa}(\mathbf{y})$ is Archimedean and as the $\epsilon$-extension of scalars of $|\cdot|_{\varphi}(\mathbf{x})$ if the absolute value on $\widehat{\kappa}(\mathbf{y})$ is non-Archimedean. Then $f^{\ast}\varphi:=(|\cdot|_{f^{\ast}\varphi}(\mathbf{y}))_{\mathbf{y}\in\ZR(X/k)^{\an}}$ is a metric on $f^{\ast}\cE$ called the \emph{pullback} of $\varphi$. Moreover, $f^{\ast}\varphi$ is usc if $\varphi$ is usc.
\end{itemize}
\end{proposition-definition}

\begin{proof}
The fact that the constructions define metrics being clear and since (4) follows from (\cite{ChenMori}, Proposition 1.1.20), we only show the assertions about the regularity. For (1) and (3), this is clear. For (5), this follows from the fact that the extension of scalars can be expressed as an infimum of functions that are usc if $\varphi$ is usc.

Let us prove (2). Let $\cX\in \cM_{X/k}$. Then by (\cite{Fujiwara-Kato}, Chapter 0, Theorem 4.2.1), the morphism $p:\cE\to\cG$ is the pullback of a surjective homomorphism $q_{\cX}:\cE_{\cX}\to\cG_{\cX}$, where $\cE_{\cX},\cG_{\cX}$ are vector bundles on some projective submodel $\cX$. Moreover, by (\emph{loc. cit.}, Chapter 0, Theorem 2.2.13), the projection $p_{\cX}:\ZR(X/k)\to \cX$ is closed. Thus $\ZR(X/k)$ admits a basis of neighbourhood consisting of inverse image of open affine subschemes of $\cX$. On such an open subset $U=p_{\cX}^{-1}(U_{\cX})$ of $\ZR(X/k)$, by (\emph{loc. cit.}, Chapter 0, Proposition 4.4.1), for any $i\geq 0$, we have an isomorphism
\begin{align*}
H^{i}(U,\ker(q)) \cong \displaystyle\varinjlim_{\alpha:\cX'\to\cX}H^{i}(\alpha^{-1}(U_{\cX}),\ker(\alpha^{\ast}p_{\cX})).
\end{align*} 
Up to passing to a dominating submodel, we may assume that $\ker(\alpha^{\ast}p_{\cX})\cong\alpha^{\ast}\ker(p_{\cX})$ for any arrow $\alpha:\cX'\to\cX$ in $\cM_{X/k}$ (cf. Lemma \ref{lemma:Tor-dimension_pullback}). Moreover, by Stein factorisation, we can factor any such arrow $\alpha:\cX'\to\cX$ as the composition of $\beta:\cX'\to\cY$ and $\gamma:\cY\to \cX$, where $\beta_{\ast}\cO_{\cX}=\cO_{\cY}$ and $\gamma$ is finite. Therefore, for any $i\geq 0$, we have isomorphisms
\begin{align*}
H^{i}(\alpha^{-1}(U_{\cX}),\alpha^{\ast}\ker(p_{\cX})) \cong H^{i}(\gamma^{-1}(U_{\cX}),\gamma^{\ast}\ker(p_{\cX})). 
\end{align*}
Since $\gamma$ is finite, it is affine and $\gamma^{-1}(U_{\cX})$ is affine. Since $\gamma^{\ast}\ker(p_{\cX})$ is coherent, we have $H^{i}(\gamma^{-1}(U_{\cX}),\gamma^{\ast}\ker(p_{\cX}))=0$ for all $i\geq 1$. Therefore, for any $i\geq 1$, we have $H^{i}(U,\ker(q))=0$. We deduce that we have a surjective map $q(U):H^{0}(U,\cE)\to H^{0}(U,\cG)$ and for any $\mathbf(x)$ and any section $s\in H^{0}(U,\cG)$, we have
\begin{align*}
|s|_{\varphi_{G}}(\mathbf{x}) = \displaystyle\inf_{t\in H^{0}(U,\cE), q(U)(t)=s} |t|_{\varphi}(\mathbf{x}).
\end{align*}
The assertion follows. 
\end{proof} 

\begin{example}
\label{example:metrised_vector_bundle_ZR}
\begin{itemize}
	\item[(1)] Assume that $k$ is the finiteness ring of a pseudo-absolute value $v\in M_{K}$. 
		\begin{itemize}
			\item[(i)] When $X=\Spec(k)$, $\ZR(X/k)\cong \Spec(k)$ and a metrised vector bundle on $\ZR(X/k)$ is the data of a free $k$-module of finite rank $\cE$ together with a norm on $\cE \otimes_{k}\widehat{\kappa_{v}}$, i.e. a pseudo-normed $K$-vector space in $v$ (cf. \S \ref{subsub:pseudo-norms}).
			\item[(ii)] In general, if $\overline{L}$ is a metrised line bundle on $\ZR(X/k)$, its pullback to $X$ via $\eta : X \to \ZR(X/k)$ will be defined as a \emph{pseudo-metrised} line bundle on $X$ (cf. Definition \ref{def:pseudo-metric} and Proposition \ref{prop:ZR_interpretation_local_pseudo-metrics}).
		\end{itemize}			
		\item[(2)] Assume that $k$ is the prime subring of $K$ equipped with the norm $|\cdot|_{\infty}$ if $k=\bZ$ and the trivial norm if $k$ is a finite field. 
			\begin{itemize}
				\item[(i)] When $X=\Spec(K)$, a metrised vector bundle on $\ZR(X/k)$ is the data of a finite-dimensional $K$-vector space equipped with a family of pseudo-norms $(\|\cdot\|_{v})_{v\in M_{K}}$. We will study these objects in more detail in \S \ref{sec:pseudo-norm_families}.
				\item[(ii)] In general, if $\overline{L}$ is a metrised line bundle on $\ZR(X/k)$, its pullback to $X$ via $\eta : X \to \ZR(X/k)$ will be defined as a line bundle on $X$ equipped with a \emph{pseudo-metric family} (cf. Definition \ref{def:pseudo-metric_family})
			\end{itemize}
		\item[(3)] Assume that $k$ is equipped with a norm $\|\cdot\|$ such that $(k,\|\cdot\|)$ is an integral structure for $K$. Metrised line bundles on $\ZR(X/k)$ are the object of \S \ref{sub:pseudo-metric_families_integral_tac}.
\end{itemize}
\end{example}

\subsection{Adelic curves}
\label{sub:reminder_adelic_curves}

\subsubsection{Adelic structures over a field}
\label{subsub:adelic_curves}

An \emph{adelic curve} is the data $S=(K,(\Omega,\cA,\nu),(\va_{\omega})_{\omega\in \Omega})$ where $K$ is a field, $(\Omega,\cA,\nu)$ is a measure space and $(\va_{\omega})_{\omega\in \Omega}$ is a family of absolute values on $K$ satisfying the following condition:
\begin{align*}
\forall a \in K^{\times}, \quad (\omega\in\Omega) \mapsto \log|a|_{\omega} \in \bR
\end{align*} 
is $\cA$-measurable and $\nu$-integrable. The adelic curve $S$ is called \emph{proper} if the product formula
\begin{align*}
\forall a\in K^{\times},\quad \int_{\Omega}\log|a|_{\omega} \nu(\diff\omega) = 0 
\end{align*}
holds. 

In that case, for any integer $n\geq 1$ and for any $(a_1,...,a_n)\in K^{n}\setminus \{0\}$, we define the \emph{height} 
\begin{align*}
h_{S}(a_1,...,a_n) := \int_{\Omega} \log\max\{|a_1|_{\omega},...,|a_n|_{\omega}\}\nu(\diff\omega). 
\end{align*}

Let $S=(K,(\Omega,\cA,\nu),(\va_{\omega})_{\omega\in \Omega})$ and $S'=(K',(\Omega',\cA',\nu'),(\va_{\omega'})_{\omega'\in \Omega'})$ be two adelic curves. A \emph{morphism} $\alpha : S' \to S$ of adelic curves is a triplet $\alpha=(\alpha^{\sharp},\alpha_{\sharp},I_{\alpha})$, where
\begin{itemize}
	\item $\alpha^{\sharp} : K \to K'$ is a field extension;
	\item $\alpha_{\sharp} : (\Omega',\cA') \to (\Omega,\cA)$ is a measurable map such that 
	\begin{align*}
	\forall \omega'\in \Omega', \quad \forall a\in K,\quad |\alpha^{\sharp}(a)|_{\omega'} = |a|_{\alpha_{\sharp}(\omega')}.
	\end{align*}
	Moreover, the direct image of $\nu'$ by $\alpha_{\sharp}$ is assumed to be equal to $\nu$, namely, for any $f\in \cL^{1}(\Omega,\nu)$, we have
\begin{align*}
\int_{\Omega}f\diff\nu = \int_{\Omega'} f\circ \alpha_{\sharp}\diff\nu';
\end{align*}
	\item  $I_{\alpha} : \cL^1(\Omega',\cA',\nu') \rightarrow \cL^1(\Omega,\cA,\nu)$ is a disintegration kernel of $\alpha_{\sharp}$, namely $I_{\alpha}$ is a linear map such that, for all $g\in \cL^{1}(\Omega',\cA',\nu')$, we have
	\begin{align*}
	\int_{\Omega} I_{\alpha}(g)\diff\nu = \int_{\Omega'} g\diff\nu',
	\end{align*}
and which, for all $f \in \cL^{1}(\Omega,\cA,\nu)$, sends the equivalence class of $f\circ \alpha_{\sharp}$ to the class of $f$.
\end{itemize}

\subsubsection{Algebraic coverings of adelic curves}

Let $S=(K,(\Omega,\cA,\nu),(\va_{\omega})_{\omega\in \Omega})$ be an adelic curve. Let $K'/K$ be an algebraic extension. Then it is possible to define an adelic curve $S\otimes_K K'$ and a morphism $\alpha_{K'/K} : S\otimes_{K} K' \to S$. Moreover, if $S$ is proper, then $S\otimes_{K}K'$ is also proper. 

\subsubsection{Examples of adelic curves}

On a field, there are many possible adelic structures. We list some examples.

\begin{itemize}
	\item Any global field (i.e. number field or function field) can be equipped with a natural adelic structure which corresponds to the classical way to perform arithmetic geometry over these fields (\cite{ChenMori}, \S 3.2.1-3.2.2).
	\item Given any field $K$ and any measure space $(\Omega,\cA,\nu)$, we can consider the proper adelic curve $S=(K,(\Omega,\cA,\nu),(|\cdot|_{\triv})_{\omega\in\Omega})$, where the family $(|\cdot|_{\triv})_{\omega\in\Omega}$ consists of copies of the trivial absolute on $K$. This kind of adelic curve is of particular importance when considering the geometry of numbers and appears naturally when considering Harder-Narasimhan filtrations of adelic vector bundles (\cite{ChenMori24}).
	\item Function fields of polarised varieties and arithmetic varieties can be endowed with proper adelic structures (\cite{ChenMori}, \S 3.2.4-3.2.6). This gives a unified approach to higher dimensional of classical arithmetic geometry over global fields. 
	\item Let $K$ be a countable field of characteristic zero. In (\cite{ChenMori21}, \S 2.7), it is proved that there exists an adelic curve $S=(K,(\Omega,\cA,\nu),(\va_{\omega})_{\omega\in \Omega})$ satisfying the following properties:
	\begin{itemize}
		\item[(1)] $S$ is proper;
		\item[(2)] for any $\omega\in\Omega$, $\va_\omega$ is a non-trivial absolute value on $K$;
		\item[(3)] the set $\{\omega\in\Omega : \va_{\omega} \text{ is non-Archimedean}\}$ is infinite and countable;
		\item[(4)] let $\overline{K}$ be an algebraic closure of $K$, for any subfield $E_0\subset \overline{K}$ which is finitely generated over $\bQ$, then the set
		\begin{align*}
		\{a\in \overline{K} : h_{S\otimes_{K} \overline{K}}(1,a) \leq C \text{ and } [K_{0}(a):K_0] \leq \delta\}
		\end{align*}
		is finite for all $C\in \bR_{\geq 0}$ and $\delta\in \bZ_{\geq 1}$.
	\end{itemize}
	
\end{itemize}

\subsection{Globally valued fields}
\label{sub:GVF}

In this subsection, we recall the essential definitions in the theory of globally valued fields introduced in \cite{GVF24}. These objects can be described as an "axiomatisation of heights" and are similar to adelic curves. 

Let $K$ be a field and $e\in \bR_{\geq 0}$. By a \emph{height} with \emph{Archimedean error} $e$ on $K$, we mean a function $h: \mathbb{A}(K):=\bigsqcup_{n\geq 0}\mathbb{A}^{n+1}(K) \to \bR\cup\{-\infty\}$ satisfying the following axioms:
\begin{itemize}
	\item height if zero: $\forall x\in \mathbb{A}(K)$, $h(x)=-\infty \Leftrightarrow x=0$;
	\item height of one: $h(1,1)=0$;
	\item invariance: $\forall x\in\mathbb{A}^{n+1}(K)$, $\forall \sigma\in \mathfrak{S}_{n+1}$, $h(\sigma(x))=h(x)$; 
	\item additivity: $\forall x,y\in \mathbb{A}(K)$, $h(x\otimes y)= h(x)+h(y)$;
	\item monotonicity: $\forall x,y\in \mathbb{A}(K)$, $h(x)\leq h(x,y)$;
	\item triangle inequality: $\forall x,y\in \mathbb{A}^{n+1}(K)$, $h(x+y)\leq h(x,y)+e$.
\end{itemize}
Such a height $h$ on $K$ is called \emph{global} if $h(x)=0$ for all $x\in K^{\times}$. In that case, we say that $(K,h)$ is a \emph{globally valued field} (\emph{GVF} for short) and $h$ induces a function $\mathbb{P}(K):=\bigsqcup_{n\geq 0}\mathbb{P}^{n}(K)\to \bR_{\geq 0}$ which is called the \emph{GVF height}.

The formalism of globally valued fields can be naturally interpreted through unbounded continuous logic, which allows the use of model theoretic methods to tackle arithmetic problems (cf. e.g. \cite{DHS24}). Moreover, there are several equivalent ways to view globally valued fields. Roughly speaking, a GVF can be interpreted either as a field equipped with a global height, a linear functional on the space of so-called \emph{lattice divisors}, or an equivalence class of proper adelic structures over the field (assuming that the latter is countable). The only precision we give, as it will be used in \S \ref{sub:examples_tac}, is as follows. A GVF structure over a countable field $K$ determines a Borel measure supported on the subspace of $M_{K}$ consisting of absolute values on $K$, yielding a proper adelic curve realising the GVF height. This means that, as long as we are interested in quantities involving heights, we may safely assume that the adelic space of an adelic curve with base field $K$ is a subset of $M_{K}$ equipped with the Borel $\sigma$-algebra and a Borel adelic measure. We refer to (\cite{GVF24}, Theorem 7.7) for the precise statements. If the base field is uncountable, we will see that GVFs are essentially equivalent to \emph{proper topological adelic curves}. 

In this article, we will be mainly interested in the GVFs/adelic curves relation. Let us just mention for now that given a proper adelic curve $S=(K,(\Omega,\cA,\nu),(\va_{\omega})_{\omega\in \Omega})$, the height $h_{S}: \bigsqcup_{n\geq 0}\mathbb{P}^{n}(K)\to \bR_{\geq 0}$ introduced in \S \ref{subsub:adelic_curves} is a GVF height on $K$. 

Globally valued fields turn out to be particularly useful in the situation where we can naturally define a height of interest on a field without \emph{a priori} knowing a possible underlying adelic structure. Then (\cite{GVF24}, Theorem 7.7) ensures that there exists a suitable (topological) adelic curve realising this height.

\section{Topological adelic curves}
\label{sec:topological_adelic_curves}

\subsection{Definitions}

\begin{definition}
\label{def:topological_adelic_curve}
We define the category $\TAC$ of \emph{topological adelic curves} as follows. An object of this category is the data $S=(K,\phi:\Omega\to M_K,\nu)$ where
\begin{itemize}
	\item $K$ is a field; 
	\item $\Omega$ is a Hausdorff topological space called the \emph{adelic space} of $S$ and $\phi:\Omega\to M_K$ is a continuous map called the \emph{structural morphism} of $S$;
	\item $\nu$ is a Borel measure on $\Omega$ satisfying the following condition: for all $f\in K^{\times}$, the map
	\begin{align*}
	\fonction{|f|_{\cdot}}{\Omega}{\bR_{\geq 0} \cup \{+\infty\}}{\omega}{|f|_{\phi(\omega)}}
	\end{align*}
is such that the function $\log|f|_{\cdot}$ is $\nu$-integrable. Note that as $\phi : \Omega \to M_{K}$ is continuous, by definition of the topology of $M_K$, for all $f\in K^{\times}$, the map $|f|_{\cdot}$ is continuous.
\end{itemize}

Then we define morphisms between topological adelic curves. Let $S_K=(K,\phi_K:\Omega_K \to M_{K},\nu_K)$ and $S_L=(L,\phi_L:\Omega_L \to M_{L},\nu_L)$ be two topological adelic curves. A \emph{morphism} $\alpha : S_L \rightarrow S_K$ is the data $(\alpha^{\sharp},\alpha_{\sharp},I_\alpha)$, where
\begin{itemize}
	\item $\alpha^{\sharp} : K \rightarrow L$ is a field extension;
	\item $\alpha_{\sharp} : \Omega_L \rightarrow \Omega_K$ is a continuous map inducing a commutative diagram 
\begin{center}
\begin{tikzcd}
\Omega_L \arrow[r, "\phi_{L}"] \arrow[d, "\alpha_{\sharp}"] & M_L \arrow[d, "\pi_{L/K}"] \\
\Omega_K \arrow[r, "\phi_{K}"]                                                     & M_K                       
\end{tikzcd}.
\end{center}
Moreover, the direct image of $\nu_L$ by $\alpha_{\sharp}$ is assumed to be equal to $\nu_K$. Namely, for any $f\in \cL^{1}(\Omega_{K},\nu_K)$, we have
\begin{align*}
\int_{\Omega_K}f\diff\nu_K = \int_{\Omega_L} f\circ \alpha_{\sharp}\diff\nu_L;
\end{align*}
	\item $I_{\alpha} : L^1(\Omega_L,\nu_L) \rightarrow L^1(\Omega_K,\nu_K)$ is a disintegration kernel of $\alpha_{\sharp}$, namely $I_{\alpha}$ is a linear map such that, for all $g\in L^{1}(\Omega_L,\nu_L)$, we have
	\begin{align*}
	\int_{\Omega_K} I_{\alpha}(g)\diff\nu_K = \int_{\Omega_L} g\diff\nu_L,
	\end{align*}
and which, for all $f \in L^{1}(\Omega_K,\nu_K)$, maps the equivalence class of $f\circ \alpha_{\sharp}$ to the class of $f$.
\end{itemize}

If $S=(K,\phi: \Omega \to M_K,\nu)$ is a topological adelic curve, for all $f\in K^{\times}$, we define the \emph{defect} $d_S(f)$ by 
\begin{align*}
d_{S}(f) := \int_{\Omega} \log|f|_{\omega} \nu(\diff\omega).
\end{align*}  
The topological adelic curve $S$ is called \emph{proper} if, for all $f\in K^\ast$ we have $d_{S}(f)=0$.
\end{definition}

\begin{remark}
\label{rem:tac_existence_of_absolute_value}
Let $S=(K,\phi:\Omega\to M_K,\nu)$ be a topological adelic curve. It is harmless to assume that there exists $\omega\in\Omega$ such that $\phi(\omega)$ is an absolute value on $K$. Indeed, if it not the the case, let $\Omega':=\Omega\sqcup\{\ast\}$ equipped with the disjoint union topology. Extend the map $\phi$ to a map $\phi':\Omega'\to M_{K}$ by sending $\ast$ to the trivial absolute value on $K$. Extend as well $\nu$ to a measure $\nu'$ on $\Omega'$ by setting $\nu'(\ast):=1$. Then $S'=(K',\phi':\Omega'\to M_{K},\nu')$ is a topological adelic curve that is proper iff $S$ is proper. From now on, we always assume that this condition is satisfied.
\end{remark}

\begin{definition}
\label{def:integral_topological_adelic_curves}
An \emph{integral topological adelic curve} is a topological adelic curve $(K,\phi:\Omega\to M_K,\nu)$ such that there exists a tame integral structure $(A,\|\cdot\|_A)$ for $K$ (cf. \S \ref{subsub:prelim_integral_structures}) such that the image of the structural morphism $\phi : \Omega \to M_K$ lies in the global space of pseudo-absolute values $V:=\cM(A,\|\cdot\|)$. The space $V$ is called the \emph{integral space} of $S$ and the integral structure $(A,\|\cdot\|_{A})$ is called the \emph{underlying integral structure} of $S$.

Let $S_K=(K,\phi_K:\Omega_K \to M_K,\nu_K)$ and $S_L=(L,\phi_L:\Omega_L \to M_{L},\nu_L)$ be two integral topological adelic curves with respective integral spaces $V_K,V_L$. A morphism $\alpha=(\alpha^{\sharp},\alpha_{\sharp},I_\alpha) : S_L \rightarrow S_K$ of topological adelic curves  is called \emph{integral} if there exists a continuous map $\tilde{\alpha_{\sharp}} : V_L\to V_K$ such that the diagram 
\begin{center}
\begin{tikzcd}
\Omega_L \arrow[r, "\phi_{L}"] \arrow[d, "\alpha_{\sharp}"] & V_L \arrow[r] \arrow[d, "\tilde{\alpha_{\sharp}}"] & M_L \arrow[d, "\pi_{L/K}"] \\
\Omega_K \arrow[r, "\phi_{K}"]                              & V_K \arrow[r]                        & M_K                       
\end{tikzcd}.
\end{center}
is commutative. 

We define the category $\ITAC$ of \emph{integral topological adelic curves} as the subcategory of $\TAC$ curves whose objects are integral topological adelic curves and whose morphisms are integral morphisms of topological adelic curves.
\end{definition}

\begin{notation}
\begin{itemize}
	\item[(1)] Let $S=(K,\phi:\Omega\to M_K,\nu)$ be a topological adelic curve. Expect mentioned otherwise, for any $\omega\in \Omega$, we denote by $A_{\omega}$ the finiteness ring of the pseudo-absolute value $\phi(\omega)$. Likewise, the kernel, resp. the residue field, resp. the underlying valuation of $\phi(\omega)$, is denoted by $\m_{\omega}$, $\kappa_{\omega}$, $v_{\omega}$.
	\item[(2)] By "let $S=(K,\phi:\Omega\to V,\nu)$ be an integral topological adelic curve, we mean that $(K,\phi:\Omega,\to M_K,\nu)$ is a topological adelic curve with integral space $V$. 
	\item[(3)] We can use the notation from \S \ref{subsub:prelim_integral_structures} in the context of topological adelic curves. Let $S=(K,\phi:\Omega\to M_K,\nu)$ be a topological adelic curve. Then we define $\Omega_{\infty}$, $\Omega_{\um}$ as the respective preimages of $M_{K,\infty},M_{K,\um}$ through the structural morphism $\phi : \Omega \to M_K$. We also have a map $\epsilon : (\omega\in\Omega_{\infty}) \mapsto \epsilon(\phi(\omega))=:\epsilon(\omega)\in ]0,1]$.
\end{itemize}
\end{notation}

\begin{proposition}
\label{prop:properties_topological_adelic_curves}
Let $S=(K,\phi:\Omega\to M_K,\nu)$ be a topological adelic curve. 
\begin{itemize}
	\item[(i)] For all $f\in K$, let $\Omega_f := \left\{\omega\in\Omega : f\in A_{\omega} \right\}$. Then we have $\nu(\Omega\setminus \Omega_f)=0$. Moreover, if $f\in K^{\times}$, then $\nu\left(\{\omega\in\Omega : |f|_{\omega} = 0\}\right) = 0$.
	\item[(ii)] $\Omega_{\infty}$, resp. $\Omega_{\um}$, is an open, resp. a closed subset of $\Omega$. 
	\item[(iii)] If $\epsilon$ is bounded from below on $\Omega_{\infty}$, then $\nu(\Omega_{\infty}) < +\infty$ and $\Omega_{\infty}$ is a closed subset of $\Omega$. 
	
\end{itemize}
\end{proposition}

\begin{proof}
Let $f\in K$. If $f= 0$, then $\Omega_f= \Omega$ and thus $\nu(\Omega\setminus \Omega_f) = 0$. Assume that $f\neq 0$. Then $\log|f|_{\cdot}$ is $\nu$-integrable and $\Omega_f = \{\omega\in\Omega : \log|f|_{\omega} \neq +\infty\}$. Hence $\nu(\Omega\setminus \Omega_f)=0$. Finally, as $f\neq 0$, $\log|f^{-1}|_{\cdot}$ is $\nu$-integrable and $\{\omega\in\Omega : |f|_{\omega} = 0\} = \Omega\setminus \Omega_{f^{-1}}$. Thus $\nu\left(\{\omega\in\Omega : |f|_{\omega} = 0\}\right) = 0$. This concludes the proof of $(i)$.

We now show $(ii)$. $\Omega_{\infty} = \phi^{-1}(M_{K,\infty})$ is open as $\phi$ is continuous. Likewise, $\Omega_{\um}$ is closed. 

To show (iii), we may assume that $\Omega_{\infty} \neq \emptyset$. Then $\mathrm{char}(K)=0$ and $\log|2|_{\cdot} \in L^1(\Omega,\nu)$. Therefore, the function $f:=\max\{0,\log|2|_{\cdot}\}$ is $\nu$-integrable. Let $0<m$ a lower bound for $\epsilon$ on $\Omega_{\infty}$. Then 
\begin{align*}
m\log(2)\int_{\Omega_{\infty}}\nu(\diff\omega) \leq \int_{\Omega} f(\omega) \nu(\diff\omega)\log < +\infty,
\end{align*}
hence $\nu(\Omega_{\infty}) < +\infty$. Let $(\omega_{\alpha})_{\alpha\in I}$ be a generalised convergent sequence in $\Omega_{\infty}$ with limit $\omega\in \Omega$. Then the generalised sequence $(\epsilon(\omega_{\alpha}))_{\alpha\in I}\in  [m,1]^{I}$ is bounded. By Bolzano-Weierstrass, up to considering a subsequence, we may assume that $(\epsilon(\omega_{\alpha}))_{\alpha\in I}$ converges to some $\epsilon\in[m,1]$. Then $|2|_{\omega}=\lim_{\alpha \in I}|2|_{\omega_{\alpha}} = \lim_{\alpha \in I}2^{\epsilon(\omega_{\alpha})}=2^{\epsilon}>1$. Therefore, $\omega\in \Omega_{\infty}$ and it follows that $\Omega_{\infty}$ is closed.
\end{proof}

\subsection{Height function on a proper topological adelic curve}
\label{sub:height_function_topological_adelic curve}

In this subsection, we fix a proper topological adelic curve $S=(K,\phi:\Omega\to M_{K},\nu)$.

\begin{definition}
\label{def:GVF_height_tac}
Let $n\in \bN$ be an integer and $f=(f_{0}:\cdots:f_{n})\in \mathbb{P}^{n}(K)$. Define the \emph{height} of $f$ w.r.t. $S$ by
\begin{align*}
h_{S}(f) := \int_{\Omega}\log\max\{|f_{0}|_{\omega},...,|f_{n}|_{\omega}\}\nu(\diff\omega).
\end{align*} 
This quantity is well-defined by virtue of the product formula. Moreover, for any $f\in K$, we set $h_{S}(f):=h_{S}(1:f)$. 
\end{definition}

The following observation relates topological adelic curves and globally valued fields (cf. \S \ref{sub:GVF}).

\begin{proposition}
\label{prop:GVF_height}
The height function $h_{S} : \bigsqcup_{n\in \bN} \mathbb{P}^{n}(K) \to \bR$ defines a GVF structure on $K$ with Archimedean error $h_{S}(2)$.
\end{proposition}

\begin{remark}
\label{rem:tac_GVF}
Using (\cite{GVF24}, Theorem 7.7), we see that topological adelic curves and globally valued fields are not too far away from each other, as long as the quantities of interest can be related in terms of height. More precisely, a GVF structure on a field $K$ yields an equivalence class of so-called \emph{admissible measures}. For any admissible measure $\nu$ in this equivalence class, $S'=(K,\mathrm{id} : M_{K}\to M_{K},\mu)$ is a proper topological adelic curve with locally compact adelic space and Radon adelic measure. Conversely, the proper topological adelic curve $S=(K,\phi:\Omega\to M_{K},\nu)$ determines a GVF structure on $K$ yielding an equivalence class of admissible measure. For any measure $\mu$ as above, we have $h_{S}=h_{S'}$, although in general $S$ and $S'$ are different as topological adelic curves. 
\end{remark}

\subsection{Zariski-Riemann spaces over adelic curves}
\label{sub:ZR_spaces_adelic_curves}

Building on the material introduced in \S \ref{sub:ZR_analytic_spaces}, we attach a "Zariski-Riemann" type space to any topological adelic curve. We fix a topological adelic curve $S=(K,\phi:\Omega\to M_{K},\nu)$ and a projective $K$-scheme $X$.

We denote by $k$ the prime subring of $K$ equipped with the norm $|\cdot|_{\infty}$ if $k=\bZ$ and the trivial norm if $k$ is a finite field. In \S \ref{sub:ZR_analytic_spaces}, we have introduced the locally ringed spaces $\ZR(X):=\ZR(X/k)$ and $\ZR(X)^{\an}:=\ZR(X/k)^{\an}$. Recall that they fit in a commutative diagram 
\begin{center}
\begin{tikzcd}
\ZR(X)^{\an} \arrow[d] \arrow[r] & \ZR(X) \arrow[d] \\
M_{K}\cong\ZR(K)^{\an} \arrow[r]          & \ZR(K)          
\end{tikzcd}.
\end{center}

\begin{definition}
\label{def:ZR_space_tac}
\begin{itemize}
	\item[(1)] Denote by $\tilde{\Omega}$ the quotient of $\Omega$ by the equivalence relation identifying points of $\Omega$ having the same image via the composition $j_{K}\circ \phi: \Omega\to M_{K}\to \ZR(K)$. $j_{K}\circ\phi$ factors through $\tilde{\Omega}$ and we equip $\tilde{\Omega}$ with the initial topology associated with this map $\tilde{\phi}:\tilde{\Omega}\to\ZR(K)$ (in this case, it is the subspace topology). We endow $\tilde{\Omega}$ with the structure of a locally ringed space whose structure sheaf is $\cO_{\tilde{\Omega}}:=\tilde{\phi}^{-1}\cO_{\ZR(K)}$.
	\item[(2)] We define the \emph{algebraic Zariski-Riemann space} of $X$ over $S$ as $\ZR(X)_{S}:=\ZR(X)\times_{\ZR(K)} \tilde{\Omega}$, where the fibre product is considered in the category of locally ringed spaces. The structure sheaf of $\ZR(X)_{S}$ is denoted by $\cO_{\ZR(X)_{S}}$. Note that (cf. e.g. \cite{Gillam11}, Theorem 9) the underlying topological space of $\ZR(X)_{S}$ is homeomorphic to the fibre product $\ZR(X)\times_{\ZR(K)}\tilde{\Omega}$ in the category of topological spaces. In this case, it is homeomorphic to a subset of $\ZR(X)$ equipped with the subspace topology. Moreover, via this homeomorphism, the structure sheaf $\cO_{\ZR(X)_{S}}$ corresponds to the inverse image of $\cO_{\ZR(X)}$ via the projection $\ZR(X)_{S}\to\ZR(X)$. Since there exists $\omega\in \Omega$ such that $\phi(\omega)$ is an absolute value on $K$, the map $\eta:\Spec(K) \to \ZR(K)$ factors through $\tilde{\Omega}$, yielding a map $\eta_{S}:\Spec(K) \to \tilde{\Omega}$. Moreover, the map $\eta_{X}:X \to \ZR(X)$ factors through $\ZR(X)_{S}$, yielding a map $\eta_{X,S}:X\to \ZR(X)_{S}$.
	\item[(3)] We equip $\Omega$ with the structure of a locally ringed space whose structure sheaf is $\cO_{\Omega}:=\phi^{-1}\cO_{\ZR(K)^{\an}}$. 	
	\item[(4)] As in bullet (2), we define the \emph{analytic Zariski-Riemann space} of $X$ over $S$ as $\ZR(X)^{\an}_{S}:=\ZR(X)^{\an}\times_{M_{K}} \Omega$, where the fibre product is again considered in the category of locally ringed spaces. Here again the underlying topological space of $\ZR(X)_{S}^{\an}$ is homeomorphic to the fibre product $\ZR(X)^{\an}\times_{M_{K}}\Omega$ in the category of topological space and the structure sheaf $\cO_{\ZR(X)_{S}^{\an}}$ is the inverse image of $\cO_{\ZR(X)^{\an}}$ via the projection $\phi_{X}:\ZR(X)_{S}^{\an}\to\ZR(X)^{\an}$. Note that the map $\ZR(X)^{\an}\to\Omega$ is proper (as the pullback of the proper map $\ZR(X)^{\an}\to M_{K}$). Moreover, for any $x\in \ZR(X)^{\an}_{S}$, we define its \emph{completed residue field} $\widehat{\kappa}(x)$ as the completed residue field of its image in $\ZR(X)^{\an}$ via the projection $\phi_{X}$.
	\item[(5)] Finally, by a \emph{metrised vector bundle} $\ZR(X)_{S}$, we mean the data $\overline{\cE}=(\cE,\varphi)$ where $\cE$ is a locally free $\cO_{\ZR(X)_{S}}$-module of finite rank together with a family $(|\cdot|_{\varphi}(x))_{x\in \ZR(X)^{\an}_{S}}$ where, for any $x\in \ZR(X)^{\an}_{S}$, $|\cdot|_{\varphi}(x)$ is a norm on the $\widehat{\kappa}(x)$-vector space $\cE(x):=\cE\otimes_{\cO_{\ZR(X)_{S}}}\widehat{\kappa}(x)$. 
\end{itemize}

We summarise everything in the following commutative diagram in the category of locally ringed spaces
\begin{center}
\begin{tikzcd}
                                        & \ZR(X)^{\an}_{S} \arrow[rr, "\phi_{X}"] \arrow[dd] \arrow[rd, "{j_{X,S}}"] &                                                      & \ZR(X)^{\an} \arrow[dd] \arrow[rd, "j_{X}"] &                   \\
X \arrow[dd] \arrow[rr, "{\eta_{X,S}}"] &                                                                            & \ZR(X)_{S} \arrow[rr, "\tilde{\phi_{X}}"] \arrow[dd] &                                             & \ZR(X) \arrow[dd] \\
                                        & \Omega \arrow[rr, "\phi"] \arrow[rd, "{j_{K,S}}"]                          &                                                      & M_{K} \arrow[rd, "j_{K}"]                   &                   \\
\Spec(K) \arrow[rr, "\eta_{S}"]         &                                                                            & \tilde{\Omega} \arrow[rr, "\tilde{\phi}"]            &                                             & \ZR(K)           
\end{tikzcd}
\end{center}
whose front right and back right squares are Cartesian. 
\end{definition}

Let us now describe finitely presented sheaves on $\ZR(X)_{S}$. 

\begin{proposition}
\label{prop:vector_bundles_ZR_adelic_curve}
Assume that $\tilde{\Omega}$ is a locally closed quasi-compact subset of $\ZR(K)$. The following assertions hold.
\begin{itemize}
	\item[(1)] The natural functor $\varinjlim_{U}\mathrm{Fp}(U) \to \mathrm{Fp}(\ZR(X)_{S})$, where $U$ runs over the quasi-compact open neighbourhoods of $\ZR(X)_{S}$ in $\ZR(X)$ is an equivalence of categories. Moreover, this equivalence of categories restricts to an equivalence of categories between $\varinjlim_{U}\mathrm{Vb}(U) \to \mathrm{Vb}(\ZR(X)_{S})$.
	\item[(2)] Assume that $X$ is integral. Then the structure sheaf $\cO_{\ZR(X)_{S}}$ is coherent and the natural functor $\varinjlim_{U}\mathrm{Coh}(U) \to \mathrm{Coh}(\ZR(X)_{S})$, where $U$ runs over the quasi-compact open neighbourhoods of $\ZR(X)_{S}$ in $\ZR(X)$, is an equivalence of categories.
\end{itemize}
\end{proposition}

\begin{proof}
\begin{claim}
\label{claim:coherent_sheaves_qc_subspace}
Let $(X,\cO_{X})$ be a locally ringed space and let $i:Z\hookrightarrow X$ be a quasi-compact subset. The following holds.
\begin{itemize}
	\item[(i)] For any sheaf $\cF$ on $X$, the natural map
	\begin{align*}
	\displaystyle\varinjlim_{U} H^{0}(U,\cF) \to H^{0}(Z,i^{-1}Z),
	\end{align*}
	where $U$ runs over the open neighbourhoods of $Z$ in $X$, is bijective.
	\item[(ii)] Assume that $Z$ admits a basis of quasi-compact open subsets. Then the natural functor between $\varinjlim_{U}\mathrm{Fp}(U)$ and $\mathrm{Fp(Z)}$, where $U$ runs over the open neighbourhoods of $Z$ in $X$, is an equivalence of categories. Moreover, this equivalence of categories restricts to an equivalence of categories between $\varinjlim_{U}\mathrm{Vb}(U)$ and $\mathrm{Vb(Z)}$.
	\item[(iii)] Assume that $Z$ admits a basis of quasi-compact open subsets and that $\cO_{X}$ is coherent. Then $i^{-1}\cO_{Z}$ is coherent. 
\end{itemize}
\end{claim}

\begin{proof}
\textbf{(i):} (\cite{KS90}, Proposition 2.5.1) implies the injectivity. Now let $s\in H^{0}(Z,i^{-1}Z)$. Then there exist a collection $(U_{i})_{i\in I}$ of open subsets of $X$ together with a family $(s_{i})_{i\in I}$ such that $Z\subset\bigcup_{i\in I}U_{i}$ and, for any $i\in I$, $s_{i}\in H^{0}(U_{i},\cF)$ is such that $s_{i|U_{i}\cap Z}=s_{|U_{i}\cap Z}$. By quasi-compactness of $Z$, we may assume that $I$ is finite. Let 
\begin{align*}
V:=\left\{x\in \displaystyle\bigcup_{i\in I}U_{i} : \forall i,j\in I \text{ s.t. }x\in U_{i}\cap U_{j}, \quad s_{i,x}=s_{j,x}\right\}.
\end{align*}
Since $I$ is finite, $V$ is open in $X$ and contains $Z$ by definition. Since for any $i,j\in I$, we have $s_{i|V\cap U_{i}\cap U_{j}}=s_{j|V\cap U_{i}\cap U_{j}}$, there exists a section $t\in H^{0}(V,\cF)$ such that $t_{|V\cap U_{i}}=s_{i|V\cap U_{i}}$ for any $i\in I$. By construction, the image of $t$ in $H^{0}(Z,i^{-1}\cF)$ is $s$. 

\item[(ii):] Let us first prove the essential surjectivity. Let $\cF$ be a finitely presented module on $Z$. There exist a finite covering $\bigcup_{i\in I}V_{i}$ of $Z$ by quasi-compact open subsets and, for any $i\in I$, a presentation
\begin{align*}
\cO_{V_{i}}^{p_{i}} \to \cO_{V_{i}}^{q_{i}} \to \cF_{|V_{i}} \to 0,
\end{align*}
where $p_{i},q_{i}$ are integers. Let $i\in I$. By (i), there exist an open neighbourhood $U_{i}$ of $V_{i}$ in $X$ and a morphism $\varphi_{i}:\cO_{U_{i}}^{p_{i}}\to \cO_{U_{i}}^{q_{i}}$ whose pullback to $V_{i}$ is isomorphic to the morphism $\cO_{V_{i}}^{p_{i}} \to \cO_{V_{i}}^{q_{i}}$ of the presentation. Let $\cG_{i}:=\coker(\varphi_{i})$, this is a finitely presented $\cO_{U_{i}}$-module. By right exactness of the pullback of sheaves of modules, the pullback of $\cG_{i}$ to $V_{i}$ is isomorphic to $\cF_{|V_{i}}$. Since $I$ is finite, up to shrinking the $U_{i}$'s, we may assume that the $\cG_{i}$'s agree on the intersections of the form $U_{i}\cap U_{j}$, where $i,j\in I$. Thus, they glue to a finitely presented sheaf $\cG$ on $\bigcup_{i\in I}U_{i}$. By construction, the pullback of $\cG$ to $Z$ is isomorphic to $\cF$. Moreover, we see that if $\cF$ is locally free of finite rank, then $\cG$ is also locally free of finite rank. 

Let us now prove the full faithfulness. We want to show that, for any objects $(\cF_{U})_{U},(\cG_{U})_{U}$ of $\varinjlim_{U}\mathrm{Fp}(U)$ with respective image $\cF,\cG$ in $\mathrm{Fp}(Z)$, the map 
\begin{align*}
\displaystyle\varinjlim_{U}\Hom_{\cO_{U}}(\cF_{U},\cG_{U}) \to \Hom_{i^{-1}\cO_{X}}(\cF,\cG)
\end{align*}
is bijective. Let $x\in Z$. Since $\cF$ and the $\cF_{U}$'s are finitely presented, we have isomorphisms
\begin{align*}
\displaystyle\left(\varinjlim_{U}\Hom_{\cO_{U}}(\cF_{U},\cG_{U})\right)_{x} &\cong \varinjlim_{U}\Hom_{\cO_{U}}(\cF_{U},\cG_{U})_{x} \cong \varinjlim_{U}\Hom_{\cO_{U,x}}(\cF_{U,x},\cG_{U,x})\\ 
&\cong \Hom_{\cO_{X,x}}(\cF_{x},\cG_{x}) \cong  \left(\Hom_{i^{-1}\cO_{X}}(\cF,\cG)\right)_{x}.
\end{align*}

\textbf{(iii):} It suffices to prove that for any quasi-compact open subset $U\subset Z$ and any morphism $\varphi : \cO_{U}^{n}\to \cO_{U}$, $\ker(f)$ is of finite type. Let $\varphi:\cO_{U}^{n}\to \cO_{U}$ be such a morphism. Since $U$ is quasi-compact, (i) implies that there exist an open neighbourhood $V$ of $U$ in $X$ and a morphism $\psi:\cO_{V}^{n}\to \cO_{V}$ such that $(i_{|i^{-1}(V)}^{-1}\psi)_{|U}$ is isomorphic to $\varphi$ as a morphism of $\cO_{U}$-modules. Since $\cO_{X}$ is coherent, $\ker(\psi)$ is of finite type. By exactness of the inverse image functor, we deduce that $\ker(\varphi)$ is also of finite type.
\end{proof}

\textbf{(1):} Since $\tilde{\Omega}$ is a locally closed quasi-compact subset of $\ZR(K)$, $\tilde{\Omega}$ is a coherent and sober topological space and the inclusion $\tilde{\Omega}\to\ZR(K)$ is quasi-compact (\cite{Fujiwara-Kato}, Chapter 0, Propositions 2.1.1 and 2.2.3). Let $\cX\in\cM_{X/k}$. Write $\tilde{\Omega}$ as an intersection $A\cap B$, where $A,B$ are respectively open in $\ZR(K)$. Then $\ZR(X)_{S}=(\ZR(X)_{S}\to\tilde{\Omega})^{-1}(\tilde{\Omega})$ is a locally closed subset of $\ZR(X)$. Let us prove that $\ZR(X)_{S}$ is quasi-compact, then by \emph{ibid.}, $\ZR(X)_{S}$ is also a coherent and sober topological space. , it suffices to prove that the map $\ZR(X)\to\ZR(K)$ is quasi-compact. Let $\cY\in \cM_{K/k}$. The full subcategory of $\cM_{X/k}$ consisting of projective submodels of $X$ over $\cY$ is cofinal in $\cM_{X/k}$ and we obtain that the map $\ZR(X)\cong\varprojlim_{\cX\in\cM_{\cY}}\to\cY$ is quasi-compact (\cite{Fujiwara-Kato}, Chapter 0, Theorem 2.2.13). Now the map $\ZR(X)\to\ZR(K)$ is the projective limit of the maps $\ZR(X)\to\cY$, when $\cY$ runs over $\cM_{K/k}$ and is quasi-compact by \emph{ibid.} Therefore, Claim \ref{claim:coherent_sheaves_qc_subspace} (ii) yields (1).

\textbf{(2):} Proposition \ref{prop:coherence_algeabraic_ZR_sheaf} implies that, for any quasi-compact open subset $U$ of $\ZR(X)$, the structure sheaf $\cO_{U}$ is coherent. Thus (2) follows from (1) and Proposition \ref{prop:coherence_algeabraic_ZR_sheaf}.
\end{proof}

Let us discuss the material introduced in Definition \ref{def:ZR_space_tac} in the case where the topological adelic curve is integral. Assume that $(A,\|\cdot\|)$ is an integral structure such that $\phi$ factors through $V:=\cM(A,\|\cdot\|)$. Then the map $\tilde{\Omega}\to\ZR(K)$ factors through $\ZR(K/A)\cong\Spec(A)$. More generally, the map $\ZR(X)_{S}\to \ZR(X)$ factors through $\ZR(X/A)$ and $\ZR(X)_{S}$ identifies with the fibre product $\ZR(X/A)\times_{\ZR(K/A)}\tilde{\Omega}$. Likewise, the map $\ZR(X)^{\an}_{S}\to \ZR(X)^{\an}$ factors through $\ZR(X/A)^{\an}$ and $\ZR(X)^{\an}_{S}$ identifies with the fibre product $\ZR(X/A)^{\an}\times_{\ZR(K/A)^{\an}}\Omega$. We can moreover describe finitely presented sheaves on $\ZR(X)_{S}$ in terms of finitely presented sheaves on models. By abuse of notation, we denote again by $\cM_{X/A}$ the category of projective models of $X$ over $A$. For any $\cX\in\cM_{X/A}$, define $\cX_{S}:=\cX\times_{\Spec(A)}\tilde{\Omega}$ and $\cX_{S}^{\an}:=\cX^{\an}\times_{\cM(A,\|\cdot\|)}\Omega$, where the fibre products are understood in the category of locally ringed spaces. Note that the respective underlying topological space of $\cX_{S},\cX_{S}^{\an}$ identifies with the topological fibre product $\cX\times_{\Spec(A)}\tilde{\Omega},\cX^{\an}\times_{\cM(A,\|\cdot\|)}\Omega$. Moreover, we have homeomorphisms 
\begin{align*}
\ZR(X)_{S}\cong \displaystyle\varprojlim_{\cX\in\cM_{X/k}}\cX_{S}, \quad \ZR(X)_{S}^{\an}\cong \varprojlim_{\cX\in\cM_{X/k}}\cX_{S}^{\an}.
\end{align*} 

\begin{proposition}
\label{prop:coherent_sheaves_locally_closed_integral}
Assume that $\tilde{\Omega}$ is a locally closed quasi-compact subset of $\ZR(K/A)\cong\Spec(A)$. The following assertions hold.
\begin{itemize}
	\item[(1)] The natural functor $\varinjlim_{\cX\in\cM_{X/A}}\mathrm{Fp}(\cX_{S}) \to \mathrm{Fp}(\ZR(X)_{S})$ is an equivalence of categories. Moreover, this equivalence of categories restricts to an equivalence of categories between $\varinjlim_{\cX\in\cM_{X/A}}\mathrm{Vb}(\cX_{S}) \to \mathrm{Vb}(\ZR(X)_{S})$.
	\item[(2)] Assume that $X$ is integral. Then the structure sheaf $\cO_{\ZR(X)_{S}}$ is coherent and the natural functor $\varinjlim_{\cX\in\cM_{X/A}}\mathrm{Coh}(\cX_{S}) \to \mathrm{Coh}(\ZR(X)_{S})$ is an equivalence of categories.
\end{itemize}
\end{proposition}

\begin{proof}
\textbf{(1):} First assume that $\tilde{\Omega}$ is open. Then $\ZR(X)_{S}$ is a quasi-compact open subset of $\ZR(X/A)$ and the result follows from Proposition \ref{prop:fp_sheaves_ZR_algebraic}. Let us now consider the general case. Let $\cX\in\cM_{X/A}$ and $V_{\cX}$ be a quasi-compact open neighbourhood of $\cX_{S}$ in $\cX$. Choose an open neighbourhood $U$ of $\ZR(X)_{S}$ in $\ZR(X/A)$ contained in $p_{\cX}^{-1}(V)$. By (\cite{Fujiwara-Kato}, Chapter 0, Proposition 2.2.9 and Corollary 2.2.12), there exists an arrow $q:\cX'\to\cX$ in $\cM_{X/A}$ and a quasi-compact open subset $U_{\cX'}\subset \cX'$ such that $\cX_{S}\subset U_{\cX'}\subset V_{\cX'}:=q^{-1}(V_{\cX})$. This fact, combined with the previous case and Proposition \ref{prop:vector_bundles_ZR_adelic_curve}, yields equivalences of categories
\begin{align*}
\mathrm{Fp}(\ZR(X)_{S})\cong \displaystyle\varprojlim_{U}\mathrm{Fp}(U) \cong \varprojlim_{U}\varprojlim_{\cX\in\cM_{U}}\mathrm{Fp}(U_{\cX}) \cong \varprojlim_{\cX\in\cM_{X/A}}\mathrm{Fp(\cX_{S})},
\end{align*}
where $U$ runs over the quasi-compact open neighbourhoods of $\ZR(X)_{S}$ in $\ZR(X)$.

\textbf{(2):} This is a consequence of Proposition \ref{prop:vector_bundles_ZR_adelic_curve} (2).
\end{proof}

\subsection{Constructions on topological adelic curves}
\label{sub:constructions_tac}

Before giving explicit examples of topological adelic curves, let us give several general constructions that one is allowed to perform over adelic curves that will be useful for applications.

\subsubsection{Restriction to a subfield}
\label{subsub:restriction_subfield}

Let $S=(K,\phi:\Omega\to M_K,\nu)$ be a topological adelic curve. Let $K_0$ be a subfield of $K$. Recall that we denote by $\pi_{K/K_0}: M_{K} \to M_{K_{0}}$ the restriction of pseudo-absolute values on $K$ to $K_0$. We obtain a topological adelic curve $S_0=(K_0,\phi\circ\pi_{K/K_0}: \Omega \to M_{K_0},\nu)$ and a morphism $S\to S_{0}$. Note that $S_0$ is proper if so is $S$. 

We now address the question for integral adelic curves. Roughly speaking, this question can be formulated as: "Can we restrict an integral structure of $K$ to subfields?". Let $(A,\|\cdot\|_A)$ be an integral structure for $K$ and denote $V=\cM(A)$.

\begin{proposition}
\label{prop:integral_structure_subfield}
Let $K_0 \subset K$ be a subfield. Let $A_0:= A\cap K_0$. We assume that $V$ is tame (cf. \S \ref{subsub:prelim_integral_structures}) and that the morphism $\Spec(A) \to \Spec(A_0)$ is surjective. Denote by $\|\cdot\|_{A_0}$ the restriction to $A_0$ of the norm $\|\cdot\|_A$. Then $(A_0,\|\cdot\|_{A_0})$ defines a tame integral structure for $K_0$.
\end{proposition}

\begin{proof}
First, note that we have $\Frac(A_0) = \Frac(A)\cap K_0 = K_0$. Since $V$ is tame, the trivial absolute value on $K$ is in $V$. Thus, the trivial absolute value  on $K_0$ is bounded from above by $\|\cdot\|_{A_0}$ and $(A_0,\|\cdot\|_{A_0})$ is a discrete normed topological ring, hence is Banach. It remains to prove that $A_0$ is Prüfer, since tameness is immediate by restriction of the norm. Let $\p_0\in \Spec(A_0)$. By hypothesis, there exists $\p\in\Spec(A)$ such that $\p\cap A_0 = \p_0$. Since $A_{\p}$ is a valuation ring of $K$, $A_{\p} \cap K_0$ is a valuation ring for $K_0$. From $\p_0 = \p \cap K_0$ and $(A_0)_{\p_0} = A_{\p}\cap K_0$, we get that $A_0$ is Prüfer. 
\end{proof}

\begin{definition}
\label{def:restriction_subfield}
Let $K_0\subset K$ be a subfield and assume that $S=(K,\phi:\Omega\to V,\nu)$ is an integral topological adelic curve with underlying integral structure $A$. Let $A_0:=A\cap K_0$ and assume that $\Spec(A)\to \Spec(A_0)$ is surjective. Proposition \ref{prop:integral_structure_subfield} implies that $A_0$ defines an integral structure for $K_0$ whose global space of pseudo-absolute values $V_0 := \cM(A_0)$ is tame. Let $\phi_0 : \Omega \to V_0$ be the composition of $\varphi$ with the projection $\pi_{K/K_0} : V \to V_0$. Then $S_0 :=(K_0, \phi_0:\Omega\to V_0,\nu)$ is an integral topological adelic curve. Furthermore, for any $f\in K_0^{\times}$, we have $d_{S_0}(f) = d_S(f)$. In particular, $S_0$ is proper if so is $S$.
\end{definition}

\subsubsection{Restriction to a Borel subset}
\label{subsub:restriction_adelic_structure}

Let $S=(K,\phi:\Omega\to M_K,\nu)$ be a topological adelic curve. Let $i:\Omega'\hookrightarrow \Omega$ be a Borel subset. Moreover, if $\nu'$ denotes the restriction of the measure $\nu$ to $\Omega'$, then $(K,\phi_{|\Omega'}:\Omega'\to M_{K},\nu')$ is a topological adelic curve called the \emph{restriction} of $S$ to $\Omega'$. Moreover, if $S$ is an integral topological adelic curve with integral space $V$, then $(K,\phi_{|\Omega'}:\Omega'\to V,\nu')$ is an integral topological adelic curve. Note that, in general, if $S$ is proper, $S'$ need not be proper. If we further assume that $\nu(\Omega\setminus\Omega')=0$, then the triple $(\mathrm{id},i,I)$, where $I:L^{1}(\Omega',\nu')\to L^{1}(\Omega,\nu)$ is defined by
\begin{align*}
\forall f\in L^{1}(\Omega',\nu'), \quad I(f) : (\omega\in\Omega) \mapsto  \left\{\begin{matrix}
f(\omega) &\text{ if }\omega\in\Omega',\\
0 &\text{ if }\omega\notin\Omega',
\end{matrix}\right.
\end{align*}
defines a morphism of topological adelic curves and $S'$ is proper iff $S$ is proper. We will make 

\begin{remark}
\label{rem:restriction_adelic_structure_locally_compact}
For applications, we will sometimes need the adelic space $\Omega$ to be locally compact. In that case, if we further assume that $\Omega'\subset \Omega$ is a locally closed subset, then $\Omega'$ is a locally compact Hausdorff $M_K$-topological space (cf. \cite{BouTG},  Chapitre I, \S 9.7, Proposition 13).  Likewise, if $S$ is an integral topological adelic curve with integral space $V$, $\Omega'$ is a locally compact Hausdorff $V$-topological space and $(K,\phi_{|\Omega'}:\Omega' \to V,\nu')$ is an integral topological adelic curve.
\end{remark}

\subsubsection{Disintegration}
\label{subsub:disintegration}

As we will see, the geometry of a topological adelic curve is related to the geometry of its Zariski-Riemann space. For technical reasons, it is sometimes more convenient to work with the full space of pseudo-absolute values as adelic space. The naive idea is to pushforward the adelic measure to the space of pseudo-absolute values. However, since we want to preserve the integration theoretic features, this amounts to the following question: given a topological adelic curve $S=(K,\phi:\Omega\to M_{K},\nu)$, does there exist a morphism of topological adelic curves $\alpha: S \to (K,\mathrm{id}:M_{K}\to M_{K},\phi_{\ast}\nu)$? This amounts to the existence of a disintegration of $\nu$ over $\phi_{\ast}\nu$. 

\begin{proposition}[\cite{Fremlin03}, Proposition 452O]
\label{prop:existence_disintegration}
Assume that $(\Omega,\nu)$ is a Radon measure space (\cite{Fremlin03}, Definition 411H). Then there exists a disintegration kernel $I_{\phi}:L^{1}(\Omega,\nu)\to L^{1}(M_{K},\phi_{\ast}\nu)$. 
\end{proposition}

\begin{remark}
\label{rem:existence_disintegration}
In all the examples we will encounter, the adelic measure is Radon. Thus using a disintegration kernel, it suffices to prove any property of topological adelic curves preserved by morphism on a topological adelic curve with adelic space the space of all pseudo-absolute values.
\end{remark}

\subsection{Examples of topological adelic curves}
\label{sub:examples_tac}

Let us now introduce the main examples of topological adelic curves that we are concerned with in this article.

\subsubsection{Number fields}
\label{subsub:example_number_field_tac}

Let $K$ be a number field with ring of integers $\cO_K$. Let $\Omega_{\um}$ be the set of closed points of $\Spec(\cO_K)$ and $\Omega_{\infty}$ be the set of all field embeddings $\sigma : K \to \bC$, both equipped with the discrete topology. Then $\Omega:= \Omega_{\um} \sqcup \Omega_{\infty}$ is a discrete topological space. For any $\omega\in \Omega_{\infty}$, let $\phi(\omega)\in M_K$ denote the Archimedean absolute value on $K$ corresponding to the complex embedding such that $\epsilon(\omega)=1$. For any $\omega\in\Omega_{\um}$ above a prime number $p$, let $\phi(\omega)$ denote the non-Archimedean extension of the $p$-adic absolute value such that the absolute value of $p$ equals $1/p$. Then the morphism $\phi : \Omega \to M_{K}$ is continuous. For any $\omega\in \Omega$, denote by $K_{\omega}$, resp. $\bQ_{\omega}$ the completion of $K$, resp. $\bQ$, w.r.t. the absolute value $\phi(\omega)$, resp. the restriction of $\phi(\omega)$ to $\bQ$. Now set $\nu(\{\omega\}):=[K_{\omega}:\bQ_{\omega}]$. Then the usual product formula implies that $S=(K,\phi:\Omega\to M_K,\nu)$ is a proper topological adelic curve. Moreover, the homeomorphism $M_K\cong\cM(\cO_K,\max_{v\in\Omega_{\infty}} \{|\cdot|_v\})$ from (\cite{Sedillot_pav}, Proposition 7.2.1) implies that $S$ is an integral topological adelic curve. Note that the discrete topology is the coarsest topology on $\Omega$ making the map $\phi : \Omega \to M_{K}$ continuous. This follows from the fact that, for any $\omega\in \Omega$, one can exhibit a neighbourhood $U_{\omega}$ of $\phi(\omega)$ in $\Omega$ such that $\phi^{-1}(U_{\omega}) \cap \Omega = \{\omega\}$. 

\begin{remark}
\label{rem:ZR_space_number_fields}
For any projective $K$-scheme $X$ the Zariski-Riemann spaces $\ZR(X)_{S}$ and $\ZR(X)^{\an}_{S}$ identify respectively to $X$ and $\bigsqcup_{\omega\in\Omega} (X\otimes_{K}K_{\omega})^{\an}$.
\end{remark}

\subsubsection{Examples of adelic curves in \cite{ChenMori,ChenMori21}}
\label{subsub:examples_tac_adelic_curves}

\begin{itemize}
	\item[(1)] All the examples of proper adelic curves in (\cite{ChenMori}, \S 3.2) naturally define proper topological adelic curves. Indeed, all the involved $\sigma$-algebras are the Borel $\sigma$-algebra of a Hausdorff topological space. For any such example $S=(K,(\Omega,\cA,\nu),\phi)$, similarly to Remark \ref{rem:ZR_space_number_fields}, for any projective $K$-scheme $X$, $\ZR(X)_{S} \cong X$. Moreover, as a set, we have a bijection $\ZR(X)^{\an}_{S}\cong\bigsqcup_{\omega\in\Omega} (X\otimes_{K}K_{\omega})^{\an}$. $\ZR(X)^{\an}_{S}$ can be seen as the topological counterpart of the \emph{global adelic space} in the sense of (\cite{ChenMori24}, Chapter 7).
	\item[(2)] More generally, let $S=(K,(\Omega,\cA,\nu),(|\cdot|_{\omega})_{\omega\in\Omega})$ be an adelic curve such that $K$ is countable, $\Omega$ is a subset of $M_{K}$ and $\cA$ is the Borel $\sigma$-algebra (and hence $\nu$ is a Borel measure). We further assume that for all $f\in K^{\times}$, the map $(\omega\in\Omega)\mapsto\log|f|_{\cdot}\in \bR$ is continuous. Then $S$ determines the data of a topological adelic curve, which is proper if so is $S$. This includes the examples of (1) above as well as all the constructions given in (\cite{ChenMori21}, \S 2.3-2.7), assuming that the admissible fibrations involved satisfy the above assumptions. 
	\item[(3)] In case further topological properties of the adelic space are required, e.g. (local) compactness for analytic purposes, one can make the following construction. Let $S=(K,(\Omega,\cA,\nu),(|\cdot|_{\omega})_{\omega\in\Omega})$ be an adelic curve satisfying the properties introduced in bullet (2) above. Recall in particular that $\Omega\subset M_{K}$. Consider the closure $\overline{\Omega}$ in $M_{K}$ (w.r.t. the topology of $M_{K}$), this is a compact Hausdorff topological space. Let $\overline{\nu}:=\iota_{\ast}\nu$ be the pushforward measure of $\nu$ by the inclusion $\iota: \Omega \hookrightarrow\overline{\Omega}$, this is a Borel measure on $\overline{\Omega}$ by construction. Let $f\in K^{\times}$. Then the function $\log|f|_{\cdot} : \overline{\Omega}\in [-\infty,+\infty]$ is continuous by definition of the topology of $M_{K}$. More generally, using e.g. (\cite{Engelking89}, Theorem 3.2.1), we can extend any continuous function $g:\Omega \to [-\infty,+\infty]$ to a continuous function $\overline{g}:\overline{\Omega}\to [-\infty,+\infty]$. Moreover, by classical measure theory (cf. e.g. Theorem 3.6.1 in \cite{Bogatchev07}), the function $\log|f|_{\cdot} : \overline{\Omega}\to[-\infty,+\infty]$ is $\overline{\nu}$-integrable. Therefore, $S$ determines a topological adelic curve $(K,\phi:\overline{\Omega}\to M_{K},\overline{\nu})$, whose adelic space is compact Hausdorff. 
	\item[(4)] Conversely, let $S=(K,\phi:\Omega\to M_{K},\nu)$ be a topological adelic curve, where $K$ is countable. For any $f\in K^{\times}$, the subset 
	\begin{align*}
	S(f) := \{\omega\in \Omega : |f|_{\omega} \in \{0,+\infty\}\} \subset \Omega	
	\end{align*}
is closed and $\nu(S(f)) = 0$. Now set $S:=\bigcup_{f\in K^{\times}} S(f)$. Since $K$ is countable, $S$ is a Borel subset of $\Omega$ such that $\nu(S)=0$. Denote $\Omega':=\Omega\setminus S$ and consider the restriction $S':=(K,\phi':\Omega'\to M_{K},\nu')$ of $S$ to $\Omega'$ (cf. \ref{subsub:restriction_adelic_structure}). Then $\phi'(\Omega')$ lies in the subset of $M_{K}$ consisting of absolute values on $K$. It is directly seen that $S'$ determines an adelic curve in the sense of Chen-Moriwaki. Moreover, this adelic curve is proper if $S$ is proper. In this case, we see that the corresponding GVF heights coincide.
\end{itemize}

\subsubsection{Nevanlinna theory of meromorphic functions on a compact disc}
\label{subsub:example_tac_Nevanlinnaç_compact_disc}

We fix $R>0$ and a real number $C_{R}>0$. We define a topological adelic curve $S_{R,C_{R}}=(K_R,\phi_R:\Omega_R\to M_{K_R},\nu_{R,C_{R}})$ (simply denoted by $S_{R}$ when $C_{R}=1$), where 
\begin{itemize}
	\item $K_R$ is the field of meromorphic functions on the closed disc $\overline{D(R)}:=\{z\in \bC : |z|_{\infty}\leq R\} \subset \bC$;
	\item $\Omega_R = \{z\in\bC:|z|_{\infty}<R\}\coprod \{z\in\bC :|z|_{\infty}=R\}$, where $\{z\in\bC:|z|_{\infty}<R\}$ is equipped with the discrete topology and $\{z\in\bC :|z|_{\infty}=R\}$ is equipped with the usual topology;
	\item the map $\phi_R : \Omega_R \to M_{K_R}$ is defined by
	\begin{align*}
	\forall z\in \Omega, \quad \phi_R(z) := \left\{\begin{matrix}
v_{z,\infty,1}=[(f\in K_R) \mapsto |f(z)|_{\infty}\in [0,+\infty]] &\text{ if } |z|_{\infty}=R,\\
(f\in K_R)\mapsto e^{-\ord(f,z)}\in \bR_{\geq 0} &\text{ if } |z|_{\infty}<R;\end{matrix}\right.
	\end{align*}
	\item $\nu_{R,C_{R}}$ is defined by
	\begin{align*}
	\forall z\in \Omega_{R,\um}, \quad \nu_{R,C_{R}}(\{z\}) := \left\{\begin{matrix}
\frac{1}{C_{R}}\log\frac{R}{|z|_{\infty}} &\text{ if } 0<|z|_{\infty}<R,\\
\frac{1}{C_{R}}\log R &\text{ if } z= 0,\end{matrix}\right.
\end{align*}		
	and $\nu_R$ is the Lebesgue with total mass $1/C_{R}$ on $\Omega_{R,\infty}$.
\end{itemize}
Moreover, results in \S \ref{subsub:example_integral_structures} imply that $S_R$ is an integral topological adelic curve w.r.t. the integral structure $(A_{R},\|\cdot\|_{R,\hyb})$.

Finally, the Jensen formula yields
\begin{align*}
\forall f\in K_R^{\times},\quad d_{R,C_{R}}(f) := d_{S_{R,C_{R}}}(f) = \frac{1}{C_{R}}\log|c(f,0)|_{\infty},
\end{align*} 
where $c(f,0)$ denotes the first non-zero coefficient in the Laurent series expansion of $f$ in $0$. 

\begin{remark}
\label{rem:ZR_space_NEvanlinna_complex_disc}
Let $X$ be a projective $K_{R}$-scheme. Then the algebraic Zariski-Riemann space $\ZR(X)_{S_{R}}$ is $\ZR(X/A_{R})$ and the analytic Zariski-Riemann space $\ZR(X)^{\an}_{S_{R}}$ is a subspace of $\ZR(X/A_{R})^{\an}$.
\end{remark}

\subsubsection{Nevanlinna theory of meromorphic functions on the complex plane}
\label{subsub:example_tac_Nevanlinna_complex}

The goal of this paragraph is to introduce a proper topological adelic curve structure on different subfields of the field of meromorphic functions on $\bC$. These subfields will be parametrised by functions $\eta : \bR_{>0} \to \bR_{>0}$ and ultrafilters on $\bR_{>0}$ satisfying additional properties. From now on, we fix a non-decreasing unbounded function $\eta : ]R_{0},+\infty[ \to \bR_{>0}$, for some $R_{0}\geq0$, and an ultrafilter $\cU$ on $]R_{0},+\infty[$ that avoids finite Lebesgue measure sets, i.e. it contains the filter from Example \ref{example:filters} (3), note that it is a free $\delta$-incomplete ultrafilter. By translation, we assume that $R_{0}=0$.

Let $R>0$, recall that we have defined a topological adelic curve $(\cM(\overline{D(R)}),\phi_{R}:\Omega_{R}\to M_{K},\nu_{R,\eta(R)})$. Define 
\begin{align*}
&\forall f\in \cM(\overline{D(R)}), \quad h_{\eta,R}(f) := \int_{\Omega_{R}} \log|f|_{\omega} \nu_{R,\eta(R)}(\diff\omega)=\frac{\log|c(f,0)|_{\infty}}{\eta(R)},\\
&\forall f\in \cM(\overline{D(R)}), \quad h_{\eta,R}(1,f) := \int_{\Omega_{R}} \log^{+}|f|_{\omega} \nu_{R,\eta(R)}(\diff\omega),
\end{align*}
where by convention, we have $h_{\eta,R}(0)=-\infty$. Using the usual notation in Nevanlinna theory as in \S \ref{sec:Nevanlinna_classical}, we have
\begin{align*}
\forall f\in \cM(\overline{D(R)}), \quad h_{\eta,R}(1,f) = \frac{T(R,f)}{\eta(R)}.
\end{align*} 
More generally, let $n$ be an integer. For $(f_{0},...,f_{n})\in \cM(\overline{D(R)})^{n+1}\setminus\{0\}$, define
\begin{align*}
h_{\eta,R}(f_{0},...,f_{n}) := \int_{\Omega_{R}} \log\displaystyle\max_{0\leq i\leq n}|f_{i}|_{\omega} \nu_{R,\eta(R)}(\diff\omega),
\end{align*}
Note that, for any $(f_{0},...,f_{n})\in \cM(\overline{D(R)})^{n+1}\setminus\{0\}$, we have the inequality
\begin{align}
\label{eq:height_sum_eta}
h_{\eta,R}(f_{0},...,f_{n}) \leq \displaystyle\sum_{i=0}^{n} h_{\eta,R}(1,f_{i}) + O(1/\eta(R)).
\end{align} 

For any $R>0$, we have a field extension $\cM(\overline{D(R)})/\cM(\bC)$ and, for any meromorphic function $f\in \cM(\bC)$, we denote its image in $\cM(\overline{D(R)})$ by $f_{R}$. Consider the subset 
\begin{align*}
\cM(\bC)_{\eta,\cU} := \{f\in \cM(\bC) : \displaystyle\lim_{R\to\cU} h_{\eta,R}(1,f)<+\infty\}.
\end{align*}
This is a subfield of $\cM(\bC)$ which contains the subfield of $\eta$-finite order functions (cf. \S \ref{subsub:asymptotic_height_Nevanlinna}). 

Let $n$ be an integer and $(f_{0},...,f_{n})\in (\cM(\bC)_{\eta,\cU})^{n+1}$. Define 
\begin{align*}
h_{\eta,\cU}(f_{0},...,f_{n}) := \displaystyle\lim_{i\to\cU} h_{\eta,R}(f_{0},...,f_{n})
\end{align*}
if $(f_{0},...,f_{n})\neq 0$ and $h_{\eta,\cU}(f_{0},...,f_{n}):=-\infty$ if $(f_{0},...,f_{n})=0$. Note that $h_{\eta,\cU}(f_{0},...,f_{n})<+\infty$ by (\ref{eq:height_sum_eta}). 

\begin{proposition}
\label{prop:height_eta_Nevanlinna}
The function $h_{\eta,\cU}$ defines a GVF height on $K_{\eta,\cU}$ with Archiemedean error $0$.
\end{proposition}

\begin{proof}
All the axioms except additivity, triangle inequality and product formula are clear from the definition. Let $f=(f_{0},...,f_{n})\in (\cM(\bC)_{\eta,\cU})^{n+1}\setminus\{0\}$ and $g=(g_{0},...,g_{m})\in (\cM(\bC)_{\eta,\cU})^{m+1}\setminus\{0\}$. Then
\begin{align*}
\forall R>0, \quad h_{\eta,R}(f\otimes g) = \int_{\Omega_{R}} \log\displaystyle\max_{i,j}|f_{i}g_{j}|_{\omega} \nu_{R,\eta(R)}(\diff\omega)= h_{\eta,R}(f)+h_{\eta,R}(g).
\end{align*}
Hence $h_{\eta,\cU}(f\otimes g) = h_{\eta,\cU}(f)+h_{\eta,\cU}(g)$ which proves the additivity axiom. 

For the triangle inequality, let $f=(f_{0},...,f_{n})\in (\cM(\bC)_{\eta,\cU})^{n+1}\setminus\{0\}$ and $g=(g_{0},...,g_{n})\in (\cM(\bC)_{\eta,\cU})^{n+1}\setminus\{0\}$. Then 
\begin{align*}
\forall R>0, \quad h_{\eta,R}(f\otimes g) = \int_{\Omega_{R}} \log\displaystyle\max_{i,j}|f_{i}g_{j}|_{\omega} \nu_{R,\eta(R)}(\diff\omega)\leq \max{h_{\eta,R}(f),h_{\eta,R}(g)}+\frac{\log 2}{\eta(R)}.
\end{align*}
Since $\lim_{R\to\cU}(\log 2)/\eta(R) = 0$, we get $h_{\eta,\cU}(f+g)\leq \max\{h_{\eta,\cU}(f),h_{\eta,\cU}(g)\} = h_{\eta,\cU}(f,g)$.

Finally, for any $f\in \cM(\bC)_{\eta,\cU}^{\times}$, 
\begin{align*}
h_{\eta,\cU}(f)=\displaystyle\lim_{R\to \cU} h_{\eta,R}(f)=\lim_{R\to \cU}\frac{\log|c(f,0)|_{\infty}}{\eta(R)}=0. 
\end{align*}
\end{proof}

\begin{definition}
\label{def:tac_Nevanlinna_complex_eta,ultrafilter}
Using Remark \ref{rem:tac_GVF}, we can represent the GVF $(\cM(\bC)_{\eta,\cU},h_{\eta,\cU})$ by a proper topological adelic curve $T_{\eta,\cU}=(\cM(\bC)_{\eta,\cU}, \mathrm{id} : M_{\cM(\bC)_{\eta,\cU}} \to M_{\cM(\bC)_{\eta,\cU}},\nu_{\eta,\cU})$, where $\nu_{\eta,\cU}$ is a Radon measure on $M_{\cM(\bC)_{\eta,\cU}}$. Moreover, since the Archimedean error is zero, the support $\Omega_{\eta,\cU} := \Supp(\nu_{\eta,\cU})$ is a closed subset of $M_{\cM(\bC)_{\eta,\cU}}$ that is contained in the ultrametric part $M_{\cM(\bC)_{\eta,\cU},\um}$. 
\end{definition}

\begin{example}
\label{example:tac_Nevanlinna_complex_eta_ultrafilter}
\begin{itemize}
\item[(1)] Assume that $\eta(R) =  o(\log(R))$. Then Proposition \ref{prop:criterion_rational_functions} yields
\begin{align*}
f\in \cM(\bC)_{\eta,\cU} \Leftrightarrow \displaystyle\lim_{R\to\cU} \frac{T(R,f)}{\eta(R)} <+\infty \Rightarrow \liminf_{R\to+\infty} \frac{T(R,f)}{\log(R)} = 0 \Leftrightarrow f\in \bC.
\end{align*}
In that case, the GVF $(\cM(\bC)_{\eta,\cU},h_{\eta,\cU})$ is the trivial GVF $(\bC,h_{\triv})$. 
	\item[(2)] Assume that $\eta=\log$. We directly have $\bC(T)\subset \cM(\bC)_{\eta,\cU}$. Conversely, Proposition \ref{prop:criterion_rational_functions} again yields
	\begin{align*}
	f\in \cM(\bC)_{\eta,\cU} \Rightarrow \liminf_{R\to+\infty} \frac{T(R,f)}{\log(R)} <+\infty \Leftrightarrow f\in \bC(T).
	\end{align*}
	Moreover, the GVF $(\cM(\bC)_{\eta,\cU},h_{\eta,\cU})$ can be represented by the proper topological adelic curve coming from the function field of $\mathbb{P}_{\bC}^{1}$ (cf. \cite{ChenMori}, \S 3.2.1).
\end{itemize} 
\end{example}

\begin{remark}
\label{rem:tac_Nevanlinna_complex_eta_ultrafilter}
\begin{itemize}
	\item[(1)] If $\cM(\bC)_{\eta,\cU}$ contains a transcendental function, and therefore $\lim_{R\to\cU} \eta(R)/\log(R)=+\infty$, it is not obvious to explicit a proper topological adelic curve representing the GVF $(M(\bC)_{\eta,\cU},h_{\eta,\cU})$. However, Definition \ref{def:tac_Nevanlinna_complex_eta,ultrafilter} together with Example \ref{example:tac_Nevanlinna_complex_eta_ultrafilter} (2) indicate that the adelic space of a proper topological adelic curve representing the GVF $\cM(\bC)_{\eta,\cU}$ will contain non-Archimedean pseudo-absolute values that are not associated with the order of vanishing at a complex point. The subset of all such pseudo-absolute values, together with the adelic measure, behaves like (the exponential of) the degree for rational functions. 
	\item[(2)] As the reader familiar with Nevanlinna theory might notice, the asymptotic height $h_{\eta,\cU}$ is not the typical asymptotic quantity that is classically considered: the order, or more generally the $\eta$-order, of a meromorphic function is arguably a richer notion. If one tries to mimic the above construction to view the ($\eta$-)order as a GVF height, one encounters a failure of the "additivity axiom". Although this may be fixed by considering GVF heights with target $\bR$ equipped with the tropical semi-ring structure instead of the usual ring structure, this construction seems, at this stage, too artificial. 
	\item[(3)] Consider the GVF $(M(\bC)_{\eta,\cU},h_{\eta,\cU})$, where $\eta(R)\geq R$. Then $\bC(T)\subset \cM(\bC)_{\eta}$ and $h_{\eta,\cU}(\bC(T))\equiv 0$. More generally, let $f\in \cM(\bC)$ such that $h_{\eta',\cU}(f)\in \bR_{>0}$, for some function $\eta':\bR_{>0}\to \bR_{>0}$. Then
	\begin{align*}
	h_{\eta,\cU}(f) \in \bR_{>0} \Leftrightarrow \displaystyle\lim_{R\to\cU}\frac{\eta(R)}{\eta'(R)} \in \bR_{>0}, \quad h_{\eta,\cU}(f) = 0 \Leftrightarrow \displaystyle\lim_{R\to\cU}\frac{\eta(R)}{\eta'(R)} = 0.
\end{align*}	 
In other words, if one needs to consider meromorphic functions whose heights do not grow the same way, only considering the asymptotic limit of the height w.r.t. to a test function $\eta$ loses too much information, as this procedure does not differentiate, for instance, a rational function from a transcendental function whose characteristic function is not growing "so fast". This suggests that a more refined formulation is required to capture the information contained in the characteristic function.
\end{itemize}
\end{remark}

\section{Families of topological adelic curves}
\label{sec:families_topological_adelic_curves}

In this section, we introduce to notion of families of topological adelic curves. It should be seen as a preliminary approach to including the analogy between Diophantine approximation and Nevanlinna theory in an Arakelov theoretic framework. Roughly speaking, we patch several topological adelic curves by means of an ultrafilter. In some sense, a family of topological adelic curves can be thought of as an ultraproduct of topological adelic curves, in the sense that the adelic field is a subfield of some ultraproduct and the adelic measure is a non-standard measure instead of a real measure. We motivate the abstract definition by giving our main example coming from Nevnalinna theory (\ref{sub:motivation_families}). Then we give several definitions and elementary properties related to families of topological adelic curves (\S \ref{sub:family_tac}-\ref{sub:morphism_family_tac}).

\subsection{Motivation: example in Nevanlinna theory}
\label{sub:motivation_families}

Let us first fix the needed notation. For any $R>0$, denote by $K_{R}$ the field of meromorphic functions on the closed disc $\overline{D(R)}$ and let $S_{R}=(K_{R},\phi_{R}:\Omega_{R}\to M_{K_{R}},\nu_{R})$ be the topological adelic curve introduced in \S \ref{subsub:example_tac_Nevanlinnaç_compact_disc}. Fix an ultrafilter on $\bR_{>0}$ avoiding sets of finite Lebesgue measure. Let
\begin{align*}
K_{\cU} := \displaystyle\prod_{\cU}K_{R}, \quad \Omega_{\cU} := \prod_{\cU}\Omega_{R}, \quad \nu_{\cU} := \prod_{\cU}\nu_{R},
\end{align*}
where the precise definition of these ultraproducts are given in \S \ref{subsub:ultraproduct_fields}-\ref{subsub:ultraproduct_measures}. Note that $\Omega_{\cU}$ is a Hausdorff space.

Let us construct a continuous map $\phi_{\cU} : \Omega_{\cU}\to M_{K_{\cU}}$. First, note that we have a map 
\begin{align*}
\iota : \prod_{\cU}M_{K_{R}} \to M_{K_{\cU}}
\end{align*}
defined by mapping any class $[(|\cdot|_{R})_{R>0}]\in\prod_{\cU}M_{K_{R}}$ to $([(f_{R})_{R>0}]\in K_{\cU}) \mapsto \lim_{R\to\cU}|f_{R}|_{R} \in [0,+\infty]$. Note that this definition makes sense by construction of ultraproducts. Since, for any $R>0$, the map $\phi_{R}:\Omega_{R}\to M_{K_{R}}$ are continuous, we have a continuous map $\Omega_{\cU}\to M_{K_{\cU}}$. We denote by $\phi_{\cU}$ the composition of this map with $\iota$. The claim below shows that $\phi_{\cU}$ is continuous.

\begin{claim}
\label{claim:ultraproduct_continuity}
The map $\iota$ is continuous.
\end{claim}

\begin{proof}
It suffices to show that for any $a\in K_{\cU}$, the map $\iota_{a}:(v\in\prod_{\cU}M_{K_{R}})\mapsto \iota(v)(a)\in [0,+\infty]$ is continuous. Fix an element $a=[(a_{R})_{R>0}]\in K_{\cU}$ and an open subset $V\subset[0,+\infty]$. For any $R>0$, set $U_{R}:=\{|\cdot|\in M_{K_{R}}: |a_{R}|\in V\}$. By definition, for any $[(|\cdot|_{R})_{R>0}]\in\prod_{\cU}M_{K_{R}}$, we have
\begin{align*}
[(|\cdot|_{R})_{R>0}] \in \iota_{a}^{-1}(V) \Leftrightarrow \{R>0 : |a_{R}|_{R} \in V\}\in \cU \Leftrightarrow [(|\cdot|_{R})_{R>0}] \in \prod_{\cU} U_{R}.
\end{align*}
Thus $\iota_{a}^{-1}(V)$ is open by definition of the topology of $\prod_{\cU}M_{K_{R}}$. 
\end{proof}

We denote $S_{\cU}:=(K_{\cU},\phi_{\cU} : \Omega_{\cU}\to M_{K_{\cU}},\nu_{\cU})$. Define the \emph{defect function} $d_{\cU} : f=[(f_{R})_{R>0}]\in K_{\cU}^{\times}\to [(d_{R}(f_{R}))_{R>0}]\in \prod_{\cU}\bR$, where we recall that
\begin{align*}
d_{R}(f_{R}) = \int_{\Omega_{R}}\log|f_{R}|_{\omega}\nu_{R}(\diff\omega) = \log|c(f_{R},0)|_{\infty}.
\end{align*}
Of course, it is not possible to expect any kind of control on the defect for arbitrary elements of $K_{\cU}$. Note that $\cM(\bC)$ is a subfield of $K_{\cU}$ and that for all $f\in \cM(\bC)^{\times}$, $d_{\cU}(f) = [(\log|c(f,0)|_{\infty})_{R>0}]$, namely the diagonal image of $\log|c(f,0)|_{\infty}$ in $\prod_{\cU}\bR$.

Let $n$ be an integer and $(f_{0},...,f_{n})\in \mathbb{A}^{n+1}(\cM(\bC))\setminus\{0\}$, we have defined in \S \ref{subsub:example_tac_Nevanlinna_complex} the \emph{height}
\begin{align*}
h_{R}(f_{0},...,f_{n}):= \int_{\Omega_{R}} \log\displaystyle\max_{0\leq i\leq n}|f_{i}|_{\omega} \nu_{R}(\diff\omega).
\end{align*} 
Now define the $\cU$\emph{-height} 
\begin{align*}
h_{\cU}(f_{0},...,f_{n}):= \left[\left(h_{R}(f_{0},...,f_{n})\right)_{R>0}\right] \in \displaystyle\prod_{\cU}\bR.
\end{align*}

Since the defect is non-identically zero, the $\cU$-height does not induce a $\cU$-height on projective spaces. As we mentioned earlier, we are interested in the asymptotic behaviour of the heights as $R\to+\infty$ and the latter diverge to $+\infty$ as long as we do not consider constant functions. To obtain a well-defined height machinery, we introduce the equivalence relation $\sim_{\fin}$ on $\prod_{\cU}\bR$ defined by
\begin{align*}
\forall f,g\in \displaystyle\prod_{\cU}\bR, \quad f\sim_{\fin}g \Leftrightarrow \left(\exists M\in \bR_{>0}, [(-M)_{R>0}]\leq f-g\leq [(M)_{R>0}]\right).
\end{align*}
In other words, $\sim_{\fin}$ identifies ultrareals that differ by a finite amount. It is straightforward to check that $\sim_{\fin}$ is indeed an equivalence relation that is compatible with the addition law. For any $f\in \cM(\bC)^{\times}$, we have
\begin{align*}
h_{\cU}(f)=d_{\cU}(f)\sim_{\fin} 0.
\end{align*}
Therefore, the $\cU$-height $h_{\cU}:\mathbb{A}(\cM(\bC))\to\prod_{\cU}\bR$ induces a well-defined map $\mathbb{P}(\cM(\bC))\to\prod_{\cU}\bR/\sim_{\fin}$ denoted again by $h_{\cU}$ by abuse of notation. 

Another important feature of the target space $\prod_{\cU}\bR/\sim_{\fin}$ of our heights is that it carries a natural total order compatible with the addition law and the $\bR$-action. Let $\overline{f},\overline{g}\in \prod_{\cU}\bR/\sim_{\fin}$ represented respectively by some $f,g\in\prod_{\cU}\bR$. We write $\overline{f}\leq \overline{g}$ if there exists $M\in \bR$ such that $a\leq b + [(M)_{R>0}]$.

\begin{claim}
\label{claim:total_order_target_height}
$\leq$ defines a total order on $\prod_{\cU}\bR/\sim_{\fin}$. Moreover, this order is compatible with the addition law and the $\bR$-action.
\end{claim}

\begin{proof}
Let us justify first that $\leq$ is well defined. Let $\overline{f},\overline{g},f,g,M$ as above and assume that $\overline{f},\overline{g}$ are also represented respectively by $f',g'\in\prod_{\cU}\bR$, i.e. $f-f'$ and $g-g'$ are finite. Let $M'>0$ such that $[(-M')_{R>0}]\leq f-f',g-g'\leq [(M')_{R>0}]$. Then $f'\leq g'+[(M+2M')_{R>0}]$. Thus $\leq$ is well defined. 

The fact that $\leq$ is an order relation satisfying the stated compatibility relations is straightforward to check and we only justify that the order is total. Let $\overline{f},\overline{g}\in \prod_{\cU}/\sim_{\fin}$, represented respectively by $[(f_{R})_{R>0}],[(g_{R})_{R>0}]$. Let $M\in \bR$. If $\{R>0:f_{R}\leq g_{R}+M\}\in \cU$, then $\overline{f}\leq\overline{g}$. Otherwise, since $\cU$ is an ultrafilter, $\{R>0:f_{R}> g_{R}+M\}\in \cU$, which implies that $\overline{g}\leq \overline{f}$.
\end{proof}

\subsection{Family of topological adelic curves}
\label{sub:family_tac}

We now give the abstract definition of a family of topological adelic curves, building on the observations of \S \ref{sub:motivation_families}.

\begin{definition}
\label{def:family_of_tac}
A \emph{family of topological adelic curves} is the data $\mathbf{S}=(I,\cU,(S_{i})_{i\in I},K)$, where
\begin{itemize}
	\item $I$ is an infinite set called the \emph{index set} of $\mathbf{S}$;
	\item $\cU$ a free ultrafilter on $I$ called the \emph{ultrafilter} of $\mathbf{S}$;
	\item for any $i\in I$, $S_{i}=(K_{i},\phi_{i}:\Omega_{i}\to M_{K_{i}},\nu_{i})$ is a topological adelic curve;
	\item $K$ is a subfield of $K_{\mathbf{S}}:=\prod_{\cU}K_{i}$ called the \emph{base field} of $\mathbf{S}$. 
\end{itemize}

The \emph{defect function} of a family of topological adelic curves $\mathbf{S}=(I,\cU,(S_{i})_{i\in I},K)$ is a function $d_{\mathbf{S}}:K^{\times}\to \prod_{\cU}\bR$ defined by
\begin{align*}
\forall f=[(f_{i})_{i\in I}]\in K^{\times}, \quad d_{\mathbf{S}}(f) := [(d_{S_{i}}(f_{i}))_{i\in I}].
\end{align*}

Moreover, a family of topological adelic curves $\mathbf{S}=(I,\cU,(S_{i})_{i\in I},K)$ is called \emph{proper} if the following \emph{product formula} holds:
\begin{align*}
\forall f\in K^{\times}, \quad d_{\mathbf{S}}(f) = 0.
\end{align*}
\end{definition}

\begin{example}
\label{example:families_of_tac}
\begin{itemize}
	\item[(1)] Let $S=(K,\phi : \Omega \to M_{K},\nu)$ be a topological adelic curve. Let $I$ be any infinite countable set and $\cU$ be an arbitrary free ultrafilter on $I$. We consider the field $K_{\cU}:=\prod_{\cU}K$. The data $\mathbf{S}=(I,\cU,(S)_{i\in I},K_{\cU})$ is a family of topological adelic curves, which is proper iff $S$ is proper. This example should be seen as the first step in the implementation of classical Diophantine approximation in the language of topological adelic curves.
	\item[(2)] The construction of \S \ref{sub:motivation_families} yields a family of topological adelic curves $\mathbf{S}=(\bR_{>0},\cU,(S_{R})_{R>0},\cM(\bC))$. As we saw earlier, $\mathbf{S}$ is not proper.
\end{itemize}
\end{example}

\subsection{Asymptotically proper family of topological adelic curves}

\begin{definition}
\label{def:family_of_tac_proper}
Let $\mathbf{S}=(I,\cU,(S_{i})_{i\in I},K)$ be a family of topological adelic curves. Let $\sim$ be an equivalence relation on $\prod_{\cU}\bR$ which is compatible with the additive group structure. We say that the family $\mathbf{S}$ is \emph{asymptotically proper} w.r.t. $\sim$ if 
\begin{align*}
\forall f\in K^{\times}, \quad d_{\cdot}(f) \sim 0.
\end{align*}
\end{definition}

\begin{example}
\label{example:family_of_tac_proper}
\begin{itemize}
	\item[(1)] A family of topological adelic curves is proper iff it is asymptotically proper w.r.t. $=$.
	\item[(2)] Consider the family $\mathbf{S}=(\bR_{>0},\cU,(S_{R})_{R>0},\cM(\bC))$ and the equivalence relation $\sim_{\fin}$ introduced in \S \ref{sub:motivation_families}. Then $\mathbf{S}$ is asymptotically proper w.r.t. $\sim_{\fin}$.
\end{itemize}
\end{example}

\subsection{Adelic space of a family of topological adelic curves}
\label{sub:adelic_space_family_tac}

\begin{definition}
\label{label:adelic_space_family_tac}
Let $\mathbf{S}=(I,\cU,(S_{i}=(K_{i},\phi_{i}:\Omega_{i}\to M_{K_{i}},\nu_{i}))_{i\in I},K)$ be a family of topological adelic curves. The \emph{adelic space} of $\mathbf{S}$ is defined as the ultraproduct $\Omega_{\mathbf{S}}:=\prod_{\cU}\Omega_{i}$. Note that this is a Hausdorff topological space. 

Moreover, using the construction done in \S \ref{sub:motivation_families} (which transposes \emph{mutatis mutandis} in this case), we obtain a continuous map $\phi_{\mathbf{S}}:\Omega_{\mathbf{S}}\to M_{K}$ called the \emph{structure morphism} of $\mathbf{S}$ and a ultraproduct measure (cf. \S \ref{subsub:ultraproduct_measures}) $\nu_{\mathbf{S}} : \cB(\Omega_{\mathbf{S}})\to\prod_{\cU}[0,+\infty]$ called the \emph{adelic measure} of $\mathbf{S}$. 

We also denote by $L^{1}(\Omega_{\mathbf{S}},\nu_{\mathbf{S}})$ the ultraproduct vector space $\prod_{\cU}L^{1}(\Omega_{i},\nu_{i})$ of $\nu_{\mathbf{S}}$\emph{-integrable functions} (modulo the null equivalence). For any $f=[(f_{i})_{i\in I}]\in L^{1}(\Omega_{\mathbf{S}},\nu_{\mathbf{S}})$, we set
\begin{align*}
\int_{\Omega_{\mathbf{S}}}f\diff\nu_{\mathbf{S}} := \left[\left(\int_{\Omega_{i}}f_{i}\diff\nu_{i}\right)_{i\in I}\right]\in \displaystyle\prod_{\cU}\bR
\end{align*}
and call it the \emph{integral} of $f$ w.r.t. $\nu_{\mathbf{S}}$. One can check in the usual fashion that it is well-defined and linear.
\end{definition}

\subsection{Morphism between families of topological adelic curves}
\label{sub:morphism_family_tac}

\begin{definition}
\label{def:morphism_family_tac}
Let $\mathbf{S}=(I,\cU,(S_{i}=(K_{i},\phi_{i}:\Omega_{i}\to M_{K_{i}},\nu_{i}))_{i\in I},K),\mathbf{S'}=(I,\cU,(S'_{i}=(K'_{i},\phi'_{i}:\Omega'_{i}\to M_{K'_{i}},\nu'_{i}))_{i\in I},K')$ be two families of topological adelic curves. A \emph{morphism} $\alpha : \mathbf{S'} \rightarrow \mathbf{S}$ is the data $(\alpha^{\sharp},\alpha_{\sharp},I_\alpha)$, where
\begin{itemize}
	\item $\alpha^{\sharp}$ is a commutative diagram of fields 
	\begin{center}
\begin{tikzcd}
K_{\mathbf{S}} \arrow[r] & K_{\mathbf{S'}} \\
K \arrow[u] \arrow[r]    & K' \arrow[u]   
\end{tikzcd};
	\end{center}
	\item $\alpha_{\sharp} : \Omega_{\mathbf{S'}} \rightarrow \Omega_{\mathbf{S}}$ is a continuous map inducing a commutative diagram 
\begin{center}
\begin{tikzcd}
\Omega_{\mathbf{S'}} \arrow[r, "\phi_{\mathbf{S'}}"] \arrow[d, "\alpha_{\sharp}"] & M_{K'} \arrow[d, "\pi_{K'/K}"] \\
\Omega_{\mathbf{S}} \arrow[r, "\phi_{\mathbf{S}}"]                                                     & M_{K}                       
\end{tikzcd}.
\end{center}
Moreover, the \emph{direct image} of $\nu_{\mathbf{S'}}$ by $\alpha_{\sharp}$ is assumed to be equal to $\nu_{\mathbf{S}}$. Namely, for any $E\in \cB(\Omega_{\mathbf{S}})$, we have $\nu_{\mathbf{S}}(E)=\nu_{\mathbf{S'}}(\alpha_{\sharp}^{-1}(E))$;
	\item $I_{\alpha} : L^1(\Omega_{\mathbf{S'}},\nu_{\mathbf{S'}}) \rightarrow L^1(\Omega_{\mathbf{S'}},\nu_{\mathbf{S'}})$ is a linear map such that, for all $g\in L^1(\Omega_{\mathbf{S'}},\nu_{\mathbf{S'}})$, we have
	\begin{align*}
	\int_{\Omega_{\mathbf{S}}} I_{\alpha}(g)\diff\nu_{\mathbf{S}} = \int_{\Omega_{\mathbf{S'}}} g\diff\nu_{\mathbf{S'}},
	\end{align*}
and which, for all $f \in L^{1}(\Omega_{\mathbf{S}},\nu_{\mathbf{S}})$, maps the equivalence class of $f\circ \alpha_{\sharp}$ to the class of $f$.
\end{itemize}
\end{definition}

\begin{remark}
\label{rem:morphism_family_tac}
Let $\alpha : \mathbf{S'} \to \mathbf{S}$ be as in Definition \ref{def:family_of_tac}. One can check directly from the definition that $\cU$-almost everywhere, $\alpha$ determines a morphism $\alpha_{i} : S'_{i}\to S_{i}$ of topological adelic curves in the sense of Definition \ref{def:topological_adelic_curve}.
\end{remark}

\section{Algebraic covering of topological adelic curves}
\label{sec:algebraic_covering_tac}

The goal of this section is to associate to any topological adelic curve $S$ with adelic field $K$ and any algebraic extension $L/K$, a topological adelic curve $S\otimes_{K}L$ with adelic field $L$ sharing many properties with $S$ (\S \ref{sub:finite_separable_extension_tac}-\ref{sub:algebraic_extension}). If the extension $L/K$ is Galois, then from the topological point of view the topological adelic curve $S\otimes_{K}L$ behaves like a covering of $S$ (\S \ref{sub:action_of_Galois_group_tac}). We also introduce the analogue notion for families of topological adelic curves (\S \ref{sub:algebraic_covering_family_tac}).

\subsection{Finite separable extension}
\label{sub:finite_separable_extension_tac}

Throughout this subsection, we fix a topological adelic curve $S_{K}=(K,\phi_K : \Omega_K \to M_K,\nu_K)$ and a finite separable extension $L/K$. We will construct a canonically determined adelic structure on $L$ whose corresponding adelic curve maps to $S$.

\subsubsection{Case of a general topological adelic curve}

Let $\Omega_L := \Omega_K \times_{M_{K}} M_{L}$ equipped with the fibre product topology. Denote by $\phi_L: \Omega_L \to M_L$ the pullback of $\phi_K$. Then $\Omega_L$ is a Hausdorff space, which is also locally compact if so is $\Omega_{K}$. By definition of the fibre product topology, for any $f\in L^{\times}$, the map $\omega\in \Omega_L \mapsto \log|f|_{\omega} \in [-\infty,+\infty]$ is continuous. 

\begin{proposition}
\label{prop:projection_morphism_finite_separable_extension}
The restriction morphism $\pi_{L/K} : M_L \to M_K$ is surjective, proper, with finite fibres.
\end{proposition}

\begin{proof}
$\pi_{L/K}$ is proper since it is continuous and both $M_K$ and $M_L$ are compact Hausdorff. $\pi_{L/K}$ is surjective by construction and has finite fibres by Proposition \ref{prop:formula_extension_sav}.
\end{proof}

We now adapt the constructions in (\cite{ChenMori}, \S 3.3) to define a measure $\nu_L$ on $\Omega_L$. For any $\omega\in \Omega_K$, let $\Omega_{L,\omega}$ denote the fibre of $\omega$ through the restriction map $\pi_{L/K} : \Omega_L \to \Omega_K$. Let $\omega\in\Omega$. Denote by $I_\omega$ the cardinality of the set of valuation rings on $L$ extending $A_{\omega}$. 

\begin{proposition}
Using the same notation as in Proposition \ref{prop:formula_extension_sav}, the measure $\Omega_{L,\omega}$ defined by
\begin{align*}
\forall x\in \Omega_{L,\omega},\quad  \bP_{L,\omega}(\{x\}) := \frac{1}{|I_\omega|} \frac{[\widehat{\kappa_x}:\widehat{\kappa_{\omega}}]}{[\kappa_x:\kappa_{\omega}]},
\end{align*}
is a probability measure. 
\end{proposition}

\begin{proof}
It is a direct consequence of (\ref{eq:formula_extension_sav}).
\end{proof}

\begin{definition}
\label{def:disintegration_kernel}
Let $f : \Omega_{L} \to \bR$. Define $I_{L/K}(f) : \Omega_{K}\to \bR$ by 
\begin{align*}
I_{L/K}(f)(\omega) := \displaystyle\sum_{x\in \pi^{-1}_{L/K}(\omega)} \bP_{L,\omega}(\{x\})f(x).
\end{align*}
\end{definition}

The $\sigma$-algebra defined in (\cite{ChenMori}, \S 3.3) coincides with the Borel $\sigma$-algebra on $\Omega_L$. Therefore, we can reproduce \emph{mutatis mutandis} the arguments of \emph{ibid.} to obtain a measure $\nu_L$ such that the equality
\begin{align}
\label{eq:measure_finite_separable_extension}
\nu_L(A) = \int_{\Omega_K} \displaystyle\sum_{x\in \Omega_{L,\omega}}\bP_{L,\omega}(\{x\})\mathbf{1}_A(x) \diff \nu_K,
\end{align}
holds for any Borel subset $A \subset \Omega_L$, where $\mathbf{1}_{A}$ denotes the characteristic function of $A$. Furthermore, $\nu_L$ satisfies all the conditions of (\cite{ChenMori}, Theorem 3.3.7). Thus, we deduce the following proposition.

\begin{proposition}
\label{prop:finite_separable_covering_tac}
We use the above notation. Then $S_L = S \otimes_K L := (L,\phi_L : \Omega_L \to M_L,\nu_L)$ is a topological adelic curve. We further have a morphism $\alpha_{L/K}: S_L \to S_K$. Moreover, for any $f\in L^{\times}$, we have
\begin{align}
\label{eq:product_formula_finite_separable_extension_tac}
d_{S_L}(f) = \frac{1}{[L:K]} d_{S_K}(N_{L/K}(f)).
\end{align}
In particular, if $S_K$ is proper, then $S_L$ is proper.
\end{proposition}

\begin{proof}
We only need to justify that we have a morphism of topological adelic curves $\alpha_{L/K} : S_L \to S_K$. Denote by $\alpha_{L/K}^{\sharp} : K \to L$ the field extension. Define $\alpha_{L/K,\sharp} : \Omega_L \to \Omega_K$ as the pullback of $\pi_{L/K}:M_L \to M_K$ by $\varphi_K : \Omega_K \to M_K$. By construction, the diagram in Definition \ref{def:topological_adelic_curve} is commutative. Moreover, (\cite{ChenMori}, Theorem 3.3.7 (3)) implies that $\alpha_{L/K,\sharp,\ast} \nu_L = \nu_K$. Finally, from (\cite{ChenMori}, Theorem 3.3.7 (1-2)), we get that $I_{L/K}$ (cf. Definition \ref{def:disintegration_kernel}) is a disintegration kernel of $\alpha_{L/K,\sharp}$.
\end{proof}

\begin{example}
\label{example:finite_separable_covering_Gubler}
In (\cite{Gubler97}, Example 2.8), Gubler gives a construction of an extension of the $M$-field structure on the field of meromorphic functions of a complex closed disc. The counterpart of this construction in the language presented here is as follows. 

Fix $R>0$ and consider the topological adelic curve $S_{R}=(K_{R},\phi_{R}:\Omega_{R}\to M_{K_{R}},\nu_{R})$ from \S \ref{subsub:example_tac_Nevanlinnaç_compact_disc}. Let $L/K_{R}$ be a finite extension. Let $S_{L}:= S_{R}\otimes_{K_{R}}L = (L,\phi_{L}:\Omega_{L}\to M_{L},\nu_{L})$ the topological adelic curve from Proposition \ref{prop:finite_separable_covering_tac}. In \cite{Gubler97}, showed that one can realise $L$ as the field of meromorphic function of a branched holomorphic covering $\pi: X_{L}\to \Omega_{R}$, where we consider the complex topology on $\Omega_{R}$. We claim that $\Omega_{L}$ and $X_{L}$ are in bijection (as sets). Indeed, the ultrametric points of $\Omega_{L}$ are absolute values associated with discrete valuations on $L$ that must be equivalent to the order of vanishing at a point of $X_{L}$ by (\cite{Isssa66}, Theorem III). Since $\mathrm{Int}(X_{L})=\pi^{-1}(\Omega_{R,\um})$, we obtain a bijection for the ultrametric points. By the same argument, we obtain the bijection for the Archimedean points. It is now straightforward to check that the measure theoretic features in (\cite{Gubler97}, Example 2.8) coincide with the ones we consider here.
\end{example}

\subsubsection{Case of an integral topological adelic curve}

Assume that $S_K$ is integral. Let $(A,\|\cdot\|_A)$ be the underlying integral structure and let $V_K=\cM(A,\|\cdot\|_A)$ be the integral space. Denote by $B$ the integral closure of $A$ in $L$. As recalled in \S \ref{subsub:prelim_integral_structures}, there exists a norm $\|\cdot\|_B$ on $B$ such that $(B,\|\cdot\|_B)$ is a tame integral structure for $L$ and $V_L := \cM(B,\|\cdot\|_B) = \pi_{L/K}^{-1}(\cM(A,\|\cdot\|_{A}))$. This is a compact Hausdorff topological space and we denote by $\pi_{L/K} : V_L \to V_K$ the (continuous) restriction morphism. Now the following proposition implies that the image of structural morphism $\phi_L : \Omega_L \to M_L$ lies in $V_L$ and $S \otimes_{K} L = (L,\phi_L : \Omega_L \to V_L,\nu_L)$ is an integral topological adelic curve.

\begin{proposition}
\label{prop:projection_morphism_finite_separable_extension_integral}
The restriction morphism $\pi_{L/K} : V_L \to V_K$ is surjective, proper, with finite fibres.
\end{proposition}

\begin{proof}
$\pi_{L/K}$ is proper since it is continuous and both $V_L$ and $V_K$ are compact Hausdorff. $\pi_{L/K}$ is surjective by construction and has finite fibres by Proposition \ref{prop:projection_morphism_finite_separable_extension}.
\end{proof}

\subsection{Finite extension}
\label{sub:finite_extension_tac}

Throughout this subsection, we fix a topological adelic curve $S_{K}=(K,\phi_K : \Omega_K \to M_K,\nu_K)$ and a finite extension $L/K$. Denote by $K'/K$ the separable closure of $K$ in $L$ and let $S_{K'} := S \otimes_K K' := (K',\phi_{K'}:\Omega_{K'}\to M_{K'},\nu_{K'})$ be the topological adelic curve constructed in \S \ref{sub:finite_separable_extension_tac}. 

\begin{proposition}
\label{prop:purely_inseparable_extension_tac}
\begin{itemize}
	\item[(1)] The spaces $M_L$ et $M_{K'}$ are homeomorphic.

	\item[(2)] There exist a topological adelic curve $S_L := (L,\phi_L:\Omega_{L}\to M_{L},\nu_{L})$ and a morphism $S_L \to S_{K'} \to S$ of topological adelic curves. 
	\item[(3)] Assume that $S_K$ is an integral topological adelic curve and denote by $(A,\|\cdot\|_A)$ its underlying integral structure and by $V_K$ its integral space. Let $A'$ and $A_L$ denote respectively the integral closures of $A$ in $K'$ and $L$. We respectively endow $A'$ and $A_L$ with the norm $\|\cdot\|_{A'}$ and $\|\cdot\|_{A_L}$ from Proposition \ref{prop:algebraic_extension_of_tame_spaces}. Then $V_L := \cM(A_L)$ and $V_{K'} = \cM(A')$ are homeomorphic and the topological adelic curve $S_L$ is integral.
\end{itemize}
\end{proposition}

\begin{proof}
\textbf{(1)} Corollary \ref{cor:purely_inseparable_extension_sav} provides a continuous bijection $\pi_{L/K'} : M_L \to M_{K'}$. Since both $M_L$ and $M_{K'}$ are compact Hausdorff, we obtain (1).

\textbf{(2)} The homeomorphism $\pi_{L/K'} : M_L \to M_{K'}$ induces by pullback a homeomorphism $\Omega_L:= \pi_{L/K'}^{-1}(\Omega_{K'})\to \Omega_{K'}$ and a continuous map $\phi_L :\Omega_{L} \to M_{L}$. Denote by $\nu_{L}$ the pushforward of $\nu_{K'}$ by the inverse of the latter homeomorphism. Then $S_L := (L,\phi_L : \Omega_L \to M_L,\nu_{L})$ is a topological adelic curve.

\textbf{(3)} $(A_L,\|\cdot\|_{A_L})$ is an integral structure for $L$ (cf. Proposition \ref{prop:algebraic_extension_of_tame_spaces}). Moreover, the restriction map $\pi_{L/K'}: \cM(A_L) \to \cM(A')$ is a restriction of the homeomorphism of (1) which is surjective. Hence it is a homeomorphism. The last part of the assertion follows directly.
\end{proof}

We now state a compatibility result for successive extensions of the base field.

\begin{proposition}
\label{prop:successive_finite_extensions_tac}
Let $K_2/K_1/K$ be successive finite extensions. Then we have a canonical isomorphism of topological adelic curves
\begin{align*}
\left(S_K \otimes_{K} K_1 \right) \otimes_{K_1} K_2 \cong S_K \otimes_{K} K_2
\end{align*}
making the below diagram commute.
\begin{center}
\begin{tikzcd}
S_K \otimes_{K} K_2 \arrow[r, "\alpha_{K_2/K_1}"] \arrow[rd, "\alpha_{K_2/K}"'] & S_K \otimes_{K} K_1 \arrow[d, "\alpha_{K_1/K}"] \\
                                                                                & S_K                                            
\end{tikzcd}.
\end{center}
Moreover, if $S_K$ is integral, the above isomorphism and diagram lie in the category of integral topological adelic curves.
\end{proposition}

\begin{proof}
Let $\left(S_K \otimes_{K} K_1 \right) \otimes_{K_1} K_2 = (K_2,\phi'_2 :\Omega'_2 \to M_{K_{2}},\nu'_2)$ and $S_K \otimes_{K} K_2 = (K_2,\phi_2 : \Omega_2 \to M_{K_{2}},\nu_2)$. The definition of $\Omega_{2},\Omega'_{2}$ implies that there exists a canonical homeomorphism $\Omega_{2} \cong \Omega'_{2}$. From (\cite{ChenMori}, (3.15)), we deduce that $\nu_2$ can be identified with $\nu'_2$ via the previous homeomorphism. Commutativity of the diagram follows from Proposition \ref{prop:purely_inseparable_extension_tac} (3).

In the integral case, we denote respectively by $A_1$ and $A_2$ the integral closures of $A$ in $K_1$ and $K_2$. Note that the construction of the extension of the norm $\|\cdot\|_A$ from Proposition \ref{prop:algebraic_extension_of_tame_spaces} is compatible with successive extension. Since $V'_2$ is obtained by considering the integral closure of $A_1$ in $L$, which is equal to $A_2$, we obtain $V_2 = V'_2$. 
\end{proof}

\subsection{Algebraic extension}
\label{sub:algebraic_extension}

Throughout this subsection, we fix a topological adelic curve $S_{K}=(K,\phi_K : \Omega_K \to M_K,\nu_K)$ and an algebraic extension $L/K$. Let $\cE_{L/K}$ be the set of all sub-extensions $L/K'/K$ with $K'/K$ finite. It is a directed set with respect to the inclusion relation and we have 
\begin{align*}
L = \displaystyle\bigcup_{K'/K \in \cE_{L/K}} K'.
\end{align*}
Then Proposition \ref{prop:successive_finite_extensions_tac} ensures that we have a cofiltered inverse system of topological adelic curves $(S_{K'})_{K'\in \cE_{L/K}}$ whose arrows are denoted by $\pi_{K''/K'} : S_{K''} \to S_{K'}$. We will prove that the inverse limit of this inverse system exists and is a topological adelic curve with field $L$.

Proposition \ref{prop:successive_finite_extensions_tac} yields an inverse system  $(S_{K'}=(K',\phi_{K'}: \Omega_{K'}\to M_{K'},\nu_{K'}))_{K'\in \cE_{L/K}}$ in $\TAC$ which induces an inverse system of Cartesian diagrams of the form
\begin{center}
\begin{tikzcd}
\Omega_{K'} \arrow[r, "\phi_{K'}"] \arrow[d] & M_{K'} \arrow[d, "\pi_{K'/K}"] \\
\Omega_{K} \arrow[r, "\phi_{K}"]             & M_{K}                        
\end{tikzcd}
\end{center}
indexed by $\cE_{L/K}$ and whose vertical arrows are proper. Hence its inverse limit can be written as
\begin{center}
\begin{tikzcd}
\Omega_{L} \arrow[r, "\phi_{L}"] \arrow[d] & M_{L} \arrow[d, "\pi_{L/K}"] \\
\Omega_{K} \arrow[r, "\phi_{K}"]           & M_{K}                        
\end{tikzcd},
\end{center}
where $\Omega_{L} := \varprojlim_{K'\in \cE_{L/K}} \Omega_{K'}$. $\Omega_{L}$ is a Hausdorff topological space. Moreover, if $\Omega_{K}$ is locally compact, since the map $\pi_{L/K}$ is proper and the $\Omega_{K'}$ are locally compact Hausdorff topological spaces, $\Omega_{L}$ is locally compact.

We can now adapt the arguments from (\cite{ChenMori}, \S 3.4) to construct a Borel measure $\nu_L$ on $\Omega_L$ and a disintegration kernel for $\pi_{L/K} : \Omega_{L} \to \Omega_{K}$.

Hence we obtain the following proposition.

\begin{proposition}
\label{prop:algebraic_extension_tac}
We use the same notation as above. Then $S_L := S_K \otimes_K L := (L,\phi_{L}:\Omega_{L} \to M_{L},\nu_L)$ is a topological adelic curve and we have an isomorphism
\begin{align*}
\varprojlim_{K'\in\cE_{L/K}} S_K \otimes_K K'.
\end{align*}
Moreover, if $S_K$ is proper, then $S_L$ is proper. Finally, if $S_K$ is integral, then $S_L$ is integral. Finally if $\Omega_{K}$ is assumed to be locally compact, $\Omega_{L}$ is again locally compact.
\end{proposition}

\begin{proof}
Except for the assertion about integrality, the proposition follows from the above paragraph. We now assume that $S_K$ is integral. Let $(A,\|\cdot\|_A)$ denote its underlying integral structure and let $V_K$ be its integral space. Let $B$ denote the integral closure of $A$ in $L$. Then Proposition \ref{prop:algebraic_extension_of_tame_spaces} and its proof yield a norm $\|\cdot\|_{B}$ as well as an isomorphism 
\begin{align*}
V_L := \cM(B,\|\cdot\|_{B}) \cong \displaystyle\varprojlim_{K'\in \cE_{L/K}} \cM(A_{K'},\|\cdot\|_{A_{K'}}) =: V_{K'},
\end{align*}
where, for all $K'\in \cE_{L/K}$, $(A_{K'},\|\cdot\|_{A_{K'}})$ denotes the extension of $(A,\|\cdot\|_{A})$ over $K'$. Therefore, we see that the image of the structural morphism $\phi_L : \Omega_L \to M_L$ lies in $V_L$ and $S_L$ is integral.
\end{proof}

\begin{example}
\label{example:algebraic_covering_Gubler}
We give an example in Nevanlinna theory building on Example \ref{example:finite_separable_covering_Gubler}. Let $R>0$ and consider the topological adelic curve $S_{R}=(K_{R},\phi_{R}:\Omega_{R}\to M_{K_{R}},\nu_{R})$ from \S \ref{subsub:example_tac_Nevanlinnaç_compact_disc}. Let $L/K_{R}$ be an algebraic extension. We have seen that in the case where $L$ is finite, our construction of the topological adelic curve $S \otimes_{K} L$ recovered the construction of (\cite{Gubler97}, Example 2.8). In the general case, we can check that the adelic space of $S\otimes_{K} L$ coincides with the $M$-field constructed by Gubler and that the same measure theoretic properties are satisfied. We refer to \emph{loc. cit.} for the interested reader. 
\end{example}

\subsection{Action of the Galois group}
\label{sub:action_of_Galois_group_tac}

Throughout this subsection, we fix a topological adelic curve $S=(K,\phi:\Omega\to M_K,\nu)$ and an algebraic extension $L/K$ whose group of $K$-linear automorphisms is denoted by $\Aut(L/K)$. Let $S_L := (L,\phi_{L} : \Omega_{L}\to M_{L},\nu_{L})$ be the topological adelic curve $S \otimes_K L$ defined in \S \ref{sub:algebraic_extension}.

\begin{proposition}
\label{prop:action_of_Galois_group_algebraic_extension}
\begin{itemize}
	\item[(i)]The action of $\Aut(L/K)$ on $M_L$ introduced in Proposition \ref{prop:action_Galois} (1) induces continuous and proper actions of $\Aut(L/K)$ on $M_L$ and $\Omega_L$. Moreover, if $S$ is integral and $V_L$ denotes the integral space of $S_L$, the action of $\Aut(L/K)$ on $M_L$ induces a continuous and proper action on $V_L$.
	\item[(ii)] We assume that $L/K$ is Galois and that, for any $v\in M_K$, the residue field $\kappa_{v}$ is perfect. Then for any $\omega\in\Omega$, the actions of $\Aut(L/K)$ on $\Omega_{L,\omega}$ and $M_{L,\phi(\omega)}$ are transitive. Moreover, if $S$ is integral with underlying global space of pseudo-absolute values $V$. Then for any $v\in V$, $\Aut(L/K)$ acts transitively on $V_{L,v}$, where $V_{L}$ denotes the underlying global space of pseudo-absolute values of $S_{L}$.
	\item[(iii)] We use the same assumptions as in (ii). Then we have homeomorphisms 
	\begin{align*}
	\Omega_L/\Aut(L/K) \cong \Omega, \quad M_{L}/\Aut(L/K) \cong M_K.
	\end{align*}
	 Moreover, if $S$ is integral, we have a homeomorphism 
	 \begin{align*}
	 V_L/\Aut(L/K) \cong V.
	 \end{align*}
\end{itemize}
\end{proposition}

\begin{proof}
By considering the trivial actions on $V$ and $\Omega$, we see that $\Aut(L/K)$ induces actions on $\Omega_L$ and $V_L$. Let us show that these actions are continuous. First, assume that $L/K$ is finite. Since $\Aut(L/K)$ is discrete, it is enough to prove that, for any $\tau\in \Aut(L/K)$, for any $a\in L$, the map $(\va_x\in M_L) \mapsto |a|_{\tau(x)}$ is continuous. This is clear by definition of the topology on $M_{L}$ and since $|a|_{\tau(x)} = |\tau(a)|_x$. If now $L/K$ is infinite, let us show that the map 
\begin{align*}
\fonction{\alpha_{L/K}}{\Aut(L/K)\times M_L}{M_L}{(\tau,v)}{v\circ\tau}
\end{align*}
is continuous. By definition of the topology on $M_L$, it is enough to prove that for any intermediate extension $L/K'/K$ with $K'/K$ finite, the map $\alpha_{L/K} \circ \pi_{L/K'}$ is continuous. Let $K'$ be such an intermediate extension. By definition of the topology on $\Aut(L/K)$ and $M_L$, the map $\beta : (\tau,v)\in \Aut(L/K)\times M_L \mapsto (\tau_{|K'},v_{|K'}) \in \Aut(K'/K)\times V_{K'}$ is continuous. The finite case implies that $\alpha_{K'/K}$ is continuous. We conclude by using the fact that $\alpha_{L/K} \circ \pi_{L/K'} = \alpha_{K'/K} \circ\beta$. Since $\phi_{L} :\Omega_L \to M_L$ is continuous, we obtain the continuity of the action on $\Omega_L$. These actions are proper since $M_L$ is compact Hausdorff and $\Aut(L/K)$ is Hausdorff and the action on $\Omega_{L}$ is the pullback of the one on $M_{L}$. In the case where $S$ is integral, the action of $\Aut(L/K)$ on $M_L$ induces a continuous and proper action on $V_L$ as $V_L$ is a compact subset of $M_L$ which is stabilised by the action. This concludes the proof of (i).

(ii) is a direct consequence of Proposition \ref{prop:action_Galois} (2).

We now prove (iii). Denote $G:=\Aut(L/K)$. Since $K = L^{G}$, the restriction maps $\pi_{L/K}: V_L \to V$ and $\pi_{L/K} : \Omega_L \to \Omega$ induce continuous maps $M_L/G\to M_K$ and $\Omega_L/G \to \Omega$ such that the diagram
\begin{center}
\begin{tikzcd}
\Omega_L \arrow[r, "\phi_L"] \arrow[d, two heads]   & M_L \arrow[d, two heads]   \\
\Omega_L/G \arrow[d, two heads] \arrow[r] & M_L/G \arrow[d, two heads] \\
\Omega \arrow[r, "\phi"]                          & M_K                         
\end{tikzcd}
\end{center}
is Cartesian. (ii) ensures that the arrows $\Omega_L/G \to \Omega$ and $M_L/G \to M_K$ are injective, hence bijective. Since both $M_{L}$ and $M_K$ are compact Hausdorff, the arrow $M_L/G \to M_K$ is a homeomorphism. Therefore, the arrow $\Omega_L/G \to \Omega$ is a homeomorphism as it is the pullback of $M_L/G \to M_K$ by $\phi : \Omega \to M_K$. In case where $S$ is integral we use that the map $V_L/G \to V$ is the pullback of $M_L/G \to M_K$ by the natural inclusion. Hence it is a homeomorphism. This concludes the proof of the proposition.
\end{proof}

\begin{remark}
\label{rem:uniqueness}
We assume that the hypotheses of Proposition \ref{prop:action_Galois} (ii) hold. Using (\cite{GVF24}, Lemma 10.4), if the measure $\nu$ is assumed to be Radon, we see that the measure $\nu_L$ is the only $\Aut(L/K)$-invariant measure on $\Omega_{L}$ whose pushforward via $\pi_{L/K}$ is $\nu$. 
\end{remark}

\subsection{Algebraic coverings for families of topological adelic curves}
\label{sub:algebraic_covering_family_tac}

In the last subsection of the first part of this article, we construct algebraic coverings for families of topological adelic curves. We fix a family of topological adelic curves $\mathbf{S}=(I,\cU,(S_{i}=(K_{i},\phi_{i}:\Omega_{i}\to M_{K_{i}},\nu_{i}))_{i\in I},K)$. Let $L/K$ be an algebraic extension. We want to construct a family of topological adelic curves $\mathbf{S}\otimes_{K} L$ with base field $L$ together with a morphism $\alpha_{L/K} : \mathbf{S}\otimes_{K}L \to \mathbf{S}$.

\subsubsection{Finite case}
\label{subsub:finite_covering_family_tac}

Assume first that $L/K$ is finite. Consider the field extension, $K_{\mathbf{S}}(L)/K_{\mathbf{S}})$. It is a finite extension. Moreover, by \L{}o\'s theorem, the field $K_{\mathbf{S}}(L)$ can be realised as an ultraproduct $\prod_{\cU}L_{i}$, where for $\cU$-almost all $i$ in $I$, $L_{i}/K_{i}$ is a field extension of degree $[L:K]$. Let $i\in I$. If $i\in \{j\in I : [L_{j}:K_{j}]=[L:K]\}$. We set $S_{L,i}:=S_{i}\otimes_{K_{i}}L_{i}$. Otherwise, we set $S_{L,i}:=S_{i}$. 

\begin{proposition-definition}
\label{prop-def:finite_separablecovering_family_tac}
$\mathbf{S}\otimes_{K} L := (I,\cU,(S_{i,L})_{i\in I},L)$ is a family of topological adelic curves that comes naturally with a morphism $\alpha_{L/K} : \mathbf{S}\otimes_{K} L\to \mathbf{S}$ obtained by combining all the morphisms $\alpha_{i}: S_{L,i} \to S_{i}$, where $i$ runs over $I$. Moreover, Proposition \ref{prop:finite_separable_covering_tac} implies that
\begin{align*}
\forall f=[(f_{i})_{i\in I}]\in L^{\times}, \quad [L:K]d_{\mathbf{\mathbf{S}\otimes_{K} L}}(f) = \left[\left(d_{S_{i}}\left(N_{L_{i}/K_{i}}(f_{i})\right)\right)_{i\in I}\right],
\end{align*}
where we set $N_{L_{i}/K_{i}}(f_{i}):=1$ if $L_{i}/K_{i}$ is not finite. Therefore, we see that $\mathbf{S}\otimes_{K} L$ is proper, or more generally asymptotically proper w.r.t. to any equivalence relation compatible with the addition law on $\prod_{\cU}\bR$, if $\mathbf{S}$ is so.
\end{proposition-definition}

\subsubsection{General algebraic case}
\label{subsub:algebraic_covering_family_tac}

Consider now the general case. Using the previous case, we can embed $K_{\mathbf{S}}(L)$ in an ultraproduct $\prod_{\cU}L_{i}$, where for $\cU$-almost all $i$ in $I$, $L_{i}$ is an algebraic extension of $K_{i}$. Let $i\in I$. As above, if $L_{i}/K_{i}$ is algebraic, we set $S_{L,i}:=S_{i}\otimes_{K_{i}}L_{i}$. Otherwise, we set $S_{L,i}:=S_{i}$. 

\begin{proposition-definition}
\label{prop-def:algebraic_covering_family_tac}
$\mathbf{S}\otimes_{K} L := (I,\cU,(S_{i,L})_{i\in I},L)$ is a family of topological adelic curves that comes naturally with a morphism $\alpha_{L/K} : \mathbf{S}\otimes_{K} L\to \mathbf{S}$ as in Proposition-Definition \ref{prop-def:finite_separablecovering_family_tac}. Proposition-Definition \ref{prop-def:finite_separablecovering_family_tac} again implies that $\mathbf{S}\otimes_{K} L$ is proper, or more generally asymptotically proper w.r.t. to any equivalence relation compatible with the addition law on $\prod_{\cU}\bR$, if $\mathbf{S}$ is so.
\end{proposition-definition}

\subsubsection{Case of a finite Galois extension}
\label{subsub:finite_Galois_covering_family_tac}

Now assume that $L/K$ is finite Galois. Realise $L$ as an ultraproduct $\prod_{\cU}L_{i}$, where for all $i\in I$, $L_{i}/K_{i}$ is a field extension that is finite of degree $[L:K]$ for $\cU$-almost all $i$ as in \S \ref{subsub:finite_Galois_covering_family_tac}. By \L{}o\'s theorem again, $L_{i}/K_{i}$ is Galois of degree $[L:K]$ for $\cU$-almost all $i$ in $I$. Moreover (\cite{Nguyen24}, Proposition 3.9) implies that $\Gal(L/K)$ identifies with the ultraproduct $\prod_{\cU}\Gal(L_{i}/K_{i})$, where we set $\Gal(L_{i}/K_{i})$ for all $i\in I$ such that $L_{i}/K_{i}$ is not Galois. Consider the family of topological adelic curves $\mathbf{S}\otimes_{K} L=(I,\cU,(S_{L,i})_{i\in I}, L)$ introduced in Proposition-Definition \ref{prop-def:finite_separablecovering_family_tac}. 

\begin{proposition}
\label{prop:galois_action_covering_family_tac}
Assume that for all $i\in I$ and for all $\omega\in \Omega_{i}$, the residue field corresponding to $\omega$ is perfect. Then $\Gal(L/K)$ acts continuously, properly and transitively on $\Omega_{\mathbf{S}\otimes_{K}L}$ and we have a homeomorphism
\begin{align*}
\Omega_{\mathbf{S}\otimes_{K}L}/\Gal(L/K) \cong \Omega_{\mathbf{S}}.
\end{align*}
\end{proposition}

\begin{proof}
Proposition \ref{prop:action_of_Galois_group_algebraic_extension} implies that for $\cU$-almost all $i\in I$, $\Gal(L_{i}/K_{i})$ acts continuously, properly and transitively on $\Omega_{L,i}$ and that we have an homeomorphism
\begin{align*}
\Omega_{L,i}/\Gal(L_{i}/K_{i}) \cong \Omega_{i}.
\end{align*}
These actions induce an action of $\Gal(L/K)\cong \prod_{\cU}\Gal(L_{i}/K_{i})$ on $\Omega_{\mathbf{S}\otimes_{K}L}=\prod_{\cU}\Omega_{L,i}$ that is continuous by definition of the topology, proper since $\Gal(L/K)$ is discrete, and transitive. Moreover, we have homeomorphisms
\begin{align*}
\Omega_{\mathbf{S}\otimes_{K}L}/\Gal(L/K) \cong \displaystyle \left(\prod_{\cU}\Omega_{L,i}\right)/\left(\prod_{\cU}\Gal(L_{i}/K_{i})\right) \cong \prod_{\cU}\left(\Omega_{L,i}/\Gal(L_{i}/K_{i})\right) \cong \Omega_{\mathbf{S}}.
\end{align*}
\end{proof}

\subsubsection{Example in Nevanlinna theory}
\label{subsub:covering_family_topological_adelic_curves_Nevanlinna}

We conclude this first part by explaining how our Nevanlinna theoretic example of family of topological adelic curves fits in the material developed above. Consider the family of topological adelic curves $\mathbf{S}=(\bR_{>0},\cU,(S_{R})_{R>0},\cM(\bC))$ from Example \ref{example:families_of_tac} (2), where for all $R>0$, $S_{R}=(K_{R},\phi_{R}:\Omega_{R}\to M_{K_{R}},\nu_{R})$ is the topological adelic curve from \S \ref{subsub:example_tac_Nevanlinnaç_compact_disc}. Let $L/\cM(\bC)$ be an algebraic extension. Let us explicit the family of topological adelic curves $\mathbf{S}\otimes_{\cM(\bC)}L$ from Proposition-Definition \ref{prop-def:algebraic_covering_family_tac}.

\textbf{Case 1: $L/\cM(\bC)$ finite.} Let $g\in L$ be a primitive element of $L/\cM(\bC)$ with minimal polynomial
\begin{align*}
P(T) := T^{d}+f_{d-1}T^{d-1}+\cdots+f_{0},
\end{align*} 
where $f_{0},...,f_{d-1}\in \cM(\bC)$. Then for all $R>0$, $L_{R}:=K_{R}(g)/K$ is a finite extension of degree $[L:\cM(\bC)]$ with primitive element $g$, and $L \subset \prod_{\cU}L_{R}$. Denote $S_{L,R}:=S_{R}\otimes_{K_{R}}L_{R}=(L_{R},\phi_{L,R}:\Omega_{L,R}\to M_{L_{R}},\nu_{L,R})$.

Arguing as in (\cite{Gubler97}, Example 2.8), or using (\cite{Forster81}, Theorem 8.9), there exists a branched holomorphic covering $\pi_{L} : X_{L} \to \bC$ of degree $[L:\cM(\bC)]$ such that $L$ identifies with $\cM(X_{L})$, the field of meromorphic functions on $X_{L}$. Moreover, as we saw in Example \ref{example:finite_separable_covering_Gubler}, we have a bijection of sets $\Omega_{L,R}\cong\pi_{L}^{-1}(\Omega_{R})$.

Thus, the family of topological adelic curves $\mathbf{S}\otimes_{\cM(\bC)}\cM(X_{L})=(\bR_{>0},\cU,(S_{L,R})_{R>0},\cM(X_{L}))$ can be seen as the formulation of the Nevanlinna theory of branched coverings of $\bC$ in our context (cf. e.g. \cite{LangCherry90}, Chapter III).

\textbf{Case 2: $L/\cM(\bC)$ infinite.} In the infinite case, We realise $L$ as the union of all the fields $\cM(X_{K'})$, where $K'$ runs over the intermediate extensions of $L/\cM(\bC)$ that are finite over $\cM(\bC)$, and for any such $K'$, $X_{K'}$ denotes a branched covering of $\bC$ such that $\cM(X_{K'})$ identifies with $K'$. Using the previous case, it turns out that the family of topological adelic curves $\mathbf{S}\otimes_{\cM(\bC)} L$ formulates the Nevanlinna theory over all the branched coverings $X_{K'}$ as above in a single object. 
 
\part{Adelic vector bundles and Harder-Narasimhan filtrations over topological adelic curves}
\label{part:geometry_of_numbers}

In this part, we study the intrinsic geometry of topological adelic curves. This is done by introducing what plays the role of a vector bundle on a curve in algebraic geometry. We introduce the counterpart of norm families in our context and introduce various regularity and dominance conditions (\S \ref{sec:pseudo-norm_families}). After that, we define adelic vector bundles on a topological adelic curve (\S \ref{sec:adelic_vector_bundle}). Finally, we study slope theory for adelic vector bundles on a topological adelic curve (\S \ref{sec:slopes_proper_case}). 

\section{Pseudo-norm families}
\label{sec:pseudo-norm_families}

In this section, we globalise the constructions of \S \ref{subsub:pseudo-norms}. Roughly speaking, this is done by glueing metrised vector bundles on Zariski-RIemann spaces. 

Throughout this section, we fix a topological adelic curve $S=(K,\phi: \Omega \to M_K, \nu)$. Recall that, for any $\omega\in\Omega$, we denote by $A_{\omega}$ and $\kappa_{\omega}$, the finiteness ring and the residue field of $\omega$ respectively. From now on, we assume that for any $\omega\in\Omega_{\infty}$, we have $\epsilon(\omega)=1$. In that case Proposition \ref{prop:properties_topological_adelic_curves} implies that $\nu(\Omega_{\infty}) <+\infty$ and $\Omega_{\infty},\Omega_{\um}$ are open subsets of $\Omega$. Note that this assumption is not too harmful to the generality since we can replace, for any $\omega\in\Omega_{\infty}$, the pseudo-absolute value $\va_{\omega}$ by the pseudo-absolute value $\va_{\omega}^{1/\epsilon(\omega)}$ and the measure $\nu$ by the measure $\tilde{\nu}:=(\mathbf{1}_{\Omega_{\um}}+\epsilon\mathbf{1}_{\Omega,\infty})\nu$. 

Recall that we denote by $j_{S}: \Omega\to\tilde{\Omega}$ the quotient map corresponding to the equivalence relation on $\omega$ identifying elements in $\Omega$ having the same finiteness ring. $\tilde{\Omega}$ is equipped with the pullback of the Zariski topology on the set $\ZR(K)$ of valuation rings of $K$. Namely, a basis of the topology on $\tilde{\Omega}$ is given by elements of the form
\begin{align*}
U(a_{1},...,a_{n}) := \{\tilde{\omega}\in\tilde{\Omega}: a_{1},...,a_{n}\in A_{\tilde{\omega}}\},
\end{align*}
where $n$ runs over the integers and $(a_{1},...,a_{n})$ runs over $K^{n}$. For any open subset $U\subset\tilde{\Omega}$, we denote $U^{\an}:=j_{S}^{-1}(U) \subset \Omega$. This is an open subset of $\Omega$.

\subsection{Definitions}

Let $E$ be a finite-dimensional $K$-vector space. For any $\omega\in\Omega$,  we call \emph{pseudo-norm} in $\omega$ on $E$ any map $\|\cdot\|_\omega : E \to [0,+\infty]$ such that $\|\cdot\|_{\omega}$ is a pseudo-norm in $\varphi(\omega)$ on $E$ (cf. \S \ref{subsub:pseudo-norms}).

\begin{definition}
\label{def:semi-norm_family}
Let $E$ be a finite-dimensional $K$-vector space. We call \emph{pseudo-norm family} on $E$ any family $\xi =(\|\cdot\|_{\omega})_{\omega\in \Omega}$ where, for any $\omega\in \Omega$, $\|\cdot\|_\omega$ is a pseudo-norm in $\omega$ on E. We assume that the following condition holds: 
\begin{itemize}
	\item[($\ast$)] for any $\omega\in \Omega$, there exist an open neighbourhood $U$ of $\tilde{\omega}:=j_{S}(\omega)$ in $\tilde{\Omega}$ and a basis $(e_{1},...,e_{r})$ of $E$ such that, for any $i\in\{1,...,r\}$, for any $\omega'\in U^{\an}=j_{S}^{-1}(U)$, we have $\|e_i\|_{\omega'} \in \bR_{>0}$. Such a basis is called \emph{adapted} to the pseudo-norm family $\xi$ in $\omega$.
\end{itemize}
Moreover, if there exists a basis $(e_{1},...,e_{r})$ of $E$ such that, for any $\omega\in\Omega$, for all $i=1,...,r$, we have $\|e_i\|_{\omega}\in \bR_{>0}$, we say that $(e_{1},...,e_{r})$ is \emph{globally adapted} to $\xi$ (on $\Omega$).
Finally, we say that the pseudo-norm family $\xi$ is \emph{ultrametric}, resp. \emph{Hermitian}, if $\|\cdot\|_{\omega}$ is ultrametric for any $\omega\in \Omega_{\um}$, resp. if $\|\cdot\|_{\omega}$ is Hermitian for any $\omega\in\Omega_{\infty}$.
\end{definition}

\begin{notation}
\label{notation:pseudo-norm family}
Let $E$ be a finite-dimensional $K$-vector space.
\begin{itemize}
	\item[(1)] By "let $\xi=(\|\cdot\|_{\omega},\cE_{\omega},N_{\omega},\widehat{E_{\omega}})_{\omega\in\Omega}$ be a pseudo-norm family on $E$", we mean that, for any $\omega\in\Omega$, the pseudo-norm $\|\cdot\|_{\omega}$ has finiteness module $\cE_{\omega}$, kernel $N_{\omega}$ and residue vector space $\widehat{E_{\omega}}$. 
	\item[(2)] In case there is no explicit notation as above, for any pseudo-norm family $\xi =(\|\cdot\|_{\omega})_{\omega\in \Omega}$ on $E$, for any $\omega\in\Omega$, we denote by
\begin{itemize}
	\item $\cE_{\omega}$ the finiteness module of $\|\cdot\|_{\omega}$;
	\item $N_{\omega}$ the kernel of $\|\cdot\|_{\omega}$;
	\item $\widehat{E_{\omega}}$ the residue vector space of $\|\cdot\|_{\omega}$.
\end{itemize}
\end{itemize}
\end{notation}

\begin{example}
\label{example:model_semi-norm_family}
\begin{itemize}
	\item[(1)] Assume that for any $\omega\in\Omega$, $\phi(\omega)$ is an absolute value on $K$, and thus that $S$ determines an adelic curve in the sense of Chen-Moriwaki. Then any norm family $\xi$ on a finite-dimensional $K$-vector space $E$ (cf. \cite{ChenMori}, \S 4.1) is a pseudo-norm family and any basis of $E$ is globally adapted to $\xi$.
	\item[(2)] Let $E$ be a finite-dimensional $K$-vector space and fix a basis $\mathbf{e}=(e_1,...,e_r)$ of $E$. We fix $\omega \in \Omega$ and we denote respectively by $\cE_{\omega}$ and $\widehat{E_{\omega}}$ the restriction of scalars of $E$ to $A_{\omega}$ and the corresponding residue vector space. Then $\mathbf{e}$ defines compatible isomorphisms
\begin{center}
\begin{tikzcd}
E \arrow[d, "\cong"] & \cE_{\omega} \arrow[d, "\cong"] \arrow[l, hook'] \arrow[r] & \widehat{E_{\omega}} \arrow[d, "\cong"] \\
K^n                  & A_{\omega}^n \arrow[l, hook'] \arrow[r]                    & \widehat{\kappa_{\omega}}^n             
\end{tikzcd}.
\end{center}
For any $\lambda_1,...,\lambda_r \in \widehat{\kappa_{\omega}}$, we set
\begin{align*}
\|\lambda_{1} e_{1}+\cdots+\lambda_{r} e_{r}\|_{\mathbf{e},\omega}=\left\{\begin{matrix}
\max\{\tilde{|\lambda_{1}|}_\omega,...,\tilde{|\lambda_{r}|}_\omega\}, &\text{ if } \omega \in \Omega_{\um},\\ 
\tilde{|\lambda_{1}|}_\omega+\cdots + \tilde{|\lambda_{r}|}_\omega, &\text{ if } \omega \in \Omega_{\infty},
\end{matrix}\right.
\end{align*}
where $\tilde{|\cdot|}$ denotes the residue absolute value on $\widehat{\kappa_{\omega}}$. Then $\|\cdot\|_{\mathbf{e},\omega}$ defines a norm on $\widehat{E_{\omega}}$. By lifting $\|\cdot\|_{\mathbf{e},\omega}$ to a pseudo-norm on $E$ in $\omega$, we obtain a pseudo-norm family $\xi_{\mathbf{e}}=(\xi_{\mathbf{e},\omega})_{\omega\in\Omega}$ on $E$ called the \emph{model pseudo-norm family} associated with the basis $\mathbf{e}$. 

We also define a Hermitian pseudo-norm family $\xi_{\mathbf{e},2}$ as follows. For any $\lambda_1,...,\lambda_r \in \widehat{\kappa_{\omega}}$, we set
\begin{align*}
\|\lambda_{1} e_{1}+\cdots+\lambda_{r}e_{r}\|_{\mathbf{e},2,\omega}=\left\{\begin{matrix}
\max\{\tilde{|\lambda_{1}|}_\omega,...,\tilde{|\lambda_{r}|}_\omega\}, &\text{ if } \omega \in \Omega_{\um},\\ 
(\tilde{|\lambda_{1}|}_\omega^2+\cdots + \tilde{|\lambda_{r}|}_\omega^2)^{1/2}, &\text{ if } \omega \in \Omega_{\infty}. 
\end{matrix}\right.
\end{align*}
By the same arguments as above, we can lift the construction to obtain a pseudo-norm family $\xi_{\mathbf{e},2}$ on $E$. Note that $\mathbf{e}$ is a basis of $E$ which is globally adapted to both $\xi_{\mathbf{e}}$ and $\xi_{\mathbf{e},2}$.
\end{itemize}
\end{example}

The following lemma studies more closely condition $(\ast)$ in Definition \ref{def:semi-norm_family}.

\begin{lemma}
\label{lemma:pseudo-norm_family_base_change}
Let $E$ be a finite-dimensional $K$-vector space equipped with a pseudo-norm family $\xi=(\|\cdot\|_{\omega})_{\omega\in\Omega}$. Let $\omega\in\Omega$ and let $(e_1,...,e_r)$ be a basis of $E$ such that, for any $i=1,...,r$, we have $\|e_i\|_{\omega}\in \bR_{>0}$. Then $(e_1,...,e_r)$ is an adapted basis to $\xi$ in $\omega$.
\end{lemma}

\begin{proof}
By hypothesis, there exist a basis $(e'_1,...,e'_r)$ of $E$ together with an open neighbourhood $U$ of $\tilde{\omega}$ in $\tilde{\Omega}$ such that,
\begin{align*}
\forall \omega'\in U^{\an},\quad \forall i=1,...,r,\quad \|e'_i\|_{\omega'} \in \bR_{>0}. 
\end{align*} 
For any $i=1,...r$, write $e_{i}=\sum_{j=1}^{r}a_{i,j}e'_{j}$, where $a_{i,1},...,a_{i,r}\in K$. Consider
\begin{align*}
V:= U \cap \{\tilde{\omega'}\in\tilde{\Omega}: \forall 1\leq i,j\leq r, a_{i,j}\in A_{\tilde{\omega'}} \text{ and } \forall 1\leq i\leq r, \exists 1\leq j\leq r \text{ s.t. }a_{i,j}\in A_{\tilde{\omega'}}^{\times}\},
\end{align*}
this is an open neighbourhood of $\tilde{\omega}$ in $\tilde{\Omega}$ since $\|e_{1}\|_{\omega},...,\|e_{r}\|_{\omega}\in\bR_{>0}$. By construction, for any $\omega'\in V^{\an}$, we have $\|e_{1}\|_{\omega'},...,\|e_{r}\|_{\omega'}\in\bR_{>0}$.
\end{proof}

We end this subsection with two results concerning pseudo-norm families that possess a globally adapted basis.

\begin{lemma} 
\label{lemma:pseudo-norm_family_globally_adapted}
Let $E$ be a finite-dimensional $K$-vector space equipped with a pseudo-norm family $\xi=(\|\cdot\|_{\omega})_{\omega\in\Omega}$. Assume that there exists a basis $(e_1,...,e_r)$ of $E$ which is globally adapted to $\xi$. Let $(e'_1,...,e'_{r})$ be another basis of $E$. Then there exists an open subset $\Omega'\subset \Omega$ such that, for all $\omega\in \Omega'$, the basis $(e'_1,...,e'_{r})$ is adapted to $\xi$ in $\omega$ and $\nu(\Omega\setminus \Omega')=0$. 
\end{lemma}

\begin{proof}
For any $i=1,...,r$, we write $e'_{i}= a_{1}^{(i)}e_{1} + \cdots a_{r}^{(i)}e_{r}$, where $a_{1}^{(i)},...,a_{r}^{(i)} \in K$. Let $\omega\in \Omega$. Note that the basis $(e'_1,...,e'_{r})$ is adapted to $\xi$ in $\omega$ iff, for any $i,j \in \{1,...,r\}$, $a_{j}^{(i)}\in A^{\times}_{\omega}$. As the functions $\log|a_{j}^{(i)}|_{\cdot} : \Omega \to [-\infty,+\infty]$ are continuous and $\nu$-integrable for $i,j=1,...,r$, we obtain the desired assertion.
\end{proof}

\begin{lemma} 
\label{lemma:pseudo-norm_family_singular}
Let $E$ be a finite-dimensional $K$-vector space equipped with a pseudo-norm family $\xi=(\|\cdot\|_{\omega})_{\omega\in\Omega}$. Let $s\in E\setminus\{0\}$. Then the set
\begin{align*}
\{\omega\in\Omega : \|s\|_{\omega} \in \{0,+\infty\}\}
\end{align*}
is a locally closed subset of $\Omega$ which has measure zero with respect to $\nu$.
\end{lemma}

\begin{proof}
Let $s\in E\setminus\{0\}$. And denote $F_s := \{\omega\in\Omega : \|s\|_{\omega} \in \{0,+\infty\}\}$. We may assume that $F_s$ is non-empty. Let $\omega_0\in F_s$. Let $(e_1,...,e_r)$ be a basis of $E$ which is globally adapted to $\xi$. 

Write $s= s_1e_1 +\cdots + s_r e_r$, where $s_1,...,s_r\in K$. Then, for any $\omega\in \Omega$, $\omega$ belongs to $F_s$ iff there exists $i\in \{1,...,r\}$ such that $s_i \in K \setminus A^{\times}_{\omega}$. For any $i=1,...,r$, denote $F_{i}:=\{\omega\in\Omega : s_i \in K \setminus A^{\times}_{\omega}\}$. As $|s_1|_{\cdot},...,|s_r|_{\cdot}$ are continuous, $F_s$ is locally closed. Thus
\begin{align*}
F_s \subset F_1 \cup \cdots \cup F_r.
\end{align*}
As the functions $\log|s_i|_{\cdot} : \Omega \to [-\infty,+\infty]$ for $i=1,...,r$ are $\nu$-integrable, $\nu(F_1)=\cdots=\nu(F_r)=0$ and therefore $F_s$ has measure zero.
\end{proof}

\subsection{Zariski-Riemann interpretation}
\label{sub:ZR_interpretation_pseudo-norm_families}

Before moving on to the definition of constructions on vector spaces equipped with a pseudo-norm family, let us interpret these objects by means of the Zariski-Riemann spaces attached to $S$ (cf. \S \ref{sub:ZR_spaces_adelic_curves}). 

\begin{proposition}
\label{prop:ZR_interpretation_pseudo-norm_family}
There is a one-to-one correspondence between the pairs of the form $(E,\xi)$ where $E$ is a finite-dimensional $K$-vector space equipped with a pseudo-norm family $\xi$ and metrised vector bundles on $\tilde{\Omega}$. Moreover, for any such pair $(E,\xi)$, there exists a basis of $E$ that is globally adapted to $\xi$ iff the corresponding metrised vector bundle on $\tilde{\Omega}$ is free. 
\end{proposition}

\begin{proof}
First, consider a vector bundle $\cE$ of rank $r>0$ on $\ZR(K)$. By definition, for any open subset $U\subset \ZR(K)$ such that $\cE_{|U}$ is free of rank $d$, there exists a family $(s_{1},...,s_{r})$ of sections of $\cE$ over $U$ such that the map of sheaves
\begin{align*}
\left((a_{1},...,a_{r})\in  \cO_{U}\right) \mapsto \displaystyle \sum_{i=1}^{r}a_{i}s_{i} \in \cE_{|U} 
\end{align*}
is an isomorphism. Thus, the stalk $E_{K}$ at the generic point of $\ZR(K)$ is a $d$-dimensional $K$-vector space and for any $v\in U$ corresponding to a valuation ring $A_{v}$ of $K$, the stalk $\cE_{v}$ is a free $A_{v}$-module of rank $d$ such that $\cE_{v}\otimes_{A_{v}}K = E$. Since $\ZR(K)$ can be covered by such $U$'s, $\cE$ gives rise to a $d$-dimensional $K$-vector spaces together with a family $(\cE_{v})_{v\in \ZR(K)}$ of free modules of rank $d$ over the $A_{v}$'s that are generically $E$. 

Now, let $\cE$ be a vector bundle of rank $r>0$ on $\ZR(K)_{S}=\tilde{\Omega}$. By definition, there exists a vector bundle $\cE'$ of rank $r$ on $\ZR(K)$ such that $\cE$ is the inverse image $\tilde{\varphi}^{-1}\cE'$. Since we assume that $\Omega$ contains an element $\omega_{0}$ such that $\phi(\omega_{0})$ is an absolute value on $K$ (cf. Remark \ref{rem:tac_existence_of_absolute_value}), $\cE$ yields a family $(\cE_{\tilde{\omega}})_{\tilde{\omega}\in \tilde{\Omega}}$ of free modules of rank $d$ over the $A_{\tilde{\omega}}$'s that are generically $E:=\cE_{\tilde{\omega_{0}}}$. For any $\omega\in \Omega$ whose equivalence class through $j_{S}:\Omega\to \tilde{\Omega}$ is $\tilde{\omega}$, we denote $\cE_{\omega}:=\cE_{\tilde{\omega}}$.

Let $\cE$ be a vector bundle of rank $r>0$ on $\ZR(K)_{S}=\tilde{\Omega}$ and $(\|\cdot\|(\omega))_{\omega\in \Omega}$ be a metric on $\cE$. By definition, for any $\omega\in\Omega$, $\|\cdot\|_{\omega}$ is a norm on $\cE(\omega):=\cE_{\omega}\otimes_{A_{\omega}}\widehat{\kappa}(\omega)$. In the terminology of \S \ref{subsub:pseudo-norms}, we see that $\|\cdot\|_{v}$ gives rise to a pseudo-norm on $E:=\cE\otimes_{\cO_{\tilde{\Omega}}}K$ with finiteness module $\cE_{\omega}$, kernel $\m_{\omega}\cE_{\omega}$ and residue vector space $\cE(\omega)$ equipped with residue norm $\|\cdot\|_{\omega}$. By abuse of notation, we denote again this pseudo-norm by $\|\cdot\|_{\omega}$. Moreover, for any $U\subset\tilde{\Omega}$ such that $\cE_{|U}$ is free, the generic stalk of any local trivialisation of $\cE$ over $U$ yields a basis of $E$ which is adapted to $\xi:=(\|\cdot\|_{\omega})_{\omega\in\Omega}$ at every point of the open subset $U^{\an}\subset \Omega$.

Conversely, consider a $r$-dimensional $K$-vector space $E$ equipped with a pseudo-norm family $\xi=(\|\cdot\|_{\omega})_{\omega\in\Omega}$, where $r>0$. Let $U\subset \tilde{\Omega}$ be an open subset such that there exists a basis $(s_{1},...,s_{r})$ of $E$ that is adapted to $\xi$ on $U^{\an}$. Let $V\subset U$ be an open subset. We set 
\begin{align*}
\cE_{U}(V):= \displaystyle\bigcap_{\tilde{\omega\in V}} \left(\bigoplus_{i=1}^{r}A_{\tilde{\omega}}\cdot s_{i}\right) \cong \cO_{\tilde{\Omega}}(V)^{r}.
\end{align*}
This defines a free $\cO_{U}$-module of rank $r$ on $U$. Moreover, if $U'$ is another open subset of $\tilde{\Omega}$ such that $\xi$ admits an adapted basis $(s'_{1},...,s_{r})$ on $V$ and $U\cap V\neq \emptyset$, since the transition matrix between $(s_{1},...,s_{r})$ and $(s'_{1},...,s'_{r})$ has coefficients in $\mathrm{GL}_{d}(\bigcap_{\tilde{\omega}\in U\cap U'}A_{\tilde{\omega}})$, we get an isomorphism $\cE_{U}(U\cap U')\cong \cE_{U'}(U\cap U')$. Since $\xi$ is a pseudo-norm family, there exists an open covering $\tilde{\Omega}=\bigcup_{i\in I} U_{i}$, where for any $i\in I$, $U_{i}$ is as above and by glueing the $\cE_{U_{i}}$'s, we get a vector bundle on $\tilde{\Omega}$ denoted by $\cE$. By construction, for any $\tilde{\omega}\in\tilde{\Omega}$, the stalk of $\cE$ at $\tilde{\omega}$ is the finiteness module of $\|\cdot\|_{\omega}$. Thus, the pseudo-norm family $\xi$ induces a metric on $\cE$. It is clear that this construction is inverse to the one above. 

The assertion concerning globally adapted bases is now clear since the existence of a globally adapted basis corresponds to a global trivialisation.
\end{proof}

\begin{remark}
\label{ZR_interpretation_pseudo-norm_families_equivalence_of_categories}
By suitably defining a notion of morphism between $K$-vector spaces equipped with a pseudo-norm family, one can prove that the correspondence constructed above yields an equivalence of categories. 
\end{remark}

Proposition \ref{prop:ZR_interpretation_pseudo-norm_family} implies that pseudo-norm families on a vector space are the same thing as metrics on a vector bundle on $\tilde{\Omega}$. In other words, we consider objects determined on the algebraic Zariski-Riemann space. We could also give a definition involving the analytic structure. This would give the following definition of a pseudo-norm family on a finite-dimensional $K$-vector space $E$: a family $\xi =(\|\cdot\|_{\omega})_{\omega\in \Omega}$ where, for any $\omega\in \Omega$, $\|\cdot\|_\omega$ is a pseudo-norm in $\omega$ on E is an \emph{analytic pseudo-norm family} if the following condition holds: 
\begin{itemize}
	\item[($\ast_{\an}$)] for any $\omega\in \Omega$, there exist an open neighbourhood $U$ of $\omega$ in $\Omega$ and a basis $(e_1,...,e_d)$ of $E$ such that, for any $i\in\{1,...,d\}$, $\|e_i\|_{\omega} \in \bR_{>0}$. 
\end{itemize}
One can note that all the results of this section can be proven \emph{mutatis mutandis} for analytic pseudo-norm families (cf. \cite{Sedillotthese}, Chapter III). Moreover, we have the following proposition.

\begin{proposition}
\label{prop:ZR_interpretation_pseudo-norm_families_GAGA}
Let $E$ be a finite-dimensional $K$-vector space and $\xi=(\|\cdot\|_{\omega})_{\omega\in\Omega}$ be a family such that, for any $\omega\in\Omega$, $\|\cdot\|_{\omega}$ is a pseudo-norm in $\omega$ on $E$. Assume that one of the following conditions hold:
\begin{itemize}
	\item[(1)] $\Omega=M_{K}$;
	\item[(2)] $S$ is integral whose underlying integral structure $(A,\|\cdot\|)$ is a geometric base ring such that the specification morphism $\cM(A,\|\cdot\|)\to\Spec(A)$ is flat and surjective and $\Omega=\cM(A,\|\cdot\|)$.
\end{itemize}
Then $\xi$ is a pseudo-norm family iff it is an analytic pseudo-norm family.
\end{proposition}

\begin{proof}
We have to show that $\xi$ satisfies $(\ast)$ iff is satisfies $(\ast_{\an})$. The direct implication being clear, it suffices to prove that if $(\ast_{\an})$ holds, then $(\ast)$ holds. In this case, by the same arguments as in the proof of Proposition \ref{prop:ZR_interpretation_pseudo-norm_family}, we see that $(E,\xi)$ defines a vector bundle $\cG$ on $\Omega$ equipped with a metric $\varphi$. By Proposition \ref{prop:GAGA_ZR_sheaf}, there exists a vector bundle $\cE$ on $\tilde{\Omega}$ such that $\cE^{\an}:=j_{S}^{\ast}\cE \cong \cG$. Moreover, $\varphi$ determines a metric on $\cE$ (since $\cE(\omega)\cong\cG(\omega)$ for any $\omega\in\Omega$) and the pair $(\cE,\varphi)$ corresponds to $(E,\xi)$ via Proposition \ref{prop:ZR_interpretation_pseudo-norm_family}. This concludes the proof.
\end{proof}

\begin{remark}
\label{rem:ZR_interpretation_pseudo-norm_families_GAGA}
As long as $\Omega\subset M_{K}$ is a Borel subset, the assumptions made in Proposition \ref{prop:ZR_interpretation_pseudo-norm_families_GAGA} are not too harmful for the generality. Indeed, we can pushforward the measure $\nu$ to $M_{K}$ and obtain a topological adelic curve whose adelic space is $M_{K}$ and such that the corresponding GVF height coincides with the one of $S$. If $S$ is integral, we do the same trick by pushing forward $\nu$ to the integral space of $S$.  
\end{remark}

\subsection{Algebraic constructions on pseudo-norm families}
\label{sub:algebraic_constructions_semi-norm_families}

\begin{proposition-definition}
\label{def:properties_semi-norm_family}
Let $E$ be a finite-dimensional $K$-vector space equipped with a pseudo-norm family $\xi =(\|\cdot\|_\omega)_{\omega\in \Omega}$.
\begin{itemize}
	\item[(1)] Let $F\subset E$ be a vector subspace of $E$. Then the family of restrictions $\xi_{|F}:=(\|\cdot\|_{\omega|F})_{\omega\in \Omega}$ is a pseudo-norm family on $F$ called the \emph{restriction} of $\xi$.
	\item[(2)] Let $G$ be a quotient of $E$. Then the family of quotient pseudo-norms $\xi_{G}:=(\|\cdot\|_{\omega,G})_{\omega\in \Omega}$ is a pseudo-norm family on $G$ called the \emph{quotient} of $\xi$.
	\item[(3)] The family of dual pseudo-norms $\xi^{\vee}:=(\|\cdot\|_{\omega,\ast})_{\omega\in \Omega}$ is a pseudo-norm family on the dual vector space $E^{\vee}$ called the \emph{dual} of $\xi$. 
	\item[(4)] Let $E'$ be another finite-dimensional $K$-vector space equipped with a pseudo-norm family $\xi' =(\|\cdot\|'_\omega)_{\omega\in \Omega}$. For any $\omega\in \Omega$, let $\|\cdot\|_{\omega,\pi}$ and $\|\cdot\|_{\omega,\epsilon}$ be respectively the $\pi$-tensor product and the $\epsilon$-tensor product of the pseudo-norms $\|\cdot\|_{\omega}$ and $\|\cdot\|'_{\omega}$. We denote by $\xi \otimes_{\epsilon,\pi} \xi'$ the pseudo-norm family on $E \otimes E'$ consisting of the pseudo-norms $\|\cdot\|_{\omega,\epsilon}$ for $\omega\in\Omega_{\um}$ and $\|\cdot\|_{\omega,\pi}$ for $\omega\in\Omega_{\infty}$. This family is called the $\epsilon,\pi$\emph{-tensor product} pseudo-norm family of $\xi$ and $\xi'$. Similarly, we define the $\epsilon$\emph{-tensor} product as well as the $\pi$\emph{-tensor product} of $\xi$ and $\xi'$.
	\item[(5)] Let $i\geq 1$ be an integer. We denote by $\Lambda^{i}\xi$ the pseudo-norm family on $\Lambda^{i}E$ as the quotient of the $\epsilon,\pi$-tensor product pseudo-norm family on $E^{\otimes i}$. This family is called the \emph{exterior power} pseudo-norm family on $\Lambda^{i}E$. If $i=\dim_{K}(E)$, the pseudo-norm family $\Lambda^{i}\xi$ is called the \emph{determinant} pseudo-norm family of $\xi$ and is denoted by $\det(\xi)$.
	\item[(6)] Let $\xi' =(\|\cdot\|'_{\omega},\cE'_{\omega},N'_{\omega},\widehat{E'_{\omega}})_{\omega\in \Omega}$ be another pseudo-norm family on $E$ Assume that for any $\omega\in\Omega$, there exists a basis $\mathbf{e}$ of $E$ such that $\mathbf{e}$ is both adapted to $\xi$ and $\xi'$ in $\omega$. Then, for any $\omega\in\Omega$, we have equalities 
	\begin{align*}
	\cE_{\omega}=\cE'_{\omega},\quad N_{\omega}=N'_{\omega},\quad \widehat{E_{\omega}} := \widehat{E'_{\omega}}.
	\end{align*}
	We define the \emph{local distance function} by
	\begin{align*}
	(\omega\in\Omega) \mapsto d_{\omega}(\xi,\xi') := \displaystyle\sup_{\overline{s}\in \widehat{E_{\omega}}\setminus\{0\}} \left|\log \|\overline{s}\|_{\omega}-\log\|\overline{s}\|'_{\omega}\right| = \sup_{s\in E} \left|\log \|s\|_{\omega}-\log\|s\|'_{\omega}\right|,
	\end{align*}
	where we use the convention that, for any $\omega\in\Omega$, for any $s\in E$ such that $\|s\|_{\omega}\in\{0,+\infty\}$, we have $\log \|s\|_{\omega}-\log\|s\|'_{\omega}=0$.
\end{itemize}
\end{proposition-definition}

\begin{proof}
It suffices to prove that condition $(\ast)$ from Definition \ref{def:semi-norm_family} holds for the pseudo-norm families in (1)-(5). Fix $\omega \in \Omega$.

\textbf{(1)} By construction of $\xi_F$, there exists a basis $(e_1,...,e_r)$ of $F$ which can be enlarged in a basis $(e_1,...,e_r,e_{r+1},...,e_d)$ such that $\|e_1\|_{\omega},...,\|e_d\|_{\omega}>0$. Then Lemma \ref{lemma:pseudo-norm_family_base_change} implies that $(e_1,...,e_r)$ is adapted to $\xi_{F}$ in $\omega$.

\textbf{(2)} Write $G = E/F$ for some vector subspace $F\subset E$. By construction of $\xi_G$, there exists a basis $(e_1,...,e_r,e_{r+1},...,e_{d})$ of $E$ such that $(e_1,...,e_r)$ is a basis of $F$, the image of $(e_{r+1},...,e_{d})$ in $E/F$ is a basis and, for any $i=r+1,...,d$, $\|\overline{e_{i}}\|_{\omega} >0$. Then one can use Lemma \ref{lemma:pseudo-norm_family_base_change} to conclude as above. 

\textbf{(3)} Let $(e_1,...,e_d)$ be a basis of $E$ which is adapted to $\xi$ in $\omega$, say on an open neighbourhood of the form $U^{\an}$ of $\omega$, where $U\subset \tilde{\Omega}$ is open. Then it follows from the construction of $\xi^{\vee}$ that, for any $\omega'\in U^{\an}$, $(e_1^{\vee},...e_{d}^{\vee})$ is a basis of $E^{\vee}$ such that $\|e_1^{\vee}\|_{\omega',\ast},...,\|e_d^{\vee}\|_{\omega',\ast}>0$. Hence the basis $(e_1^{\vee},...e_{d}^{\vee})$ is adapted to $\xi^{\vee}$ in $\omega$. 

\textbf{(4)} Let $(e_1,...,e_d)$ and $(e'_1,...,e'_{d'})$ respectively be basis of $E$ and $E'$ that are adapted to $\xi$ and $\xi'$ in $\omega$ on respective open neighbourhood $U^{\an}$ and $V^{\an}$ of $\omega$, where $U,V\subset \tilde{\Omega}$ are open subsets. The construction of $\pi$-tensor product and $\epsilon$-tensor product implies that the tensor product basis $(e_1,...,e_d)\otimes (e'_{1},...,e'_{d'})$ is an adapted basis of $E \otimes E'$ in $\omega$ on $U^{\an} \cap V^{\an}=(U\cap V)^{\an}$.

\textbf{(5)} This follows from (2) and (4).
\end{proof}

\begin{remark}
\label{rem:ZR_interpretation_algebraic_operations_pseudo-norm_families}
Using the interpretation given in Proposition \ref{prop:ZR_interpretation_pseudo-norm_family}, we can interpret the operations described in Proposition-Definition \ref{def:properties_semi-norm_family} in terms of metrised vector bundles on $\tilde{\Omega}$.
\end{remark}

\subsection{Dominated pseudo-norm families}
\label{sub:dominated_semi-norm_families}

In this subsection, we introduce a domination condition for pseudo-norm families. Most of its content is an adaptation of (\cite{ChenMori}, \S 4.1.2) in our context. As the reader can get easily convinced from our definition, all the results presented here can be proven exactly the same way as in \emph{loc. cit.} and we will simply refer to the corresponding result for the proof. 

\begin{definition}
\label{def:dominated_semi-norm_family}
Let $E$ be a finite-dimensional $K$-vector space equipped with a pseudo-norm family $\xi =(\|\cdot\|_\omega)_{\omega\in \Omega}$. The family is called \emph{dominated} if the pseudo-norm families $\xi$ and $\xi^{\vee}$ are \emph{upper dominated}, namely we have 
\begin{align*}
\forall f\in E\setminus\{0\},\quad\upint_{\Omega} \log \|f\|_\omega \nu(\diff \omega) < +\infty,
\end{align*}
and
\begin{align*}
\forall \varphi\in E^{\ast}\setminus\{0\},\quad\upint_{\Omega} \log \|\varphi\|_{\omega,\ast} \nu(\diff \omega) < +\infty.
\end{align*}

$\xi$ is called \emph{strongly dominated} if $\xi$ is dominated and the distance function $(\omega\in\Omega) \mapsto d_{\omega}(\xi,\xi^{\vee\vee})\in \bR_{\geq 0}$ (cf. Definition \ref{def:properties_semi-norm_family} (6)) is $\nu$-dominated.  
\end{definition}

\begin{example}
\label{example:model_pseudo-norm_family_dominated}
\begin{itemize}
\item[(1)] In the situation of Example \ref{example:model_semi-norm_family} (1), a pseudo-norm family is dominated iff it is dominated in the sense of (\cite{ChenMori}, Definition 4.1.2). 
\item[(2)] Using the notation of Example \ref{example:model_semi-norm_family} (2), we show that for any basis $\mathbf{e}=(e_1,...,e_r)$ of a $K$-vector space $E$, the pseudo-norm families $\xi_{\mathbf{e}},\xi_{\mathbf{e},2}$ are dominated (\cite{ChenMori}, Example 4.1.5).
\end{itemize} 
\end{example}

\begin{remark}[\cite{ChenMori}, Rem. 4.1.12]
\label{rem_ChenMori_4.1.12}
If a pseudo-norm family $\xi$ consists of ultrametric pseudo-norms $\xi_{\omega}$ for any $\omega\in \Omega_{\um}$, being dominated is equivalent to being strongly dominated.
\end{remark}

\begin{proposition}[\cite{ChenMori}, Proposition 4.1.16]
\label{prop:dominated_pseud-norm_family_dimension_1}
Let $E$ be a finite-dimensional $K$-vector space equipped with a pseudo-norm family $\xi =(\|\cdot\|_\omega)_{\omega\in \Omega}$. Assume that $\dim_{K}(E) = 1$. Then the following assertions are equivalent.
\begin{itemize}
	\item[(i)] $\xi$ is dominated.
	\item[(ii)] For any $s\in E\setminus\{0\}$, the function $(\omega\in\Omega)\mapsto\log\|s\|_{\omega}$ is $\nu$-dominated.
	\item[(iii)] There exists $s\in E\setminus\{0\}$, the function $(\omega\in\Omega)\mapsto\log\|s\|_{\omega}$ is $\nu$-dominated.
\end{itemize}
\end{proposition}

\begin{proposition}[\cite{ChenMori}, Proposition 4.1.19]
\label{prop:constructions_dominated_semi-norm_families}
\begin{itemize}
	\item[(1)] Let $E$ be a finite-dimensional $K$-vector space equipped with a pseudo-norm family $\xi=(\|\cdot\|_{\omega})_{\omega\in\Omega}$.  If $\xi$ is dominated (resp. strongly dominated), then for any vector subspace $F\subset E$, the restriction of $\xi$ to $F$ is a dominated (resp. strongly dominated) pseudo-norm family on $F$.
	\item[(2)] Let $E$ be a finite-dimensional $K$-vector space equipped with a pseudo-norm family $\xi=(\|\cdot\|_{\omega})_{\omega\in\Omega}$.  If $\xi$ is dominated (resp. strongly dominated), then for any vector subspace $F\subset E$, the quotient pseudo-norm family $\xi_{E/F}$ is dominated (resp. strongly dominated).
	\item[(3)] Let $E$ be a finite-dimensional $K$-vector space equipped with a dominated pseudo-norm family $\xi=(\|\cdot\|_{\omega})_{\omega\in\Omega}$.  Then the dual pseudo-norm family $\xi^{\vee}$ on $E^{\vee}$ is strongly dominated.
	\item[(4.a)] Let $E,F$ be finite-dimensional $K$-vector spaces equipped with dominated pseudo-norm families $\xi,\xi'$ respectively. Then the $\epsilon,\pi$-tensor product pseudo-norm family $\xi \otimes_{\epsilon,\pi} \xi'$ is a strongly dominated pseudo-norm family on $E\otimes_K F$.
	\item[(4.b)] Let $E,F$ be finite-dimensional $K$-vector spaces equipped with dominated pseudo-norm families $\xi,\xi'$ respectively. Then the $\epsilon$-tensor product pseudo-norm family $\xi \otimes_{\epsilon,\pi} \xi'$ is a strongly dominated pseudo-norm family on $E\otimes_K F$.
	\item[(5)] Let $E$ be a finite-dimensional $K$-vector space equipped with a dominated pseudo-norm family $\xi$. Let $i\in \bN$. Then the pseudo-norm family $\Lambda^i \xi$ is strongly dominated. 
	\item[(6)] Let $E$ be a finite-dimensional $K$-vector space equipped with a pseudo-norm family $\xi=(\|\cdot\|_{\omega})_{\omega\in\Omega}$. If $\xi$ is dominated, then the determinant pseudo-norm family $\det(\xi)$ is strongly dominated.
	\item[(7)] Let $K'/K$ be a finite extension of fields. Let $S'=(K',\phi' : \Omega' \to M_{K'},\nu')$ be the topological adelic curve constructed in \S \ref{sub:finite_extension_tac}. Denote by $\pi_{K'/K} : \Omega' \to \Omega$ the projection. Let $E$ be a finite-dimensional $K$-vector space and let $E' := E \otimes_{K} K'$. Let $\xi=(\|\cdot\|_{\omega})_{\omega\in\Omega}$ and $\xi'=(\|\cdot\|'_{\omega'})_{\omega'\in\Omega'}$ be pseudo-norm families on $E$ and $E'$ respectively such that
	\begin{align*}
	\forall \omega\in \Omega,\quad \forall \omega'\in \pi_{K'/K}^{-1}(\{\omega\}),\quad \forall s\in E,\quad \|s\|_{\omega'}=\|s\|_{\omega}.
	\end{align*}
	If $\xi'$ is dominated, then $\xi$ is dominated.
\end{itemize}
\end{proposition}

\begin{proposition}[\cite{ChenMori}, Proposition 4.1.6]
\label{prop:CM_4.1.6}
Let $E$ be a finite-dimensional $K$-vector space equipped with two pseudo-norm families $\xi_1,\xi_2$ such that, for any $\omega\in\Omega\in \bR_{\geq 0}$, there exists a basis $\mathbf{e}_{\omega}$ of $E$ which is both adapted to $\xi_1$ and $\xi_2$ in $\omega$ (so that the local distance function $(\omega\in\Omega)\mapsto d_{\omega}(\xi_1,\xi_2)\in \bR_{\geq 0}$ is well-defined). Assume that $\xi_1$ is dominated and that the local distance function $(\omega\in\Omega)\mapsto d_{\omega}(\xi_1,\xi_2)\in\bR_{\geq 0}$ is $\nu$-dominated. Then $\xi_2$ is dominated.
\end{proposition}

\begin{proposition}[\cite{ChenMori}, Corollary 4.1.8]
\label{cor:CM_4.1.8}
Let $E$ be a finite-dimensional $K$-vector space equipped with two pseudo-norm families $\xi_1,\xi_2$ such that there exists a basis $\mathbf{e}$ which is both globally adapted to $\xi_1$ and $\xi_2$. Assume that $\xi_1$ and $\xi_2$ are ultrametric on $\Omega_{\um}$. If $\xi_1$ and $\xi_2$ are both dominated, then the local distance function $(\omega\in\Omega)\mapsto d_{\omega}(\xi_1,\xi_2)\in\bR_{\geq 0}$ is $\nu$-dominated.
\end{proposition}

\begin{corollary}[\cite{ChenMori}, Corollary 4.1.10]
\label{cor:analogue_ChenMori_4.1.10}
Let $E$ be a finite-dimensional $K$-vector space equipped with a pseudo-norm family $\xi$. Assume that there exists a basis of $E$ which is globally adapted to $\xi$. Then the following assertions are equivalent: 
\begin{itemize}
	\item[(i)] the pseudo-norm family $\xi$ is dominated and the local distance function $(\omega\in\Omega) \mapsto d_{\omega}(\xi,\xi^{\vee\vee}) \in \bR_{\geq 0}$ is $\nu$-dominated;
	\item[(ii)] for any basis $\mathbf{e}$ of $E$ such that $\mathbf{e}$ is globally adapted to $\xi$, the local distance function $(\omega\in\Omega) \mapsto d_{\omega}(\xi,\xi_{\mathbf{e}})\in\bR_{\geq 0}$ is $\nu$-dominated;
	\item[(iii)] there exists a $\mathbf{e}$ basis of $E$ such that $\mathbf{e}$ is globally adapted to $\xi$ and the local distance function $(\omega\in\Omega) \mapsto d_{\omega}(\xi,\xi_{\mathbf{e}})\in\bR_{\geq 0}$ is $\nu$-dominated.
\end{itemize}
\end{corollary}

\subsection{Regularity and measurability conditions for pseudo-norm families}
\label{sub:regularity_measurability_smie-norm_families}

In this subsection, we introduce regularity conditions for pseudo-norm families. Here the theory differs from the classical theory of adelic curves, due to the topological nature of the adelic space. 

\begin{definition}
\label{def:continuous_measurable_semi-norm_family}
Let $E$ be a finite-dimensional $K$-vector space equipped with a pseudo-norm family $\xi =(\|\cdot\|_\omega)_{\omega\in \Omega}$. 
\begin{itemize}
	\item[(1)] The family $\xi$ is called \emph{continuous}, respectively \emph{usc}, \emph{lsc}  if, for any $s\in E$, the map $(\omega\in \Omega)\mapsto \log\|s\|_{\omega} \in [-\infty,+\infty]$ is continuous, respectively usc, lsc. We denote by $\cN(E)^{\cont}$, respectively $\cN(E)^{\usc}$, $\cN(E)^{\lsc}$ the set of continuous, respectively usc, lsc pseudo-norm families on $E$.
	\item[(2)] The family $\xi$ is called $\nu$\emph{-measurable} if, for any $s\in E$, the map $(\omega\in \Omega)\mapsto \log\|s\|_{\omega} \in [-\infty,+\infty]$ is $\nu$-measurable. We denote by $\cN(E)^{\nu}$ the set of $\nu$-measurable pseudo-norm families on $E$.
\end{itemize}
\end{definition}

\begin{remark}
\label{rem:ZR_interpretation_regular_pseudo-norm_families}
Using the interpretation given in Proposition \ref{prop:ZR_interpretation_pseudo-norm_family}, one can prove that a pseudo-norm family on a $K$-vector space $E$ is continuous/usc/lsc iff the associated metric on the corresponding vector bundle $\cE$ on $\tilde{\Omega}$ is continuous/usc/lsc. This is a consequence of the fact that any element $s\in E$ are interpreted as a meromorphic section of $\cE$.
\end{remark}

\begin{example}
\label{example:continuous_semi-norm_family_general}
\begin{itemize}
	\item[(1)] Let $\xi$ be a pseudo-norm family on a finite-dimensional $K$-vector space $E$. Assume that $\xi$ is either usc or lsc. Then $\xi$ is a $\nu$-measurable norm family on $E$. In particular, if $\xi$ is continuous, then $\xi$ is $\nu$-measurable.
	\item[(2)] Assume that $\Omega$ is discrete. Let $E$ be a finite-dimensional $K$-vector space equipped with a pseudo-norm family $\xi =(\|\cdot\|_\omega)_{\omega\in \Omega}$. Then $\xi$ is continuous.
	\item[(3)] In the situation of Example \ref{example:model_semi-norm_family} (1), a pseudo-norm family is $\nu$-measurable iff it is measurable in the sense of (\cite{ChenMori}, \S 4.1.3).
	\item[(4)] Model pseudo-norm families constructed in Example \ref{example:model_semi-norm_family} (2) are continuous. Indeed, as $\Omega_{\infty}$ and $\Omega_{\um}$ are open, it follows from the definition of such pseudo-norm families on each of these open sets.
\end{itemize}
\end{example}

\begin{proposition} 
\label{prop:constructions_continuous_semi-norm_families}
Let $E$ be a finite-dimensional $K$-vector space equipped with a pseudo-norm family $\xi =(\|\cdot\|_\omega)_{\omega\in \Omega}$. 
\begin{itemize}
	\item[(1)] Let $F \subset E$ be any vector subspace. Assume that $\xi$ is lsc, resp. usc. Then the restriction pseudo-norm family $\xi_{F}=(\|\cdot\|_{F,\omega})_{\omega\in \Omega}$ on $F$ induced by $\xi$ is lsc, resp. usc.
	\item[(2)] Assume that $\xi$ is usc. Then the dual pseudo-norm family $\xi^{\vee}=(\|\cdot\|_{\omega,\ast})_{\omega\in \Omega}$ on $E^{\vee}$ is lsc, hence is $\nu$-measurable.
	\item[(3)] Let $\pi: E \twoheadrightarrow G$ be a quotient vector space of $E$. Assume that $\xi$ is usc. Then the quotient pseudo-norm family $\xi_{G}=(\|\cdot\|_{G,\omega})_{\omega\in \Omega}$ induced by $\xi$ is usc, hence is $\nu$-measurable.
	\item[(4)] Let $n\geq 1$ be an integer. For any $i=1,...,n$, let $E_i$ be a finite-dimensional $K$-vector space equipped with a pseudo-norm family $\xi_i =(\|\cdot\|_{i,\omega})_{\omega\in \Omega}$. 
	\begin{itemize}
		\item[(4.a)] Assume that $\xi_1,...,\xi_n$ are usc. Then the $\pi$-tensor product $\xi_1 \otimes_{\pi} \cdots \otimes_{\pi} \xi_n$ is a usc and $\nu$-measurable pseudo-norm family on $ E_1\otimes \cdots\otimes E_n$.
		\item[(4.b)] Assume that $\xi_1,...,\xi_n$ are usc and ultrametric on $\Omega_{\um}$. Then the $\epsilon$-tensor product $\xi_1 \otimes_{\epsilon} \cdots \otimes_{\epsilon} \xi_n$ is a usc and $\nu$-measurable pseudo-norm family on $E_1\otimes \cdots\otimes E_n$. Moreover, if the dual pseudo-norm families $\xi_1^{\vee},...,\xi_n^{\vee}$ are continuous. Then the $\epsilon$-tensor product $\xi_1 \otimes_{\epsilon} \cdots \otimes_{\epsilon} \xi_n$ is a continuous pseudo-norm family on $E_1\otimes \cdots\otimes E_n$.

		\item[(4.c)] Assume that the dual pseudo-norm families $\xi_1^{\vee},...,\xi_n^{\vee}$ are continuous. Then the $\epsilon,\pi$-tensor product $\xi_1 \otimes_{\epsilon,\pi} \cdots \otimes_{\epsilon,\pi} \xi_n$ is a $\nu$-measurable pseudo-norm family on $E_1\otimes \cdots\otimes E_n$.
	\end{itemize}
	\item[(5)] Assume that $\xi$ is usc. The determinant pseudo-norm family $\det(\xi)$ on $\det(E)$ is usc and $\nu$-measurable. 
	\item[(6)] Let $L/K$ be an algebraic extension. Then the extension of scalars pseudo-norm family $\xi_L=(\|\cdot\|_{x})_{x\in \Omega_L}$ on $E_L := E \otimes_{K} L$ induced by $\xi$ is lsc, hence $\nu$-measurable. 
\end{itemize}
\end{proposition}

\begin{proof}
\textbf{(1):} It is immediate from the definition of a continuous pseudo-norm family.

\textbf{(2):} Let $\varphi\in E^{\vee}\setminus\{0\}$ and $a\in\bR_{>0}$. Assume that $\|\varphi\|_{\omega,\ast}\in ]a,+\infty]$. Let us prove that $\|\varphi\|_{\cdot,\ast}^{-1}(]a,+\infty])$ is open. Let $x\in \cE_{\omega}\setminus\m_{\omega}\cE_{\omega}$. By upper semi-continuity of $\|x\|_{\cdot}$ on $\Omega$, there exists an open neighbourhood $U$ of $\omega$ such that $\|x\|_{\cdot}$ has value in $\bR_{\geq 0}$ on $U$. Then the function $(\omega'\in U)\mapsto \frac{|\varphi(x)|_{\omega'}}{\|x\|_{\omega'}}\in\bR_{>0}$ is lower semi-continuous. Hence up to shrinking $U$, we may assume that for any $\omega'\in U$, we have
\begin{align*}
a<\frac{|\varphi(x)|_{\omega'}}{\|x\|_{\omega'}} \leq \|\varphi\|_{\omega',\ast}.
\end{align*}
We deduce the desired lower semi-continuity.

\textbf{(3):} Let $\pi: E \twoheadrightarrow G$ be a quotient vector space of $E$. By definition, for any $\overline{x}\in E/F$, we have
\begin{align*}
\|\overline{x}\|_{\omega,E/F} = \displaystyle\inf_{x\in \pi^{-1}(\overline{x})} \|x\|_{\omega}.
\end{align*}
Thus $\|\overline{x}\|_{\cdot,E/F}$ is the infimum of a family of usc functions and hence is usc. Since $\log$ is non-decreasing, $\log\|\overline{x}\|_{\cdot,E/F}$ is usc.

\textbf{(4.a):} Let $x\in E_1 \otimes\cdots\otimes E_n$ and $a\in \bR_{>0}$. We prove that $\|x\|_{\cdot,\pi}^{-1}([0,a[)$ is open. Let $\omega \in \|x\|_{\pi}^{-1}([0,a[)$. By definition of $\|x\|_{\omega,\pi}$, there exists a decomposition $x=\sum_{i=1}^N x_1^{(i)}\otimes\cdots\otimes x_n^{(i)}$, with $\|x_j^{(i)}\|_{j,\omega}<+\infty$ for any $i=1,...,N$ and $j=1,...,n$, such that
\begin{align*}
\displaystyle\sum_{i=1}^{N}\|x_1^{(i)}\|_{1,\omega}\cdots\|x_n^{(i)}\|_{n,\omega} < a.
\end{align*}
Let $U$ be an open neighbourhood of $\omega$ such that, for any $\omega'\in U$, for any $i=1,...,N$ and $j=1,...,n$, we have $\|x_j^{(i)}\|_{j,\omega'}<+\infty$. Up to shrinking $U$, we may assume that, for any $\omega'\in U$, we have
\begin{align*}
\|x\|_{\omega',\pi}\leq \displaystyle\sum_{i=1}^{N}\|x_1^{(i)}\|_{1,\omega'}\cdots\|x_n^{(i)}\|_{n,\omega'} < a.
\end{align*}
Thus $\|x\|_{\cdot,\pi}$ and $\log\|x\|_{\cdot,\pi}$ are usc and $\nu$-measurable.

\textbf{(4.b):} Let $x\in E_1 \otimes\cdots\otimes E_n$ and $a\in \bR_{>0}$. We prove that $\|x\|_{\cdot,\epsilon}^{-1}([0,a[)$ is open. Let $\omega \in \|x\|_{\epsilon}^{-1}([0,a[)$. Since the $\xi_{i}$'s are ultrametric on $\Omega_{\um}$, there exists a decomposition $x=\sum_{i=1}^N x_1^{(i)}\otimes\cdots\otimes x_n^{(i)}$, with $\|x_j^{(i)}\|_{j,\omega}<+\infty$ for any $i=1,...,N$ and $j=1,...,n$, such that
\begin{align*}
\displaystyle\max_{1\leq i\leq N}\|x_1^{(i)}\|_{1,\omega}\cdots\|x_n^{(i)}\|_{n,\omega} < a.
\end{align*}
Let $U$ be an open neighbourhood of $\omega$ such that, for any $\omega'\in U$, for any $i=1,...,N$ and $j=1,...,n$, we have $\|x_j^{(i)}\|_{j,\omega'}<+\infty$. Up to shrinking $U$, we may assume that, for any $\omega'\in U$, we have
\begin{align*}
\|x\|_{\omega',\epsilon}\leq \displaystyle\max_{1\leq i\leq N}\|x_1^{(i)}\|_{1,\omega'}\cdots\|x_n^{(i)}\|_{n,\omega'} < a.
\end{align*}
Thus $\|x\|_{\cdot,\epsilon}$ and $\log\|x\|_{\cdot,\epsilon}$ are usc and $\nu$-measurable.

Now assume that the $\xi_{i}$'s are continuous. Let $x\in E_1 \otimes\cdots\otimes E_n$ and $a\in \bR_{>0}$. We prove that $\|x\|_{\cdot,\epsilon}^{-1}(]a,+\infty])$ is open. Let $\omega \in \|x\|_{\epsilon}^{-1}(]a,+\infty])$. By definition of $\|x\|_{\omega,\epsilon}$, there exists $(f_1,...,f_n)\in E_1^{\vee}\times\cdots\times E_n ^{\vee}$ such that $0<\|f_1\|_{1,\omega,\ast},...,\|f_n\|_{n,\omega,\ast}<+\infty$ and 
\begin{align*}
a< \frac{|x(f_1,...,f_n)|_{\omega}}{\|f_1\|_{1,\omega,\ast}\cdots\|f_n\|_{n,\omega,\ast}}.
\end{align*}
By continuity $|x(f_1,...,f_n)|_{\cdot}$ and Lemma \ref{lemma:pseudo-norm_family_base_change}, there exists an open neighbourhood $U$ of $\omega$ such that, for any $\omega'\in U$, we have
\begin{align*}
0 < \|f_1\|_{1,\omega',\ast},...,\|f_n\|_{n,\omega',\ast} < +\infty.
\end{align*}
Up to shrinking $U$, by continuity of  $\|f_1\|_{1,\cdot,\ast},...,\|f_n\|_{n,\cdot,\ast}$, we may assume that
\begin{align*}
\forall \omega'\in U, \quad a< \frac{|x(f_1,...,f_n)|_{\omega'}}{\|f_1\|_{1,\omega',\ast}\cdots\|f_n\|_{n,\omega',\ast}} \leq \|x\|_{\omega',\epsilon}.
\end{align*}
Hence $\|x\|_{\cdot,\epsilon}$ and $\log\|x\|_{,\cdot,\epsilon}$ are lsc and thus continuous.

\textbf{(4.c):} Let $x\in E_1 \otimes\cdots\otimes E_n$ and $A \subset [-\infty,+\infty]$ be a measurable set. By (4.a) and (4.b), the sets
\begin{align*}
\left(\log\|x\|_{\cdot,\pi}\right)^{-1}(A), \quad \left(\log\|x\|_{\cdot,\epsilon}\right)^{-1}(A)
\end{align*}
are $\nu$-measurable. We can then conclude by using the equality
\begin{align*}
\left(\log\|x\|_{\cdot,\epsilon,\pi}\right)^{-1}(A) = \left(\left(\log\|x\|_{\cdot,\pi}\right)^{-1}(A)\cap \Omega_{\infty}\right) \cup \left(\left(\log\|x\|_{\cdot,\epsilon}\right)^{-1}(A)\cap \Omega_{\um}\right)
\end{align*}
combined with the fact that $\Omega_{\infty}$ and $\Omega_{\um}$ are both $\nu$-measurable.

\textbf{(5):} Let $\eta\in\det(E)$ and $\omega\in\Omega$. Assume that $\|\eta\|_{\omega,\det}<+\infty$. Let $a>\|\eta\|_{\omega,\det}$. We show that $\|\eta\|_{\cdot,\det}^{-1}([0,a[)$ is an open neighbourhood of $\omega$. By Proposition \ref{prop:Hadamrd_inequality_local_semi-norm} (3), there exist $x_1,...,x_r\in \cE_{\omega}$ such that $\eta=x_1\wedge\cdots\wedge x_r$ and 
\begin{align*}
\|\eta\|_{\omega,\det}\leq\|x_1\|_{\omega}\cdots\|x_r\|_{\omega} < a.
\end{align*}
Let $\mathbf{e}=(e_1,...,e_r)$ a basis of $E$ which is adapted to $\xi$. Thus, for any $i=1,...,r$, $x_i = x_{i}^{(1)}e_1 + \cdots + x_{i}^{(r)}e_r \in \bigoplus_{i=1}^{r} A_{\omega}e_i$. Let $U$ be an open neighbourhood of $\omega$ such that, for any $\omega'\in U$, for any $i,j\in \{1,...,r\}$, $x_{i}^{(j)} \in A_{\omega'}$. By semi-continuity of $\|x_1\|_{\cdot},...,\|x_r\|_{\cdot}$, up to shrinking $U$, we may assume that, for any $\omega'\in U$, we have
\begin{align*}
\|\eta\|_{\omega',\det} \leq \|x_1\|_{\omega'}\cdots\|x_r\|_{\omega'} < a. 
\end{align*} 
Hence, for any $a>0$, the set $\|\eta\|_{\cdot,\det}^{-1}([0,a[)$ is open and $\det(\xi)$ is usc and $\nu$-measurable.
\end{proof}

\section{Adelic vector bundles}
\label{sec:adelic_vector_bundle}
Throughout this section, we fix a topological adelic curve $S=(K,\phi:\Omega\to M_K, \nu)$. 

\subsection{Definitions}
\label{sub:adelic_vector_bundles_definitions}

\begin{definition}
\label{def:adelic_vector_bundle}
Let $\xi$ be a pseudo-norm family on a finite-dimensional $K$-vector space $E$. Then $\overline{E}=(E,\xi)$ is called 
\begin{itemize}
	\item[(1)] a \emph{upper/lower semi-continuous adelic vector bundle} (\emph{usc/lsc adelic vector bundle} for short), if the pseudo-norm families $\xi$ and $\xi^{\vee}$ are upper/lower semi-continuous and $\xi$ is dominated;
	\item[(2)] an \emph{adelic vector bundle} if the pseudo-norm families $\xi$ and $\xi^{\vee}$ are continuous and $\xi$ is dominated;
	\item[(3)] a \emph{measurable adelic vector bundle} on $S$, if the pseudo-norm families $\xi$ and $\xi^{\vee}$ are $\nu$-measurable and $\xi$ is dominated. 
\end{itemize}
Note that any (usc/lsc) adelic vector bundle is a measurable adelic vector bundle.

If $\overline{E}=(E,\xi)$ is a (usc/lsc/measurable) adelic vector bundle on $S$ and if the pseudo-norm family $\xi$ is strongly dominated, then $\overline{E}$ is called a \emph{strongly (usc/lsc/measurable) adelic vector bundle} on $S$. 

Moreover, an adelic vector bundle $\overline{E}=(E,\xi)$ on $S$ is called \emph{Hermitian} if $\xi$ is Hermitian.
\end{definition}

\begin{remark}
\label{rem:adelic_vector_bundle_Zariski-Riemann}
\begin{itemize}
	\item[(1)] The terminology "adelic vector bundle" is justified by Proposition \ref{prop:ZR_interpretation_pseudo-norm_family}. Indeed, as a vector space equipped with a pseudo-norm family can be seen as a metrised vector bundle on $\tilde{\Omega}=\ZR(K)_{S}$, when one assumes further global conditions on the pseudo-norm family in order to perform Arakelov geometric constructions, we add the adjective "adelic". In the case where $K$ is an algebraic extension of $\bQ$ equipped with the topological adelic curve structure $S_{\bQ}\otimes_{\bQ} K$ using the constructions of \S \ref{sec:algebraic_covering_tac}, an adelic vector bundle on $S_{\bQ}\otimes_{\bQ} K$ is an \emph{integrable adelic vector space} in the sense of Gaudron (\cite{Gaudron20}, Definition 7).
	\item[(2)] In the situation of Example \ref{example:model_semi-norm_family} (1). Any adelic vector bundle on $S$ is an adelic vector bundle on the corresponding adelic curve (\cite{ChenMori}, Definition 4.1.28).
\end{itemize}
\end{remark}

\begin{proposition}
\label{prop:comparaison_usc_adeli_vector_bundle}
Let $\xi$ be a pseudo-norm family on a finite-dimensional $K$-vector space $E$ and assume that $\xi$ is ultrametric on $\Omega_{\um}$. Then $\overline{E}=(E,\xi)$ is a usc adelic vector bundle on $S$ iff $\overline{E}=(E,\xi)$ is an adelic vector bundle on $S$. 
\end{proposition}

\begin{proof}
An adelic vector bundle is a usc adelic vector bundle. We show the converse implication. Assume that $\overline{E}$ is usc. It suffices to prove that $\xi$ and $\xi^{\vee}$ are both lsc. As $\xi$ is usc, Proposition \ref{prop:constructions_continuous_semi-norm_families} (2) implies that $\xi^{\vee}$ is lsc. Moreover, since $\xi$ is ultrametric on $\Omega_{\um}$, $\xi=\xi^{\vee\vee}$ and, as $\xi^{\vee}$ is usc, we obtain that $\xi$ is lsc.
\end{proof}

\begin{proposition}
\label{prop:adelic_line_bundle_tac}
Let $\xi=(\|\cdot\|_{\omega})_{\omega\in\Omega}$ be a pseudo-norm family on a $K$-vector space $E$ of dimension $1$. Then $(E,\xi)$ is an adelic vector bundle on $S$ iff $\xi$ is continuous and dominated.
\end{proposition}

\begin{proof}
The direct implication being immediate, it suffices to show the converse. Assume that $\xi$ is continuous and dominated. Let $\varphi\in E^{\vee} \setminus\{0\}$. Let $x\in E$ such that $\varphi(x)=1$. Then, for any $\omega\in\Omega$, we have $\|\varphi\|_{\omega,\ast}=1/\|x\|_{\omega}$, with the convention $1/0=+\infty$ and $1/+\infty=0$. Thus $\xi^{\vee}$ is continuous.
\end{proof}

\begin{proposition}
\label{prop:constructions_adelic_vector_bundles}
Let $\overline{E}=(E,\xi)$ be an adelic vector bundle on $S$. 
\begin{itemize}
	\item[(1)] Let $F\subset E$ be any vector subspace. Then $(F,\xi_{F})$ is an adelic vector bundle on $S$.
	\item[(2)] $\overline{E}^{\vee} = (E^{\vee},\xi^{\vee})$ is an lsc adelic vector bundle on $S$. Moreover, if $\xi^{\vee\vee}$ is continuous, then $\overline{E}^{\vee}$ is an adelic vector bundle on $S$. In particular, if $\xi$ is ultrametric on $\Omega_{\um}$, then $\overline{E}^{\vee}$ is an adelic vector bundle on $S$.
	\item[(3)] Let $\pi: E \twoheadrightarrow G$ be a quotient vector space of $E$. Let $\xi_{G}$ denote the quotient pseudo-norm family on $G$ induced by $\xi$. Then $(G,\xi_G)$ is a usc adelic vector bundle on $S$. Moreover, if $\xi$ is ultrametric on $\Omega_{\um}$, then $(G,\xi_G)$ is an adelic vector bundle on $S$.
	\item[(4)] Let $n\geq 1$ be an integer. For any $i=1,...,n$, let $\overline{E_i}=(E_i,\xi_i)$ be an adelic vector bundle on $S$. We assume that the pseudo-norm families $\xi_1,...,\xi_n$ are ultrametric on $\Omega_{\um}$. 
	\begin{itemize}
		\item[(4.a)] The $\epsilon$-tensor product $\overline{E_1}\otimes_{\epsilon}\cdots\otimes_{\epsilon}\overline{E_n}$ is a measurable adelic vector bundle on $S$. 
		\item[(4.b)] The $(\epsilon,\pi)$-tensor product $\overline{E_1}\otimes_{\epsilon,\pi}\cdots\otimes_{\epsilon,\pi}\overline{E_n}$ is a measurable adelic vector bundle on $S$. 
	\end{itemize}
	\item[(5)] $(\det(E),\det(\xi))$ is a measurable adelic vector bundle on $S$. Moreover, if $\xi$ is Hermitian, then $(\det(E),\det(\xi))$ is an adelic vector bundle on $S$.	
\end{itemize}
\end{proposition}

\begin{proof}
This is a consequence of Propositions \ref{prop:constructions_dominated_semi-norm_families} and \ref{prop:constructions_continuous_semi-norm_families}. We leave the details to the reader.

\end{proof}

\subsection{Arakelov degree}

\begin{definition}
\label{def:arakelov_degree_section}
Let $\overline{E}:= (E,\xi=(\|\cdot\|_{\omega})_{\omega\in\Omega})$ be a measurable adelic vector bundle on $S$. Let $s\in E\setminus\{0\}$, then the map $\log\|s\|_{\cdot} : \Omega \to [-\infty,+\infty]$ is $\nu$-integrable and we define the \emph{Arakelov degree} of $s$ as
\begin{align*}
\widehat{\deg}_{\xi}(s) := -\int_{\Omega} \log\|s\|_{\omega}\nu(\diff\omega).
\end{align*}
\end{definition}

\begin{definition}
\label{def:Arakelov_degree_proper_adelic_curve}
We assume that the adelic curve $S$ is proper. Let $\overline{E}:= (E,\xi=(\|\cdot\|_{\omega})_{\omega\in\Omega})$ be an usc adelic vector bundle on $S$. Then the quantity
\begin{align*}
\widehat{\deg}(\overline{E}):= \widehat{\deg}_{\det(\xi)}(\eta) = - \int_{\Omega} \log\|\eta\|_{\omega,\det}\nu(\diff\omega),
\end{align*}
where $\eta\in\det(E)\setminus\{0\}$, is independent of the choice of $\eta$. We call it the \emph{Arakelov degree of} $\overline{E}$.
\end{definition}

\subsection{Example in Nevanlinna theory}
\label{sub:example_in_Nevanlinna_theory_avb}

In this subsection, we give an example of adelic vector bundles in the context of Nevanlinna theory. Fix $R>0$ and consider the topological adelic curve $S_R =(K_R,\phi_R:\Omega_R \to M_{K_R},\nu_R)$ constructed in \S \ref{subsub:example_tac_Nevanlinnaç_compact_disc}. For any $\omega\in\Omega_{R}$, we denote by $A_{R,\omega}$ the finiteness ring on $\omega$. Let $E$ be an arbitrary finite-dimensional $\bC$-vector space equipped with a norm $\|\cdot\|$. Let us see that this data induces an adelic vector bundle on $S_R$. Let $E_R := E \otimes_{\bC} K_R$. 

Let $\omega\in \Omega_{R,\infty}$, we denote $\cE_{R,\omega} := E \otimes_{\bC} A_{R,\omega}$. It is a free $A_{R,\omega}$ of rank $\dim_{\bC}(E)$ and $\widehat{E_{R,\omega}} := \cE_{R,\omega} \otimes_{A_{R,\omega}} \bC$ identifies with $E$. By lifting the norm $\|\cdot\|$ on $E$, we obtain a local-pseudo-norm on $E_R$ in $\omega$, which is denoted by $\|\cdot\|_{\omega}$. 

Let $\omega\in \Omega_{R,\um}$. Recall that $A_{R,\omega} = K_R$ as $|\cdot|_{\omega}$ is a usual absolute value. Then the completion $K_{R,\omega}$ of $K_{R}$ w.r.t. $|\cdot|_{\omega}$ is isomorphic to $\bC((T))$. Denote by $K_{R,\omega}^{\circ}\cong\bC[[T]]$ the corresponding valuation ring of $K_{R,\omega}$. Then $K_{R,\omega}^{\circ}$ is the completion of the discrete valuation ring
\begin{align*}
\{f\in K_R : f(\omega) \in \mathbb{P}^{1}_{\bC} \setminus\{\infty\}\}.
\end{align*}
Let $\cE_{R,\omega} := E \otimes_{\bC} K_{R,\omega}^{\circ}$, it is a free sub-$K_{R,\omega}^{\circ}$-module of rank $\dim_{\bC}(E)$ of $E_{R,\omega}:= E \otimes_{K_{R}} K_{R,\omega}$, hence it is a (finitely generated) lattice of $E_R$. We consider the ultrametric lattice norm $\|\cdot\|_{\omega}$ induced by $\cE_{R,\omega}$ on $E_{R,\omega}$ (cf. e.g. Definition 1.1.23 in \cite{ChenMori}). Recall that, since $|\cdot|_{\omega}$ is discrete, $\cE_{R,\omega}$ coincides with the unit ball of $E_{R,\omega}$ w.r.t. $\|\cdot\|_{\omega}$. 

From the two above paragraphs, we have a collection $\xi=(\|\cdot\|_{\omega})_{\omega\in \Omega_{R}}$ of pseudo-norms on $E_R$. 

\begin{proposition}
\label{prop:example_avb_Nevanlinna}
We use the same notation as above. Then $\overline{E_R}:=(E_R,\xi_R)$ is an adelic vector bundle on $S_R$.
\end{proposition}

\begin{proof}
Let us first prove that $\xi$ is a pseudo-norm family on $E_R$ (in the sense of Definition \ref{def:semi-norm_family}). By construction of the pseudo-norms $\|\cdot\|_{\omega}$, any basis of $E$ defines a basis of $E_R$ which is adapted to $\|\cdot\|_{\omega}$ for all $\omega\in\Omega_{R}$. Thus $\xi_R$ is a pseudo-norm family on $E_R$. 

Let us now show that $\xi_R$ is dominated. Fix an arbitrary basis $(e_1,...,e_r)$ of $E$ (which is globally adapted to $\xi_R$). Let $s=s_1e_1 + \cdots + s_re_r \in E_R\setminus\{0\}$, where $s_1,...,s_r\in K_{R}$. Let $\omega\in \Omega_{R,\um}$, then $\|s\|_{\omega}\neq 1$ iff $\omega$ is either a zero or a pole of $s_i$ for some $i\in \{1,...,r\}$. Since elements of $K_R$ only have a finite number of zeroes and poles, this means that $\|s\|_{\omega}=1$ for all but a finite number of $\omega\in\Omega_{R,\um}$. As $\Omega_{R,\um}$ is a discrete topological space equipped with a counting measure, $(\omega\in \Omega_{R,\um}) \mapsto \log\|s\|_{\omega}\in[-\infty,+\infty]$ is a $\nu_R$-integrable function. Now let $\omega\in \Omega_{R,\infty}$. Then we have 
\begin{align*}
\|s\|_{\omega} \leq \displaystyle\max_{1\leq i \leq r} |s_i|_{\omega} \max_{1\leq i \leq r} \|e_i\|.
\end{align*}
As $\nu_{R}(\Omega_{R,\infty}) = 1$, the function $(\omega\in \Omega_{R,\infty}) \mapsto \max_{1\leq i \leq r}\log\|e_i\|\in [-\infty,+\infty]$ is $\nu_R$-integrable. Moreover, for any $i=1,..,r$, the function $\log|s_i|_{\cdot}$ is $\nu_R$-integrable. Hence $(\omega\in \Omega_{R,\infty}) \mapsto \max_{1\leq i \leq r}\log|s_i|_{\omega}\in[-\infty,+\infty]$ is $\nu_R$-integrable. Finally, we obtain that $\xi_R$ is upper dominated. 

Let $\alpha=\alpha_1e_1^{\vee} + \cdots \alpha_{r}e^{\vee}_{r}\in E_R^{\vee} \setminus\{0\}$, where $\alpha_1,...,\alpha_r\in K_R$. Let $\omega\in \Omega_{R,\um}$. Denote $\cE_{R,\omega}^{\vee}:= \Hom_{K_{R,\omega}^{\circ}}(\cE_{R,\omega},K_{R,\omega}^{\circ})$. Then (\cite{ChenMori}, Proposition 1.1.34) yields the equality 
\begin{align*}
\|\cdot\|_{\omega,\ast} = \|\cdot\|_{\cE_{R,\omega}^{\vee}}
\end{align*}
of norms on $E_{R,\omega}$. Therefore, the same argument used in the above paragraph for the lattice $\cE^{\vee}_{R,\omega}$ instead of $\cE_{R,\omega}$ implies that $\|\alpha\|_{\omega,\ast} = 1$ for almost all $\omega\in\Omega_{R,\um}$ and thus $[(\omega\in \Omega_{R,\um}) \mapsto \log\|\alpha\|_{\omega,\ast}]$ is a $\nu_R$-integrable function. Now let $\omega\in \Omega_{R,\infty}$. Then we have 
\begin{align*}
\|\alpha\|_{\omega,\ast} \leq \displaystyle\max_{1\leq i \leq r} |\alpha_i|_{\omega} \max_{1\leq i \leq r} \|e^{\vee}_i\|_{\ast},
\end{align*}
and the same arguments of the above paragraph imply that $\xi_R^{\vee}$ is upper dominated. Hence $\xi_R$ is a dominated pseudo-norm family on $E$. 

Finally, let us prove that $\xi_R$ and $\xi_R^{\vee}$ are continuous. We start by showing the continuity of $\xi_R$. Let $s\in E_R$. By discreteness of $\Omega_{R,\um}$, it suffices to prove that $\|s\|_{\cdot}$ is continuous on $\Omega_{R,\infty}$. Fix an arbitrary basis $(e_1,...,e_r)$ of $E$ (which is globally adapted to $\xi_R$). Write $s=s_1e_1 + \cdots + s_re_r \in E_R\setminus\{0\}$, where $s_1,...,s_r\in K_{R}$. By definition, for any $\omega\in\Omega_{R,\infty}$, we have 
\begin{align*}
\|s\|_{\omega} = \left\{\begin{matrix}
&\|s_1(\omega)e_1+\cdots s_{r}(\omega)e_r\| \quad &\text{if } s_1(\omega),...,s_{r}(\omega)\in \mathbb{P}^{1}(\bC)\setminus\{\infty\},\\
&+\infty \quad &\text{otherwise}.
\end{matrix}\right.
\end{align*}
Note that the set $U := \{\omega\in \Omega_{R,\infty} : s_1(\omega),...,s_{r}(\omega)\in \mathbb{P}^{1}(\bC)\setminus\{\infty\}\}$ is the complement of a finite set, hence it is open. The continuity of $\|s\|_{\cdot}$ on $U$ is equivalent to the continuity of the map
\begin{align*}
f : (\omega\in U) \mapsto s_1(\omega)e_1 + \cdots s_{r}(\omega) e_r \in E,
\end{align*}
where $E$ is equipped with the topology induced by the norm $\|\cdot\|$. By continuity of the maps $|s_1|_{\cdot},...,|s_r|_{\cdot}$ on $U$, we see that $f$ is continuous w.r.t. the topology on $E$ induced by the infinite norm w.r.t. the basis $(e_1,...,e_r)$, which is the same as the desired topology by equivalence of norms. Therefore, $\|s\|_{\cdot}$ is continuous on $U$. It remains to prove that $\|s\|_{\cdot}$ is continuous at any point of the finite set $\Omega_{R,\infty}\setminus U$. Let $\omega_0 \in \Omega_{R,\infty}\setminus U$. Since $\Omega_{R,\infty}\setminus U$ is discrete, it suffices to prove the continuity of $\|s\|_{\cdot}$ on a neighbourhood $V$ of $\omega_0$ such that $V\cap (\Omega_{R,\infty}\setminus U) = \{\omega_0\}$. To show that $\|s\|_{\cdot}$ is continuous at $\omega_0$, it suffices to show that 
\begin{align*}
\displaystyle\lim_{V \ni\omega\to \omega_0}\|s\|_{\omega} = \|s\|_{\omega_0} = +\infty. 
\end{align*}
By symmetry, we may assume that $s_1$ has minimal valuation, i.e. has a pole of the greatest order in $\omega_0$, among $s_1,...,s_r$. Denote this order by $p<0$. We see that, up to shrinking $V$, we can write $s=(T-\omega_0)^{p}s'$, where $s'$ satisfies 
\begin{align*}
\forall \omega\in V, \quad \|s'\|_{\omega}<+\infty. 
\end{align*}
By the above case of continuity and compactness of $\Omega_{R,\infty}$, we obtain that $\|s'\|_{\cdot}$ is bounded on $V$. As 
\begin{align*}
\displaystyle\lim_{V \ni\omega\to \omega_0}(\omega-\omega_{0})^{p} = +\infty,
\end{align*}
we obtain the desired continuity. 

To conclude the proof, we prove that $\xi_R^{\vee}$ is continuous. As in the previous case, it suffices to prove that, for any $\alpha\in E_R^{\vee}$, the function $\|\alpha\|_{\cdot,\omega}$ is continuous on $\Omega_{R,\infty}$. Note that, by construction, for $\omega\in\Omega_{R,\infty}$, the pseudo-norm $\|\cdot\|_{\omega,\ast}$ is the pseudo-norm constructed on $E_R^{\vee}$ by lifting the norm $\|\cdot\|_{\ast}$ on $E^{\vee}$. Therefore, the above case can be applied and implies that $\xi_R^{\vee}$ is continuous. 
\end{proof}

\begin{definition}
\label{def:example_avb_Nevanlinna}
Let $E$ be a finite-dimensional complex vector space equipped with a norm $\|\cdot\|$. Let $R>0$. Let $E_R:= E_{\bC} \otimes_{\bC} K_R$ and $\xi_R = (\|\cdot\|_{\omega})_{\omega\in\Omega_R}$ be the pseudo-norm family on $E_R$ constructed as above. Then the adelic vector bundle $\overline{E_R}:=(E_R,\xi_R)$ is called the \emph{induced adelic vector bundle on} $S_R$ by the complex normed vector space $(E,\|\cdot\|)$. Note that any basis of $E$ is globally adapted to $\xi_R$.
\end{definition}

\begin{remark}
\label{rem:example_avb_Nevanlinna_finite_extension}
Let $E$ be a finite-dimensional complex vector space equipped with a norm $\|\cdot\|$. Let $R>0$. Consider a finite extension $K'/K_R$. Then one can use the same arguments as above to prove that $(E,\|\cdot\|)$ induces an adelic vector bundle on the topological adelic curve $S_R \otimes_{K_R} K'$.
\end{remark}

\subsection{Adelic vector bundles on families of topological adelic curves}
\label{sub:adelic_vector_bundles_families_tac}

\begin{proposition-definition}
\label{prop-def:adelic_vector_bundles_families_tac}
Let $\mathbf{S}=(I,\cU,(S_{i}=(K_{i},\phi_{i}:\Omega_{i}\to M_{K_{i}},\nu_{i}))_{i\in I},K)$ be a family of topological adelic curves. 

\begin{itemize}
	\item[(1)] By a $K$-\emph{vector space equipped with a pseudo-norm family} $(E,\xi)$ on $\mathbf{S}$, we mean a finite-dimensional $K$-vector space $E$ and an equivalence class $\xi=[(E_{i},\xi_{i})_{i\in I}]$, where the $(E_{i},\xi_{i})$'s are finite-dimensional $K_{i}$-vector spaces equipped with a pseudo-norm family over the $S_{i}$'s such that $E\otimes_{K}\prod_{\cU}K_{i}\cong \prod_{\cU}E_{i}$ and two such families $(E_{i},\xi_{i})_{i\in I},(E'_{i},\xi'_{i})_{i\in I}$ are declared to be equivalent if there exist $\cU$-almost everywhere isomorphisms $E_{i}\cong E'_{i}$ yielding compatible isomorphisms $\xi_{i}\cong\xi'_{i}$. In this case, $\xi$ is called a \emph{pseudo-norm family} on $\mathbf{S}$. Moreover, $\xi$ is said to be \emph{ultrametric at non-Archimedean places} if $\xi_{i}$ is ultrametric on $\Omega_{i,\um}$ $\cU$-almost everywhere.
	\item[(2)] Let $E$ be a finite-dimensional $K$-vector space. A pseudo-norm family $\xi=[(E_{i},\xi_{i})_{i\in I}]$ on $E$ is said to be \emph{dominated}, resp. \emph{usc/lsc/continuous/measurable} if $\xi_{i}$ is dominated, resp. usc/lsc/measurable $\cU$-almost everywhere. This is independent of the choice of the representative of $\xi$.
	\item[(3)] A $K$-vector space equipped with a pseudo-norm family $(E,[(\overline{E_{i}})_{i\in I}])$ is said to be a \emph{usc/lsc/measurable adelic vector bundle} on $\mathbf{S}$ if the $\overline{E_{i}}$'s are usc/lsc/measurable adelic vector bundles on the $S_{i}$'s $\cU$-almost everywhere.
	\item[(4)] Let $\overline{E}=(E,\xi)$ be $K$-vector space equipped with a pseudo-norm family on $\mathbf{S}$. 
	\begin{itemize}
	\item[(i)] Let $F\subset E$ be a vector subspace. Then $\xi$ induces by restriction, a pseudo-norm family $\xi_{F}$ on $F$. Moreover, if $\overline{E}$ is a usc/lsc/measurable adelic vector bundle on $\mathbf{S}$, then so is $\overline{F}:=(F,\xi_{F})$.
		\item[(ii)] $\xi$ induces a pseudo-norm family $\xi^{\vee}$ on $E^{\vee}$. Moreover, if $\overline{E}$ is a usc adelic vector bundle on $\mathbf{S}$, then $\overline{E}^{\vee}:=(E^{\vee},\xi^{\vee})$ is an lsc adelic vector bundle on $\mathbf{S}$. If we further assume that $\xi$ and $\xi^{\vee\vee}$ are continuous, then $\overline{E}^{\vee}$ is an adelic vector bundle on $\mathbf{S}$.
		\item[(iii)] Let $G$ be a quotient vector space of $E$. Then $\xi$ induces a pseudo-norm family $\xi_{G}$ on $G$. Moreover, if $\overline{E}$ is a usc adelic vector bundle $\mathbf{S}$, then so is $\overline{G}:=(G,\xi_{G})$. If we further assume that $\overline{E}$ is an adelic vector bundle on $\mathbf{S}$ and that $\xi=\xi^{\vee}$, then $\overline{G}$ is also an adelic vector bundle on $\mathbf{S}$. 
		\item[(iv)] Then $\xi$ induces a pseudo-norm family $\det(\xi)$ on $\det(E)$. Moreover, if $\overline{E}$ is a usc adelic vector bundle on $\mathbf{S}$, then $\det(\overline{E}):=(\det(E),\det(\xi))$ is a measurable adelic vector bundle on $\mathbf{S}$.
	\end{itemize}
\end{itemize}
\end{proposition-definition}

\begin{proof}
The definitions of pseudo-norm families and the corresponding algebraic constructions follow from \L{}o\'s theorem. The assertions concerning adelic vector bundles follow from Proposition \ref{prop:constructions_adelic_vector_bundles} combined with the fact that all the involved algebraic constructions are compatible with the extension of scalars.
\end{proof}

\begin{example}
\label{example:adelic_vector_bundles_families_tac}
\begin{itemize}
	\item[(1)] Let $S=(K,\phi:\Omega\to M_{K},\nu)$ be an adelic curve and consider the family of topological adelic curves $\mathbf{S}=(I,\cU,(S)_{i\in I},\prod_{\cU}K)$, where $I$ is an infinite set and $\cU$ is a free ultrafilter on $I$. For any $i\in I$, let $\overline{E_{i}}=(E_{i},\xi_{i})$ be a (usc/lsc/measurable) adelic vector bundle on $S$. Then $(\prod_{\cU}E_{i},[(\overline{E_{i}})_{i\in I}])$ is a (usc/lsc/measurable) adelic vector bundle on $\mathbf{S}$.
	\item[(2)] Consider the family of topological adelic curves $\mathbf{S}=(\bR_{>0},\cU,(S_{R})_{R>0},\cM(\bC))$ from Example \ref{example:families_of_tac} (2). Let $\overline{E}$ be a finite-dimensional complex normed vector space. Let $R>0$. In \S \ref{sub:example_in_Nevanlinna_theory_avb}, we have constructed the induced adelic vector bundle $\overline{E_R}=(E_R,\xi_R)$ on $S_R$ by $\overline{E}$ (cf. Definition \ref{def:example_avb_Nevanlinna}). By construction, $\overline{E_{\cM(\bC)}}=(E\otimes_{\bC}\cM(\bC),[(\overline{E_{R}})_{R>0}])$ is an adelic vector bundle on $\mathbf{S}$. We call it the adelic vector bundle \emph{induced} by $\overline{E}$ on $\mathbf{S}$.
\end{itemize}
\end{example}

\begin{definition}
\label{def:degree_adelic_vector_bundles_families_tac}
Let $\mathbf{S}=(I,\cU,(S_{i}=(K_{i},\phi_{i}:\Omega_{i}\to M_{K_{i}},\nu_{i}))_{i\in I},K)$ be a family of topological adelic curves and $\mathbf{\overline{E}}=(E,[(\overline{E_i})_{i\in I}]$ be a measurable adelic vector bundle on $\mathbf{S}$. Let $s\in E\setminus \{0\}$, writing $s=\prod_{\cU}s_{i}$, where $\cU$-almost everywhere $s_{i}\in E_{i}\setminus\{0\}$, we define the \emph{Arakelov degree} of $s$ w.r.t. $\overline{E}$ by
\begin{align*}
\widehat{\deg}_{\overline{E}}(s) := \left[\left(\widehat{\deg}_{\overline{E_{i}}}(s_{i})\right)_{i\in I}\right]\in\displaystyle\prod_{\cU}\bR.
\end{align*}

Let $\sim$ be an equivalence relation on $\prod_{\cU}\bR$ which is compatible with the additive group structure. Assume that $\mathbf{S}$ is asymptotically proper w.r.t. $\sim$ (cf. Definition \ref{def:family_of_tac_proper}). Assume that $\overline{E}$ is a usc adelic vector bundle on $\mathbf{S}$. Let $\eta,\eta'\in \det(E)\setminus\{0\}$. Then 
\begin{align*}
\widehat{\deg}_{\det(\overline{E})}(\eta) - \widehat{\deg}_{\det(\overline{E})}(\eta') \sim 0.
\end{align*}
Therefore, the class of $\widehat{\deg}_{\det(\overline{E})}(\eta)$ in $\prod_{\cU}\bR/\sim$ yields a well-defined element denoted by $\widehat{\deg}(\overline{E})$ called the \emph{Arakelov degree} of $\overline{E}$. 
\end{definition}

\begin{example}
\label{example:degree_adelic_vector_bundles_families_tac_Nevanlinna}
\begin{itemize}
	\item[(1)] Let $S=(K,\phi:\Omega\to M_{K},\nu)$ be a proper adelic curve and consider the family of topological adelic curves $\mathbf{S}=(I,\cU,(S)_{i\in I},\prod_{\cU}K)$, where $I$ is an infinite set and $\cU$ is a free ultrafilter on $I$. For any $i\in I$, let $\overline{E_{i}}=(E_{i},\xi_{i})$ be a adelic vector bundle on $S$. Consider the adelic vector bundle $\overline{E}:=(E_{\mathbf{S}}=\prod_{\cU}E_{i},[(\overline{E_{i}})_{i\in I}])$ on $\mathbf{S}$. Then 
	\begin{align*}
	\widehat{\deg}(\overline{E}) = \left[\left(\widehat{\deg}(\overline{E_{i}})\right)_{i\in I}\right] \in \displaystyle\prod_{\cU}\bR.
	\end{align*}
	\item[(2)] Consider the family of topological adelic curves $\mathbf{S}=(\bR_{>0},\cU,(S_{R})_{R>0},\cM(\bC))$ from Example \ref{example:families_of_tac} (2). Let $\overline{E}$ be a one-dimensional complex normed vector space. Consider the adelic vector bundle $\overline{E_{\cM_{\bC}}}:=(E\otimes_{\bC}\cM(\bC),[(\overline{E_{R}})_{R>0}])$ induced by $\overline{E}$ on $\mathbf{S}$ (cf. Example \ref{example:adelic_vector_bundles_families_tac} (2)). Let $e\in E\setminus\{0\}$, then
	\begin{align*}
	\widehat{\deg}(\overline{E_{\cM_{\bC}}}) = \left[\left(-\log\|e\|\right)_{R>0}\right] \in \displaystyle \prod_{\cU}\bR /\sim_{\fin}.
	\end{align*}
\end{itemize}
\end{example}

\section{Slopes of adelic vector bundles on a proper topological adelic curve}
\label{sec:slopes_proper_case}

In this section, we introduce slope theory for adelic vector bundles on a fixed proper topological adelic curve adelic curve $S=(K,\phi : \Omega \to M_{K},\nu)$ (\S \ref{sub:degree_positive_degree}-\ref{sub:HN_general}). We then give the Nevanlinna theoretic variant (\S \ref{sub:slopes_adelic_vector_bundles_family_tac}).

\subsection{Degree, positive degree}
\label{sub:degree_positive_degree}

\begin{proposition}[\cite{ChenMori}, Propositions 4.3.8, 4.3.10, 4.3.19 and 4.3.13]
\label{prop:CM_4.3.2}
Let $\overline{E}=(E,\xi)$ be an adelic vector bundle on $S$. The following assertions hold.
\begin{itemize}		
	\item[(1)] Assume that $\overline{E}$ is Hermitian. Then we have the equality
	\begin{align*}
	\widehat{\deg}(E,\xi) = - \widehat{\deg}(E^{\vee},\xi^{\vee}).
	\end{align*}
	\item[(2)] In general, we have the inequality 
	\begin{align*}
	0 \leq \widehat{\deg}(E,\xi) + \widehat{\deg}(E^{\vee},\xi^{\vee}) \leq \frac{1}{2}\dim_{K}(E)\log\dim_{K}(E)\nu(\Omega_{\infty}).
	\end{align*}
	\item[(3)] Let $\overline{E'}=(E',\xi')$ be another adelic vector bundle on $S$. Assume that the double-dual pseudo-norm families $\xi^{\vee\vee},\xi'^{\vee\vee}$ are continuous (e.g. if $\xi$ and $\xi'$ are ultrametric on $\Omega_{\um}$). Then we have the equality
	\begin{align*}
	\widehat{\deg}(\overline{E}\otimes_{\epsilon,\pi}\overline{E'}) = \dim_{K}(E')\widehat{\deg}(\overline{E}) + \dim_{K}(E)\widehat{\deg}(\overline{E'}).
	\end{align*}
	\item[(4)] Let 
	\begin{align*}
	0 = E_0 \subset E_1 \subset \cdots \subset E_n = E
	\end{align*}
be a flag of vector subspaces of $E$. For any $i=1,...,n$, denote by $\xi_i$ the restriction of $\xi$ to $E_i$ and by $\eta_{i}$ the quotient pseudo-norm family induced by $\xi_i$ on $E_{i}/E_{i-1}$. Then  we have the inequality
\begin{align*}
\displaystyle\sum_{i}^{n} \widehat{\deg}(E_i/E_{i-1},\eta_i) \leq \widehat{\deg}(E,\xi).
\end{align*}
Moreover, if $\xi$ is Hermitian, then the above inequality is an equality.
\end{itemize}
\end{proposition}

\begin{definition}
\label{def:positive_degree_proper}
Let $\overline{E}=(E,\xi)$ be an adelic vector bundle on $S$. Let $F\subset E$ be any vector subspace and denote by $\xi_F$ the restriction of $\xi$ to $F$. Then Proposition \ref{prop:constructions_adelic_vector_bundles} (1) implies that $\overline{F}:=(F,\xi_F)$ is an adelic vector bundle on $S$. We define the \emph{positive degree} of $\overline{E}$ as 
\begin{align*}
\widehat{\deg}_{+}(\overline{E}) := \displaystyle\sup_{F\subseteq E} \widehat{\deg}(\overline{F}),
\end{align*}
where $F$ runs over the set of all vector subspaces of $E$.
\end{definition}

\begin{remark}
\label{rem:positive_degree_proper}
The positive degree of an adelic vector bundle plays the role of the number of "small sections" in the classical framework of Arakelov geometry over number fields.
\end{remark}

\subsection{Slopes}

\begin{definition}
\label{def:slopes_proper}
Let $\overline{E}=(E,\xi)$ be a usc adelic vector bundle on $S$. Assume that $E\neq \{0\}$.
\begin{itemize}
	\item[(1)] We define the \emph{slope} of $\overline{E}$ as 
	\begin{align*}
	\widehat{\mu}(\overline{E}) := \frac{\widehat{\deg}(\overline{E})}{\dim_{K}(E)}.
	\end{align*}
	\item[(2)] We define the \emph{maximal slope} of $\overline{E}$ as
	\begin{align*}
	\widehat{\mu}_{\max}(\overline{E}) := \displaystyle\sup_{\{0\}\neq F \subseteq E} \widehat{\mu}(\overline{F}),
	\end{align*}
	where $F$ runs over the set of non-zero vector subspaces of $E$. 
	\item[(3)] Assume that $\overline{E}$ is an adelic vector bundle. Let $E \twoheadrightarrow G$ be a quotient vector space of $E$. Denote by $\xi_G$ the quotient pseudo-norm family on $G$ induced by $\xi$ and let $\overline{G}:=(G,\xi_G)$. Proposition \ref{prop:constructions_adelic_vector_bundles} (2) implies that $\overline{G}$ is a usc adelic vector bundle on $S$. We define the \emph{minimal slope} of $\overline{E}$ as 
	\begin{align*}
	\widehat{\mu}_{\min}(\overline{E}) := \displaystyle\inf_{E \twoheadrightarrow G \neq \{0\} } \widehat{\mu}(\overline{G}),
	\end{align*}
	where $G$ runs over the set of non-zero quotient vector spaces of $E$.
\end{itemize}
\end{definition}

\subsection{Harder-Narasimhan filtration over proper topological adelic curves}
\label{sub:HN_general}

In the recent work \cite{ChenJeannin23}, Chen and Jeannin developed a very general framework for proving the existence of Harder-Narasimhan filtrations which is inspired by game theory. We construct Harder-Narasihman filtrations over proper topological adelic curves using (\emph{loc. cit.}, Theorem 1.1).

Let $\overline{E}=(E,\xi)$ be an adelic line bundle on $S$. We assume that the pseudo-norm family $\xi$ is ultrametric on $\Omega_{\um}$. We consider the set $\cL(E)$ of vector subspaces of $E$, equipped with the ordering defined by the inclusion relation, it is a bounded lattice. We also consider the totally ordered set $[-\infty,+\infty]$ with the usual ordering. For any $(F',F) \in \cP_{<}(\cL(E))$, namely $F' \subsetneq F$ are vector subspaces of $E$, recall that the subquotient $\overline{F/F'}$ is an adelic vector bundle on $S$ (cf. Proposition \ref{prop:constructions_adelic_vector_bundles} (3)). Therefore we can define
\begin{align*}
\mu(F',F) := \widehat{\mu}(\overline{F/F'}) \in \bR.
\end{align*}
Then we obtain a Harder-Narasimhan game on $\cL(E)$ with pay-off function $\mu$ (cf. \cite{ChenJeannin23}, \S 2.1). Note that, for any $(F',F) \in \cP_{<}(\cL(E))$, we have
\begin{align*}
\mu_{A}(F',F) = \widehat{\mu}_{\min}(\overline{F/F'})\in [-\infty,+\infty].
\end{align*}

\begin{definition}
\label{def:semi-stable_adelic_vector_bundle}
We say that $\overline{E}$ is \emph{semi-stable} if for any non-zero vector subspace $F \subset E$, we have the inequality
\begin{align*}
\widehat{\mu}_{\min}(\overline{F}) \leq \widehat{\mu}_{\min}(\overline{E}).
\end{align*}
Note that $\overline{E}$ is semi-stable iff the Harder-Narasimhan game above is semi-stable is the sense of (\emph{loc. cit.}, Definition 3.6).
\end{definition}

\begin{theorem}
\label{th:HN_general_proper}
Let $\overline{E}=(E,\xi)$ be an adelic vector bundle on $S$. We assume that the pseudo-norm family $\xi$ is ultrametric on $\Omega_{\um}$. Then there exists a unique flag 
\begin{align*}
0=E_0 \subsetneq E_1 \subsetneq \cdots \subsetneq E_n = E,
\end{align*}
of $E$, such that
\begin{itemize}
	\item[(1)] for any $i=1,...,n$, $\overline{E_i/E_{i-1}}$ is semistable;
	\item[(2)] we have the inequalities 
	\begin{align*}
	\widehat{\mu}(\overline{E_1/E_0}) > \cdots > \widehat{\mu}(\overline{E_n/E_{n-1}}).
	\end{align*}
\end{itemize}
\end{theorem}

\begin{proof}
Let us see that the hypotheses (i-iv) of (\emph{loc. cit.}, Theorem 1.1) are satisfied. (i), (iii) and (iv) are trivially true. We only need to check that the pay-off function $\mu$ is convex, namely, for any vector subspaces $F,F'\subset E$ such that $F' \nsubseteq F$, we have the inequality
\begin{align*}
\widehat{\mu}(\overline{F/(F\cap F')}) \leq \widehat{\mu}(\overline{(F+F')/F})
\end{align*}
of slopes of adelic vector bundles on $S$. The canonical isomorphism 
\begin{align*}
f : F'/(F\cap F') \to (F+F')/F
\end{align*}
is constructed as follows. An element $\alpha \in F'/(F\cap F')$, represented by some $x'\in F'$, is mapped to the class of $x'$, viewed as an element of $F+F'$, in $(F+F')/F$. This is independent of the choice of $x'$. Write $\xi=(\|\cdot\|_{\omega})_{\omega\in\Omega}$. For any $\omega\in \Omega$, denote respectively by $\|\cdot\|_{\omega,1}$ and $\|\cdot\|_{\omega,2}$ the subquotient pseudo-norm induced by $\|\cdot\|_{\omega}$ on $F'/(F\cap F')$ and $(F+F')/F$. Let $\alpha \in F'/(F\cap F')$. For any representative $x'\in F'$ of $\alpha$, by construction of $f$, we have the inequality 
\begin{align*}
\|f(\alpha)\|_{\omega,2} \leq \|x'\|_{\omega}.
\end{align*}
As $x'$ is arbitrary, we obtain
\begin{align*}
\forall \alpha\in F'/(F\cap F'), \quad \|f(\alpha)\|_{\omega,2} \leq \|\alpha\|_{\omega,1}.
\end{align*} Therefore, we have
\begin{align*}
\widehat{\deg}(\overline{F'/(F\cap F')}) \leq \widehat{\deg}(\overline{(F+F')/F}).
\end{align*}
We can conclude by using the fact that 
\begin{align*}
\dim_{K}(F'/(F\cap F')) &= \dim_{K}(F') -\dim_{K}(F\cap F')) \\
&= \dim_{K}(F+F') - \dim_{K}(F) = \dim_{K}((F+F')/F).
\end{align*}
\end{proof}

\begin{definition}
\label{def:HN_flag_general_proper}
Let $\overline{E}=(E,\xi)$ be an adelic line bundle on $S$. We assume that the pseudo-norm family $\xi$ is ultrametric on $\Omega_{\um}$. Then the flag constructed in Theorem \ref{th:HN_general_proper} is called the \emph{Harder-Narasimhan flag} of $\overline{E}$. 
\end{definition}

\subsection{Slope of adelic vector bundles on an asymptotically proper family of topological adelic curves in Nevanlinna theory}
\label{sub:slopes_adelic_vector_bundles_family_tac}

In the last subsection of the second part of this article, we introduce the notion of slopes and Harder-Narasimhan filtrations for the family of topological adelic curves arising in Nevanlinna theory (Example \ref{example:families_of_tac} (2)). 

Throughout this section, we consider the family of topological adelic curves $\mathbf{S}=(\bR_{>0},\cU,(S_{R})_{R>0},\cM(\bC))$ from Example \ref{example:families_of_tac} (2). It is asymptotically proper w.r.t. the equivalence relation $\sim_{\fin}$ on $\prod_{\cU}\bR$. Recall from Claim \ref{claim:total_order_target_height}, $\prod_{\cU}\bR/\sim_{\fin}$ has the structure of a totally ordered $\bR$-vector space. 

Let $\bR_{\mathbf{S}}$ denote the Dedekind-MacNeille completion of $\prod_{\cU}\bR/\sim_{\fin}$. This is a complete lattice. Moreover, for any $x\in \bR_{\mathbf{S}}\setminus\{0\}$ and $\lambda\in \bR_{>0}$, there exists a unique element $y\in \bR_{\mathbf{S}}$ such that $\lambda\cdot y=x$. We denote it by $x/\lambda$. 

\subsubsection{Slopes}

\begin{definition}
\label{def:slope_adelic_vector_bundles_family}
Let $\overline{E}=(E,\xi)$ be a usc adelic vector bundle on $\mathbf{S}$. Assume that $E\neq\{0\}$.
\begin{itemize}
	\item[(1)] Define the \emph{slope} of $\overline{E}$ as 
	\begin{align*}
	\widehat{\mu}(\overline{E}) := \frac{\widehat{\deg}(\overline{E})}{\dim_{K}(E)}\in \displaystyle\prod_{\cU}\bR/\sim_{\fin}.
	\end{align*}
	\item[(2)] We define the \emph{maximal slope} of $\overline{E}$ as
	\begin{align*}
	\widehat{\mu}_{\max}(\overline{E}) := \displaystyle \sup_{\{0\}\neq F \subseteq E} \widehat{\mu}(\overline{F})\in \bR_{\mathbf{S}},
	\end{align*}
	where $F$ runs over the set of non-zero vector subspaces of $E$ and for any such $F$, $\overline{F}$ denote the adelic vector bundle constructed in Proposition-Definition \ref{prop-def:adelic_vector_bundles_families_tac} (4.i).
	\item[(3)] Assume that $\overline{E}$ is a usc adelic vector bundle. Let $G$ be a quotient vector space of $E$. With the notation of Proposition-Definition \ref{prop-def:adelic_vector_bundles_families_tac} (4.iii), $\overline{G}$ is a usc adelic vector bundle on $\mathbf{S}$. We define the \emph{minimal slope} of $\overline{E}$ as 
	\begin{align*}
	\widehat{\mu}_{\min}(\overline{E}) := \displaystyle\inf_{E \twoheadrightarrow G \neq \{0\} } \widehat{\mu}(\overline{G})\in\bR_{\mathbf{S}},
	\end{align*}
	where $G$ runs over the set of non-zero quotient vector spaces of $E$.
\end{itemize}
\end{definition}

\subsubsection{Harder-Narasimhan filtration}

Let $\overline{E}=(E,\xi)$ be an adelic vector bundle on $\mathbf{S}$. Assume that $\xi$ is ultrametric at non-Archimedean places. Consider, as in the previous subsection, the bounded lattice $\cL(E)$ of vector subspaces of $E$. For any vector subspaces $\{0\}\neq F,F'$ of $E$ such that $F'\subsetneq F$, with the notation of Proposition-Definition \ref{prop-def:adelic_vector_bundles_families_tac} (4.i) and (4.iii), $\overline{F/F'}$ is an adelic vector bundle on $\mathbf{S}$. Thus, we can define
\begin{align*}
\mu(F',F) := \widehat{\mu}(\overline{F/F'}) \in \displaystyle\prod_{\cU}\bR/\sim_{\fin}. 
\end{align*}
We obtain a Harder-Narasimhan game on $\cL(E)$ with pay-off function $\mu$, for which
\begin{align*}
\mu_{A}(F',F) := \widehat{\mu}_{\min}(\overline{F/F'}) \in \bR_{\mathbf{S}},
\end{align*}
where $F,F'$ are as above. 

\begin{definition}
\label{def:semi-stability_adelic_vector_bundle_Nevanlinna}
We say that $\overline{E}$ is \emph{semi-stable} if the Harder-Narasimhan game above is semi-stable, i.e. if for any non-zero vector subspace $F\subset E$, we have
\begin{align*}
\widehat{\mu}_{\min}(\overline{F}) \leq \widehat{\mu}_{\min}(\overline{E}).
\end{align*}
\end{definition}

\begin{theorem}
\label{th:HN_Nevanlinna}
Let $\overline{E}=(E,\xi)$ be an adelic vector bundle on $\mathbf{S}$. Assume that $\xi$ is ultrametric at non-Archimedean places. Then there exists a unique flag 
\begin{align*}
0=E_{0} \subsetneq E_{1} \subsetneq \cdots \subsetneq E_{n} = E,
\end{align*}
of $E$, such that
\begin{itemize}
	\item[(1)] for any $i=1,...,n$, $\overline{E_i/E_{i-1}}$ is semistable;
	\item[(2)] we have the inequalities 
	\begin{align*}
	\widehat{\mu}(\overline{E_1/E_0}) > \cdots > \widehat{\mu}(\overline{E_n/E_{n-1}})
	\end{align*}
	in $\bR_{\mathbf{S}}$.
\end{itemize}
\end{theorem}

\begin{proof}
This is a consequence of (\cite{ChenJeannin23}, Theorem 1.1). As in the proof of Theorem \ref{th:HN_general_proper}, we only need to check that the pay-off function is convex. Looking at the proof of Theorem \ref{th:HN_general_proper}, we can directly reproduce the argument and use the definition of the order on $\bR_{\mathbf{S}}$ to conclude. 
\end{proof}

\part{Arithmetic varieties over topological adelic curves: adelic line bundles and heights}

In this final part, we study the analogue of adelic vector bundles over projective scheme defined over a topological adelic curve. We start by studying the local aspects of the theory (\S \ref{sec:metrics_MA_intersection_local}). Then we globalise these local ingredients (\S \ref{sec:pseudo-metric families}). We finish by constructing global height functions (\S \ref{sec:global_heights}-\ref{sec:arithmetic_intersection_product}).

\section{Pseudo-metrics: local case}
\label{sec:metrics_MA_intersection_local}

In this section, we introduce the analogue of pseudo-absolute values and pseudo-norms on a projective schemes defined over a topological adelic curve. This is done in two steps. We first introduce the notion of \emph{model pseudo-metrics}. Roughly speaking, this is done by considering models over the finiteness ring of a pseudo-valued field and by using the classical theory of Berkovich spaces on the special fibre of the model. The second step is considering equivalence classes of such model pseudo-metrics to obtain the notion of \emph{pseudo-metric} (\S \ref{sub:local_pseudo-metric}). These equivalence classes can be interpreted as metrics on line bundles over Zariski-Riemann spaces (\S \ref{sub:ZR_interpretation_local_pseudo-metrics}). We then extend the usual notions of Fubini-Study, semi-positive and plurisubharmonic metrics to the (model) pseudo-metric framework (\S \ref{sub:FS_local}-\ref{sub:semi-positive_local}).

Throughout this section, we fix a field $K$ equipped with a pseudo-absolute value $v=(\va,A,\m,\kappa)\in M_{K}$. Recall that we denote by $\widehat{\kappa}$ the completed residue field of $v$.

\subsection{Pseudo-metrics}
\label{sub:local_pseudo-metric}

\subsubsection{Model pseudo-metrics}
\label{subsub:model_pseudo-metric}

We first recall some notions from (\cite{Sedillot_pav}, \S 8.1). Fix a field $K$ equipped with a pseudo-absolute value $v=(\va,A,\m,\kappa)\in M_{K}$. Recall that we denote by $\widehat{\kappa}$ the completed residue field of $v$. Fix a projective $K$-scheme $X$ and a projective model $\cX$ of $X/A$. The \emph{completed special fibre} of $\cX$ is defined by $\widehat{\cX_s}:=\cX \otimes_{A} \widehat{\kappa}$. Then the \emph{(local) model analytic space} associated with $\cX$ is the Berkovich analytification $\widehat{\cX_s}^{\an}$. 

\begin{definition}
\label{def:local_pseudo-metric}
Let $L$ be a line bundle on $X$. A \emph{(local) model pseudo-metric} on $L$ (over $v$) is the data $((\cX,\cL),\varphi)$, where $(\cX,\cL)$ is a projective model of $(X,L)$ over $A$ and $\varphi$ is a metric on the pullback of $\cL$ to $\widehat{\cX_s}$. In this case, we say that $\cX$ is the \emph{underlying model} of the model pseudo-metric $((\cX,\cL),\varphi)$. Such a model pseudo-metric $(\cL,((\cX,\cL),\varphi))$ is respectively called \emph{lsc}, \emph{usc}, \emph{continuous} if the metric $\varphi$ is lsc, usc, continuous. When there is no possible confusion on the model $\cX$, we allow ourselves to denote the model pseudo-metric by $(\cL,\varphi)$.
\end{definition}

\begin{remark}
\label{rem:pseudo-metric_absolute_value}
Assume that $v$ is a usual absolute value on $K$. Then any model pseudo-metric on a line bundle $L$ on $X$ is a metric on $L$ over the Berkovich space $(X\otimes_{K} K_{v})^{\an}$. 
\end{remark}

\begin{proposition}
\label{prop:existence_local_model-pseudo-metrics}
Let $L$ be a line bundle on $X$. Then there exists a model pseudo-metric on $L$.
\end{proposition}

\begin{proof}
Write $L$ as a difference $L_{1}-L_{2}$, where $L_{1},L_{2}$ are very ample line bundles on $X$. For $i=1,2$, choose a closed embedding $\iota_{i}:X \to \mathbb{P}^{n_{i}}_{K}$ defined by $L_{i}$ and denote by $\cX_{i}$ the schematic closure of $X$ in $\mathbb{P}^{n_{i}}_{A}$, this is a projective model of $X$ over $A$. Moreover, the pullback $\cL_{i}$ of $\cO_{\mathbb{P}^{n_{i}}_{A}}$ to $\cX_{i}$ gives a model of $L_{i}$. Choose a projective model $\cX$ of $X$ over $A$ dominating $\cX_{1}$ and $\cX_{2}$. Taking the difference of the pullback of the $\cL_{i}$'s to $X$, we get a model of $L$ on $\cX$. Choosing arbitrary metrics on the pullback of the $\cL_{i}$'s to the completed special fibre of the $\cX_{i}$'s, we get a model pseudo-metric on $L$.
\end{proof}

\begin{proposition-definition}[\cite{Sedillot_pav}, \S 8.1]
\label{propdef:supnorm_pseudo-metric}
Assume that $X$ is geometrically reduced if $A=K$. Let $L$ be a line bundle on $X$ and $((\cX,\cL),\varphi)$ be a model pseudo-metric on $L$, where $\cX$ is a flat coherent model of $X/A$. Then $\varphi$ induces a pseudo-norm $\|\cdot\|_{(\cL,\varphi)}$ on $H^{0}(X,L)$ with finiteness module $H^{0}(\cX,\cL)$, kernel $\m H^{0}(\cX,\cL)$ and residue norm $\|\cdot\|_{\varphi}$, namely the supnorm over $\widehat{\cX_{s}}^{\an}$ induced by the metric $\varphi$. When no confusion may arise, we simply denote $\|\cdot\|_{(\cL,\varphi)}$ by $\|\cdot\|_{\varphi}$.
\end{proposition-definition}


\begin{notation}
Let $L$ be a line bundle on $X$ and let $(\cX,\cL)$ be a model of $(X,L)$. Then we denote by $L_v$ the pullback of $\cL$ to the completed special fibre $X_v := \cX\otimes_{A} \widehat{\kappa}$.
\end{notation}

\begin{proposition-definition}
\label{propdef:local_model_pseudo-metric}
\begin{itemize}
	\item[(1)] Let $L$ be a line bundle on $X$. Let $(\cL,\varphi)$ be a model pseudo-metric on $L$. Then $(-\cL,-\varphi)$ is a model pseudo-metric on $-L$. Moreover, if $(\cL,\varphi)$ is respectively continuous/usc/lsc, then $(-\cL,-\varphi)$ is continuous/lsc/usc.
	\item[(2)] Let $L,L'$ be two line bundles on $X$. Let $(\cL,\varphi),(\cL',\varphi')$ be model pseudo-metrics on $L,L'$ respectively with the same underlying model. Then $(\cL+ \cL',\varphi + \varphi')$. Moreover, if $(\cL,\varphi)$ and $(\cL',\varphi')$ are both respectively continuous/usc/lsc, then $(\cL+ \cL',\varphi + \varphi')$ is continuous/usc/lsc.
	\item[(3)] Let $f:Y \to X$ be a projective morphism of $K$-schemes and $\cX$ be a projective model of $X$ over $A$. Let $\cY$ be a projective model of $Y$ over $A$ and let $g:\cY \to \cX$ be a projective morphism such that $g$ extends $f$.  Then we have a commutative diagram with Cartesian squares
	\begin{center}
\begin{tikzcd}
Y \arrow[d] \arrow[r, "f"]               & X \arrow[d]       \\
\cY \arrow[r, "g"]                       & \cX               \\
\cY_{s} \arrow[u] \arrow[r, "\tilde{g}"] & \cX_{s} \arrow[u]
\end{tikzcd}.
	\end{center}
Let $L$ be a line bundle on $X$. Let $(\cL,\varphi)$ be a model pseudo-metric on $L$ with underlying model $\cX$. Then $(g^{\ast}\cL,\tilde{g}^{\ast}\varphi)$ is a model pseudo-metric on $L$ with underlying model $\cY$. By abuse of notation, we denote this model pseudo-metric by $g^{\ast}(\cL,\varphi)$ and $\tilde{g}^{\ast}\varphi$ by $g^{\ast}\varphi$. $g^{\ast}(\cL,\varphi)$ is called the \emph{pullback} of $(\cL,\varphi)$ w.r.t. the morphism $g : \cY \to \cX$. Moreover, $g^{\ast}(\cL,\varphi)$ is continuous/usc/lsc if $\varphi$ is continuous/usc/lsc. 
	\item[(4)] Let $K'/K$ be a field extension and let $v'=(\va',A',\m',\kappa')$ be a pseudo-absolute value on $L$ such that $v'$ extends $v$. Let $\cX$ be a projective model of $X$ over $A$. Consider the fibre products $f : X' := X \otimes_K K' \to X$ and $g:\cX':=\cX \otimes_{A} A'$. 
	\begin{itemize}
		\item[(i)] Then $\cX'$ is a projective model of $X'$ over $A'$. Moreover, $\cX'$ is flat/coherent if $\cX$ is flat/coherent. Moreover, we have a natural isomorphism $\cX'_s \cong \cX_{s} \otimes_{\kappa} \kappa'$.
		\item[(ii)] Let $(\cL,\varphi)$ be a model pseudo-metric on $L$ with underlying model $\cX$. Let $L',\cL'$ denote the pullbacks of $L,\cL$ to $X',\cX'$ respectively. Then $(\cX',\cL')$ is a model of $(X',L')$ and denote by $\varphi'$ the metric on $L'_{v'}$ induced by $\varphi$ on $L_v$. Then $f^{\ast}(\cL,\varphi) := (\cL',\varphi')$ defines a model pseudo-metric on $L'$ with underlying model $\cX'$ which is called the \emph{extension} of $(\cL,\varphi)$ w.r.t. the field extension $K'/K$.
\end{itemize}		
\end{itemize}
\end{proposition-definition}

\begin{proof}
\textbf{(1):} Since $(\cX,\cL)$ is a model of $(X,L)$, we deduce that $(\cX,\cL^{\vee})$ is a model of $(X,L^{\vee})$. Moreover, $
\varphi^{\vee}$ defines a metric on $L^{\vee}_{v}$ which is respectively continuous/lsc/usc if $\varphi$ is continuous/usc/lsc.

\textbf{(2):} It is clear that $(\cX,\cL +\cL')$ is a model of $(X,L+L')$. The assertion about the regularity of $(\cL+ \cL',\varphi + \varphi')$ follows from the definition of the sum of two metrics on a Berkovich space.

\textbf{(3):} We first justify that the squares are Cartesian. For the first one, we have isomorphisms
\begin{align*}
\cY \times_{\cX} X \cong \cY \times_{\cX} (\cX \otimes_{A} K) \cong \cY \otimes_{A} K \cong Y.
\end{align*}
For the second one, we have isomorphisms
\begin{align*}
\cY_s \cong \cY \otimes_{A} \kappa \cong \cY \times_{\cX} (\cX \otimes_{A} \kappa) \cong \cY\times_{\cX} \cX_s.
\end{align*}
Finally, the assertion about model pseudo-metrics is clear from the properties of usual metrics on Berkovich spaces.

\textbf{(4.i):} We have isomorphisms 
\begin{align*}
\cX' \otimes_{A'} K' \cong (\cX \otimes_{A} A') \otimes_{A'} K' \cong \cX \otimes_{A} K' \cong (\cX \otimes_{A} K) \otimes_{K} K' \cong X',
\end{align*}
hence $\cX'$ is a model of $X'$ over $A$. Moreover, $\cX'$ is a projective $A'$-scheme. Since being flat and of finite presentation is preserved by base change, $\cX'$ is flat/coherent if $\cX$ is flat/coherent. We now justify the assertion about the Cartesian squares. Finally, we have isomorphisms 
\begin{align*}
\cX_s \otimes_{\kappa} \kappa' = (\cX \otimes_{A} \kappa) \otimes_{\kappa} \kappa' \cong \cX \otimes_{A} \kappa' \cong (\cX \otimes_{A} A') \otimes_{A'} \kappa' \cong \cX' \otimes_{A'} \kappa' = \cX'_s.
\end{align*}

\textbf{(4.ii):} This is clear using (4.i) and the properties of usual metrics on Berkovich Spaces.
\end{proof}

\begin{remark}
\label{remark:models_extension}
Let $f : Y \to X$ be a projective morphism of $K$-schemes. Let $L$ be a line bundle on $X$ equipped with a model pseudo-metric $(\cL,\varphi)$ with underlying model denoted by $\cX$. Then the extension of $(\cL,\varphi)$ to $f^{\ast}L$ is subject to the choice of a model $\cY$ of $Y$ over $A$ extending $f$. Let $\cY'$ be an arbitrary model of $Y$ over $A$. Then consider the schematic closure $\cY$ of the graph of $f$ in $\cY' \times_{\Spec(A)} \cX$. Then we see that $\cY$ is a model of $Y$ over $A$ such that $\cY$ dominates $\cY'$ and $g:\cY \to \cX$ extends $f$. 
\end{remark}

\begin{definition}
\label{def:distance_local_pseudo-metric}
Let $L$ be a line bundle on $X$. Let $(\cL,\varphi),(\cL,\varphi')$ be two continuous model pseudo-metrics on $L$ that have the same underlying model $\cX$. Then $(\cO_{\cX},\varphi-\varphi')$ is a model pseudo-metric on $\cO_{X}$ with underlying model $\cX$. Then we define the \emph{distance} between $\varphi$ and $\varphi'$ as
\begin{align*}
d_{v}(\varphi,\varphi'):= \displaystyle\sup_{x\in \widehat{\cX_{s}}^{\an}} |\log|\cdot|_{\varphi}(x)-\log|\cdot|_{\varphi'}(x)|,
\end{align*}
where, for all $x\in  \widehat{\cX_{s}}^{\an}$, $\log |\cdot|_{\varphi}(x)-\log|\cdot|_{\varphi'}(x):=\log |\ell|_{\varphi}(x)-\log|\ell|_{\varphi'}(x)$ for some $\ell\in \cL_{s}(x)=\cL_{s}\otimes_{\cO_{\cX_{s}}}\widehat{\kappa}(x)$ (this value does not depend on the choice of $\ell$). 
\end{definition}

We end this paragraph by defining an equivalence relation on model-pseudo-metrics on a fixed line bundle. 

\begin{proposition-definition}
\label{propdef:equivalence_relation_model_pseudo-metrics}
Let $L$ be a line bundle on $X$. We define a relation $\sim$ on model pseudo-metrics on $L$ as follows. For any two model pseudo-metrics $((\cX,\cL),\varphi),((\cX',\cL'),\varphi')$, we write $((\cX,\cL),\varphi)\sim ((\cX',\cL'),\varphi')$ iff there exists a projective model $\cX''$ of $X$ over $A$ and $A$-morphisms $p:\cX''\to \cX$, $q:\cX''\to \cX'$ such that we have an isomorphism $p^{\ast}(\cL,\varphi)\cong q^{\ast}(\cL',\varphi')$.
\end{proposition-definition}

\begin{proof}
It is clear that $\sim$ is reflexive and symmetric. The transitivity follows from the fact that for any two projective models $\cX,\cX'$ of $X$ over $A$, there exist a projective model $\cX''$ of $X$ over $A$ and $A$-morphisms $p:\cX''\to \cX$, $q:\cX''\to \cX'$.
\end{proof}

\subsubsection{Pseudo-metrics}
\label{subsub:pseudo-metric}

We now make the general definition of a pseudo-metric on a projective $K$-scheme. These objects are built from the model pseudo-metrics introduced in the previous paragraph. We fix a projective $K$-scheme $X$.

\begin{definition}
\label{def:pseudo-metric}
Let $L$ be a line bundle on $X$. A \emph{(local) pseudo-metric} (in $v$) on $L$ is an equivalence class of model pseudo-metrics on $L$ w.r.t. the equivalence relation $\sim$ introduced in Proposition \ref{propdef:equivalence_relation_model_pseudo-metrics}. A line bundle on $X$ equipped with a pseudo-metric is called a \emph{pseudo-metrised} line bundle on $X$. A pseudo-metric is called \emph{continuous/lsc/usc} if it can be represented by a continuous/lsc/usc. Note that in this case, any representative is continuous/lsc/usc.
\end{definition}

\begin{remark}
\label{rem:existence_local_pseudo-metrics}
\begin{itemize}
	\item[(1)] Proposition \ref{prop:existence_local_model-pseudo-metrics} ensures that there exists a pseudo-metric on any line bundle on $X$.
	\item[(2)] The terminology "pseudo-metric" will be justified in Proposition \ref{prop:local_pseudo-metrics_justification_terminology}.
\end{itemize}
\end{remark}

\begin{notation}
\label{notation:pseudo-metric}
Let $L$ be a line bundle on $X$ and let $((\cX,\cL),\varphi)$ be a model pseudo-metric on $L$. 
\begin{itemize}
	\item[(1)] By "let $[((\cX,\cL),\varphi)]$ be a pseudo-metric on $L$", we mean that the pseudo-metric is the equivalence class of the model pseudo-metric $((\cX,\cL),\varphi)$ and we say that $[((\cX,\cL),\varphi)]$ is \emph{represented} on the projective model $\cX$.
	\item[(2)] By "let $[(\cL,\varphi)]$ be a pseudo-metric on $L$ represented on a projective model $\cX$", we mean that $[(\cL,\varphi)]$ is the pseudo-metric $[((\cX,\cL),\varphi)]$. We will also omit the mention of the projective model when no confusion may arise.
\end{itemize}
\end{notation}

\begin{remark}
\label{rem:pseudo-metric_Zariski-Riemann}
Although we will see that pseudo-metrics are more convenient than model pseudo-metrics from the technical viewpoint, they appear naturally in light of Remark \ref{rem:pseudo-norm_metrised_vector_bundle}. Indeed, consider the space $X_{v}^{\an}:=\lim_{\cX} \widehat{\cX_{s}}^{\an}$, where $\cX$ runs over all projective models of $X$ over $A$ as in (\cite{Sedillot_pav}, \S 8.3). We can equip $X_{v}^{\an}$ with a structure sheaf by taking the direct limit of the inverse image of the structure sheaves on the $\widehat{\cX_{s}}^{\an}$'s. Then a pseudo-metrised line bundle on $X$ is exactly a "metrised line bundle on $X_{v}^{\an}$. 
\end{remark}

Let us now illustrate why pseudo-metrics are technically more convenient than model pseudo-metrics.

\begin{proposition-definition}
\label{propdef:local_pseudo-metric}
Let $L$ be a line bundle on $X$ equipped with a pseudo-metric represented by some model pseudo-metric $(\cL,\varphi)$ whose underlying model is denoted by $\cX$.
\begin{itemize}
	\item[(1)] The equivalence class of $(-\cL,-\varphi)$ is a pseudo-metric on $-L$ called the \emph{dual} pseudo-metric. Moreover, it is continuous/lsc/usc if so is $(\cL,\varphi)$.
	\item[(2)] Let $L'$ be another line bundle on $X$ equipped with a pseudo-metric $[(\cL',\varphi')]$ represented on some projective model $\cX'$ of $X$ over $A$. Define a pseudo-metric $[(\cL,\varphi)]+[(\cL',\varphi')]$ on $L+L'$ as follows. Let $\cX''$ be a projective model of $X$ over $A$ such that there exists $A$-morphisms $p:\cX''\to \cX,q:\cX''\to \cX'$. We then set 
	\begin{align*}
	[(\cL,\varphi)]+[(\cL',\varphi')]:=[(p^{\ast}\cL+q^{\ast}\cL',p^{\ast}\varphi+q^{\ast}\varphi')],
	\end{align*}
	this is a pseudo-metric on $L+L'$ represented on $\cX''$ by Proposition \ref{propdef:local_model_pseudo-metric} (2).
	\item[(3)] Let $f:Y \to X$ be a projective morphism of $K$-schemes. Using Remark \ref{remark:models_extension} and Proposition-Definition \ref{propdef:local_model_pseudo-metric} (3), we can pullback the model pseudo-metric $(\cL,\varphi)$ to obtain a model pseudo-metric $f^{\ast}(\cL,\varphi)$ on $f^{\ast}L$. We then define $f^{\ast}[(\cL,\varphi)]:=[f^{\ast}(\cL,\varphi)]$ and call it the \emph{pullback pseudo-metric} via $f$.
	\item[(4)] Let $K'/K$ be a field extension and let $v'=(|\cdot|',A',\m',\kappa')$ be a pseudo-absolute value on $K'$ extending $v$ on $K$. Consider the fibre product $f:X':=X\otimes_{K}K'\to X$ and $\cX':=\cX\otimes_{A}A'$. By Proposition-Definition \ref{propdef:local_model_pseudo-metric} (4), we have model pseudo-metric $f^{\ast}(\cL,\varphi)$ on $f^{\ast}L$ whose underlying model is $\cX'$. We define $f^{\ast}[(\cL,\varphi)]:=[f^{\ast}(\cL,\varphi)]$ and call it the \emph{extension} of $[(\cL,\varphi)]$ w.r.t. the field extension $K'/K$.
	\item[(5)] Assume that $X$ is geometrically reduced if $A=K$. Assume that $\cX$ is flat and coherent, then the pseudo-norm $\|\cdot\|_{(\cL,\varphi)}$ on $H^{0}(X,L)$ from Proposition \ref{propdef:supnorm_pseudo-metric} depends only on the equivalence class $[(\cL,\varphi)]$. In the case where $\cX$ is not necessarily coherent, we use (\cite{ChenMori21}, Lemma 3.2.17) to find a coherent model $(\cX',\cL')$ of $(X,L)$ such that $\cX$ is a closed subscheme of $\cX'$, the special fibres of $\cX,\cX'$ coincide and $\cL$ is the pullback of $\cL'$ to $\cX$. One can define $\|\cdot\|_{(\cL,\varphi)}$ to be $\|\cdot\|_{(\cL',\varphi)}$. Then $[(\cL,\varphi)]=[(\cL',\varphi)]$ and thus, we get a well-defined pseudo-norm on $H^{0}(X,L)$ We denote it by $\|\cdot\|_{[(\cL,\varphi)]}$ or simply by $\|\cdot\|_{\varphi}$ when no confusion may arise and call it the \emph{supremum pseudo-norm} on $H^{0}(X,L)$ induced by $[(\cL,\varphi)]$.
	\item[(6)] Let $[(\cL',\varphi')]$ be another pseudo-metric on $L$ represented on some projective model $\cX'$ of $X$ over $A$. We say that $[(\cL,\varphi)]$ and $[(\cL',\varphi')]$ have the same \emph{model} if there exist a projective model $\cX''$ of $X$ over $A$ and $A$-morphisms $p:\cX''\to \cX,q:\cX''\to \cX'$ such that $p^{\ast}\cL\cong q^{\ast}\cL'$. In this case, the distance $d_{v}(p^{\ast}\varphi,q^{\ast}\varphi')$ introduced in Definition \ref{def:distance_local_pseudo-metric} does not depend on the choice of $\cX'',p,q$. We denote it by $d_{v}([(\cL,\varphi)],[(\cL',\varphi')])$, or by $d_{v}(\varphi,\varphi')$ when no confusion may arise. This defines a distance on the collection of pseudo-metrics on $L$ having the same model. 
\end{itemize}
\end{proposition-definition}

\subsection{Zariski-Riemann interpretation}
\label{sub:ZR_interpretation_local_pseudo-metrics}

Similarly to the interpretation given in \S \ref{sub:ZR_interpretation_pseudo-norm_families}, pseudo-metrised line bundles over $v$ are nothing else than metrised line bundles on a Zariski-Riemann space.

\begin{proposition}
\label{prop:ZR_interpretation_local_pseudo-metrics}
Let $X$ be a projective $K$-scheme. There is a one-to-one correspondence between pseudo-metrised line bundles in $v$ on $X$ and metrised line bundles on $\ZR(X/A)$ (cf. Definition \ref{def:metrised_vector_bundle_Zariski-Riemann}).
\end{proposition}

\begin{proof}
Let $L$ be a line bundle on $X$ and $[(\cL,\varphi)]$ be a pseudo-metric on $L$ represented on some projective model $\cX$ of $X$ over $A$. Then, by definition, $\varphi$ induces a metric on the pullback of $\cL$ to $\ZR(X/A)$, which gives a metrised line bundle on $\ZR(X/A)$.

Conversely, let $(\cL,\varphi)$ be a metrised line bundle on $\ZR(X/A)$. By Proposition \ref{prop:fp_sheaves_ZR_algebraic}, there exists a projective model $\cX$ of $X$ over $A$ and a line bundle $\cL_{\cX}$ on $\cX$ such that $\cL\cong p_{\cX}^{\ast}\cL_{\cX}$, where $p_{\cX}:\ZR(X/A)\to\cX$ denotes the projection. Then the metric $\varphi$ induces a model pseudo-metric on $L:=\eta_{X}^{\ast}\cL$ represented on $\cX$, where $\eta_{X}:X\to \ZR(X/A)$ is the generic fibre. This construction is the converse of the above one.
\end{proof}

Let us now give a third equivalent way to interpret pseudo-metrised line bundles on a projective $K$-scheme, which will justify the terminology "pseudo-metric". 

\begin{proposition}
\label{prop:local_pseudo-metrics_justification_terminology}
Let $X$ be a projective $K$-scheme. The data of a pseudo-metrised line bundle $(L,\varphi)$ on $X$ is equivalent to the data $(L,\psi)$, where $\psi=(|\cdot|_{\psi}(\mathbf{x}))_{\mathbf{x}\in \ZR(X/A)^{\an}}$ is a family such that, for any $\mathbf{x}\in \ZR(X/A)^{\an}$ with underlying scheme point $x\in X$, $|\cdot|_{\psi}(\mathbf{x})$ is a pseudo-norm on $L(x)$ and $\psi$ satisfies the condition:
\begin{itemize}
	\item[($\ast$)] for any $\mathbf{p}\in \ZR(X/A)$, there exists an open neighbourhood $U\subset \ZR(X/A)$ of $\mathbf{p}$ and a section $s\in H^{0}(\eta_{X}^{-1}(U),L)$ such that, for any $\mathbf{x}\in U^{\an}:=j_{X}^{-1}(U) \subset \ZR(X/A)^{\an}$, $|s|_{\psi}(\mathbf{x})\in \bR_{>0}$.
\end{itemize}
\end{proposition}

\begin{proof}
Let $(L,\varphi)$ be a pseudo-metrised line bundle on $X$. Using Proposition \ref{prop:ZR_interpretation_local_pseudo-metrics}, we consider it as a metrised line bundle on $\ZR(X/A)$. Let $\mathbf{x}\in \ZR(X/A)^{\an}$. Recall that $\mathbf{x}=(\mathbf{p},|\cdot|_{\mathbf{x}})$, where $\mathbf{p}=(p,A_{p},\phi_{p})\in \ZR(X/A)$ and $|\cdot|_{\mathbf{x}}$ is an absolute value on the residue field $\kappa(\mathbf{p})$ of $A_{p}$. Then Proposition \ref{prop:local_semi-norms} and Remark \ref{rem:pseudo-norm_metrised_vector_bundle} yield a family $\psi=(|\cdot|_{\psi}(\mathbf{x}))_{\mathbf{x}\in \ZR(X/A)^{\an}}$ of pseudo-norms on the fibres of $L$ over $\ZR(X/A)^{\an}$. The condition ($\ast$) is then automatically satisfied using the fact that $L$ is the pullback of a line bundle on $\ZR(X/A)$. 

Conversely let $L$ be a line bundle on $X$ equipped with a family $\psi=(|\cdot|_{\psi}(\mathbf{x}))_{\mathbf{x}\in \ZR(X/A)^{\an}}$ as above. Let $U\subset \ZR(X/A)$ be an open subset such that there exists a section $s\in H^{0}(\eta_{X}^{-1}(U),L)$ such that, for any $\mathbf{x}\in U^{\an} \subset \ZR(X/A)^{\an}$, $|s|_{\psi}(\mathbf{x})\in \bR_{>0}$. Then consider the free sheaf of $\cO_{U}$-modules of rank one $\cL_{U} := \cO_{U}\cdot s$. By construction, for any $\mathbf{p}=(p,A_{p},\phi_{p})\in U$, $s(p)$ generates the free $A_{p}$-module of rank one $\cL_{U}\otimes_{\cO_{U}}A_{p}$ and the image of $s(p)$ in the fibre $\cL_{U}\otimes_{\cO_{U}}\kappa(\mathbf{p})$ is non-zero. Assume now that $V\subset \ZR(X/A)$ is another open subset such that there exists a section $t\in H^{0}(\eta_{X}^{-1}(V),L)$ such that, for any $\mathbf{x}\in V^{\an} \subset \ZR(X/A)^{\an}$, $|t|_{\psi}(\mathbf{x})\in \bR_{>0}$. The previous construction gives rise to a free sheaf of $\cO_{V}$-modules of rank one, denoted by $\cL_{V}$. Then multiplication by $t/s$ gives an isomorphism between $(\cL_{U})_{|U\cap V}$ and $(\cL_{V})_{|U\cap V}$. Since $\ZR(X/A)$ is assumed to be covered by such $U$'s, the $\cL_{U}$'s glue to obtain a line bundle $\cL$ on $\ZR(X/A)$ which satisfies by construction $\eta_{X}^{\ast}\cL\cong L$. 

Let $\mathbf{x}=(\mathbf{p},|\cdot|_{\mathbf{x}})\in \ZR(X/A)^{\an}$, where $\mathbf{p}=(p,A_{p},\phi_{p})\in \ZR(X/A)$ and $|\cdot|_{\mathbf{x}}$ is an absolute value on the residue field $\kappa(\mathbf{p})$ of $A_{p}$. Choose an open neighbourhood $U$ of $\mathbf{p}$ in $\ZR(X/A)$ and a section $s\in H^{0}(\eta_{X}^{-1}(U),L)$ such that, for any $\mathbf{x}\in U^{\an} \subset \ZR(X/A)^{\an}$, $|s|_{\psi}(\mathbf{x})\in \bR_{>0}$. By construction, the finiteness module of $|\cdot|_{\psi}(x)$ is generated by $s(p)$, and is thus isomorphic to $\cL\otimes_{\cO_{\ZR(X/A)}}A_{p}$. Therefore $|\cdot|_{\psi}(x)$ induces a norm on the fibre $\cL(\mathbf{x})$. The fact that this construction is inverse to the above one is again given by Proposition \ref{prop:local_semi-norms} and Remark \ref{rem:pseudo-norm_metrised_vector_bundle}.
\end{proof}

\begin{remark}
\label{rem:ZR_interpretation_sup_pseudo-norm}
Let $X$ be a projective $K$-scheme that is assumed to be geometrically reduced if $A=K$. Let $(L,\psi)$ be the data of a line bundle $L$ on $X$ equipped with a family $\psi=(|\cdot|_{\psi}(\mathbf{x}))_{\mathbf{x}\in \ZR(X/A)^{\an}}$ such that, for any $\mathbf{x}\in \ZR(X/A)^{\an}$ with underlying scheme point $x\in X$, $|\cdot|_{\psi}(\mathbf{x})$ is a pseudo-norm on $L(x)$ and $\psi$ satisfies the condition ($\ast$). Using Propositions \ref{prop:ZR_interpretation_local_pseudo-metrics} and \ref{prop:local_pseudo-metrics_justification_terminology} we can view $(L,\psi)$ as a pseudo-metrised line bundle $(L,\varphi)$ on $X$. Assume that the pseudo-metric $\varphi$ is continuous. Then the supremum pseudo-norm $\|\cdot\|_{\varphi}$ on $H^{0}(X,L)$ is given by
\begin{align*}
\forall s\in H^{0}(X,L),\quad \|s\|_{\varphi} = \displaystyle\sup_{x\in \ZR(X/A)^{\an}} |s|_{\psi}(x).
\end{align*}
\end{remark}

\subsection{Fubini-Study pseudo-metric and Fubini-Study operator}
\label{sub:FS_local}

The following construction is a crucial example of pseudo-metrics. Let $E$ be a finite-dimensional $K$-vector space. Let $(\|\cdot\|,\cE,N,\widehat{E})$ be a pseudo-norm on $E$. 

Consider the projective bundle $\mathbb{P}(\cE) \to \Spec(A)$ associated with $\cE$ and denote by $\cO_{\cE}(1)$ the universal line bundle on $\mathbb{P}(\cE)$. Likewise, we have the projective bundles $\mathbb{P}(E)\to\Spec(K)$ and $\mathbb{P}(\widehat{E}) \to \Spec(\widehat{\kappa})$ for which we denote respectively by $\cO_{E}(1)$ and $\cO_{\widehat{E}}(1)$ the associated universal line bundles. Then $(\mathbb{P}(\cE),\cO_{\cE}(1))$ is a model of $(\mathbb{P}(E),\cO_{E}(1))$ over $A$ and we have a commutative diagram whose horizontal arrows are surjective
\begin{center}
\begin{tikzcd}
E \otimes_{K} \cO_{\mathbb{P}(E)} \arrow[r, two heads]                                   & \cO_{E}(1)                       \\
\cE \otimes_{A} \cO_{\mathbb{P}(\cE)} \arrow[r, two heads] \arrow[u] \arrow[d]           & \cO_{\cE}(1) \arrow[d] \arrow[u] \\
\widehat{E}\otimes_{\widehat{\kappa}} \cO_{\mathbb{P}(\widehat{E})} \arrow[r, two heads] & \cO_{\widehat{E}}(1)            
\end{tikzcd}.
\end{center}

For any $x\in \mathbb{P}(\widehat{E})^{\an}$, we denote by $|\cdot|_{\overline{E}}(x)$ 
\begin{itemize}
	\item the $\epsilon$-extension of scalars of the residue norm induced by $\|\cdot\|$ on $\widehat{E}\otimes_{\widehat{\kappa}}\widehat{\kappa}(x)$ if $v$ is non-Archimedean;
	\item the $\pi$-extension of scalars of the residue norm induced by $\|\cdot\|$ on $\widehat{E}\otimes_{\widehat{\kappa}}\widehat{\kappa}(x)$ if $v$ is Archimedean.
\end{itemize}

For any $x\in \mathbb{P}(\widehat{E})^{\an}$, we denote by $|\cdot|_{\overline{E},\mathrm{FS}}(x)$ the quotient norm on $\cO_{\widehat{E}}(1)(x) := \cO_{\widehat{E}}(1) \otimes_{\widehat{\kappa}} \widehat{\kappa}(x)$ induced by the norm $|\cdot|_{\overline{E}}(x)$ on $\widehat{E}\otimes_{\widehat{\kappa}} \widehat{\kappa}(x)$ constructed above. Then the family $\varphi_{\overline{E},\mathrm{FS}} := (|\cdot|_{\overline{E},\mathrm{FS}}(x))_{x\in \mathbb{P}(\widehat{E})^{\an}}$ defines a metric on $\cO_{\widehat{E}}(1)$. Therefore, we obtain a model pseudo-metric $(\cO_{\cE}(1),\varphi_{\overline{E},\mathrm{FS}})$ on $\cO_{E}(1)$.

\begin{proposition-definition}
\label{propdef:Fubini-Study_local_pseudo_metric}
We use the same notation as above. The model pseudo-metric $(\cO_{\cE}(1),\varphi_{\overline{E},\mathrm{FS}})$ on $\cO_{E}(1)$ is called the \emph{Fubini-Study model pseudo-metric} on $\cO_{E}(1)$ associated with the locally pseudo-normed vector space $\overline{E}$.
\end{proposition-definition}

\begin{proposition}[\cite{ChenMori}, Proposition 2.2.12]
We use the same notation as above. Then the Fubini-Study model pseudo-metric $(\cO_{\cE}(1),\varphi_{\overline{E},\mathrm{FS}})$ is continuous.
\end{proposition}

\begin{example}
\label{example:FS_local_pseudo-metric}
Let $E$ be a finite-dimensional $K$-vector space and fix a basis $(e_1,...,e_d)$ of $E$. Fix $\lambda_1,...,\lambda_r$. Let $\|\cdot\|$ denote the diagonalisable pseudo-norm on $E$ such that $(e_1,...,e_d)$ is an orthogonal basis of $\|\cdot\|$ such that, for any $i=1,...,d$, we have $\|e_i\|=e^{-\lambda_i}$. Namely, for any $x=x_1 e_1+\cdots+x_d e_d \in E$, where $x_1,...,x_d\in K$, we have
\begin{align*}
\|x\| =  \left\{\begin{matrix}
\displaystyle\max_{i=1,...,d}|x_i|e^{-\lambda_i} \quad \text{if }v\in M_{K,\um},\\
\displaystyle\sqrt{\sum_{i=1}^{d} |x_i|^{2}e^{-2\lambda_i}} \quad \text{if }v\in M_{K,\infty}.
\end{matrix}\right.
\end{align*}
Let $\overline{E}:=(E,\|\cdot\|)$. Then we can see $e_1,...,e_d$ as global sections of $\cO_{E}(1)$ without common zeroes and $(e_1,...,e_d)$ is adapted to $\|\cdot\|$. Moreover, the Fubini-Study pseudo-metric $(\cO_{\cE}(1),\varphi)$ is given as follows (cf. \cite{BoucksomEriksson}, Lemma 7.17). Let $\tau$ be a local trivialisation of $\cO_{\widehat{E}}(1)$, then for any $i=1,...,d$ we consider the function $f_i := s_i/\tau$ and we have
\begin{align*}
-\log|\tau|_{\varphi} = \left\{\begin{matrix}
\displaystyle \max_{i=1,...,d}\log|f_i| + \lambda_i \quad \text{if }v\in M_{K,\um},\\
\displaystyle\frac{1}{2}\log\sqrt{\sum_{i=1}^{d} |f_i|^{2}e^{2\lambda_i}} \quad \text{if }v\in M_{K,\infty}.
\end{matrix}\right.
\end{align*}
\end{example}

We now consider the general situation where we have a line bundle $L$ on the projective $K$-variety $X$. Fix a model $(\cX,\cL)$ of $(X,L)$. Assume that there exist a finite-dimensional $K$-vector space $E$ equipped with a pseudo-norm $(\|\cdot\|,\cE,N,\widehat{E})$ and surjective morphism of sheaves $\beta : \cE \otimes_{A} \cO_{\cX} \to \cL$. Denote $\overline{E}:=(E,\|\cdot\|)$. Then $\beta$ yields a morphism of schemes $g : \cX \to \mathbb{P}(\cE)$ such that $\cL$ is isomorphic to $g^{\ast}\cO_{\cE}(1)$. By considering generic fibres, $g$ induces a morphism $f : X \to \mathbb{P}(E)$ which is extended by $g$ by construction. 

\begin{definition}
\label{def:quotient_local_model_pseudo-metric}
We use the same notation as above. 
\begin{itemize}
	\item[(1)] The pullback of the Fubini-Study model pseudo-metric on $\cO_{E}(1)$ by $g$ is a continuous model pseudo-metric on $L$ which is called the \emph{quotient model pseudo-metric} induced by $\overline{E}$ and $\beta$. 
	\item[(2)] Consider the particular case where $E\subset H^{0}(X,L)$ is a basepoint free vector subspace. Let $\|\cdot\|$ be a pseudo-norm on $E$. Let $\overline{E}=(E,\|\cdot\|)$. Then the quotient model pseudo-metric on $L$ induced by $\overline{E}$ and $\beta$ is called a \emph{Fubini-Study model pseudo-metric} on $L$.
	\item[(3)] More generally, if $L$ is a semi-ample line bundle on $X$. We say that a model pseudo-metric $(\cL,\varphi)$ is \emph{Fubini-Study} if there exists an integer $n\geq 1$ such that $nL$ is globally generated and a Fubini-Study model pseudo metric $(n\cL,\psi)$ on $nL$  such that $\varphi=n^{-1}\psi$. 
\end{itemize}
\end{definition} 

\begin{proposition-definition}
\label{propdef:quotient_local_pseudo-metric}
Let $L$ be a line bundle on $X$ equipped with a pseudo-metric $[(\cL,\varphi)]$ represented on some projective model $\cX$ of $X$ over $A$.
\begin{itemize}
	\item[(1)] Assume that there exist a finite-dimensional $K$-vector space $E$ equipped with a pseudo-norm $(\|\cdot\|,\cE,N,\widehat{E})$ and surjective morphism of sheaves $\beta : \cE \otimes_{A} \cO_{\cX} \to \cL$. Denote $\overline{E}:=(E,\|\cdot\|)$. Then for any $A$-morphism $p:\cX'\to\cX$ of projective models of $X$ over $A$, the quotient model pseudo-metric induced by $\overline{E}$ and $\beta$ on $L$ coincides with the quotient model pseudo-metric induced by $\overline{E}$ and $p^{\ast}\beta$ on $L$. We say that the pseudo-metric $[(\cL,\varphi)]$ is a \emph{quotient pseudo-metric} induced by $\overline{E}$ if there exists a model $(\cX',\cL')$ of $(X',L')$ representing the pseudo-metric and a surjective morphism $\beta : \cE \otimes_{A} \cO_{\cX'} \to \cL'$ such that $[(\cL,\varphi)]$ is the class of the quotient model pseudo-metric induced by $\overline{E}$ and $\beta$ on $L$. 
	\item[(2)] Assume that $E\subset H^{0}(X,L)$ is a basepoint free vector subspace. Let $\|\cdot\|$ be a pseudo-norm on $E$. Let $\overline{E}=(E,\|\cdot\|)$. Then the quotient pseudo-metric induced by $\overline{E}$ on $L$ is called a \emph{Fubini-Study pseudo-metric} on $L$.
	\item[(3)] Assume that $L$ is a semi-ample line bundle on $X$. A pseudo-metric on $L$ is called \emph{Fubini-Study} if it is the equivalence class of a Fubini-Study model pseudo-metric on $L$. The class of Fubini-Study pseudo-metrics on $L$ is denoted by $\FS(L)$.
\end{itemize}
\end{proposition-definition}

\begin{proof}
Let us prove the assertion in bullet (1). Since $p^{\ast}$ is a right exact functor, $p^{\ast}\beta$ is surjective. The assertion on the quotient norms is checked directly using the surjectivity of $\cX'\to\cX$. This ensures that the definitions in bullets (1)-(3) are independent of the choice of the representative of the pseudo-metric on $L$.
\end{proof}

We now introduce the Fubini-Study operator for pseudo-metrics. 

\begin{proposition-definition}
\label{propdef:Fubini-Study_operator_local_pseudo-metric}
Let $L$ be a line bundle on $X$. Let $[(\cL,\varphi)]$ be a pseudo-metric on $L$ represented on some projective model $\cX$ of $X$ over $A$. We assume that $\cL$ is globally generated. Thus, we have a surjection $\beta : H^{0}(\cX,\cL) \otimes_{A} \cO_{\cX} \to \cL$.  Moreover, the pseudo-metric $[(\cL,\varphi)]$ induces a pseudo-norm $\|\cdot\|_{\varphi}$ on $H^{0}(X,L)$ (cf. Proposition-Definition \ref{propdef:local_pseudo-metric} (5)). Then we can use the construction of Proposition-Definition \ref{propdef:quotient_local_pseudo-metric} (3) to produce a continuous Fubini-Study pseudo-metric $[(\cL,\varphi_{\FS})]$ on $L$, it is called the \emph{Fubini-Study pseudo-metric} associated with $[(\cL,\varphi)]$. 
\end{proposition-definition}

\begin{proposition}[\cite{ChenMori}, Proposition 2.2.23]
\label{prop:CM_2.2.23}
Let $L$ be a line bundle on $X$. Let $[(\cL,\varphi)]$ be a pseudo-metric on $L$ represented on some projective model $\cX$ of $X$ over $A$. We assume that the model $\cX$ is flat and coherent and that $[(\cL,\varphi)]$ is a Fubini-Study pseudo-metric on $L$. Then, for any integer $n\geq 1$, we have $[(n\cL,n\varphi)]=[(n\cL,(n\varphi)_{\mathrm{FS}})]$.
\end{proposition}

\subsection{Semi-positive pseudo-metrics}
\label{sub:semi-positive_local}

\begin{definition}
\label{propdef:psh_local_pseudo_metric}
Let $L$ be a line bundle on $X$. Fix a continuous model pseudo-metric $(\cL,\varphi)$ on $L$ with underlying projective model $\cX$. For any integer $n\geq 1$, we denote by $(n\cL,\varphi_n)$ the pseudo-metric $(n\cL,(n\varphi)_{\mathrm{FS}})$. 

\begin{itemize}
	\item[(1)] Assume that $\cL$ is semi-ample. Choose an integer $n\geq 1$ such that $n\cL$ is globally generated. Then the model pseudo-metric $(\cL,\varphi)$ is called \emph{semi-positive} if the sequence
\begin{align*}
\frac{d_{v}(nk\varphi,\varphi_{nk})}{nk}, \quad k\geq 1
\end{align*}
converges to $0$. Note that this definition does not depend on the choice of $n$. 
	\item[(2)] We say that a continuous pseudo-metric on $L$ is \emph{semi-positive} if it is the equivalence class of a semi-positive model pseudo-metric on $L$ (cf. Proposition-Definition \ref{propdef:local_pseudo-metric} (6) for the justification that this is a well-defined notion). 
	\item[(3)] Assume that $\cL$ is semi-ample. Then $(\cL,\varphi)$ is called \emph{plurisubharmonic} (\emph{psh} for short) if there exists a sequence $(\cL,\varphi_i)_{i\geq 1}$ of Fubini-Study model pseudo-metrics such that the sequence of function $(\varphi-\varphi_{i})_{i\geq 1}$ converges uniformly to the null function. 
	\item[(4)] Like in bullet (2), we say that a pseudo-metric on $L$ is \emph{plurisubharmonic} if it is the equivalence class of a psh model pseudo-metric on $L$.
	\item[(5)] $(\cL,\varphi)$ is called \emph{integrable} if there exist line bundles $L_1,L_2$ on $X$ equipped respectively with plurisubharmonic model pseudo-metrics $(\cL_1,\varphi_1),(\cL_2,\varphi_2)$ whose underlying model is $\cX$ such that $\cL_1$ and $\cL_2$ are very ample, $L=L_1-L_2$ and $\varphi=\varphi_1-\varphi_2$.
	\item[(6)] Like in bullets (2) and (4), we say that a pseudo-metric on $L$ is \emph{integrable} if it is the equivalence class of an integrable model pseudo-metric on $L$.
\end{itemize}
\end{definition}

\begin{remark}
\label{rem:psh_semi-positive}
\begin{itemize}
	\item[(1)] A plurisubharmonic model pseudo-metric is, at first glance, a special case of semi-positive pseudo-metric. If $\widehat{\kappa}=\bC$, then for any semi-positive model pseudo-metric $(\cL,\varphi)$, $\varphi$ is a plurisubharmonic metric in the usual sense (cf. \cite{Zhang95}, Theorem 3.5). 
	\item[(2)] The terminology and definition of plurisubharmonic (model) pseudo-metrics is inspired by \cite{BoucksomEriksson}.
\end{itemize}
\end{remark}

\begin{proposition}
\label{prop:semi-positive_scalar_extension}
Let $L$ be a line bundle on $X$. Let $(\cL,\varphi)$ be a continuous model pseudo-metric on $L$ with underlying model $\cX$. Let $K'/K$ be a field extension and let $v'=(|\cdot|',A',\m',\kappa')$ be a pseudo absolute value on $K'$ extending $v$. Consider the fibre product $f : X':=X \otimes_{K} K' \to X$ and denote by $L'$ the pullback of $L$ to $X'$. If $(\cL,\varphi)$ is semi-positive, then $f^{\ast}(\cL,\varphi)$ is a semi-positive model pseudo-metric on $L'$. Moreover, $f^{\ast}[(\cL,\varphi)]$ is a semi-positive pseudo-metric on $L'$.
\end{proposition}

\begin{proof}
Recall that the model pseudo-metric $f^{\ast}(\cL,\varphi)$ is equal to $(\cL',\varphi')$, where $\cL'$ is the pullback of $\cL$ to $\cX':=\cX \otimes_{A} A'$ and $\varphi'$ is the pullback of $\varphi$ to $\cX'_{s} \otimes_{\kappa'} \widehat{\kappa'}$. (\cite{ChenMori21}, Remark 3.2.7) implies that the metric $\varphi'$ is semi-positive. Therefore, the local model pseudo metric $(\cL',\varphi')$ is semi-positive. The assertion concerning $f^{\ast}[(\cL,\varphi)]$ then follows from its definition (cf. Proposition-Definition \ref{propdef:local_pseudo-metric} (4)).
\end{proof}

\begin{proposition}
\label{prop:equivalence_psh_semi-positive}
\begin{itemize}
	\item[(1)] Let $\cX$ be a projective model of $X$ over $A$. Assume that the special fibre $\cX_{s}$ is geometrically reduced. Let $L$ be a line bundle on $X$. Let $(\cL,\varphi)$ be a continuous model pseudo-metric on $L$ with underlying model $\cX$ such that $\cL$ is semi-ample. Then $(\cL,\varphi)$ is psh iff it is semi-positive. 
	\item[(2)] Assume that $A$ is a rank $1$ valuation ring, $K$ is algebraically closed and $X$ is reduced. Let $L$ be a line bundle on $X$. A pseudo-metric on $L$ is semi-positive iff it is psh.
\end{itemize}
\end{proposition}

\begin{proof}
\textbf{(1):} First, note that it suffices to prove that if $(\cL,\varphi)$ is semi-positive, then $(\cL,\varphi)$ is psh. Thus, we assume that $(\cL,\varphi)$ is semi-positive. The $\widehat{\kappa}=\bC$ is treated in Remark \ref{rem:psh_semi-positive}. The $\widehat{\kappa}=\bR$ is dealt with by combining Proposition \ref{prop:semi-positive_scalar_extension} with (\cite{BoucksomEriksson}, Theorem 7.31).

The non-Archimedean and non trivially valued case follows from (\cite{ChenMori21}, Proposition 3.2.19). For the trivially valued case, we choose a transcendental extension $K'/K$ and a pseudo-absolute value $v'=(|\cdot|',A',\m',\kappa')$ extending $v$ such that $\kappa'$ is non-trivially valued. Consider the fibre product $f:X':=X \otimes_{K} K' \to X$ and the pullback $L'$ of $L$ to $X'$. Then the pullback of $(\cL,\varphi)$ to $X'$ is semi-positive by Proposition \ref{prop:semi-positive_scalar_extension}. Then by the non-trivially valued case, $f^{\ast}(\cL,\varphi)=:(\cL',\varphi')$ is psh. By (\cite{BoucksomEriksson}, Theorem 7.31), $\varphi'$ is continuous psh metric on $L'_{v'}$ and therefore $\varphi$ is psh and $(\cL,\varphi)$ is a (continuous) psh model pseudo-metric on $L$.

\textbf{(2):} Note that by assumption $\kappa$ and thus $\widehat{\kappa}$ are algebraically closed. Let $[(\cL,\varphi)]$ be a pseudo-metric represented on some projective model $\cX$ of $X$ over $A$. (\cite{BoucksomEriksson}, Theorem 4.20) implies that the integral closure $\cX'$ of $\cX$ in $X$ is a model of $X$ over $A$ dominating $\cX$ such that $\cX'_{s}$ is (geometrically) reduced. Thus, we get (2) using (1). 
\end{proof}

\section{Pseudo-metrics: global case}
\label{sec:pseudo-metric families}

In this section, we introduce the global counterpart of \S \ref{sec:metrics_MA_intersection_local}. We first introduce the global version of local pseudo-metrics (\S \ref{sub:pseudo-metric_families_definitions}). They also admit an interpretation in terms of metrics on line bundles on a (global) Zariski-Riemann space (\S \ref{sub:ZR_interpretation_pseudo-metric}). Then we introduce integrability and regularity conditions for pseudo-metric families (\S \ref{sub:dominated_pseudo-metric_family}-\ref{sub:regularity_pseudo-metric_families}). After that, we introduce the pushforward of a pseudo-metric family (\S \ref{sub:pushforward_pseudo-metric_families}). We are now able to define the notion of adelic line bundles (\S \ref{sub:adelic_line_bundles}). All the previously introduced notions are studied in the case of an integral topological adelic curve (\S \ref{sub:pseudo-metric_families_integral_tac}). Then we introduce adelic line bundles for families of topological adelic curves (\S \ref{sub:adelic_line_bundles_families_of_tac}.  Finally, we define the arithmetic volume and the $\chi$-volume of adelic line bundles (\S \ref{sub:volume_functions_proper_tac}).
 
\subsection{Definitions}
\label{sub:pseudo-metric_families_definitions}

In this subsection, we fix a topological adelic curve $S=(K,\phi:\Omega\to M_K,\nu)$ and a projective $K$-scheme $X$. Recall that we have defined two Zariski-Riemann spaces $\ZR(X)_{S}=\tilde{\Omega}\times_{\ZR(K)}\ZR(X)$ and $\ZR(X)^{\an}_{S}=\Omega\times_{M_{K}}\ZR(X)^{\an}$ fitting in a commutative diagram 
\begin{center}
\begin{tikzcd}
                                        & \ZR(X)^{\an}_{S} \arrow[rr, "\phi_{X}"] \arrow[dd] \arrow[rd, "{j_{X,S}}"] &                                                      & \ZR(X)^{\an} \arrow[dd] \arrow[rd, "j_{X}"] &                   \\
X \arrow[dd] \arrow[rr, "{\eta_{X,S}}"] &                                                                            & \ZR(X)_{S} \arrow[rr, "\tilde{\phi_{X}}"] \arrow[dd] &                                             & \ZR(X) \arrow[dd] \\
                                        & \Omega \arrow[rr, "\phi"] \arrow[rd, "{j_{K,S}}"]                          &                                                      & M_{K} \arrow[rd, "j_{K}"]                   &                   \\
\Spec(K) \arrow[rr, "\eta_{S}"]         &                                                                            & \tilde{\Omega} \arrow[rr, "\tilde{\phi}"]            &                                             & \ZR(K)           
\end{tikzcd}.
\end{center}

\begin{definition}
\label{def:pseudo-metric_family}
Let $L$ be a line bundle on $X$. A \emph{pseudo-metric} on $L$ \emph{over} $S$ is a family $\varphi=(|\cdot|_{\varphi}(\mathbf{x}))_{\mathbf{x}\in \ZR(X)^{\an}_{S}}$ where, for any $\mathbf{x}\in \ZR(X)_{S}^{\an}$ with underlying scheme point $x\in X$, $|\cdot|_{\varphi}(\mathbf{x})$ is a pseudo-norm on $L(x)$ and $\varphi$ satisfies:
\begin{itemize}
	\item[($\ast$)] for any $\mathbf{p}\in \ZR(X)_{S}$, there exists an open neighbourhood $U\subset \ZR(X)_{S}$ of $\mathbf{p}$ and a section $s\in H^{0}(\eta_{X}^{-1}(U),L)$ such that, for any $\mathbf{x}\in U^{\an}:=j_{X}^{-1}(U) \subset \ZR(X)_{S}^{\an}$, $|s|_{\varphi}(\mathbf{x})\in \bR_{>0}$.
	\end{itemize}
A line bundle on $X$ equipped with a pseudo-metric over $S$ is called a \emph{pseudo-metrised line bundle} on $X$ over $S$.
\end{definition}

\begin{example}
\label{example:pseudo-metric_adelic_curve}
In the situation of Example \ref{example:model_semi-norm_family} (1), namely where $\phi(\omega)$ is an absolute value on $K$ for all $\omega\in\Omega$, any metric family on a line bundle $L$ on $X$ in the sense of (\cite{ChenMori}, \S 6.1) is a pseudo-metric on $L$. 
\end{example}

\begin{remark}
\label{rem:CM_6.1.3_general}
We consider the case where $X=\Spec(K')$, where $K'/K$ is a finite extension of fields. Let $L$ be a line bundle on $X$, i.e. a one-dimensional $K'$-vector space. Then the data of a pseudo-metric $\varphi=(|\cdot|_{\varphi}(\mathbf{x}))_{\mathbf{x}\in \ZR(X)^{\an}_{S}}$ on $L$ over $S$ is the same as a pseudo-norm family $L$, relatively to the topological adelic curve $S' := S \otimes_{K} K'$. It is a consequence of the fact that $\ZR(X)^{\an}_{S}$ identifies with the adelic space of the topological adelic curve $S'$ (cf. \S \ref{sub:finite_extension_tac}).
\end{remark}

\begin{proposition}
\label{prop:pseudo-metric_family_global_local}
Let $(L,\varphi=(|\cdot|_{\varphi}(\mathbf{x}))_{\mathbf{x}\in \ZR(X)^{\an}_{S}})$ be a pseudo-metrised line bundle on $X$ over $S$. Then for any $\omega\in \Omega$, $\varphi_{\omega}:=(|\cdot|_{\varphi}(\mathbf{x}))_{\mathbf{x}\in\ZR(X/A_{\omega})^{\an}}$ is a local pseudo-metric on $L$ over $\omega$.
\end{proposition}

\begin{proof}
Let $\omega\in\Omega$. Recall that $\ZR(X/A_{\omega})^{\an}$ identifies with the fibre over $\omega$ of the morphism $f^{\an}_{S}:\ZR(X)^{\an}_{S}\to\Omega$. By Proposition \ref{prop:local_pseudo-metrics_justification_terminology}, it suffices to prove that, for any $\mathbf{p}\in \ZR(X/A_{\omega})$, there exists an open neighbourhood $U\subset \ZR(X/A_{\omega})$ of $\mathbf{p}$ and a section $s\in H^{0}(\eta_{X}^{-1}(U),L)$ such that, for any $\mathbf{x}\in U^{\an} \subset \ZR(X/A_{\omega})^{\an}$, $|s|_{\varphi}(\mathbf{x})\in \bR_{>0}$. This is clear from the condition ($\ast$) together with the fact that the inclusion $\ZR(X/A_{\omega})^{\an} \to \ZR(X)^{\an}_{S}$ is continuous.
\end{proof}

\begin{remark}
\label{rem:pseudo-metric_family_local_global}
Proposition \ref{prop:pseudo-metric_family_global_local} suggests that there exists an interpretation of pseudo-metrics as a family of local pseudo-metrics satisfying a "globalisation" condition. This was the approach chosen in \cite{Sedillotthese}. In \S \ref{sub:pseudo-metric_families_integral_tac}, we will give such an interpretation, which agrees with the classical point of view of Arakelov geometry.
However, in the general case, our use of Zariski-Riemann spaces and their geometry allows us to work directly with pseudo-metrics defined globally.
\end{remark}

\begin{proposition-definition}
\label{def:pseudo-metric_family_algebraic_constructions}
Let $L$ be a line bundle on $X$ equipped with a pseudo-metric $\varphi=(|\cdot|_{\varphi}(\mathbf{x}))_{\mathbf{x}\in \ZR(X)^{\an}_{S}}$ over $S$.
\begin{itemize}
	\item[(1)] The family $-\varphi:=(|\cdot|_{\varphi,\ast}(\mathbf{x}))_{f{x}\in \ZR(X)^{\an}_{S}}$, where, for any $\mathbf{x}\in \ZR(X)^{\an}_{S}$, $|\cdot|_{\varphi,\ast}(\mathbf{x})$ denotes the dual pseudo-norm of $|\cdot|_{\varphi}(\mathbf{x})$, is a pseudo-metric on $-L$ over $S$ called the \emph{dual} pseudo-metric of $\varphi$.
	\item[(2)] Let $L'$ be another line bundle on $X$ equipped with a pseudo-metric $\varphi'=(|\cdot|_{\varphi'}(\mathbf{x}))_{\mathbf{x}\in \ZR(X)^{\an}_{S}}$ over $S$. Then the family $\varphi+\varphi':=(|\cdot|_{\varphi}(\mathbf{x})\cdot|\cdot|_{\varphi'}(\mathbf{x}))_{\mathbf{x}\in \ZR(X)^{\an}_{S}}$ is a pseudo-metric on $L+L'$ over $S$.
	\item[(3)] Let $f : Y \to X$ be a projective morphism of $K$-schemes. Denote by $f_{S}:\ZR(Y)_{S}\to\ZR(X)_{S}$ and $f_{S}^{\an}:\ZR(Y)^{\an}_{S}\to\ZR(X)^{\an}_{S}$ the induced morphisms (cf. \ref{sub:ZR_analytic_spaces}). For any $\mathbf{y}\in \ZR(Y)^{\an}_{S}$ with underlying scheme point $y\in Y$, consider the pseudo-norm $|\cdot|_{f^{\ast}\varphi}(\mathbf{y})$ on $(f^{\ast}L)(y)$ defined by
	\begin{align*}
	\forall s\in (f^{\ast}L)(y)\cong L(f(y)), \quad |s|_{f^{\ast}\varphi}(\mathbf{y}) := |s|_{\varphi}(f^{\an}_{S}(\mathbf{y})).
	\end{align*}
Then $f^{\ast}\varphi=(|\cdot|_{f^{\ast}\varphi}(\mathbf{y}))_{\mathbf{y}\in\ZR(Y)^{\an}_{S}}$ is a pseudo-metric on $f^{\ast}L$ over $S$.
	\item[(4)] Let $P : \Spec(K') \to X$ be a closed point. Denote by $S'=(K',\phi': \Omega' \to M_{K'},\nu')$ the topological adelic curve $S \otimes_{K} K'$. Using (3) and Remark \ref{rem:CM_6.1.3_general}, we interpret the pseudo-metric $P^{\ast}\varphi$ on $P^{\ast}L$ over $S$ as a pseudo-norm family on the one-dimensional $K'$-vector space $P^{\ast}L$ w.r.t. the topological adelic curve $S'$.
\end{itemize}
\end{proposition-definition}

\begin{proof}
It suffices to check the condition ($\ast$) in bullets (1)-(3).  For (2), for any $\mathbf{p}\in\ZR(X)_{S}$, we may assume that there exists an open neighbourhood $U$ of $\mathbf{p}$ in $\ZR(X)_{S}$ and sections $s\in H^{0}(\eta_{X}^{-1},L)$ and $s'\in H^{0}(\eta_{X}^{-1}(U),L')$ such that ($\ast$) holds on $U$ for both $s$ and $s'$. Then $s\cdot s'\in H^{0}(\eta_{X}^{-1}(U),L+L')$. For (1), we argue similarly, considering inverses. For (3), we argue similarly by pulling back local sections of $L$ to local sections of $f^{\ast}L$.  
\end{proof}

\begin{example}
\label{example:pseudo-metric_family_quotient}
A crucial example of pseudo-metric is the \emph{quotient pseudo-metric}, which is defined as follows. 

Let $\overline{E}=(E,\xi)$, where $E$ is a finite-dimensional $K$-vector space and $\xi=(\|\cdot\|_{\omega})_{\omega\in\Omega}$ is a pseudo-norm family on $E$. Let $L$ be a line bundle on $X$. Assume that there exists a surjective homomorphism $\beta: E\otimes_{K}\cO_{X} \to L$. Let $\mathbf{x}\in \ZR(X)^{\an}_{S}$ with underlying scheme point $x\in X$. Denote $\omega:=f^{\an}_{S}(\mathbf{x})\in \Omega$ Then the surjection $E \otimes_{K} \kappa(x) \to L(x)$ of $\kappa(x)$-vector spaces induces a pseudo-norm $|\cdot|_{\xi}(\mathbf{x})$ on $L(x)$, where we consider the extension of scalars of $\|\cdot\|_{\omega}$ on $E\otimes_{K}\kappa(x)$  (w.r.t. the extension $\kappa(x)/K$ of pseudo-valued fields). The family $(|\cdot|_{\xi}(\mathbf{x}))_{\mathbf{x}\in\ZR(X)^{\an}_{S}}$ is then a pseudo-metric on $L$ over $S$ called the \emph{quotient pseudo-metric} induced by $\overline{E}$ and $\beta$. 

Let us justify that the condition ($\ast$) is satisfied. Let $\mathbf{p}\in\ZR(X)_{S}$ with $f_{S}(\mathbf{p})=:\tilde{\omega}\in\tilde{\Omega}$. Since $\xi$ is a pseudo-norm family on $E$, there exists an open neighbourhood $V\subset \tilde{\Omega}$ of $\tilde{\omega}$ such that there exists a basis of $E$ which is adapted to $\xi$ on $V$. Let $U:=f_{S}^{-1}(V)$, this is an open neighbourhood of $\mathbf{p}$. By Lemma \ref{lemma:pseudo-norm_family_base_change}, the construction of quotient pseudo-norms and the definition of the extension of scalars of pseudo-norms, up to shrinking $U$, there exists a basis $(s_{1},...,s_{r})$ of $E$ such that, for any $\mathbf{x}\in \ZR(X)^{\an}_{S}$, we have $|\beta(s_{1})|_{\xi}(\mathbf{x})\in \bR_{>0}$. This is the condition ($\ast$). 
\end{example}

\begin{proposition}
\label{prop:distance_quotient_pseudo-metric_families_dominated}
Let $L$ be a line bundle on $X$. Assume that there exist a finite-dimensional $K$-vector space $E$ equipped with two dominated pseudo-norm families $\xi,\xi'$ such that there exists a basis of $E$ that is globally adapted to both $\xi$ and $\xi'$, together with a surjective homomorphism $\beta :E\otimes_{K}\cO_{X}\to L$. Let $\varphi,\varphi'$ be respectively the quotient pseudo-metric on $L$ over $S$ induced by $(E,\xi),\beta$ and $(E,\xi'),\beta$. Then the local distance function $(\omega\in\Omega)\mapsto d_{\omega}(\varphi,\varphi')$ is $\nu$-dominated. 
\end{proposition}

\begin{proof}
By (\cite{ChenMori}, Proposition 2.2.20), for any $\omega\in\Omega$, we have the inequality
\begin{align*}
d_{\omega}(\varphi,\varphi') \leq d_{\omega}(\xi^{\vee\vee},\xi'^{\vee\vee}).
\end{align*}
Using Proposition \ref{cor:CM_4.1.8}, as $\xi^{\vee\vee}_1$ and $\xi^{\vee\vee}_2$ are both ultrametric on $\Omega_{\um}$, the local distance function $(\omega\in\Omega)\mapsto d_{\omega}(\xi^{\vee\vee},\xi'^{\vee\vee})$ is $\nu$-dominated. Hence so is $(\omega\in\Omega)\mapsto d_{\omega}(\varphi,\varphi')$.
\end{proof}

\subsection{Zariski-Riemann interpretation}
\label{sub:ZR_interpretation_pseudo-metric}

Similarly to \S \ref{sub:ZR_interpretation_local_pseudo-metrics}, pseudo-metric families have an interpretation in terms of metrised line bundles on Zariski-Riemann spaces. 

\begin{proposition}
\label{prop:ZR_interpretation_pseudo-metric_families}
There is a one-to-one correspondence between line bundles on $X$ equipped with a pseudo-metric over $S$ and metrised line bundles on $\ZR(X)_{S}$.
\end{proposition}

\begin{proof}
Let $L$ be a line bundle on $X$ and $\varphi=(|\cdot|_{\varphi}(\mathbf{x}))_{\mathbf{x}\in\ZR(X)^{\an}_{S}}$ be a pseudo-metric on $L$ over $S$. Arguing as in the proof of Proposition \ref{prop:local_pseudo-metrics_justification_terminology}, the condition ($\ast$) yields a line bundle $\cL$ on $\ZR(X)_{S}$ such that $\eta_{X}^{\ast}\cL\cong L$ and such that, for any $\mathbf{x}=((p,A_{p},\phi_{p}),|\cdot|_{\mathbf{x}})\in\ZR(X)^{\an}_{S}$, the finiteness module of $|\cdot|_{\varphi}(\mathbf{x})$ coincides with $\cL\otimes_{\cO_{\ZR(X)_{S}}}A_{p}$. Thus $\varphi$ induces a metric on $\cL$. 

Conversely, let $(\cL,\varphi)$ be a metrised line bundle on $X$. Let $L:=\eta_{X}^{-1}\cL$. The metric $\varphi$ induces a family of pseudo-norms on the fibres of $L$ by Proposition \ref{prop:local_semi-norms} and Remark \ref{rem:pseudo-norm_metrised_vector_bundle}. Pulling back trivialising sections of $\cL$, we obtain open subsets of $\ZR(X)_{S}$ and sections of $L$ fulfilling the condition ($\ast$).
\end{proof}

\subsection{Dominated pseudo-metric family}
\label{sub:dominated_pseudo-metric_family}

In this subsection, we fix a topological adelic curve $S=(K,\phi:\Omega\to M_K,\nu)$ and a projective $K$-scheme $X$. We introduce a domination condition on pseudo-metrics similar to the one from (\cite{ChenMori}, \S 6.1.2). As the reader can be easily convinced from the definition, all the results presented in the section can be proved using the same ideas as in \emph{loc. cit.} and we will simply refer to the corresponding proposition for the proof. 

\begin{definition}
\label{def:dominated_pseudo-metric_family}
Let $L$ be a line bundle on $X$ equipped with a pseudo-metric $\varphi$ over $S$.
\begin{itemize}
	\item[(1)] Assume that $L$ is very ample. Then $\varphi$ is called \emph{dominated}, if there exist 
\begin{itemize}
	\item[(i)] a pair $\overline{E}=(E,\xi)$, where $E$ is a finite-dimensional $K$-vector space and $\xi=(\|\cdot\|_{\omega})_{\omega\in\Omega}$ is a pseudo-norm family on $E$;
	\item[(ii)] a surjective homomorphism $\beta: E\otimes_{K}\cO_{X} \to L$;
\end{itemize}	
such that the quotient pseudo-metric family $\varphi'$ induced by $(E,\xi)$ and $\beta$ satisfies: the local distance function $(\omega\in\Omega)\mapsto d_{\omega}(\varphi,\varphi')$ is $\nu$-dominated. 
	\item[(2)] In general, we say that $\varphi$ \emph{dominated} if there exist two very ample line bundles $L_{1},L_{2}$ on $X$, respectively equipped with two dominated pseudo-metric $\varphi_1,\varphi_2$ over $S$, such that $L=L_1-L_2$ and $\varphi=\varphi_1-\varphi_2$. 
\end{itemize}
\end{definition}

\begin{example}
\label{example:dominated_pseudo-metric_adelic_curve}
In the situation of Example \ref{example:pseudo-metric_adelic_curve}, a pseudo-metric on a line bundle $L$ on $X$ determined by a metric family in the sense of Chen-Moriwaki is dominated iff it is so as a metric family (\cite{ChenMori}, \S 6.1.2).  
\end{example}

\begin{proposition}[\cite{ChenMori}, Proposition 6.1.8]
\label{prop:tensor_product_dominated_pseudo-metric_families_very_ample}
Let $L_1,L_2$ be two very ample line bundles on $X$ equipped respectively with pseudo-metrics $\varphi_1,\varphi_2$ over $S$. If $\varphi_1$ and $\varphi_2$ are dominated, then $\varphi_1+\varphi_2$ is dominated. 
\end{proposition}

\begin{remark}
\label{rem:CM_6.1.10} We can adapt (\cite{ChenMori}, Remark 6.1.10) in our context to see that (1) and (2) in Definition \ref{def:dominated_pseudo-metric_family} are equivalent when the line bundle is very ample. We leave the details to the reader.

\end{remark}

\begin{proposition}[\cite{ChenMori}, Proposition 6.1.11]
\label{prop:quotient_pseudo-metric_family_dominated}
Let $L$ be a line bundle on $X$. Let $E$ be a finite-dimensional $K$-vector space equipped with a pseudo-norm family $\xi$. Assume that there exists a surjective homomorphism $\beta: E\otimes_{K}\cO_{X} \to L$. If $\xi$ is dominated, then the quotient pseudo-metric $\varphi$ induced by $\overline{E}=(E,\xi)$ and $\beta$ is dominated.
\end{proposition}

\begin{proposition}[(\cite{ChenMori}, Proposition 6.1.12)]
\label{prop:dominated_pseudo-metric_families}
Let $L$ be a line bundle on $X$ equipped with a pseudo-metric $\varphi$ over $S$. 
\begin{itemize}
	\item[(1)] If $\varphi$ is dominated, then $-\varphi$ is dominated.
	\item[(2)] Let $L'$ be another line bundle on $X$ equipped with a pseudo-metric $\varphi'$ over $S$. If $\varphi$ and $\varphi'$ are dominated, then $\varphi+\varphi'$ is dominated.
	\item[(3)] Let $\varphi'$ be another pseudo-metric on $L$ over $S$. If $\varphi'$ is dominated and the local distance function $(\omega\in\Omega)\mapsto d_{\omega}(\varphi,\varphi')$ is $\nu$-dominated, then $\varphi$ is dominated.
	\item[(4)] Let $f : Y \to X$ be a projective morphism of $K$-schemes. If $\varphi$ is dominated, then $f^{\ast}\varphi$ is a dominated pseudo-metric on $f^{\ast} L$ over $S$.
	\item[(5)] Let $\varphi'$ be another pseudo-metric on $L$ over $S$. If $\varphi$ and $\varphi'$ are both dominated, then the local distance function $(\omega\in\Omega)\mapsto d_{\omega}(\varphi,\varphi')$ is $\nu$-dominated.
\end{itemize}
\end{proposition}

\subsection{Regularity conditions for pseudo-metric families}
\label{sub:regularity_pseudo-metric_families}

In this subsection, we fix a topological adelic curve $S=(K,\phi:\Omega\to M_K,\nu)$ and a projective $K$-scheme $X$. 

\begin{definition}
\label{def:regularity_pseudo-metric_families}
Let $L$ be a line bundle on $X$. A pseudo-metric $\varphi=(|\cdot|_{\varphi}(\mathbf{x}))_{\mathbf{x}\in\ZR(X)_{S}^{\an}}$ on $L$ over $S$ is called respectively \emph{usc/lsc/continuous} if the corresponding metric on the associated metrised line bundle on $\ZR(X)_{S}$ is usc/lsc,/continuous (cf. \S \ref{subsub:metrised_vector_bundle_ZR_space}). This is equivalent to say that, for any open subset $U\subset\ZR(X)_{S}$ and any section $s\in H^{0}(\eta_{X}^{-1}(U),L)$, the map $(\mathbf{x}\in U^{\an})\mapsto |s|_{\varphi}(\mathbf{x})\in[0,+\infty]$ is usc/lsc/ continuous.
\end{definition}

\begin{remark}
\label{rem:regularity_pseudo-metric_adelic_curve}
In the situation of Example \ref{example:pseudo-metric_adelic_curve}, any usc/lsc/continuous pseudo-metric yields a measurable metric family in the sense of (\cite{ChenMori}, Definition 4.1.27).
\end{remark}

\begin{proposition}
\label{prop:regularity_pseudo-metric_families}
\begin{itemize}
	\item[(1)] Let $L$ be a line bundle on $X$ equipped with a respectively usc/lsc/continuous pseudo-metric $\varphi$. Then the pseudo-metric $-\varphi$ on $-L$ is respectively lsc/usc/continuous.
	\item[(2)] Let $L_1,L_2$ be two line bundles on $X$ respectively equipped with usc/lsc/continuous pseudo-metrics $\varphi_1,\varphi_2$. Then $\varphi_1+\varphi_2$ is a usc/lsc/continuous pseudo-metric on $L_1+L_2$.
	\item[(3)] Let $f : Y \to X$ be a projective morphism of $K$-schemes. Let $L$ be a line bundle equipped with a respectively usc, lsc, continuous pseudo-metric $\varphi$. Then the pseudo-metric $f^{\ast}\varphi$ on $f^{\ast}L$ is respectively usc, lsc, continuous.
\end{itemize}
\end{proposition}

\begin{proof}
\textbf{(1):} Let $U\subset\ZR(X)_{S}$ be an open subset such that $L$ and $-L$ are respectively trivialised by sections $s,\alpha$ on $\eta_{X}^{-1}(U)$ and the functions $|s|_{\varphi}(\cdot),|\alpha|_{\varphi}(\cdot)$ have positive real values on $U^{\an}$. Then since
\begin{align*}
|\alpha|_{\varphi}(\cdot)=\frac{|\alpha(s)|(\cdot)}{|s|_{\varphi}(\cdot)}
\end{align*}
on $U^{\an}$ and $|\alpha(s)|(\cdot)$ is continuous on $U^{\an}$, we get (1).

\textbf{(2) and (3):} It is seen directly from the construction of the sum and pullback operation on pseudo-metrics.
\end{proof}

\begin{proposition}
\label{prop:regularity_quotient_pseudo-metric_families}
Let $L$ be a line bundle on $X$. Let $\overline{E}=(E,\xi)$ where $E$ is a finite-dimensional $K$-vector space and $\xi$ is a pseudo-norm family on $E$ which is assumed to be ultrametric on $\Omega_{\um}$. Assume that there exists a surjective homomorphism $\beta : E \otimes_{K} \cO_{X} \to L$. Denote by $\varphi$ the corresponding quotient pseudo-metric on $L$. The following assertion holds.
\begin{itemize}
	\item[(1)] If $\xi$ is usc, then $\varphi$ is a usc pseudo-metric on $L$.
	\item[(2)] If $\xi$ and $\xi^{\vee}$ are both usc, then $\varphi$ is continuous
\end{itemize}
\end{proposition}

\begin{proof} 
\textbf{(1):} $\beta$ corresponds to a morphism of schemes $g:X\to \mathbb{P}(E)$ such that $L\cong g^{\ast}\cO_{\mathbb{P}(E)}(1)$. Thus, it suffices to prove that $\cO_{\mathbb{P}(E)}(1)$ equipped with the quotient metric induced by $\overline{E}$ is usc. Let $\pi:\mathbb{P}(E) \to \Spec(K)$ be the structural morphism. Consider $\pi^{\ast}\overline{E}=(\pi^{\ast}E,\pi^{\ast}\xi)$, this is a metrised vector bundle on $\ZR(\mathbb{P}(E))_{S}$ (cf. Definition \ref{def:ZR_space_tac} (5)). Since $\xi$ is usc, then $\pi^{\ast}\xi$ is usc as well (cf. Proposition-Definition \ref{prop-def:_metrised_vector_bundles_constructions} (5)). The quotient metric on $\cO_{\mathbb{P}(E)}(1)$ is given as the quotient metric induced by the universal surjection $\pi^{\ast}E \to \cO_{\mathbb{P}(E)}(1)$, where $\pi^{\ast}E$ is equipped with $\pi^{\ast}\xi$. Since $\pi^{\ast}\xi$ is usc, the quotient metric on $\cO_{\mathbb{P}(E)}(1)$ is usc as well (Proposition-Definition (2)). 

\textbf{(2):} We use the same notation as in the proof of (1). By Proposition-Definition \ref{prop-def:_metrised_vector_bundles_constructions} (5), the dual metric $-\varphi$ is obtained as the restriction of the metric $\pi^{\ast}\xi^{\vee}$ to $-L$. By Proposition-Definition \ref{prop-def:_metrised_vector_bundles_constructions} (1) and (3), $-\varphi$ is usc. Thus $\varphi=-(-\varphi)$ is lsc, and thus continuous by (1).
\end{proof}

\subsection{Pushforward of pseudo-metric families}
\label{sub:pushforward_pseudo-metric_families}

This subsection is devoted to studying the behaviour of the supremum pseudo-norm family determined by a pseudo-metric. This notion is crucial in view of developing volume functions on topological adelic curves. We fix a topological adelic curve $S=(K,\phi:\Omega\to M_K,\nu)$ and a projective $K$-scheme $X$. 

\begin{definition}
\label{def:pushforward_pseudo-metric_family_general}
Assume that $X$ is geometrically reduced. Let $L$ be a line bundle on $X$. Let $\varphi$ be a pseudo-metric on $L$ over $S$ determining a family $(\varphi_{\omega})_{\omega\in\Omega}$ as in Proposition \ref{prop:pseudo-metric_family_global_local}. For any $\omega\in \Omega$, the pseudo-metric $\varphi_{\omega}$ on $L$ in $\omega$ induces a pseudo-norm $\|\cdot\|_{\varphi_{\omega}}$ on $H^{0}(X,L)$ (cf. Proposition-Definition \ref{propdef:supnorm_pseudo-metric}). In the case where the family $\xi_{\varphi}:=(\|\cdot\|_{\varphi_{\omega}})_{\omega\in\Omega}$ is a pseudo-norm family in the sense of Definition \ref{def:semi-norm_family}, we call it the \emph{pushforward pseudo-norm family} of $\varphi$. In that case, we say that the pushforward pseudo-norm family of $\varphi$ is \emph{well-defined}.
\end{definition}

\begin{remark}
We use the notation of Definition \ref{def:pushforward_pseudo-metric_family_general}. In general, it is not clear that the condition $(\ast)$ in Definition \ref{def:semi-norm_family} is satisfied for the family $\xi$ of supremum pseudo-norms. We will give a sufficient condition to ensure this (Proposition \ref{prop:existence_pushforward_pseudo-norm_family}). 
\end{remark}

\begin{example}
\label{example:pushforward_pseudo-metric_family}
Let $E$ be a finite-dimensional $K$-vector space and let $\xi$ be a pseudo-norm family on $E$. Assume that $X$ is geometrically reduced. Let $L$ be a line bundle on $X$. Assume that we have a surjective homomorphism $\beta:E\otimes_{K}\cO_{X}\to L$. Denote by $\varphi$ the quotient pseudo-metric on $L$ over $S$ defined by $(E,\xi)$ and $\beta$. By construction of the quotient pseudo-metric, we see that any adapted basis of $E$ to $\xi$ (on a Zariski open subset $U\subset \tilde{\Omega}$) is also adapted to the pushforward pseudo-norm family $\xi_{\varphi}$. Thus $\xi_{\varphi}$ is well-defined. Moreover, if $\xi$ possesses a globally adapted basis, then $\xi_{\varphi}$ possesses a globally adapted basis. 
\end{example}

\begin{proposition}
\label{prop:CM_4.1.18}
Assume that $X$ is geometrically reduced. Let $L$ be a line bundle on $X$. Let $\varphi$ be a pseudo-metric on $L$ over $S$. Assume that the pushforward pseudo-norm family $\xi_{\varphi}$ of $\varphi$ is well-defined. Let $K'/K$ be a finite field extension. Let $X_{K'}:=X \otimes_{K} K'$, $L':=L\otimes_{K}K'$ and $S' := S \otimes_{K} K'=(K',\phi':\Omega' \to M_{K'},\nu')$ denote the topological adelic curve constructed in \S \ref{sub:finite_extension_tac}. Denote by $\varphi'$ the pseudo-metric on $L'$ by $\varphi$ induced by extension of scalars. Then the following hold.
\begin{itemize}
	\item[(1)] The pushforward pseudo-norm family $\xi_{\varphi'}$ of $\varphi'$ is well-defined.
	\item[(2)] If $\xi_{\varphi'}$ is dominated, then $\xi_{\varphi}$ is dominated.
\end{itemize}
\end{proposition}

\begin{proof}
(1) follows directly from the fact that the $\xi_{\varphi}$ is well-defined. Moreover, the pseudo-norm families $\xi_{\varphi},\xi_{\varphi'}$ satisfy the hypotheses of Proposition \ref{prop:constructions_dominated_semi-norm_families} (7). Hence (2) holds.
\end{proof}

We are now able to prove the analogue of (\cite{ChenMori}, Theorem 6.1.13) in our setting.

\begin{theorem}
\label{th:pushforward_pseudo-metric_family_dominated_general}
Assume that $X$ is geometrically integral. Let $L$ be a line bundle on $X$ equipped with a pseudo-metric $\varphi$ over $S$. Assume that the pushforward pseudo-norm family $\xi$ of $\varphi$ is well-defined. If $\varphi$ is dominated, then $\xi$ is strongly dominated. 
\end{theorem}

\begin{proof}

As $\xi$ is ultrametric on $\Omega_{\um}$, it suffices to prove that $\xi$ is dominated. Moreover, we claim that we may assume that $\xi$ admits a globally adapted basis. Indeed, by quasi-compactness of $\tilde{\Omega}$, we can cover $\tilde{\Omega}$ by finitely many open subsets such that on each of these subsets, there is a basis of $H^{0}(X,L)$ that is globally adapted. By restricting the adelic structure to the analytification of these subsets, we obtain a finite open covering of $\Omega$ and if the statement holds when restricting the adelic structure to each of the members of the covering, then the statement holds over $S$.

\begin{claim}
\label{claim_1}
Assume that $L$ is very ample. Then $\xi$ is dominated.
\end{claim}

\begin{proof}
Let $E= H^{0}(X,L)$ and $r:= \dim_{K}(E)$. As in the proof of (\cite{ChenMori}, Theorem 6.1.13), there exist a finite extension $K'/K$ and closed points $P_1,...,P_r$ of $X$ such that $\kappa(P_i) \subset K'$ for all $i=1,...,r$. Moreover, we have a strictly decreasing sequence of $K$-vector spaces
\begin{align*}
\{0\} = E_r \subsetneq \cdots \subsetneq E_r = E \otimes_{K} K',
\end{align*} 
such that
\begin{align*}
\forall i\in\{1,...,r\}, \quad E_i = \{s\in E_0 : s(P_1) = \cdots s(P_r) = 0\}.
\end{align*}
By Proposition \ref{prop:CM_4.1.18}, we may assume that $K=K'$. 

Let $\alpha_1,...,\alpha_r$ denote respectively local bases of $L$ around $P_1,...,P_r$. Then we define a basis $(\theta_1,...,\theta_r)$ of $E^{\vee}$ as follows. 
\begin{align*}
\forall i\in\{1,...,r\}, \quad \forall s \in E, \quad \theta_i(s) := f_{s}(P_i), 
\end{align*}
where, for any $s\in E$, $s=f_{s}\alpha_i$ around $P_i$. Denote by $(e_1,...,e_r)$ the corresponding dual basis of $E$. 

Let $\omega\in\Omega$. Define a pseudo-norm on $E$ in $\omega$ as follows. Let $s\in E$ written as $s=s_1e_1 +\cdots s_re_r$, where $s_1,...,s_r\in K$. We set
\begin{align*}
\|s\|'_{\omega} := \displaystyle\max_{i=1,...,r} |s_i|_{\omega}.
\end{align*}
Then $\xi'=(\|\cdot\|'_{\omega})_{\omega}$ is a pseudo-norm family on $E$ and the basis $(e_1,...,e_r)$ is globally adapted basis to $\xi'$. Note that $\xi'$ is dominated. Moreover, Lemma \ref{lemma:pseudo-norm_family_globally_adapted} implies that, there exists an open subset $\Omega'\subset \Omega$ such that, for any $\omega\in\Omega'$, the finiteness modules of $\|\cdot\|_{\omega}$ and $\|\cdot\|'_{\omega}$ coincide and $\nu(\Omega')=0$. From now on, we work on the topological adelic curve $S'=(K,\phi' : \Omega' \to M_{K},\nu_{|\Omega'})$ obtained by restriction of the adelic structure on $\Omega'$ (cf. \S \ref{subsub:restriction_adelic_structure}). 

We have a  surjective homomorphism $\beta:E\otimes_{K}\cO_{X}$ and we denote by $\varphi'$ the quotient pseudo-metric induced by $(E,\xi')$ and $\beta$. By Proposition \ref{prop:pseudo-metric_family_global_local}, we can write $\varphi=(\varphi'_{\omega})_{\omega\in\Omega'}$ as a family of local pseudo-metrics. Let us prove that, for any $\omega\in \Omega'$, we have $\|\cdot\|'_{\omega}=\|\cdot\|_{\varphi'_{\omega}}$. Let $\omega\in\Omega'$. Using (\cite{ChenMori}, Proposition 2.2.23), we have the inequality $\|\cdot\|_{\varphi'_{\omega}}\leq\|\cdot\|'_{\omega}$. We prove the converse inequality. For any $i=1,...,r$ and $\omega\in\Omega'$, the rational point $P_{i}$ defines a unique point $\mathbf{P_{i,\omega}}$ in $\ZR(X/A_{\omega})_{S}^{\an}$. Let $(i,j)\in \{1,...,r\}$. Then we have
\begin{align*}
|e_j|_{\varphi'_{\omega}}(\mathbf{P_{i,\omega}}) =  \left\{\begin{matrix}
1 \quad \text{if } i=j, \\
0 \quad \text{if } i\neq j.\end{matrix}\right.
\end{align*}
Let $s=s_1e_1+\cdots+s_re_r \in E$, where $s_1,...,s_r\in K$. Then
\begin{align*}
\forall i=1,...,r, \quad |s|_{\varphi'_{\omega}}(\mathbf{P_{i,\omega}}) = \|s\|'_{\omega} \leq \|s\|_{\varphi'_{\omega}}.
\end{align*} 

Therefore, the following holds.
\begin{align*}
\forall \omega\in \Omega', \quad d_{\omega}(\xi,\xi') \leq d_{\omega}(\|\cdot\|_{\varphi_{\omega}},\|\cdot\|_{\varphi_{\omega}}) \leq d_{\omega}(\varphi,\varphi').
\end{align*}
As $\varphi$ is dominated, the local distance function $(\omega\in\Omega') \mapsto d_{\omega}(\varphi,\varphi')$ is $\nu_{|\Omega'}$-dominated. Therefore, using Proposition \ref{prop:CM_4.1.6}, we obtain that $\xi$ is dominated. This concludes the proof of the claim.
\end{proof}
 
\begin{claim}
\label{claim_2}
For any $s\in H^{0}(X,L)\setminus\{0\}$, the function $(\omega\in\Omega)\mapsto \log\|s\|_{\varphi_{\omega}}$ is $\nu$-dominated.
\end{claim}

\begin{proof}
Let $s\in H^{0}(X,L)\setminus\{0\}$. Lemma \ref{lemma:pseudo-norm_family_singular} implies that there exists a locally closed subset $\Omega'\subset \Omega$ such that $\nu(\Omega\setminus \Omega')=0$ and, for any $\omega\in\Omega'$, $\|s\|_{\varphi_{\omega}} \notin \{0,+\infty\}$. We denote by $S'=(K,\phi':\Omega' \to M_K,\nu_{|\Omega'})$ the restriction of the adelic structure introduced in \S \ref{subsub:restriction_adelic_structure}. 

Choose a very ample line bundle $L_{1}$ on $X$ such that $L_2:=L+L_1$ is very ample. For $i=1,2$, $L_{i}$ define a surjective homomorphism $\beta'_{i} : H^{0}(X,L_{i})\otimes_{K}\cO_{X}\to L_{i}$. Note that multiplication by $s$ yields an injective homomorphism $H^{0}(X,L_{1})\to H^{0}(X,L_{2})$. 

We consider an arbitrary dominated pseudo-norm family $\xi'_{2}$ on $H^{0}(X,L_{2})$ that admits a globally adapted basis. Denote by $\xi'_{1}$ the restriction of $\xi'_{2}$ to $H^{0}(X,L_{1})$. Note that Proposition \ref{prop:constructions_dominated_semi-norm_families} (1) implies that $\xi'_{1}$ is dominated. Let $i=1,2$. Denote by $\varphi'_{i}$ the quotient pseudo-metric induced by $\xi'_{i}$ and $\beta_{i}$. Let $\varphi' := \varphi'_2-\varphi'_1$, this is a pseudo-metric on $L$. Hence Proposition \ref{prop:dominated_pseudo-metric_families} (5) implies that the local distance function $(\omega\in \Omega') \mapsto d_{\omega}(\varphi,\varphi')$ is $\nu_{|\Omega'}$-dominated. 

Arguing as in the proof of (\cite{ChenMori}, Claim 6.1.16), we obtain that for any $\omega\in\Omega'$, we have $\|s\|_{\varphi'_{\omega}}\leq 1$. Moreover, for any $u\in E$ such that $\|u\|_{\varphi'_{1,\omega'}}\in\bR_{>0}$ for any $\omega'\in\Omega'$  (which exists up to shrinking $\Omega'$), we have
\begin{align*}
\|su\|_{\varphi'_{2,\omega}} \leq \|s\|_{\varphi'_{\omega}} \|u\|_{\varphi'_{1,\omega}}. 
\end{align*}
Using the fact that $\xi'_{1}$ and $\xi'_{2}$ are dominated. we deduce that the function $(\omega\in\Omega')\mapsto \log\|s\|_{\varphi'_{\omega}}$ is $\nu_{|\Omega'}$-dominated. We can conclude using the fact that
\begin{align*}
\forall \omega\in\Omega', \quad d_{\omega}(\|\cdot\|_{\varphi_{\omega}},\|\cdot\|_{\varphi'_{\omega}})\leq d_{\omega}(\varphi,\varphi').
\end{align*}
\end{proof}

We now prove Theorem \ref{th:pushforward_pseudo-metric_family_dominated_general}. Choose a very ample line bundle $L_{1}$ on $X$ such that $L_2:=L+L_1$ is very ample. Fix a global section $t\in H^{0}(X,L_{1})$. Let $\varphi_1$ be a dominated pseudo-metric such that the pushforward metric family of $\varphi_{1}$ is well-defined and possesses a globally adapted basis and, for any $\omega\in\Omega$, we have $\|t\|_{\varphi_{1,\omega}}\leq 1$ (it is possible to do so by choosing a suitable quotient pseudo-metric family cf. Example \ref{example:pushforward_pseudo-metric_family}). Let $\varphi_{2} := \varphi +\varphi_{1}$ and denote by $\xi_{2}$ the pushforward pseudo-metric family of $\varphi_{2}$. Proposition \ref{prop:dominated_pseudo-metric_families} (2) implies that $\varphi_{2}$ is dominated. As $L_{2}$ is very ample, Claim \ref{claim_1} implies that $\xi_{2}$ is strongly dominated.

Fix a basis $(s_{1},...,s_{r})$ of $H^{0}(X,L)$ which is globally adapted to $\xi$. For $i=1,...,r$, write $t_{i}:=ts_{i} \in H^{0}(X,L_{2})$. Enlarge $(t_{1},...,t_{r})$ to a basis $\mathbf{t}:=(t_{1},...,t_{n})$ of $H^{0}(X,L_2)$ which is globally adapted to $\xi_2$ (this is possible up to removing a measure $0$ set from $\Omega$). Let $\xi_{\mathbf{t},2}=(\|\cdot\|_{\mathbf{t},2,\omega})_{\omega\in\Omega}$ denote the model pseudo-norm family on $H^{0}(X,L_2)$ defined by the basis $\mathbf{t}$ (cf. Example \ref{example:model_semi-norm_family} (2)). Corollary \ref{cor:analogue_ChenMori_4.1.10} implies that the local distance function $(\omega\in\Omega) \mapsto d_{\omega}(\xi_2,\xi_{\mathbf{t},2})$ is $\nu$-dominated. Thus, there exists a $\nu$-dominated function $A:\Omega \to [-\infty,+\infty]$ such that, for almost all $\omega\in\Omega$, for any $(\lambda_1,...,\lambda_r)\in K^{r}$, we have
\begin{align}
\label{eq:pushforward_1}
\log\|\lambda_1 s_1 + \cdots + \lambda_r s_r\|_{\varphi_{\omega}} &\geq \log\|\lambda_1 t_1 + \cdots + \lambda_r t_r\|_{\varphi_{2,\omega}} \nonumber\\
& \geq \log\|\lambda_1 t_1 + \cdots + \lambda_r t_r\|_{\varphi_{\mathbf{t},2,\omega}} - A(\omega), 
\end{align}
where the first inequality comes from the fact that $\|t\|_{\varphi_{1,\omega}}\leq 1$. Moreover, we have the inequality
\begin{align*}
\log\|\lambda_{1}s_1 +\cdots \lambda_r s_r\|_{\varphi_{\omega}} \leq \log\|\lambda_{1}t_1 +\cdots \lambda_r t_r\|_{\varphi_{\mathbf{t},2,\omega}} + \displaystyle\max_{1\leq i \leq r} \log\|s_i\|_{\varphi_{\omega}}.
\end{align*}
Claim \ref{claim_2} implies that there exists a $\nu$-dominated function $B : \Omega \to [-\infty,+\infty]$ such that the inequality
\begin{align}
\label{eq:pushforward_2}
\log\|\lambda_{1}s_1 +\cdots \lambda_r s_r\|_{\varphi_{\omega}} \leq \log\|\lambda_{1}t_1 +\cdots \lambda_r t_r\|_{\varphi_{\mathbf{t},2,\omega}} + B(\omega)
\end{align} 
holds for almost all $\omega\in\Omega$.

Denote by $\xi'$ the restriction of the pseudo-norm family $\xi_{\mathbf{t},2}$ to $H^{0}(X,L)$. Then the basis $(s_1,...,s_r)$ is both globally adapted to $\xi$ and $\xi'$. Moreover (\ref{eq:pushforward_1}) and (\ref{eq:pushforward_2}) imply that the local distance function $(\omega\in\Omega) \mapsto d_{\omega}(\xi,\xi')$ is $\nu$-dominated. By Corollary \ref{cor:analogue_ChenMori_4.1.10}, the pseudo-norm family $\xi$ is strongly dominated.
\end{proof}

\begin{proposition}
\label{prop:CM_6.1.18_general}
We consider the case where $X=\Spec(K')$, where $K'/K$ is a finite extension of fields. We denote $S':=S \otimes_{K} K' =(K',\phi' : \Omega'\to M_{K'},\nu')$. Let $L$ be a line bundle on $X$ equipped with a pseudo-metric $\varphi$. Denote by $\xi_L$ the pseudo-norm family on $L$ (w.r.t. the topological adelic curve $S'$) introduced in Remark \ref{rem:CM_6.1.3_general}. Then $\varphi$ is dominated iff $\xi_L$  is dominated.
\end{proposition}

\begin{proof}
As in the proof of Theorem \ref{th:pushforward_pseudo-metric_family_dominated_general}, we may assume that $\xi_{L}$ has a globally adapted basis. One can then argue the same way as in the proof of (\cite{ChenMori}, Proposition 6.1.18).

\end{proof}

Let us now give a criterion on pseudo-metrics that ensures that the corresponding pushforward pseudo-norm family is well-defined. 

\begin{proposition}
\label{prop:existence_pushforward_pseudo-norm_family}
Assume that $\Omega=M_{K}$. Let $L$ be a line bundle on $X$ equipped with a continuous pseudo-metric $\varphi$ over $S$. Then the pushforward pseudo-norm family $\xi=(\|\cdot\|_{\varphi_{\omega}})_{\omega\in \Omega}$ associated with $\varphi$ is well defined and is an usc pseudo-norm family on $H^{0}(X,L)$. 
\end{proposition}

\begin{proof}
Let $s\in H^{0}(X,L)$, let us prove that the function $(\omega\in\Omega) \mapsto \|s\|_{\varphi_{\omega}}$ is upper semi-continuous. We will use the following lemma.

\begin{lemma}[\cite{ChenMori24}, Proposition 2.2.4]
\label{lemma:pushforward_adelic_line_bundle}
Let $\pi : A \to B$ be a closed continuous mapping of topological spaces and let $f:A \to [-\infty,+\infty]$ be an upper semi-continuous function. Assume that, for any $y\in B$, the fibre $\pi^{-1}(y)$ is compact. Then the function 
\begin{align*}
\fonction{\phi_f}{B}{[0,1]}{y}{\displaystyle\sup_{x\in \pi^{-1}(y)}f(x)}
\end{align*}
is upper semi-continuous.
\end{lemma}

Recall that the map $\pi_{S}:\ZR(X)_{S}^{\an} \to \Omega$ is a proper map, thus $\pi_{S}$ is closed and, for all $\omega\in \Omega$, $\pi_{S}^{-1}(\omega)$ is compact. Let $s\in  H^{0}(X,L)$. Apply the above lemma with $A=\ZR(X)_{S}^{\an}$, $B=\Omega$ and $f:((x,\omega)\in \ZR(X)_{S}^{\an}) \mapsto |s|_{\varphi_{\omega}}(x)$, which is continuous. Since, for any $\omega\in\Omega$, we have 
\begin{align*}
\displaystyle\sup_{(x,\omega)\in p^{-1}(\omega)}|s|_{\varphi_{\omega}}(x) = \|s\|_{\varphi_{\omega}},
\end{align*}
we deduce the desired upper semi-continuity.

Thus, it suffices to prove that for any $\omega\in\Omega$, there exists an adapted basis for $\xi$ in $\omega$. We fix $\omega_{0}\in\Omega$. Since $\|\cdot\|_{\varphi_{\omega_{0}}}$ is a pseudo-norm on $E$, there exists a basis $(s_1,...,s_d)$ of $E$ such that, for any $i=1,...,d$, we have $\|s_i\|_{\varphi_{\omega_{0}}}\in\bR_{>0}$. As, for any $i=1,...,d$, the map $(\omega\in\Omega)\mapsto \|s_i\|_{\varphi_{\omega}}\in [0,+\infty]$ is upper semi-continuous, there exists an open neighbourhood $U$ of $\omega_{0}$ in $\Omega$ such that, for any $\omega\in U$, we have 
\begin{align*}
\forall i =1,...,d,\quad \|s_i\|_{\varphi_{\omega}} < +\infty.
\end{align*}  
Therefore, it remains to prove that, up to shrinking $U$, we have 
\begin{align*}
\forall i =1,...,d,\quad \|s_i\|_{\varphi_{\omega}} > 0.
\end{align*}
Choose a projective model $\cX$ of $X$ over $A_{\omega_{0}}$ such that $\varphi_{\omega_{0}}$ is represented over $\cX$ (cf. Notation \ref{notation:pseudo-metric} (1)). For any closed point $P$ of $X$ and $x$ any element $\in\Omega_{\kappa(P),\omega_0}$, by the valuative criterion of properness, we obtain a point in $(\cX\otimes_{A_{\omega_{0}}}\widehat{\kappa}(\omega_{0}))^{\an}$ which is denoted by $(P,x)$ and whose image in $\cX\otimes_{A_{\omega_{0}}}\widehat{\kappa}(\omega_{0})$ is a closed point. Since closed points of $X_{\omega}$ are dense in $X_{\omega}^{\an}$, for any $i=1,...,d$, there exist a closed point $P$ of $X$ and an element $x\in\Omega_{\kappa(P),\omega_0}$ such that $|s_i|_{\varphi_{\omega_{0}}}(P,x)>0$. By continuity of $\varphi$, up to shrinking $U$, we may assume that, for any $\omega\in U$, there exists a point $P_{i,\omega}\in X_{\omega}^{\an}$ such that $|s|_{\varphi_{\omega}}(P_{i,\omega})>0$. In particular, we obtain 
\begin{align*}
\forall i =1,...,d,\quad \|s_i\|_{\varphi_{\omega}} > 0.
\end{align*}
Therefore, $\xi$ satisfies the condition $(\ast_{\an})$ from \S \ref{sub:ZR_interpretation_pseudo-norm_families}. Since $\Omega=M_{K}$, Proposition \ref{prop:ZR_interpretation_pseudo-norm_families_GAGA} implies that $\xi$ is a pseudo-norm family on $H^{0}(X,L)$.
\end{proof}

\subsection{Adelic line bundles}
\label{sub:adelic_line_bundles}

In this subsection, we fix a topological adelic curve $S=(K,\phi:\Omega\to M_K,\nu)$ and a projective $K$-scheme $X$. 

\begin{definition}
\label{def:adelic_line_bundles}
Let $L$ be a line bundle on $X$ equipped with a pseudo-metric $\varphi$ over $S$. We say that $\overline{L}=(L,\varphi)$ is respectively a \emph{usc/lsc adelic line bundle} on $X$ if $\varphi$ is both usc/lsc and dominated. Moreover, we say that $\overline{L}$ is a \emph{(continuous) adelic line bundle} if $\varphi$ is both continuous and dominated.

Moreover, an adelic line bundle $(L,\varphi=(\varphi_{\omega})_{\omega\in\Omega})$ is called \emph{semi-positive}, resp. \emph{integrable}, if, for any $\omega\in\Omega$, the (locally) pseudo-metrised line bundle $(L,\varphi_{\omega})$ is semi-positive, resp. integrable.
\end{definition}

\begin{remark}
\label{rem:adelic_line_bundles_adelic_curves}
In the situation of Example \ref{example:pseudo-metric_adelic_curve}, any adelic line bundle on $S$ yields an adelic line bundle on the corresponding adelic curve in the sense of (\cite{ChenMori}, Definition 6.2.1). Moreover, this association preserves the notions of semi-positivity and integrability. This implies that in the case where $K$ is countable, we can use all the results in \cite{ChenMori,ChenMori21,ChenMori24}. 
\end{remark}

\begin{proposition}
\label{prop:constructions_adelic_line_bundles}
\begin{itemize}
	\item[(1)] Let $\overline{L}=(L,\varphi)$ be a usc/lsc/continuous adelic line bundle on $X$. Then $-\overline{L}:=(-L,-\varphi)$ is respectively a usc/lsc/continuous adelic line bundle on $X$.
	\item[(2)] Let $\overline{L_1}=(L_1,\varphi_1),\overline{L_2}=(L_2,\varphi_2)$ be both usc/lsc/continuous adelic line bundles on $X$. Then $\overline{L_1}+\overline{L_2}:=(L_1+L_2,\varphi_1+\varphi_2)$ is respectively a usc/lsc/continuous adelic line bundle on $X$.
	\item[(3)] Let $f : Y \to X$ be a projective morphism of $K$-schemes. Let $\overline{L}=(L,\varphi)$ be an adelic line bundle on $X$. Then $f^{\ast}\overline{L}:(f^{\ast}L,f^{\ast}\varphi)$ is respectively a usc/lsc/continuous adelic line bundle on $X$.
\end{itemize}
\end{proposition}

\begin{proof}
This is a direct consequence of Propositions \ref{prop:dominated_pseudo-metric_families} and \ref{prop:regularity_pseudo-metric_families}.
\end{proof}

\begin{proposition} 
\label{prop:quotient_adelic_line_bundle}
Let $\overline{E}=(E,\xi)$ where $E$ is a finite-dimensional $K$-vector space and $\xi=(\|\cdot\|_{\omega},\cE_{\omega},N_{\omega},\widehat{E_{\omega}})_{\omega\in\Omega}$ is a pseudo-norm family on $E$ which is assumed to be ultrametric on $\Omega_{\um}$. Let $L$ be a line bundle on $X$. Assume that there exists a surjective homomorphism $\beta:E\otimes_{K}\cO_{X}\to L$. Denote by $\varphi$ the corresponding quotient pseudo-metric on $L$. Assume that $\overline{E}=(E,\xi)$ is an adelic line bundle on $S$. Then $\overline{L}=(L,\varphi)$ is an adelic line bundle on $X$.
\end{proposition}

\begin{proof}
This is a consequence of Propositions \ref{prop:regularity_quotient_pseudo-metric_families} and \ref{prop:quotient_pseudo-metric_family_dominated}. 
\end{proof}

\subsection{Pseudo-metrics and adelic line bundles over integral topological adelic curves}
\label{sub:pseudo-metric_families_integral_tac}

In this subsection, we give another interpretation of pseudo-metrised line bundles over an integral topological adelic curve. The upshot is the following: the Zariski-Riemann space attached to such a topological adelic can be described as a projective limit of models of the projective models over the integral structure. To any such model, one can define the notion of a model pseudo-metric being a family of local model pseudo-metric satisfying a "globalisation condition", and a pseudo-metric can be interpreted as an equivalence class of such model pseudo-metrics. This point of view is more classical in the sense of Arakelov geometry. Moreover, this point of view will give a convenient way to give examples of adelic line bundles arising in Nevanlinna theory.

We fix an integral topological adelic curve $S=(K,\phi:\Omega\to V,\nu)$ with underlying integral structure $(A,\|\cdot\|)$ and $V:= \cM(A,\|\cdot\|)$ and a projective $K$-scheme $X$. We further assume that $(A,\|\cdot\|)$ is a geometric base ring. 

\subsubsection{Zariski-Riemann spaces}

Recall that in our situation, the Zariski-Riemann spaces of interest admit the following description (we use the notations of \S \ref{sub:ZR_spaces_adelic_curves}):
\begin{align*}
\ZR(X/A)\cong\displaystyle\varprojlim_{\cX\in\cM_{X}}\cX, \quad \ZR(X/A)^{\an}\cong\varprojlim_{\cX\in\cM_{X}}\cX^{\an}, \quad \ZR(X)_{S} \cong \varprojlim_{\cX\in\cM_{X}}\cX_{S}, \quad \ZR(X)_{S}^{\an} \cong \varprojlim_{\cX\in\cM_{X}}\cX_{S}^{\an},
\end{align*}
where $\cM_{X}$ denote the full subcategory of $\cM_{X/A}$ consisting of projective flat and coherent projective models of $X$ over $A$. Moreover, these Zariski-Riemann spaces fit in the following diagram of locally ringed spaces  
\begin{center}
\begin{tikzcd}
                                        & \ZR(X)^{\an}_{S} \arrow[rr, "\phi_{X}"] \arrow[dd] \arrow[rd, "{j_{X,S}}"] &                                                      & \ZR(X/A)^{\an} \arrow[dd] \arrow[rd, "j_{X}"] &                       \\
X \arrow[dd] \arrow[rr, "{\eta_{X,S}}"] &                                                                            & \ZR(X)_{S} \arrow[rr, "\tilde{\phi_{X}}"] \arrow[dd] &                                               & \ZR(X/A) \arrow[dd]   \\
                                        & \Omega \arrow[rr, "\phi"] \arrow[rd, "{j_{K,S}}"]                          &                                                      & \ZR(K/A)^{\an}\cong V \arrow[rd, "j_{K}"]     &                       \\
\Spec(K) \arrow[rr, "\eta_{S}"]         &                                                                            & \tilde{\Omega} \arrow[rr, "\tilde{\phi}"]            &                                               & \ZR(K/A)\cong\Spec(A)
\end{tikzcd}.
\end{center}

\subsubsection{Pseudo-metrics}

\begin{definition}
\label{def:pseudo-metric_family_integral}
Let $L$ be a line bundle on $X$. 
\begin{itemize}
	\item[(1)] Let $\cX\in\cM_{X}$. A \emph{metrised line bundle} on $\cX_{S}$ is the data $(\cL,\varphi)$, where $\cL$ is a line bundle on $\cX_{S}$ and $\varphi$ is a metric on $\cL$.
	\item[(2)] A \emph{model pseudo-metric} $\varphi$ on $L$ over is the data $((\cX,\cL),\varphi)$, where $\cX\in M_{X}$, $\cL$ is a model of $L$ on $\cX$ and $\varphi$ is a metric on $\cL$. $(\cX,\cL)$ is called the \emph{model} of the model pseudo-metric $((\cX,\cL),\varphi)$. Moreover, by "let $\varphi$ be a model pseudo-metric with model $(\cX,\cL)$", we mean that $(\cX,\cL)$ is a model of $(X,L)$ over $A$ with $\cX\in\cM_{X}$, $\varphi$ is a metric on $\cL$ and that we are concerned with the model pseudo-metric $((\cX,\cL),\varphi)$.
	\item[(3)] Let $\left((\cX,\cL),\varphi\right),\left((\cX',\cL'),\varphi'\right)$ be two model pseudo-metrics on $L$, we say that these two model pseudo-metrics are \emph{equivalent} if there exist a model $(\cX'',\cL'')$ of $(X,L)$ with $\cX''\in\cM_{X}$ and arrows $p:\cX''\to\cX, q:\cX''\to\cX'$ in $\cM_{X}$ such that $p^{\ast}\cL\cong q^{\ast}\cL'$ $p^{\ast}\varphi\cong q^{\ast}\varphi'$.
	Using Proposition-Definition \ref{propdef:equivalence_relation_model_pseudo-metrics}, it is straightforward to check that this defines an equivalence relation on model pseudo-metrics on $L$. The equivalence class of a model-pseudo-metric $\varphi$ on $L$ is denoted by $[\varphi]$. Moreover, by "let $[(\cX,\cL,\varphi)]$ be an equivalence class of model-pseudo-metrics on $L$", we mean that we are concerned with the equivalence class of a model-pseudo-metric $(\cX,\cL,\varphi)$ and we say that $[(\cX,\cL,\varphi)]$ is \emph{represented} on $(\cX,\cL)$.
\end{itemize}
\end{definition} 

\begin{proposition}
\label{prop:ZR_interpretation_pseudo-metrics_integral_tac}
Assume that $\tilde{\Omega}$ is a locally closed subset of $\Spec(A)$. Let $L$ be a line bundle on $X$. There is a one-to-one correspondence between pseudo-metrics on $L$ (in the sense of Definition \ref{def:pseudo-metric_family}) and equivalence classes of model pseudo-metrics.
\end{proposition}

\begin{proof}
Let $[(\cL,\varphi)]$ be an equivalence class of model-pseudo-metrics on $L$ represented on a model $\cX\in\cM_{X}$. Then by pulling back $(\cL,\varphi)$ to $\ZR(X)_{S}$, we get a metrised line bundle on $\ZR(X)_{S}$, and thus a metrised line bundle on $X$ over $S$ by Proposition \ref{prop:ZR_interpretation_pseudo-metric_families}.

Conversely, let $(\cL,\varphi)$ be a metrised line bundle on $\ZR(X)_{S}$. By Proposition \ref{prop:coherent_sheaves_locally_closed_integral}, there exists a projective model $\cX\in\cM_{X}$ and a line bundle $\cL_{\cX}$ on $\cX$ such that $p_{\cX}^{-1}\cL_{\cX}\cong \cL$, where $p_{\cX}:\ZR(X)_{S}\to \cX_{S}$ denotes the projection. Then the metric $\varphi$ induces a model pseudo-metric on $L:=\eta_{X,S}^{-1}\cL$ and we can consider its equivalence class. This construction is inverse to the previous one. 
\end{proof}

\subsubsection{Example of adelic line bundles in Nevanlinna theory}
\label{subsub:adelic_line_bundle_Nevanlinna}

Let $R>0$. Consider the topological adelic curve $S_{R}=(K_{R},\phi_{R} : \Omega_{R} \to V_{R},\nu_{R})$ defined in \S \ref{subsub:example_tac_Nevanlinnaç_compact_disc}. Note that $\tilde{\Omega_{R}}=\ZR(K/A)\cong\Spec(A)$. Let $X$ be a reduced projective $\bC$-scheme and denote $X_{R}=X \otimes_{\bC} K_{R}$. Set $\cX_{R}:=X \otimes_{\bC} A_{R}$, it is a projective model of $X_{R}$ over $A_{R}$ and $(\cX_{R})_{S_{R}}=\cX_{R}$. Let $(L,\varphi)$ be an lsc/usc/continuous metrised line bundle on $X$, namely $L$ is a line bundle on $X$ and $\varphi$ is an lsc/usc/continuous metric on $L$. Let $L_{R} := L \otimes_{\cO_X} \cO_{X_{R}}$, $\cL_{R}:=L\otimes_{\cO_{X}}\cO_{\cX_{R}}$ and, for any $\omega\in\Omega_{R}$. %

\begin{proposition}
\label{prop:adelic_line_bundle_Nevanlinna}
We use the same notation as above. Then $(L,\varphi)$ induces an lsc/usc/continuous adelic line bundle $(L_{R},\varphi_{R})$ on $X_{R}$ over $S_{R}$. Moreover, if $(L,\varphi)$ is semi-positive, resp. integrable, then so is $(L_{R},\varphi_{R})$.
\end{proposition}

\begin{proof}
\textbf{Definition of the pseudo-metric:} Let $\varphi_{R}:=(|\cdot|_{\varphi_{R}}(x))_{x\in(\cX_{R})_{S_{R}}^{\an}}$ be the family defined as follows. Let $x=(p,|\cdot|_{x})\in(\cX_{R})_{S_{R}^{\an}}$, where $p\in \cX_{R}$ and $|\cdot|_{x}$ is an absolute value on the residue field $\kappa(x)$ mapping to an element of the image of $\Omega$ in $V_{R}$. First assume that $|\cdot|_{x}$ is Archimedean. Thus $\cL(p)$ identifies with $L(q)$, where $q$ denotes the image of $p$ in $X$ and we define $|\cdot|_{\varphi_{R}}(x)$ to be $|\cdot|_{\varphi}(q)$. Now assume that $|\cdot|_{x}$ is non-Archimedean, mapping to an element $\omega\in\Omega_{\um}$. Then $\phi_R(\omega)$ is a usual absolute value on $K_{R}$ and the completion $K_{R,\omega}$ of $K_{R}$ w.r.t. $|\cdot|_{\omega}$ is isomorphic to $\bC((T))$. Denote by $K_{R,\omega}^{\circ}\cong\bC[[T]]$ the corresponding valuation ring of $K_{R,\omega}$. Let $\tilde{\cX_{R,\omega}}:=X\otimes_{\bC}K_{R,\omega}^{\circ}$ and $\tilde{\cL_{R,\omega}}:=L \otimes_{\cO_{X}} \cO_{\tilde{\cX_{R,\omega}}}$. Then $(\tilde{\cX_{R,\omega}},\tilde{\cL_{R,\omega}})$ is a model of $X_R \otimes_{K_R} K_{R,\omega}$ over $K_{R,\omega}^{\circ}$ and the model metric $\varphi_{R,\omega}$ determined by this model yields a norm $|\cdot|_{\varphi_{R}}(x)$ on $\cL_{R}(p)$. Thus $\varphi_{R}$ is a metric on $\cL_{R}$ and taking its equivalence class, we get a pseudo-metric on $L_{R}$ (cf. Proposition \ref{prop:ZR_interpretation_pseudo-metrics_integral_tac}) that we denote again by $\varphi_{R}$.

\textbf{$\varphi_{R}$ is dominated:} It suffices to prove it in the case where $L$ is very ample, up to passing to a multiple of $L$, we may assume that $L_{R}$ and $\cL_{R}$ are globally generated. Denote $E := H^0(X,L)$ and
\begin{align*}
E_R := E \otimes_{\bC} K_{R},\quad \forall \omega\in \Omega_{R,\infty}, \quad \cE_{R,\omega} := E \otimes_{\bC} A_{R,\omega}, \quad\forall \omega\in \Omega_{R,\um},\quad \tilde{\cE_{R,\omega}} := E \otimes_{\bC} K_{R,\omega}^{\circ}.
\end{align*}
Let $\|\cdot\|_{\varphi}$ denote the supremum norm on $E$ induced by $\varphi$ and denote by $\varphi_{\mathrm{FS}}$ the Fubini-Study (usual) metric on $L$ associated with $\varphi$. Recall that $\varphi_{\mathrm{FS}}$ is the quotient metric associated with the complex normed vector space $(E,\|\cdot\|_{\varphi})$. Moreover, using the results of \S \ref{sub:example_in_Nevanlinna_theory_avb}, we have an adelic vector bundle $\overline{E_{R}}=(E_{R},\xi_{R})$ on $S_{R}$. We denote by $\varphi_{R,\mathrm{FS}}$ the quotient pseudo-metric on $L_{R}$ determined by the surjective homomorphism $E_{R}\otimes_{K_{R}} \cO_{X_{R}} \to L_{R}$ and $\overline{E_{R}}$. 

Let $\omega\in \Omega_{R,\um}$. As $\bC \to K_{R,\omega}^{\circ}$ is flat, for any $\omega\in \Omega_{R,\um}$, flat base change yields $\tilde{\cE_{R,\omega}}= H^{0}(\tilde{\cX_{R,\omega}},\tilde{\cL_{R,\omega}})$ and the evaluation morphism $\tilde{\cE_{R,\omega}} \otimes_{K_{R,\omega}^{\circ}} \cO_{\tilde{\cX_{R,\omega}}} \to \tilde{\cL_{R,\omega}}$ is surjective. Using (\cite{ChenMori}, Proposition 2.3.12), we obtain that $\varphi_{R,\omega}$ identifies with the quotient metric induced by the lattice norm defined by the lattice $\tilde{\cE_{R,\omega}}$ inside $\tilde{E_{R,\omega}}:=\tilde{\cE_{R,\omega}} \otimes_{K_{R,\omega}^{\circ}} K_{R,\omega}$. Note that from the description of $\xi_R$ given in \S \ref{sub:example_in_Nevanlinna_theory_avb}, we have $\varphi_{R,\mathrm{FS},\omega}=\varphi_{\omega}$. 

Moreover, we have 
\begin{align*}
\forall \omega\in \Omega_{R}, \quad d_{\omega}(\varphi_R,\varphi_{R,\mathrm{FS}}) = \mathbf{1}_{\Omega_{R,\infty}}(\omega)d(\varphi,\varphi_{\mathrm{FS}}),
\end{align*}
where $d(\varphi,\varphi_{\mathrm{FS}})$ denotes the usual distance between the (complex) continuous metrics $\varphi$ and $\varphi_{\mathrm{FS}}$. As $\nu_R(\Omega_{R,\infty})<+\infty$, we obtain that the local distance function $[(\omega\in\Omega) \mapsto d_{\omega}(\varphi_R,\varphi_{R,\mathrm{FS}})]$ is $\nu_R$-dominated. Using Proposition \ref{prop:example_avb_Nevanlinna} and Proposition \ref{prop:quotient_pseudo-metric_family_dominated}, we obtain that $\varphi_{R,\FS}$ is dominated. Therefore, by using Proposition \ref{prop:dominated_pseudo-metric_families} (3), we see that the pseudo-metric $\varphi_R$ is dominated.

\textbf{Regularity of $\varphi_{R}$:} Since $\Omega_{R,\um}$ is discrete, we have to prove the regularity of $\varphi_{R}$ at points lying over $\Omega_{R,\infty}$. This follows from the construction of the norms $|\cdot|_{\varphi}(\cdot)$ and the regularity of $\varphi$. 

\textbf{Semi-positivity/integrability of $\varphi_{R}$:} It suffices to treat the case where $L$ is very ample and $\varphi$ is semi-positive. Then $L_{R}$ and $\cL_{R}$ are nef. Let $\omega\in\Omega_{R,\um}$. Then $\varphi_{R,\omega}$ comes from a model metric on a nef model of $X_{R}$. Thus by (\cite{ChenMori21}, Theorem 3.2.19), $\varphi_{R,\omega}$ is semi-positive. The semi-positivity of $\varphi_{R,\omega}$, where $\omega\in\Omega_{R,\infty}$ follows from the semi-positivity of $\varphi$.
\end{proof}

\begin{definition}
\label{def:adelic_line_bundle_Nevanlinna}
Let $R>0$. Then the lsc/usc/continuous adelic line bundle $(L_{R},\varphi_{R})$ constructed in Proposition \ref{prop:adelic_line_bundle_Nevanlinna} is called the \emph{lsc/usc/continuous adelic line bundle induced} by $(L,\varphi)$ on $X_{R}$. 
\end{definition}

More generally, we have the following.

\begin{proposition-definition}
\label{prop-def:adelic_line_bundle_Nevanlinna_finite covering} 
Let $F/K_{R}$ be an algebraic extension. Then the usc/lsc/continuous metrised line bundle $(L,\varphi)$ induces a usc/lsc/continuous adelic line bundle on $S_{R}\otimes_{K_{R}}F$ which is given by extension of scalars of $(L_{R},\varphi_{R})$ to $F$. Moreover, this usc/lsc/continuous adelic line bundle is semi-positive/integrable if so is $(L,\varphi)$.
\end{proposition-definition}

\begin{proposition}
\label{prop:adelic_line_bundle_Nevanlinna_pushforward}
Assume that $X$ is reduced. Let $R>0$ and let $(L,\varphi)$ be a continuous metrised line bundle on $X$. We denote by $\overline{L_R}=(L_R,\varphi_R)$ the adelic line bundle on $X_R$ induced by $(L,\varphi)$. Denote by $\xi_R$ the collection of supremum pseudo-norms on $E_R:=H^{0}(X_R,L_R)$ induced by the pseudo-metric $\varphi$. Then $\overline{E_R}:=(E_R,\xi_R)$ is a continuous adelic line bundle on $S_R$.
\end{proposition}

\begin{proof}
We first note that $\xi_R$ is a well-defined pseudo-norm family on $E_R$. Indeed, any basis of the space of global sections $E:=H^{0}(X,L)$ yields a globally adapted basis for $\xi_R$. Denote by $\|\cdot\|$ the supremum norm on $H^{0}(X,L)$ induced by the continuous metric $\varphi$.

Write $\xi_R=(\|\cdot\|_{\varphi_{\omega}})_{\omega\in\Omega_{R}}$. Let $\omega\in\Omega_{R,\infty}$. Then $\|\cdot\|_{\varphi_{\omega}}$ coincides with the local pseudo-norm on $H^{0}(X,L)$ in $\omega$ whose residue vector space is $E$ and whose residue norm is $\|\cdot\|$. Let $\omega\in \Omega_{R,\um}$. As $X$ is reduced, (\cite{ChenMori}, Proposition 2.3.16 (3)) implies that $\|\cdot\|_{\varphi_{\omega}}$ coincides with the lattice norm used in \S \ref{sub:example_in_Nevanlinna_theory_avb}. 

Using this description, we see that $(E_R,\xi_R)$ coincides with the adelic vector bundle induced by the normed vector space $(E,\|\cdot\|)$ (cf. Definition \ref{def:example_avb_Nevanlinna}). This allows us to conclude.
\end{proof}

\subsection{Adelic line bundles on families of topological adelic curves}
\label{sub:adelic_line_bundles_families_of_tac}

Throughout this subsection, we consider the following setting. Let $\mathbf{S}=(I,\cU,(S_{i}=(K_{i},\phi_{i}:\Omega_{i}\to M_{K_{i}},\nu_{i}))_{i\in I},K)$ be a family of topological adelic curves. Let $X$ be a projective $K$-scheme. Set $X_{\mathbf{S}}:=X\otimes_{K}\prod_{\cU}K_{i}$. 

\begin{proposition-definition}
\label{prop-def:pseudo-metric_family_tac}
\begin{itemize}
	\item[(1)] By a \emph{pseudo-metrised} line bundle $(L,\varphi)$ on $X$ over $\mathbf{S}$, we mean the equivalence class of a family $(X_{i},L_{i},\varphi_{i})_{i\in I}$, where the $X_{i}$'s are projective schemes over the $K_{i}$'s such that $\prod_{\cU}X_{i}\cong X\otimes_{K}\prod_{\cU}K_{i}=:X_{\mathbf{S}}$ and the $(L_{i},\varphi_{i})$'s are pseudo-metrised line bundles on the $X_{i}$'s such that $L\otimes_{\cO_{X}}\cO_{X_{\mathbf{S}}}\cong \prod_{\cU}L_{i}$, and where two such families $(X_{i},L_{i},\varphi_{i})_{i\in I},(X'_{i},L'_{i},\varphi'_{i})_{i\in I}$ are declared to be equivalent if there exist $\cU$-almost everywhere isomorphisms $X_{i}\cong X_{i}'$ yielding compatible isomorphisms $(L_{i},\varphi_{i})\cong(L'_{i},\varphi'_{i})$. We use the notation $(L,\varphi)=[(X_{i},L_{i},\varphi_{i})_{i\in I}]$ to denote such an equivalence class as above and $\varphi=[(\varphi_{i})_{i\in I}]$ is called a \emph{pseudo-metric} on $L$ over $\mathbf{S}$.
	\item[(2)] Let $L$ be a line bundle on $X$. A pseudo-metric $\varphi$ on $L$ over $\mathbf{S}$ is called \emph{dominated}, resp. \emph{usc/lsc/continuous}, resp. \emph{semi-positive/integrable} if there exists a family of pseudo-metric $(\varphi_{i})_{i\in I}$ as in (1) representing $\varphi$ such that $\varphi_{i}$ is dominated, resp. usc/lsc/continuous, resp. semi-positive/integrable $\cU$-almost everywhere. This is independent of the choice of the representative.
	\item[(3)] A pseudo-metrised line bundle $(L,\varphi)$ on $X$ over $\mathbf{S}$ is called a \emph{usc/lsc/continuous adelic line bundle} on $X$ over $\mathbf{S}$ if $\varphi$ is dominated and usc/lsc/continuous. Moreover, such a usc/lsc/continuous adelic line bundle is called \emph{semi-positive/integrable} if $\varphi$ is semi-positive/integrable.
	\item[(4)] Let $\overline{L}=[(X_{i},\overline{L_{i}})_{i\in I}]$ be a pseudo-metrised line bundle on $X$ over $\mathbf{S}$. Then $-\overline{L}:= [(X_{i},-\overline{L_{i}})_{i\in I}]$ is a pseudo-metrised on $X$ over $\mathbf{S}$. Moreover, if $\overline{L}$ is a usc/lsc/continuous adelic line bundle on $X$ over $\mathbf{S}$, then $-\overline{L}$ is also a usc/lsc/continuous line bundle on $X$ over $\mathbf{S}$
	\item[(5)] Let $\overline{L^{(1)}}=(L^{(1)},\varphi^{(1)}),\overline{L^{(2)}}=(L^{(2)},\varphi^{(2)})$ be two pseudo-metrised line bundle on $X$ over $\mathbf{S}$. For $j=1,2$, write $\overline{L^{(j)}}=[(X^{(j)}_{i},\overline{L_{j}^{(j)}})_{i\in I}]$. By \L{}o\'s theorem, $X_{i}^{(1)}\cong X_{i}^{(2)}$ $\cU$-almost everywhere and we define $\overline{L^{(1)}}+\overline{L^{(2)}}=[(X^{(1)}_{i},\overline{L_{i}^{(1)}}+\overline{L_{i}^{(2)}})_{i\in I}]$ (after identification of the isomorphic $X^{(j)}_{i}$'s). This is a pseudo-metrised line bundle on $X$ over $\mathbf{S}$. Moreover, if both $\overline{L^{(1)}}$ and $\overline{L^{(2)}}$ are usc/lsc/continuous adelic line bundles on $X$ over $\mathbf{S}$, then so is $\overline{L^{(1)}}+\overline{L^{(2)}}$.
	\item[(6)] Let $f:Y\to X$ be a morphism between projective $K$-schemes. Let $\overline{L}=[(X_{i},\overline{L_{i}})_{i\in I}]$ be a pseudo-metrised line bundle on $X$ over $\mathbf{S}$. By \L{}o\'s theorem, we realise $Y_{\mathbf{S}}$ as an ultraproduct $\prod_{\cU}Y_{i}$, where the $Y_{i}$'s are projective schemes over the $K_{i}$'s. By \L{}o\'s theorem again, $f$ induces $\cU$-almost everywhere a morphism $f_{i}:Y_{i}\to X_{i}$. Define $f^{\ast}\overline{L}:=[(Y_{i},f^{\ast}\overline{L_{i}})_{i\in I}]$. This is a pseudo-metrised line bundle on $Y$ over $\mathbf{S}$. Moreover, if $\overline{L}$ is a usc/lsc/continuous adelic line bundle on $X$ over $\mathbf{S}$, then $f^{\ast}\overline{L}$ is a usc/lsc/continuous adelic line bundle on $Y$ over $\mathbf{S}$.
	\item[(7)] Let $P:\Spec(K')\to X$ be a closed point. Let $\overline{L}$ be a usc/lsc/continuous adelic line bundle on $X$ over $\mathbf{S}$. Then $P^{\overline{L}}$ is a usc/lsc/continuous adelic line bundle on the family of topological adelic curves $\mathbf{S}\otimes_{K}K'$ (cf. Proposition-Definition \ref{prop-def:algebraic_covering_family_tac}).
\end{itemize}
\end{proposition-definition}

\begin{proof}
All the introduced notions are well-defined. The assertions concerning the adelic line bundles are a consequence of Proposition \ref{prop:constructions_adelic_line_bundles} and Remark \ref{rem:CM_6.1.3_general}.
\end{proof}

\begin{example}
\label{example:adelic_line_bundle_Nevanlinna_family}
Consider the family of topological adelic curves $\mathbf{S}=(\bR_{>0},\cU,(S_{R})_{R>0},\cM(\bC))$ from Example \ref{example:families_of_tac} (2). Let $X$ be a reduced projective $\bC$-scheme and let $(L,\varphi)$ be an lsc/usc/continuous metrised line bundle on $X$. Denote by $L_{\cM(\bC)}$ the pullback of $L$ on $X_{\cM(\bC)}$. Using the construction of \S \ref{subsub:adelic_line_bundle_Nevanlinna}, for any $R>0$, we obtain an adelic line bundle $\overline{L_{R}}=(L_{R},\varphi_{R})$ on $X_{R}:=X \otimes_{\bC} K_{R}$. Then $\overline{L_{\cM_{\bC}}}:=[(X_{R},\overline{L_{R}})_{R>0})$ is an lsc/usc/continuous adelic line bundle on $X\otimes_{\bC}\cM_{\bC}$ over $\mathbf{S}$.

More generally, let $K/\cM(\bC)$ be an algebraic extension. Then the lsc/usc/continuous metrised line bundle $(L,\varphi)$ on $X$ induces for any $R>0$ a usc/lsc/continuous adelic line bundle on $X\otimes_{\bC}K$ over $\mathbf{S}\otimes_{\cM(\bC)}K$ (cf. Proposition-Definitions \ref{prop-def:algebraic_covering_family_tac} and \ref{prop-def:adelic_line_bundle_Nevanlinna_finite covering}).
\end{example}

\subsection{Volume functions on a proper topological adelic curve}
\label{sub:volume_functions_proper_tac}

In this subsection, we fix a proper topological adelic curve $S=(K,\phi:\Omega\to M_{K},\nu)$ and a geometrically integral projective $K$-scheme $\pi : X \to \Spec(K)$ of dimension $d$. 

\begin{definition}
\label{def:volume_function_proper_tac}
Let $\overline{L}=(L,\varphi)$ be an adelic line bundle on $X$. For any integer $n\geq 1$, let $\pi_{\ast}(n\overline{L}) := (H^{0}(X,nL),\pi_{\ast}(n\varphi))$, where $\pi_{\ast}(n\varphi)$ denotes the collection of supremum pseudo-norms on $H^{0}(X,nL)$ induced by $n\varphi$. Assume that, for any integer $n\geq 1$, $\pi_{\ast}(n\varphi)$ is well defined and is dominated and usc (e.g. if $\Omega=M_{K}$ by Theorem \ref{th:pushforward_pseudo-metric_family_dominated_general} and Proposition \ref{prop:existence_pushforward_pseudo-norm_family}), so that the Arakelov degree $\widehat{\deg}(\pi_{\ast}(n\overline{L}))$ makes sense. 

Then we define the $\chi$\emph{-volume}
\begin{align*}
\widehat{\vol}_{\chi}(\overline{L}) := \displaystyle\limsup_{n \to +\infty}\frac{\widehat{\deg}(\pi_{\ast}(n\overline{L}))}{n^{d+1}/(d+1)!},
\end{align*} 
and the \emph{arithmetic volume}
\begin{align*}
\widehat{\vol}(\overline{L}) := \displaystyle\limsup_{n \to +\infty}\frac{\widehat{\deg}_{+}(\pi_{\ast}(n\overline{L}))}{n^{d+1}/(d+1)!}.
\end{align*}
\end{definition}

\section{Heights of closed points}
\label{sec:global_heights}

In this section, we introduce global heights over topological adelic curves. We start by constructing height functions for closed points over a proper topological adelic curve (\S \ref{sub:height_of_closed_points_proper}). Then we give the family counterpart for asymptotically proper families of topological adelic curves (\S \ref{sub:height_of_closed_points_family}). 

\subsection{Heights of closed points on a proper topological adelic curve}
\label{sub:height_of_closed_points_proper}

In this subsection, we fix a proper topological adelic curve $S=(K,\phi:\Omega\to M_K,\nu)$ and a projective $K$-scheme $X$. 

\begin{definition}
\label{def:height_closed_points_proper}
 Let $\overline{L}=(L,\varphi)$ be an adelic line bundle on $X$. Let $P$ be a closed point of $X$. The pseudo-norm family $P^{\ast}\varphi$ is continuous and dominated (cf. Proposition \ref{prop:CM_6.1.18_general}). As $P^{\ast}L$ is a $\kappa(P)$-vector space of dimension $1$, Proposition \ref{prop:adelic_line_bundle_tac} implies that $P^{\ast}\overline{L}:=(P^{\ast}L,P^{\ast}\varphi)$ is an adelic line bundle on $S_P := S\otimes_{K} \kappa(P)$. As $S$ is proper, $S_P$ is proper as well (cf. Proposition \ref{prop:algebraic_extension_tac}). We define the \emph{height} of $P$ w.r.t. $\overline{L}$ as 
\begin{align*}
h_{\overline{L}}(P) := \widehat{\deg}_{S}(P^{\ast}\overline{L}).
\end{align*}
\end{definition}

\begin{theorem}
\label{th:height_closed_points_proper}
Let $X$ be a projective $K$-scheme and let $\overline{L_1}=(L_1,\varphi_1),\overline{L_2}=(L_2,\varphi_2)$ be adelic line bundles on $X$. Then the following assertions hold.
\begin{itemize}
	\item[(1)] For any closed point $P$ of $X$ we have
	\begin{align*}
	h_{\overline{L_1}+\overline{L_2}}(P) = h_{\overline{L_1}}(P) + h_{\overline{L_2}}(P).
	\end{align*}
	\item[(2)] Let $P$ be a closed point of $X$. Assume that $L_1=L_2$. Then we have 
	\begin{align*}
	\left|h_{\overline{L_1}}(P)-h_{\overline{L_2}}(P)\right| \leq \upint_{\Omega}d_{\omega}(\varphi_1,\varphi_2)\nu(\diff\omega) < +\infty.
	\end{align*}
\end{itemize}
\end{theorem}

\begin{proof}
\textbf{(1):} This follows from Proposition \ref{prop:CM_4.3.2} (3). 

\textbf{(2):} Let $P$ be a closed point of $X$. Write $S \otimes_{K} \kappa(P)=(\kappa(P),\phi_P : \Omega_{P} \to M_{\kappa(P)},\nu_P)$. By definition, for any $\omega\in \Omega$, for any $x \in \pi_{\kappa(P)/K}^{-1}(\omega)$, we have
\begin{align*}
d_{x}(P^{\ast}\varphi_1,P^{\ast}\varphi_2) \leq d_{\omega}(\varphi_1,\varphi_2)
\end{align*}
By definition of $h_{\overline{L_1}}(P),h_{\overline{L_2}}(P)$, we obtain
\begin{align*}
\left|h_{\overline{L_1}}(P)-h_{\overline{L_2}}(P)\right| \leq \upint_{\Omega_P}d_{x}(P^{\ast}\varphi_1,P^{\ast}\varphi_2)\nu_P(\diff x)\leq \upint_{\Omega}d_{\omega}(\varphi_1,\varphi_2)\nu(\diff\omega).
\end{align*}
The finiteness assertion follows from Proposition \ref{prop:dominated_pseudo-metric_families} (5).
\end{proof}

\begin{remark}
\label{rem:height_closed_points_proper_analogy_Nevanlinna}
Theorem \ref{th:height_closed_points_proper} should be seen as the analogue of Nevanlinna's first main theorem in Arakelov geometry (cf. Theorem \ref{th:Nevanlinna_first_main_theorem}).
\end{remark}

\subsection{Heights of closed points on asymptotically proper families of topological adelic curves}
\label{sub:height_of_closed_points_family}

In this subsection, we consider the following setting. Let $\mathbf{S}=(I,\cU,(S_{i}=(K_{i},\phi_{i}:\Omega_{i}\to M_{K_{i}},\nu_{i}))_{i\in I},K)$ be a family of topological adelic curves. Let $X$ be a projective $K$-scheme. Set $X_{\mathbf{S}}:=X\otimes_{K}\prod_{\cU}K_{i}$. We also fix an equivalence relation $\sim$ on $\prod_{\cU}\bR$ which is compatible with the additive group structure and assume that the family $\mathbf{S}$ is asymptotically proper w.r.t. $\sim$ (cf. Definition \ref{def:family_of_tac_proper}). 

\begin{definition}
\label{def:height_closed_points_family}
Let $\overline{L}$ be an adelic line bundle on $X$ over $\mathbf{S}$. Let $P: \Spec(\kappa(P))\to X$ be a closed point. By Proposition-Definition \ref{prop-def:pseudo-metric_family_tac} (7), $P^{\ast}\overline{L}$ is an adelic line bundle on $\mathbf{S}\otimes_{K}\kappa(P)$ and we define
\begin{align*}
h_{\overline{L}}(P) := \widehat{\deg}(P^{\ast}\overline{L}) \in \displaystyle\prod_{\cU}\bR/\sim
\end{align*}
and call it the \emph{height} of $P$ w.r.t. $\overline{L}$.
\end{definition}

\begin{theorem}
\label{th:height_of_closed_points_families}
Let $\overline{L^{(1)}}, \overline{L^{(2)}}$ adelic line bundles on $X$ over $\mathbf{S}$. Let $P$ be a closed point of $X$. Then the following assertions hold.
\begin{itemize}
	\item[(1)] We have the equality
	\begin{align*}
	h_{\overline{L^{(1)}}+\overline{L^{(2)}}}(P) = h_{\overline{L^{(1)}}}(P) + h_{\overline{L^{(2)}}}(P)
	\end{align*}
	in $\prod_{\cU}\bR/\sim$.
	\item[(2)] Assume that there exists a total ordering on $\prod_{\cU}\bR/\sim$ that is compatible with the equivalence relation $\sim$, i.e. the quotient map $\prod_{\cU}\bR \to \prod_{\cU}\bR/\sim$ is increasing. For $j=1,2$, write $\overline{L^{(j)}}=(L^{(j)},[(L^{(j)}_{i},\varphi^{(j)}_{i})_{i\in I}])$. Assume that, $L_{i}^{(1)}=L_{i}^{(2)}$ $\cU$-almost everywhere and that
	\begin{align*}
	d(\overline{L^{(1)}},\overline{L^{(2)}}) :=  \left[\left(\upint_{\Omega_i} d_{\omega}(\varphi^{(1)}_{i},\varphi^{(2)}_{i})\nu_i(\diff\omega)\right)_{i\in I}\right] \in \displaystyle\prod_{\cU}\bR
	\end{align*}
	satisfies 
	\begin{align*}
	d(\overline{L^{(1)}},\overline{L^{(2)}}) \sim 0.
	\end{align*}
	Then we have the equality 
	\begin{align*}
	h_{\overline{L^{(1)}}}(P) = h_{\overline{L^{(2)}}}(P)
	\end{align*}
	in $\prod_{\cU}\bR/\sim$.
\end{itemize}
\end{theorem}

\begin{proof}
\textbf{(1):} This is a consequence of the definition of the height and Proposition \ref{prop:CM_4.3.2} (3).

\textbf{(2):} Write $\mathbf{S}\otimes_{K}\kappa(P)= (I,\cU,(S_{i,P}=(K_{i,P},\phi_{i,P}:\Omega_{i,P}\to M_{K_{i},P},\nu_{i,P}))_{i\in I},\kappa(P))$, Let $\omega=[(\omega_{i})_{i\in I}]\in \prod_{\cU}\Omega_{i}$. Define
\begin{align*}
d_{\omega}(\overline{L^{(1)}},\overline{L^{(2)}}) := \left[\left(d_{\omega_{i}}(\varphi^{(1)}_{i},\varphi_{i}^{(2)})\right)_{i\in I}\right] \in  \prod_{\cU}\bR,
\end{align*}
this is well-defined. As in the proof of Theorem \ref{th:height_closed_points_proper} (2), for any $x\in \prod_{\cU}\Omega_{i,P}$ mapped to $\omega$, we have
\begin{align*}
d_{x}(P^{\ast}\overline{L^{(1)}},P^{\ast}\overline{L^{(2)}}) \leq d_{\omega}(\overline{L^{(1)}},\overline{L^{(2)}}).
\end{align*}
By definition of the height, we have
\begin{align*}
-d(\overline{L^{(1)}},\overline{L^{(2)}})\leq h_{\overline{L^{(1)}}}(P)-h_{\overline{L^{(2)}}}(P) \leq d(\overline{L^{(1)}},\overline{L^{(2)}}).
\end{align*}
Using the assumption on the equivalence relation $\sim$, we obtain the desired result.
\end{proof}

\begin{example}
\label{example:height_of_closed_points_Nevanlinna}
Consider the family of topological adelic curves $\mathbf{S}=(\bR_{>0},\cU,(S_{R})_{R>0},\cM(\bC))$ from Example \ref{example:families_of_tac} (2). Recall that $\mathbf{S}$ is asymptotically proper w.r.t. the equivalence relation $\sim_{\fin}$ as defined in \S \ref{sub:motivation_families}. Moreover, there exists a total ordering on $\prod_{\cU}\bR/\sim_{\fin}$ that is compatible with the usual ordering on $\prod_{\cU}\bR$.

Let $X_{0}$ be a reduced projective $\bC$-scheme and let $(L_{0},\varphi_{0})$ be a continuously metrised line bundle on $X_{0}$. Set $X:=X_{0}\otimes_{\bC}\cM_{\bC}$, $X_{R}:=X_{0}\otimes_{\bC}K_{R}$ and denote by $\overline{L}:=(L,\varphi):=(L_{0}\otimes_{\cO_{X_{0}}}\cO_{X},[(X_{R},L_{R},\varphi_{R})_{R>0}])$ the adelic line bundle on $X$ associated with $(L_{0},\varphi_{0})$ constructed in Example \ref{example:adelic_line_bundle_Nevanlinna_family}.

Let $P\in X$ be a closed point. It determines a family of closed points $(P_{R})_{R>0}$ of the $X_{R}$'s. From the point of view explained in \S \ref{sub:Nevanlinna_holomorphic_curves}, $P$ corresponds to a holomorphic curve $f : \bC \to X_{0}$ and for any $R>0$, $P_{R}$ corresponds to the restriction of $f$ to the closed disc of radius $R$ in $\bC$. Let $s_{0}$ be an regular meromorphic section of $L_{0}$ such that $f(\bC) \nsubseteq |\div(s_{0})|$. Then we have the equality
\begin{align*}
h_{\overline{L}}(P)=\left[\left(T_{f,(L_{0},\varphi_{0},s_{0})}(R)\right)_{R>0}\right]\in \displaystyle\prod_{\cU}\bR/\sim_{\fin},
\end{align*}
(cf. \S \ref{sub:Nevanlinna_holomorphic_curves} for the notation $T_{f,(L_{0},\varphi_{0},s_{0})}(R)$).

For any other continuous metric $\varphi_{0}'$ on $L_{0}$ determining another adelic line bundle $\overline{L'}$ on $X$ over $\mathbf{S}$, we have
\begin{align*}
0 \leq d(\overline{L},\overline{L'}) \leq \left[\left(\left|\displaystyle\max_{x\in X_{0}(\bC)} \log\frac{|s_{0}|_{\varphi_{0}'}(x)}{|s_{0}|_{\varphi_{0}}(x)}\right|\right)_{R>0}\right] \sim_{\fin} 0,
\end{align*}
where $s_{0}$ denotes an arbitrary regular meromorphic section of $L_{0}$ such that $f(\bC)\nsubseteq |\div(s_{0})|$. By compatibility of the orderings, we get $d(\overline{L},\overline{L'}) \sim_{\fin} 0$. Therefore, we can apply Theorem \ref{th:height_of_closed_points_families} and we see that it gives a generalisation of Theorem \ref{th:Nevanlinna_first_main_theorem} in our context. 

More generally, consider an arbitrary algebraic extension $K/\cM(\bC)$. Consider the adelic line bundle $(L_{K},\varphi_{K})$ on $X_{K}:=X\otimes_{\cM(\bC)}K$ over $\mathbf{S}_{K}:=\mathbf{S}\otimes_{\cM(\bC)}K$ given by Proposition-Definition \ref{prop-def:adelic_line_bundle_Nevanlinna_finite covering}. Then a closed point $P\in X_{K}$ corresponds to a family of holomorphic curves $f=(f_{K'}:A_{K'}\to X_{0})_{K'}$, where $K'$ runs over finite extensions of $\cM_{\bC}$ contained in $K$ and for any such extension $K'$, $A_{K'}$ is a finite covering of $\bC$ such that $\cM(A_{K'})\cong K'$, satisfying some compatibility condition (cf. \S \ref{subsub:covering_family_topological_adelic_curves_Nevanlinna}). Then Theorem \ref{th:Nevanlinna_first_main_theorem} yields (\cite{Gubler97}, Theorem 3.18) for $0$-cycles on $X_{K}$. 
\end{example}

\section{Arithmetic intersection theory and heights of cycles}
\label{sec:arithmetic_intersection_product}

In this final section, we introduce heights of cycles. As it is done classically in Arakelov geometry, this is done by taking intersection products of adelic line bundles. To define the arithmetic intersection product on a topological adelic curve, in a similar way to \cite{Gubler97}, we do it on a finitely generated (hence countable) subfield and using the arithmetic intersection on adelic curves introduced in \cite{ChenMori21}. We then use the fact that arithmetic intersection numbers are invariant w.r.t. coverings of topological adelic curves (\S \ref{sub:arithmetic_intersection_theory_tac}). Using the arithmetic intersection product, we obtain a notion of height of cycles over an asymptotically proper family of topological adelic curves (\S \ref{sub:arithmetic_intersection_theory_family_tac}). In our Nevanlinna theoretic example, we recover (\cite{Gubler97}, Theorem 3.18).

\subsection{Arithmetic intersection theory and heights of cycles on topological adelic curves}
\label{sub:arithmetic_intersection_theory_tac}

Throughout this subsection, we fix a topological adelic curve $S=(K,\phi:\Omega\to M_{K},\nu)$ and a projective $K$-scheme $X$ of dimension $d$. 

\begin{theorem-definition}
\label{th-def:arithmetic_intersection_product_tac_countable}
Assume that $K$ is countable and that for any $\omega\in\Omega$, $\phi(\omega)$ is an absolute value on $K$. Let $\overline{L^{(0)}}=(L^{(0)},\varphi^{(0)}),...,\overline{L^{(d)}}=(L^{(d)},\varphi^{(d)})$ be integrable adelic line bundles on $X$ over $S$. By Remark \ref{rem:adelic_line_bundles_adelic_curves}, the $\overline{L^{(i)}}$'s are integrable adelic line bundles on $X$ over the adelic curve $\tilde{S}$ determined by $S$. Let $s^{(0)},...,s^{(d)}$ be respectively regular meromorphic sections of $L^{(0)},...,L^{(d)}$ such that the Cartier divisors $\div(s^{(0)}),...,\div(s^{(d)})$ intersect properly on $X$.
\begin{itemize}
	\item[(1)] Using (\cite{ChenMori21}, Theorem 4.2.11), we define the \emph{arithmetic intersection number}
\begin{align*}
\left((\overline{L^{(0)}},s^{(0)})\cdots(\overline{L^{(d)}},s^{(d)})\right)_{S} := (\widehat{\div}(s^{(0)})\cdots\widehat{\div}(s^{(d)}))_{\tilde{S}},
\end{align*}
with the notation of \emph{loc. cit.}.
	\item[(2)] Assume that $S$ is proper. Then $\left((\overline{L^{(0)}},s^{(0)})\cdots(\overline{L^{(d)}},s^{(d)})\right)_{S}$ is independent of the choice of $s^{(0)},...,s^{(d)}$ and we simply denote it by $\left(\overline{L^{(0)}}\cdots\overline{L^{(d)}}\right)_{S}$. This arithmetic intersection number is also called the \emph{multi-height} of $X$ w.r.t. $\overline{L^{(0)}},...,\overline{L^{(d)}}$ and is denoted by $h_{\overline{L^{(0)}},...,\overline{L^{(d)}}}(X)$. If $\overline{L^{(0)}},...,\overline{L^{(d)}}$ are all equal to the same integrable adelic line bundle $\overline{L}$ on $X$, this multi-height is denoted by $h_{\overline{L}}(X)$ and is called the \emph{height} of $X$ w.r.t. $\overline{L}$.
	\item[(3)] Assume that $S$ is proper. Let $Z$ be a $l$-dimensional cycle on $X$. Define the \emph{multi-height} of $Z$ w.r.t. $\overline{L^{(0)}},...,\overline{L^{(l)}}$ by 
	\begin{align*}
	h_{\overline{L^{(0)}},...,\overline{L^{(l)}}}(Z) := \left(\overline{L^{(0)}}\cdots\overline{L^{(l)}}|Z\right)_{\tilde{S}},
	\end{align*}
	with the notation of (\cite{ChenMori21}, \S 4.4). Moreover, if $\overline{L^{(0)}},...,\overline{L^{(l)}}$ are all equal to the same integrable adelic line bundle $\overline{L}$ on $X$, this multi-height is denoted by $h_{\overline{L}}(Z)$ and is called the \emph{height} of $Z$ w.r.t. $\overline{L}$.
	\item[(4)] Assume that $S$ is proper. Then the arithmetic intersection product is a symmetric and multi-linear pairing on the group of integrable adelic line bundles on $X$. Moreover, for any projective $K$-morphism $f:Y\to X$ and $l$-dimensional cycle $Z$ on $Y$, we have the following \emph{projection formula}
	\begin{align*}
	h_{f^{\ast}\overline{L^{(0)}},...,f^{\ast}\overline{L^{(l)}}}(Z) = h_{\overline{L^{(0)}},...,\overline{L^{(l)}}}(f_{\ast}Z).
	\end{align*}
	\item[(5)] Assume that $S$ is proper, that $L^{(0)},...,L^{(d)}$ are semi-ample and that $\varphi^{(0)},...,\varphi^{(d)}$ are semi-positive. Let $\psi^{(0)},...,\psi^{(d)}$ be respectively semi-positive pseudo-metrics on $L^{(0)},...,L^{(d)}$. For any $i=0,...,d$, we set $\overline{M^{(i)}}:=(L^{(i)},\psi^{(i)})$. Then we have
	\begin{align*}
	\left|\left(\overline{L^{(0)}}\cdots\overline{L^{(d)}}\right)_{S} - \left(\overline{M^{(0)}}\cdots\overline{M^{(d)}}\right)_{S}\right| \leq \displaystyle\sum_{i=0}^{d} \upint_{\Omega}d_{\omega}(\varphi^{(i)},\psi^{(i)})\nu(\diff\omega)\left(L^{(0)}\cdots L^{(i-1)}\cdot L^{(i+1)}\cdots L^{(d)}\right).
	\end{align*}
\end{itemize}
\end{theorem-definition}

\begin{proof}
The independence of the choice of the regular meromorphic sections is (\cite{ChenMori21}, Proposition 4.4.2). The assertion concerning the symmetry and multi-linearity of the arithmetic intersection product is (\emph{loc. cit.}, Proposition 4.4.4 (1)), the projection formula is (\emph{loc. cit.}, Theorem 4.4.9) and the final assertion concerning the change of pseudo-metrics is obtained by integration of (\emph{loc. cit.}, Corollary 3.5.7).
\end{proof}

\begin{theorem-definition}
\label{th-def:arithmetic_intersection_product_tac}
Let $\overline{L^{(0)}}=(L^{(0)},\varphi^{(0)}),...,\overline{L^{(d)}}=(L^{(d)},\varphi^{(d)})$ be integrable adelic line bundles on $X$ over $S$ and $s^{(0)},...,s^{(d)}$ be respectively regular meromorphic sections of $L^{(0)},...,L^{(d)}$ such that the Cartier divisors $\div(s^{(0)}),...,\div(s^{(d)})$ intersect properly on $X$.
\begin{itemize}
	\item[(1)] Choose a finitely generated subfield $K_{0}\subset K$ such that $X,L^{(0)},...,L^{(d)},s^{(0)},...,s^{(d)}$ are defined over $K_{0}$, namely there exist a projective $K_{0}$-scheme $X_{0}$ and line bundles $L^{(0)}_{0},...,L^{(d)}_{0}$ respectively equipped with regular meromorphic sections $s_{0}^{(0)},...,s_{0}^{(d)}$ such that $X\cong X_{0}\otimes_{K_{0}}K$ and for any $i=0,...,d$, we have $\pi_{0}^{\ast}(L^{(i)}_{0},s_{0}^{(i)})\cong (L^{(i)},s^{(i)})$. Consider the topological adelic curve $S_{0}=(K_{0},\phi_{0}:\Omega_{0}\to M_{K_{0}},\nu_{0})$ constructed in \S \ref{subsub:examples_tac_adelic_curves} (4). Recall that $\Omega_{0}\subset\Omega$ is a Borel subset such that $\nu(\Omega\setminus\Omega_{0})=0$ and $\nu_{0}$ is the restriction of $\nu$ to $\Omega_{0}$. The following assertions hold.
	\begin{itemize}
		\item[(i)] The Cartier divisors $\div(s_{0}^{(0)}),...,\div(s_{0}^{(d)})$ intersect properly on $X_{0}$.
		\item[(ii)] For any $i=0,...,d$, the pseudo-metric $\varphi^{(i)}$ induces a metric family $\varphi^{(i)}_{0}$ on $L_{0}^{(i)}$ such that the pullback pseudo-metric $\pi_{0}^{\ast}\varphi^{(i)}_{0}$ obtained by pointwise extension of scalars identifies with $\varphi^{(i)}$.
	\end{itemize}
	We then define the \emph{arithmetic intersection number}
	\begin{align*}
	\left((\overline{L^{(0)}},s^{(0)})\cdots(\overline{L^{(d)}},s^{(d)})\right)_{S} := \left((\overline{L_{0}^{(0)}},s_{0}^{(0)})\cdots(\overline{L_{0}^{(d)}},s_{0}^{(d)})\right)_{S_{0}},
	\end{align*}
	with the notation of Proposition-Definition \ref{th-def:arithmetic_intersection_product_tac_countable}. This does not depend on the choice of $K_{0},X_{0},L_{0}^{(0)},...,L_{l}^{(l)},s_{0}^{(0)},...,s_{0}^{(d)}$.
	\item[(2)] Assume that $S$ is proper. Then $\left((\overline{L^{(0)}},s^{(0)})\cdots(\overline{L^{(d)}},s^{(d)})\right)_{S}$ is independent of the choice of $s^{(0)},...,s^{(d)}$ and we simply denote it by $\left(\overline{L^{(0)}}\cdots\overline{L^{(d)}}\right)_{S}$. This arithmetic intersection number is also called the \emph{multi-height} of $X$ w.r.t. $\overline{L^{(0)}},...,\overline{L^{(d)}}$ and is denoted by $h_{\overline{L^{(0)}},...,\overline{L^{(d)}}}(X)$. If $\overline{L^{(0)}},...,\overline{L^{(d)}}$ are all equal to the same integrable adelic line bundle $\overline{L}$ on $X$, this multi-height is denoted by $h_{\overline{L}}(X)$ and is called the \emph{height} of $X$ w.r.t. $\overline{L}$.
	\item[(3)] Assume that $S$ is proper. Use the same notation as in (1). Let $Z$ be a $l$-dimensional cycle on $X$ assume that it is the base change of an $l$-cycle $Z_{0}$ on $X_{0}$. Define the \emph{multi-height} of $Z$ w.r.t. $\overline{L^{(0)}},...,\overline{L^{(l)}}$ by 
	\begin{align*}
	h_{\overline{L^{(0)}},...,\overline{L^{(l)}}}(Z) := h_{\overline{L_{0}^{(0)}},...,\overline{L_{l}^{(l)}}}(Z_{0}),
	\end{align*}
	with the notation of Theorem-Definition \ref{th-def:arithmetic_intersection_product_tac_countable} (3). This multi-height does not depend on the choice of $K_{0},X_{0},L_{0}^{(0)},...,L_{l}^{(l)}$ and $Z_{0}$. Moreover, if $\overline{L^{(0)}},...,\overline{L^{(l)}}$ are all equal to the same integrable adelic line bundle $\overline{L}$ on $X$, this multi-height is denoted by $h_{\overline{L}}(Z)$ and is called the \emph{height} of $Z$ w.r.t. $\overline{L}$.
	\item[(4)] Assume that $S$ is proper. Then the arithmetic intersection product is a symmetric and multi-linear pairing on the group of integrable adelic line bundles on $X$. Moreover, for any projective $K$-morphism $f:Y\to X$ and $l$-dimensional cycle $Z$ on $Y$, we have the following \emph{projection formula}
	\begin{align*}
	h_{f^{\ast}\overline{L^{(0)}},...,f^{\ast}\overline{L^{(l)}}}(Z) = h_{\overline{L^{(0)}},...,\overline{L^{(l)}}}(f_{\ast}Z).
	\end{align*}
	\item[(5)] Assume that $S$ is proper, that $L^{(0)},...,L^{(d)}$ are semi-ample and that $\varphi^{(0)},...,\varphi^{(d)}$ are semi-positive. Let $\psi^{(0)},...,\psi^{(d)}$ be respectively semi-positive pseudo-metrics on $L^{(0)},...,L^{(d)}$. For any $i=0,...,d$, we set $\overline{M^{(i)}}:=(L^{(i)},\psi^{(i)})$. Then we have
	\begin{align*}
	\left|\left(\overline{L^{(0)}}\cdots\overline{L^{(d)}}\right)_{S} - \left(\overline{M^{(0)}}\cdots\overline{M^{(d)}}\right)_{S}\right| \leq \displaystyle\sum_{i=0}^{d} \upint_{\Omega}d_{\omega}(\varphi^{(i)},\psi^{(i)})\nu(\diff\omega)\left(L^{(0)}\cdots L^{(i-1)}\cdot L^{(i+1)}\cdots L^{(d)}\right).
	\end{align*}
\end{itemize} 
\end{theorem-definition}

\begin{proof}
Note that it suffices to prove (1) since (2), (3), (4) and (5) follow from (1) and respectively Theorem-Definition \ref{th-def:arithmetic_intersection_product_tac_countable} (2), (3), (4) and (5).  

\textbf{(1):} (i) follows from (\cite{ChenMori21}, Remark 1.3.5) combined with the fact that the Cartier divisors $\div(s^{(0)}),...,\div(s^{(d)})$ intersect properly on $X$. Let $\mathbf{x}\in\ZR(X)^{\an}_{S}$ with underlying scheme point $p\in X$. Denote by $\mathbf{x_{0}}$ its image in $\ZR(X_{0})_{S_{0}}^{\an}$ with underlying scheme point $p_{0}\in X_{0}$. For any $i=0,...,d$, we have an inclusion $L^{(i)}_{0}(p_{0})\hookrightarrow L^{(i)}(p)$ and an isomorphism $L^{(i)}(p) \cong L^{(i)}_{0}(p_{0})\otimes_{\kappa(p_{0})}\kappa(p)$. Then (ii) follows easily since the vector spaces at stake are one-dimensional. 

Let us now prove that the arithmetic intersection number $\left((\overline{L_{0}^{(0)}},s_{0}^{(0)})\cdots(\overline{L_{0}^{(d)}},s_{0}^{(d)})\right)_{S_{0}}$ does not depend on the choice of $K_{0},X_{0},L_{0}^{(0)},...,L_{l}^{(l)},s_{0}^{(0)},...,s_{0}^{(d)}$. Let $K_{1}/K_{0}$ be a field extension, where $K_{1}$ is a finitely generated subfield of $K$. Denote by $\pi_{1}:X_{10}:=X_{0}\otimes_{K_{0}}K_{1} \to X_{0}$ the projection and, for any $i=0,...,d$, $(L_{1}^{(i)},s_{1}^{(i)}):= \pi_{10}^{\ast}(L_{0}^{(i)},s_{0}^{(i)})$. By (\cite{ChenMori21}, Remark 1.3.5), the Cartier divisors $\div(s_{1}^{(0)}),...,\div(s_{1}^{(d)})$ intersect properly on $X_{1}$. Moreover, the construction \S \ref{subsub:examples_tac_adelic_curves} (4) yields a topological adelic curve $S_{1}=(K_{1},\phi_{1}:\Omega_{1}\to M_{K_{1}},\nu_{1})$, where $\Omega_{1}$ is a Borel subset of $\Omega$ contained in $\Omega_{0}$ such that $\nu(\Omega\setminus\Omega_{1})=0$ and for any $\omega\in\Omega_{1}$, $\phi_{1}$ is an absolute value on $K_{1}$. Thus, we get a morphism $S_{1}\to (K_{0},\phi_{0|\Omega_{1}},\Omega_{1}\to M_{K_{0}},\nu_{1})$ of topological adelic curves, inducing a morphism between the corresponding adelic curves. By (\cite{ChenMori21}, Theorem 4.3.6), we have equalities
\begin{align}
\label{eq:invariance_intersection_number_covering}
\left((\overline{L_{0}^{(0)}},s_{0}^{(0)})\cdots(\overline{L_{0}^{(d)}},s_{0}^{(d)})\right)_{S_{0}} &= \left((\overline{L_{0}^{(0)}},s_{0}^{(0)})\cdots(\overline{L_{0}^{(d)}},s_{0}^{(d)})\right)_{(K_{0},\phi_{0|\Omega_{1}},\Omega_{1}\to M_{K_{0}},\nu_{1})} \nonumber\\
&= \left((\overline{L_{1}^{(0)}},s_{1}^{(0)})\cdots(\overline{L_{1}^{(d)}},s_{1}^{(d)})\right)_{S_{1}}.
\end{align}
Now consider any other finitely generated subfield $K_{1}\subset K$ such that $X,L^{(0)},...,L^{(d)},s^{(0)},...,s^{(d)}$ are defined over $K_{1}$ by a projective $K_{1}$-scheme $X_{1}$, line bundles $L_{1}^{(0)},...,L_{1}^{(d)}$ respectively equipped with regular meromorphic sections $s_{1}^{(0)},...,s_{1}^{(d)}$. By the construction from \S \ref{subsub:examples_tac_adelic_curves} (4), we get a topological adelic curve $S_{1}$. Using $(ii)$, we define metrics $\varphi_{1}^{(0)},...,\varphi_{1}^{(d)}$ on $L_{1}^{(0)},...,L_{1}^{(d)}$ such that their pullback to $X$ yields respectively $\varphi^{(0)},...,\varphi^{(d)}$. Let $K_{2}$ denote the composition of $K_{0}$ and $K_{1}$, this is a finitely generated subfield of $K$ containing both $K_{0}$ and $K_{1}$. Moreover, if $X_{2}$ denote the scheme theoretic image of $X\to (X_{1}\otimes_{K_{1}}K_{2})\times_{\Spec(K_{2}}(X_{0}\otimes_{K_{0}}K_{2})$, we have $X_{2}\otimes_{K_{2}}K\cong X$ and by faithfully flat descent (\cite{Gortz10}, Proposition 14.51), that $X_{1}\otimes_{K_{1}}K_{2}$ and $X_{0}\otimes_{K_{0}}K_{2}$ are isomorphic. Thus, we may assume that $K_{0}=K_{1}$ and $X_{0}=X_{1}$. Since, for any $i=0,...,d$, the pullbacks of $L_{0}^{(i)}$ and $L_{1}^{(i)}$ to $X$ are isomorphic, (\cite{Fujiwara-Kato}, Chapter 0 Theorem 4.2.1) ensures that there exists a finitely generated subfield $K'\subset K$ containing $K_{0}$ such that, for any $i=0,...,d$, the pullbacks of $L_{0}^{(i)},L_{1}^{(i)}$ and $s_{0}^{(i)},s_{1}^{(i)}$ to $X_{0}\otimes_{K_{0}}K'$ are both isomorphic. We can thus conclude using (\ref{eq:invariance_intersection_number_covering}).
\end{proof}

\begin{remark}
\label{rem:continuity_of_the_arithmetic_intersection_product_GVF_topology}
The definition of the arithmetic intersection product over a proper topological adelic curve allows to use the results of \cite{DHS24} and it is possible to show that the arithmetic intersection is definable w.r.t. the GVF topology (\emph{loc. cit.}, Theorem 1.4).
\end{remark}

\subsection{Arithmetic intersection theory and heights of cycles on families of topological adelic curves}
\label{sub:arithmetic_intersection_theory_family_tac}

In this subsection, we consider the following setting. Let $\mathbf{S}=(I,\cU,(S_{i}=(K_{i},\phi_{i}:\Omega_{i}\to M_{K_{i}},\nu_{i}))_{i\in I},K)$ be a family of topological adelic curves. Let $X$ be a projective $K$-scheme of dimension $d$. Set $X_{\mathbf{S}}:=X\otimes_{K}\prod_{\cU}K_{i}$. 

\begin{theorem-definition}
\label{th-def:arithmetic_intersection_product_family_tac}
Let $\overline{L^{(0)}}=(L^{(0)},[(\overline{L_{i}^{(0)}}=(L_{i}^{(0)},\varphi_{i}^{(0)}))_{i\in I}]),...,\overline{L^{(d)}}=(L^{(d)},[(\overline{L_{i}^{(d)}}=(L_{i}^{(d)},\varphi_{i}^{(d)}))_{i\in I})])$ be integrable adelic line bundles on $X$ over $\mathbf{S}$ and $s^{(0)},...,s^{(d)}$ be respectively regular meromorphic sections of $L^{(0)},...,L^{(d)}$ such that the Cartier divisors $\div(s^{(0)}),...,\div(s^{(d)})$ intersect properly on $X$. 
\begin{itemize}
	\item[(1)] For any $j=0,...,d$, pulling back $s^{(j)}$ to $X_{\mathbf{S}}$, yielding a regular meromorphic section $s^{(j)}_{\mathbf{S}}$, and using \L{}o\'s theorem, we can write $s^{(j)}_{\mathbf{S}}=[(s_{i}^{(j)})_{i\in I}]$, where the $s_{i}^{(j)}$'s are regular meromorphic sections of the $L^{(j)}_{i}$'s $\cU$-almost everywhere. By \L{}o\'s theorem, the Cartier divisors $\div(s_{i}^{(0)}),...,\div(s_{j}^{(d)})$ intersect properly on $X_{i}$ $\cU$-almost everywhere. We define the \emph{arithmetic intersection number} 
	\begin{align*}
	\left((\overline{L^{(0)}},s^{(0)})\cdots(\overline{L^{(d)}},s^{(d)})\right)_{S} := \left[\left(\left((\overline{L_{0}^{(0)}},s_{i}^{(0)})\cdots(\overline{L_{0}^{(d)}},s_{i}^{(d)})\right)_{S_{i}}\right)_{i\in I}\right]\in\displaystyle\prod_{\cU}\bR.
	\end{align*}
	This quantity is independent of the choice of the $L_{i}^{(j)}$'s and $s_{i}^{(j)}$'s.
	\item[(2)] We fix an equivalence relation $\sim$ on $\prod_{\cU}\bR$ which is compatible with the additive group structure and assume that the family $\mathbf{S}$ is asymptotically proper w.r.t. $\sim$. Then the class of $\left((\overline{L^{(0)}},s^{(0)})\cdots(\overline{L^{(d)}},s^{(d)})\right)_{S}$ in $\prod_{\cU}\bR/\sim$ is independent of the choice of $s^{(0)},...,s^{(d)}$ and we denote it by $\left(\overline{L^{(0)}}\cdots\overline{L^{(d)}}\right)_{S}$. This equivalence class is independent of the choice of the $X_{i}$'s and is called the \emph{multi-height} of $X$ w.r.t. $\overline{L^{(0)}},...,\overline{L^{(d)}}$ and we denote it by $h_{\overline{L^{(0)}},...,\overline{L^{(d)}}}(X)$. If $\overline{L^{(0)}},...,\overline{L^{(d)}}$ are all equal to the same integrable adelic line bundle $\overline{L}$ on $X$, this multi-height is denoted by $h_{\overline{L}}(X)$ and is called the \emph{height} of $X$ w.r.t. $\overline{L}$.
	\item[(3)] We keep the same assumption as in (2). Similarly to Theorem-Definition \ref{th-def:arithmetic_intersection_product_tac}, we extend by linearity the definition of multi-height for cycles on $X$. Let $Z$ be a $l$-dimensional cycle on $X$, we denote the \emph{multi-height} of $Z$ w.r.t. $\overline{L^{(0)}},...,\overline{L^{(l)}}$ by $h_{\overline{L^{(0)}},...,\overline{L^{(l)}}}(Z)\in\prod_{\cU}\bR/\sim$. Moreover, if $\overline{L^{(0)}},...,\overline{L^{(l)}}$ are all equal to the same integrable adelic line bundle $\overline{L}$ on $X$, this multi-height is denoted by $h_{\overline{L}}(Z)$ and is called the \emph{height} of $Z$ w.r.t. $\overline{L}$.
	\item[(4)] We keep the same assumption as in (2). Then the arithmetic intersection product is a symmetric and multi-linear pairing on the group of integrable adelic line bundles on $X$. Moreover, for any projective $K$-morphism $f:Y\to X$ and $l$-dimensional cycle $Z$ on $Y$, we have the following \emph{projection formula}
	\begin{align*}
	h_{f^{\ast}\overline{L^{(0)}},...,f^{\ast}\overline{L^{(l)}}}(Z) = h_{\overline{L^{(0)}},...,\overline{L^{(l)}}}(f_{\ast}Z).
	\end{align*}
	\item[(5)] We keep the same assumption as in (2). Assume that $L^{(0)},...,L^{(d)}$ are semi-ample and that the adelic line bundles $\overline{L^{(0)}},...,\overline{L^{(d)}}$ are semi-positive. Let $\psi^{(0)},...,\psi^{(d)}$ be respectively pseudo-metrics on $L^{(0)},...,L^{(d)}$ over $\mathbf{S}$. For $j=0,...,d$, we set $\overline{M^{(j)}}:=(L^{(j)},\psi^{(j)})$ and assume that 
	\begin{align*}
	d(\overline{L^{(j)}},\overline{M^{(j)}}) \sim 0.
	\end{align*}
	Then we have the equality
	\begin{align*}
	h_{\overline{L^{(0)}},...,\overline{L^{(d)}}}(X) = h_{\overline{M^{(0)}},...,\overline{M^{(d)}}}(X)
	\end{align*}
	in $\prod_{\cU}\bR/\sim$.
\end{itemize}
\end{theorem-definition}

\begin{proof}
The fact that the arithmetic intersection number is well-defined follows from \L{}o\'s theorem. (2) and (3) follows follow from a direct computation using the asymptotic properness of $\mathbf{S}$ w.r.t. $\sim$. (4) follows by integration of (\cite{ChenMori}, Proposition 3.5.4) and the techniques of the proof of (\cite{ChenMori}, Theorem 4.4.9). Finally, (5) is proven the same way as Theorem \ref{th:height_of_closed_points_families} (2).
\end{proof}

\begin{example}
\label{example:arithmetic_intersection_product_Nevanlinna}
Consider the family of topological adelic curves $\mathbf{S}=(\bR_{>0},\cU,(S_{R})_{R>0},\cM(\bC))$ from Example \ref{example:families_of_tac} (2). Recall that $\mathbf{S}$ is asymptotically proper w.r.t. the equivalence relation $\sim_{\fin}$ as defined in \S \ref{sub:motivation_families}. Moreover, there exists a total ordering on $\prod_{\cU}\bR/\sim_{\fin}$ that is compatible with the usual ordering on $\prod_{\cU}\bR$.

Let $X_{0}$ be a reduced projective $\bC$-scheme of dimension $d$ and let $(L^{(0)}_{0},\varphi^{(0)}_{0}),...,(L^{(0)}_{0},\varphi^{(0)}_{0})$ be integrable continuously metrised line bundles on $X_{0}$. Set $X:=X_{0}\otimes_{\bC}\cM_{\bC}$, $X_{R}:=X_{0}\otimes_{\bC}K_{R}$ and for any $j=0,...,d$, denote by $\overline{L^{(j)}}:=(L^{(j)},\varphi^{(j)}):=(L^{(j)}_{0}\otimes_{\cO_{X_{0}}}\cO_{X},[(X_{R},L^{(j)}_{R},\varphi^{(j)}_{R})_{R>0}])$ the integrable adelic line bundle on $X$ associated with $(L_{0},\varphi_{0})$ constructed in Example \ref{example:adelic_line_bundle_Nevanlinna_family}. For any $l$-cycle $Z$ on $X$, we have the multi-height $h_{\overline{L^{(0)}},...,\overline{L^{(t)}}}(Z)\in\prod_{\cU}\bR/\sim$. Then Theorem-Definition \ref{th-def:arithmetic_intersection_product_family_tac} shows that this multi-height does not depend on the choice of the metrics $\varphi^{(0)}_{0},...,\varphi^{(d)}_{0}$.

More generally, let algebraic extension $K/\cM(\bC)$. Consider the integrable adelic line bundles $\overline{L^{(0)}_{K}},...,\overline{L^{(d)}_{K}}$ on $X_{K}:=X\otimes_{\cM(\bC)}K$ over $\mathbf{S}\otimes_{\cM(\bC)}K$ respectively induced by $(L^{(0)}_{0},\varphi^{(0)}_{0}),...,(L^{(0)}_{0},\varphi^{(0)}_{0})$. Theorem-Definition \ref{th-def:arithmetic_intersection_product_family_tac} shows that, for any $l$-cycle $Z$ on $X_{K}$, the multi-height $h_{\overline{L_{K}^{(0)}},...,\overline{L_{K}^{(t)}}}(Z)$ does not depend on the choice of the metrics $\varphi^{(0)}_{0},...,\varphi^{(d)}_{0}$. Thus, our result gives a generalisation of (\cite{Gubler97}, Theorem 3.18).
\end{example}

\appendix
\part*{Appendix}
\section{Nevanlinna theory of complex functions}
\label{sec:Nevanlinna_classical}

\subsection{Classical Nevanlinna theory}

\subsubsection{The Nevnalinna height function}

We denote $K=\cM(\bC)$, namely the field of meromorphic functions on $\bC$, which is the fraction field of the ring of entire functions $A:=\cO(\bC)$. Let $E = \sum a_i [z_i]$ be a divisor on $\bC$, with the $z_i$ pairwise distinct. For any $k\in \bN \cup\{\infty\}$, we define the \emph{truncated counting functions} 
\begin{align*}
\forall t >0,\quad n_k(t,E) := \displaystyle\sum_{|z_i|_{\infty} < t} \min\{k,a_i\},
\end{align*} 
and
\begin{align*}
\forall r>1,\quad N_k(r,E) := \int_{1}^{r} \frac{n_k(t,E)}{t} dt.
\end{align*}
For ease of notation, we let $n(t,E) := n_{\infty}(t,E)$ and $N(r,E) := N_{\infty}(r,E)$.

Let $f\in K$, for any $a\in \bP^1(\bC)$, we denote by $(f)_a$ the divisor of associated to $(f-a)$, if $a\neq \infty$, and to $1/f$ if $a=\infty$. Then the \emph{proximity function} in $\infty$ of $f$ is defined by
\begin{align*}
\forall r>0,\quad m(r,f) := \frac{1}{2\pi}\int_{0}^{2\pi} \log^{+}|f(re^{i\theta})|_{\infty}d\theta.
\end{align*}
The function $(r\in\bR_{>0})\mapsto m(r,1/(f-a))$ is called the proximity function of $f$ in $a\in\bC$. Finally, the \emph{height} of $f$ (with respect to $\infty$) is defined by
\begin{align*}
\forall r>1,\quad T(r,f) := m(r,f) + N(r,(f)_{\infty}).
\end{align*}

Morally, the proximity function of $f$ measures the mean approximation of $f$ to $\infty$ on a circle of given radius. The counting function counts how many times $f$ attains $\infty$, i.e. has a pole, in an open disc of given radius. The philosophy of Nevanllina is that these two functions contain all the necessary information concerning the behaviour of $f$ with respect to $\infty$. The characteristic function, namely the sum of the two previous ones, behaves as a height function in Diophantine geometry:
\begin{align*}
\forall f_1,f_2\in K,\quad T(r,f_1+f_2) \leq T(r,f_1) + T(r,f_2) + \log 2.
\end{align*}

\begin{theorem}[Nevanlinna's first main theorem, \cite{Noguchi14}, Theorem 1.1.17]
Let $f\in K$ and $a\in \bC$. Then
\begin{align*}
T(r,\frac{1}{f-a}) = T(r,f) + O(1),
\end{align*}
where the bound $O(1)$ is a bounded function of $r$, with bound depending only on $f$ and $a$.
\end{theorem}

\begin{proposition}[Cartan's formula, \cite{BombieriGubler}, Proposition 13.2.13]
\label{prop:Cartan_formula}
Let $f\in K$. Let $C:=\log^{+}|f(0)|_{\infty}$ if $f(0)\neq \infty$ and $C:=\log|c(f,0)|_{\infty}$ otherwise. Then 
\begin{align*}
\forall r>0, \quad T(r,f) = \frac{1}{2\pi}\int_{0}^{2\pi}N\left(r, \frac{1}{f-e^{i\theta}}\right)\diff\theta + C.
\end{align*}
In particular, the function $T(\cdot,f)$ is an increasing function that is convex in $\log r$.
\end{proposition}

In fact, other height functions may be defined. For our purposes, it is more convenient to work in a more geometric framework, namely, we will consider any $f\in K$ as a holomorphic curve $f:\bC \to \bP^{1}(\bC)$. Let $L$ be a holomorphic line bundle on $P^1(\bC)$, equipped with a hermitian metric $\varphi$ and an invertible meromorphic section $s$. Denote $D:=\div(s)$. We define
\begin{align*}
&N_f(r,D) := \ord(f^\ast D,0)\log r + \displaystyle\sum_{0<|z|_{\infty}<r} \ord(f\ast D,z)\log \left|\frac{r}{z}\right|_{\infty},\\
&m_f(r,D) := -\frac{1}{2\pi} \int_0^{2\pi} \log |s(f(re^{i\theta}))|_{\varphi} d\theta,\\
&T_f(r,D) := N_f(r,D) + m_f(r,D).
\end{align*} 
The previous constructions are now the special case $L=\cO(1)$, $\varphi$ the standard metric, and $D=[\infty]$. Then Nevanlinna's first main theorem is rephrased as follows. Consider $L$ and $f$ fixed. Then the height function $T_f(r,D)$ does not depend, up to a bounded function of $r$, neither on the choices of the metric $\varphi$ nor on the meromorphic section $s$.

\subsubsection{Asymptotic of the height function}
\label{subsub:asymptotic_height_Nevanlinna}

The Northcott property has a counterpart in Nevanlinna theory: this is the Liouville theorem, which states that an entire function is constant iff $T(r,f)$ is a bounded function of $r$. The following result makes the analogy more precise.

\begin{proposition}[\cite{BombieriGubler}, Proposition 13.2.17 and Example 13.2.18]
\label{prop:criterion_rational_functions}
Let $f\in K$. Then $f$ is constant iff $T(\cdot,f)$ is bounded. Moreover, $f$ is a rational function iff
\begin{align*}
\displaystyle\liminf_{r\to+\infty} \frac{T(r,f)}{\log r} < +\infty,
\end{align*}
in that case, $T(r,f)=\deg(f)\log(r) + O(1)$.
\end{proposition}

The above proposition indicates that in order to study the behaviour of transcendental meromorphic functions, one has to compare the (unbounded) height to functions that grow "faster than $\log$". 

Let $f\in K$. Define the \emph{order} of $f$ as
\begin{align*}
\rho(f):= \limsup_{r\to \infty} \frac{\log T(r,f)}{\log r} \in [0,+\infty].
\end{align*}
Nevanlinna's first theorem shows that this quantity does not depend on the choice of height function.

\begin{lemma}
The subset of $K$ consisting of meromorphic functions of finite order is a subfield of $K$.
\end{lemma}

\begin{proof}
This is a consequence of (\cite{Nevanlinna70}, p.216)
\begin{align*}
&\forall f,g\in K, \quad \rho(f+g) \leq \rho(f)+\rho(g) \text{ and }\rho(fg) \leq \max\{\rho(f),\rho(g)\},\\
&\forall f\in K^{\times}, \quad \rho\left(\frac{1}{f}\right) = \rho(f).
\end{align*}
\end{proof}

More generally (cf. e.g. \cite{Heittokangas2021}), for any non-decreasing unbounded function $\eta: ]R_{0},+\infty\to \bR_{>0}$, for some $R_{0}\geq 0$, we define the $\eta$\emph{-order} $\rho_{\eta}(f)\in [0,+\infty]$ of a meromorphic function $f\in K$ as 
\begin{align*}
\rho_{\eta}(f):= \limsup_{r\to \infty} \frac{\log T(r,f)}{\eta(r)}.
\end{align*}
As in the case of the classical order $\rho_{\log}$ above, the subset of $K$ consisting of functions of finite $\eta$-order is a field, called the field of \emph{finite} $\eta$\emph{-order} functions. If $\eta = \log\circ\log$, Proposition \ref{prop:criterion_rational_functions} implies that the field of finite $\eta$-order functions is the field $\bC(T)$ of rational functions. 

\subsubsection{Second main theorem}

As it was previously mentioned, Nevanlinna's first theorem yields an upper bound for counting functions of meromorphic functions. Nevanlinna's second theorem gives a result in the other direction: namely, it yields a lower bound. To state it, we need to consider multiple proximity functions.

\begin{theorem}[Nevanlinna's second theorem, \cite{Noguchi14}, Theorem 1.2.5]
\label{th:Nevanlinna_second_main_theorem}
Let $f\in K$. Let $q\geq 1$ be an integer and let $a_1,...,a_q\in\bP^1(\bC)$. Then
\begin{align*}
\displaystyle\sum_{i=1}^q m_f(r,a_i) \leq_{\text{exc}} 2T_{f}(r) + O(\log^+ T_f(r)) + o(\log r),
\end{align*}
or, equivalently,
\begin{align*}
\displaystyle\sum_{i=1}^q N_f(r,(f)_{a_i}) \geq_{\text{exc}} (q-2)T_f(r) - O(\log^+ T_f(r)) - o(\log r),
\end{align*}
where $\leq_{\text{exc}}$ means that the inequalities hold for all $r>0$ except on a set of finite Lebesgue measure.
\end{theorem}

Roughly speaking, Nevanlinna's second main theorem says that a meromorphic function on $\bC$ cannot avoid "too many points", in that case, at most $2$. 

It is possible to measure to what extent the counting function at a point is significantly smaller than the height via the so-called defect. More precisely, let $f\in K$, and let $a\in \bP^1(\bC)$. Then the \emph{defect} of $f$ at $a$ is defined by
\begin{align*}
\delta_f(a) := \displaystyle\liminf_{r\to\infty} \frac{m_f(r,a)}{T_f(r)} = 1 - \limsup_{r\to\infty}\frac{N_f(r,(f)_a)}{T_f(r)},
\end{align*}
where the second equality comes from Nevanlinna's first theorem. The same theorem implies that we have the inequality $0\leq\delta_f(a)\leq 1$ and Nevanlinna's second theorem yields
\begin{align*}
\displaystyle\sum_{a\in\bP^1(\bC)} \delta_f(a) \leq 2.
\end{align*}
In the $\delta_f(a) > 0$ case, we say that $a$ is a \emph{deficient} value for $f$. With this notion at hand, Nevanlinna's second theorem implies that a given meromorphic function cannot have too many deficient values. 

\subsection{Holomorphic curves on a projective variety}
\label{sub:Nevanlinna_holomorphic_curves}

In this subsection, we fix a complex projective variety $X$ and we study holomorphic maps $f : \bC \to X$. The latter are called \emph{holomorphic curves} in $X$. Since $X$ is projective, $f$ induces a holomorphic map $f : \bC \to \bP^n(\bC)$ and Weierstrass factorisation theorem (\cite{Ahlfors66}, Chapter 5, Theorem 8) implies that there exist entire functions $f_0,...,f_n$ without common zeroes such that $f(x) = [f_0(x):\cdots:f_n(x)]$ for all $x\in \bC$. This way we can see $f : \bC \to X$ as a $K$-point $f\in X(K)$, where $K = \cM(\bC)$. In this context, we can extend the definitions of counting, proximity and height functions.

Let $D$ be a divisor on $X$. Assume that $f(\bC) \not\subset |D|$, so that $f^\ast D$ is a divisor on $\bC$. Define the \emph{counting function}
\begin{align*}
N_f(r,D) := \ord(f^\ast D,0)\log r + \displaystyle\sum_{0<|z|_{\infty}<r} \ord(f^\ast D,z)\log\frac{r}{|z|_{\infty}},
\end{align*}
for all $r>0$. Denote $L:= \cO_X(D)$ and let $s$ be a regular meromorphic section of de $L$ such that $D = \div(s)$. Let $\varphi$ be a continuous Hermitian metric on $L$. Define the \emph{proximity function}
\begin{align*}
m_f(r,(L,\varphi,s)) := -\frac{1}{2\pi}\int_{0}^{2\pi} \log |s|_{\varphi}(f(re^{i\theta})) d\theta,
\end{align*}
for all $r>0$. Finally, the \emph{height function} is defined by
\begin{align*}
T_{f,(L,\varphi,s)}(r) := m_f(r,(L,\varphi,s)) + N_f(r,D),
\end{align*}
for all $r>0$. 

We can now state the two main theorems of Nevanlinna theory for holomorphic curves.

\begin{theorem}[Nevanlinna's first theorem (\cite{BombieriGubler}, Theorem 13.2.9)]
\label{th:Nevanlinna_first_main_theorem}
Let $f : \bC \to X$ be a holomorphic curve, with $X$ a projective complex variety. Let $L$ be a line bundle on $X$, let $s$ be a regular meromorphic section of $L$ and let $\varphi$ be a continuous Hermitian metric on $L$. Denote $D:=\div(s)$ and assume that $f(\bC)\not\subset |D|$. This data defines a height function $(r>0) \mapsto T_{f,(L,\varphi,s)}(r) \in \bR$.
\begin{itemize}
	\item[(1)] Let $\varphi'$ be another continuous Hermitian metric on $L$. Then we have 
	\begin{align*}
	T_{f,(L,\varphi,s)}(r) - T_{f,{L,\varphi',s}}(r) = \frac{1}{2\pi}\int_{0}^{2\pi} \log \frac{|s|_{\varphi'}}{|s|_{\varphi}}(f(re^{i\theta})) d\theta,
	\end{align*}
for all $r>0$. In particular, the LHS in the above equality is a bounded function of $r$. 
	\item[(2)] Let $s'$ be another regular meromorphic section of $L$ such that $f(\bC) \not\subset |\div(s')|$. Then we have
	\begin{align*}
	T_{f,(L,\varphi,s)}(r) - T_{f,{L,\varphi,s'}}(r) = \log |c((s'/s)\circ f, 0)|_{\infty},
	\end{align*}
for all $r>0$, where $c((s'/s)\circ f, 0)$ denotes the first non-zero coefficient in the Laurent series expansion of the meromorphic function $(s'/s) \circ f$.
\end{itemize}
\end{theorem}

As in the complex case, we study the asymptotic behaviour of the height function. The latter tends to $+\infty$ as $r\to +\infty$ if the holomorphic curve is not constant. Therefore, we consider height functions associated with different choices of metrics and regular meromorphic sections as equivalent. Properties of height functions are similar to those appearing in Diophantine geometry. 

\begin{proposition}
Let $f : \bC \to X$ be a holomorphic curve, where $X$ denotes a complex projective variety. 
\begin{itemize}
	\item[(1)] Let $(L_1,\varphi_1,s_1)$, $(L_2,\varphi_2,s_2)$ be continuous Hermitian metrised Cartier divisors on $X$. Let $\varphi_{1+2}$ be a continuous Hermitian metric on $L_1\otimes L_2$. Assume that $f(\bC) \not\subset |\div(s_1)|\cup|\div(s_2)|$. Then
	\begin{align*}
	T_{f,(L_1\otimes L_2,\varphi_{1+2},s_1 \otimes s_2)} (r) = T_{f,(L_1,\varphi_1,s_1)}(r) + T_{f,(L_2,\varphi_2,s_2)}(r) + O(1).
	\end{align*}
	\item[(2)] Let $\alpha : X \to Y$ be a morphism between complex projective varieties. Let $(L,\varphi)$ be a continuous Hermitian line bundle $Y$ and let $s$ be a regular meromorphic section of $L$. Assume that $f(\bC) \not\subset |\div(s)|$ and that $\alpha(X) \not\subset |\div(s)|$. Then
	\begin{align*}
	T_{f\circ \alpha,(L,\varphi,s)}(r) = T_{f,(f^{\ast}L,f^{\ast}\varphi,f^\ast s)}(r) + O(1).
	\end{align*}
	\item[(3)] Let $(L,\varphi)$ be a Hermitian line bundle $L$ on $X$ which is globally generated. Let $s$ be a regular meromorphic section of $L$. Then the height function $T_{f,(L,\varphi,s)}$ is bounded from below.
\end{itemize}
\end{proposition}

In the context of holomorphic curves, Nevanlinna's second theorem is not known. Its conjectural statement is known as the Griffith conjecture.
\begin{conjecture}[Griffiths' Conjecture]
 Let $A$ be an ample line bundle on a complex projective variety $X$. Denote $K_X := \wedge^{\dim(X)}T_{X}^{\ast}$ the canonical line bundle on $X$. Let $D$ be a normal crossing divisor on $X$. 
\begin{itemize}
	\item[(1)] Then, for any holomorphic curve $f : \bC \to X$ with Zariski dense image, the inequality
\begin{align*}
m_{f,D}(r) + T_{f,K_X}(r) \leq_{\exc} O(\log^+ T_{f,A}(r)) + o(\log(r))
\end{align*}
holds.
	
	\item[(2)] For any $\epsilon>0$, there exists an algebraic subset $Z \nsubseteq X$ such that, for any holomorphic curve $f : \bC \to X$ such that $f(\bC) \not\subset Z$, we have 
\begin{align*}
\forall C\in \bR,\quad m_{f,D}(r) + T_{f,K_X}(r) \leq_{\exc} \epsilon T_{f,A}(r) + C.
\end{align*}
\end{itemize}
\end{conjecture}

This conjecture is known in the case where $X$ is a curve and in the $X = \bP^{n}(\bC)$ case (cf. \cite{Vojta97} refining a result of Cartan). In the recent preprint \cite{DongHu22}, Dong and Hu announced a proof of the Griffiths conjecture. Unfortunately, there seems to be a gap in the paper. 

\subsection{Analogy with Diophantine approximation}

We conclude this appendix by giving the idea of the (heuristic) analogy between Diophantine approximation and Nevanlinna theory. For more details, we refer to \cite{Vojta87,BombieriGubler,Vojta10}.

\begin{center}
\begin{tabular}{|c|c|}
  \hline  
  Diophantine approximation & Nevanlinna theory \\
  \hline 
  $\bZ$ & $\cO(\bC)$\\  
  \hline 
  $\bQ$ &  $\cM(\bC)$\\  
  \hline
  $\{b_i : i\in I\} \subset \bQ$ infinite & $f\in \cO(\bC)$ non constant\\
  \hline
  $i\in I$ & $r>0$ \\
  \hline
  $\{\va_{\infty}\}\subset S$ finite set of absolute values on $\bQ$ & $\{\theta : \theta \in [0,2\pi]\}$ \\
  \hline 
  $M_{\bQ} \setminus S$ & $\{z\in \bR : |z|_{\infty} < r\}$\\
  \hline
  $|b_i|_{v}$, $v\in S$  & $|f(re^{i\theta})|_{\infty}$, $\theta \in [0,2\pi]$\\
  \hline
  $|b_i|_{v}$, $v\notin S$ & $\ord(f,z)$, $|z|_{\infty} < r$\\
  \hline 
   \end{tabular}
  
  \begin{tabular}{|c|c|}
  \hline  
  Diophantine approximation & Nevanlinna theory \\
  \hline
  Height  & Characteristic function \\ 
   $h(b_i) = \sum_{v\in M_{\bQ}} \log^{+}|b_i|_v$ & $T(f,\infty,r)$\\
  \hline
  Proximity function  & Proximity function \\ 
  $\forall a\in \bQ$, $m_{S}(a,b_i) =   \sum_{v\in S}\log^{+}|\frac{1}{b_i-a}|_v$ & $\forall a\in \bC$, $m(f,a,r)$\\
  \hline
  Counting function  & Counting function \\ 
  $\forall a\in \bQ$, $N_{S}(a,b_i) =   \sum_{v\notin S}\log^{+}|\frac{1}{b_i-a}|_v$ & $\forall a\in \bC$, $N(f,a,r)$\\
  \hline
  Product formula & Jensen's formula \\ 
  $\sum_{v\in S}\log|b_i|_v=0$ & $T(f,\infty,r) - T(f,0,r) = \log|c(f,0)|_{\infty}$\\
  \hline
  Height theory & First main theorem \\ 
  $\forall a\in \bQ$, $m_S(a,b_i)+N_{S}(a,b_i)= h(b_i) + O(1)$ & $\forall a\in\bC$, $T(f,a,r) = T(f,\infty,r) + O(1)$\\
  \hline
  Roth's theorem & Second main theorem (weak form) \\ 
  $\forall \epsilon>0$,  $\forall m\in\bZ{\geq 1}$, $\forall a_1,...,a_m\in \bQ$, & $\forall \epsilon>0$, $\forall m\in\bZ_{\geq 1}$, $\forall a_1,...,a_m\in\bC$,\\
 $\sum_{j=1}^n m(b_i,a_j) \leq_{\exc} (2+\epsilon)h(b_i)$ & $\sum_{j=1}^n m(f,a_j,r) \leq_{\text{exc}} (2+\epsilon)T(f,\infty,r)$\\
  \hline
  Defect & Defect \\ 
  $\forall a\in \bQ$, $\delta(a):= \liminf_{i\in I} \frac{m_{S}(a,b_i)}{h(b_i)}$ & $\forall a\in \bC$, $\delta(a):= \liminf_{r\to +\infty} \frac{m(f,a,r))}{T(f,\infty,r)}$\\
  \hline
\end{tabular}
\end{center}
In the line concerning Roth's theorem, $\leq_{\text{exc}}$ means that the inequality holds for all $i\in I$ except a finite number.

\section{Ultrafilters and ultraproducts}
\label{sec:ultrafilters}

\subsection{(Ultra)filters}
\label{sub:filters}

Let $I$ be a set. A \emph{filter} on $I$ is a subset $\cF \subset \cP(E)$ such that 
\begin{itemize}
	\item[(i)] $\emptyset \notin \cF$;
	\item[(ii)] for any $X,Y\in \cP(I)$ such that $X\subset Y$, if $X\in \cF$ then $Y\in \cF$;
	\item[(iii)] for any $X,Y\in \cF$, $X\cap Y\in \cF$. 
\end{itemize}
A filter $\cF$ on $I$ is called an \emph{ultrafilter} if for any $X\in \cP(I)$, either $X\in \cF$ or $I\setminus X \in \cF$. Zorn lemma implies that any filter on $I$ is contained in an ultrafilter (\cite{BouTG}, Chapitre I, \S 6.4, Théorème 1). An filter $\cF$ on $I$ is called fixed if there exists an element $x\in I$ such that $x\in \bigcap_{X\in\cF} X$. In that case, $\cF$ is called \emph{fixed}. A filter that is not fixed is called \emph{free}. 

\begin{example}
\label{example:filters}
Let $I$ be a set. The following $\cF \subset \cP(I)$ are filters on $I$.
\begin{itemize}
\item[(1)] $\cF:=\cF_{x}:=\{X\in\cP(I) : x\in I\}$, for any $x\in I$. This is a fixed ultrafilter on $I$. Moreover, all fixed ultrafilters arise this way. 
\item[(2)] Assume that $I$ is infinite and set $\cF$ be the subset of all subsets of $I$ of finite complement. This is a free filter.
\item[(3)] $I=\bR_{>0}$ and $\cF$ is the subset of all subsets of $\bR_{>0}$ whose complement is of finite Lebesgue measure. This is also a free filter. 
\end{itemize}
\end{example}

Assume that $I$ is a topological space and $\cF$ be a filter on $X$. We say that $x\in I$ is a \emph{limit point} of $\cF$ if for any open neighbourhood $U$ of $x$ in $I$, $U\in \cF$. Note that $I$ is Hausdorff, resp. compact, iff any filter on $I$ has at most one limit point, resp. has a limit point (\emph{loc. cit.}, Chapitre I, \S 8.1-\S 9.1). 

Let $f:I\to J$ be a map of sets $\cF$ be a filter on $I$. Then define the \emph{pushforward filter} 
$$f_{\ast}\cF := \{Y\in \cP(J) : f^{-1}(Y)\in \cF\}.$$
It is straightforward to check that $f_{\ast}(\cF)$ is a filter on $J$ that is an ultrafilter if $\cF$ is an ultrafilter. Assume that $J$ is a topological space. We say that $f$ admits a limit $x\in X$ along $\cF$ if $x$ is a limit point of $f_{\ast}\cF$. 

Let $\cF$ be a filter on a set $I$. We say that $\cF$ is $\delta$\emph{-incomplete} if there exists a countable family $(X_{n})_{n\in \bN}\in \cF^{\bN}$ such that $\bigcap_{n\in\bN}X_{n}=\emptyset$. For instance, any free filter on a countable set is $\delta$-incomplete (\cite{Vath07}, Proposition 4.11). Moreover, the filter $\cF$ on $\bR_{>0}$ defined in Example \ref{example:filters} (3) is $\delta$-incomplete (take $X_{n}:=]n,+\infty[$ for all $n\in \bN$).

\subsection{Ultraproducts}
\label{sub:ultraproducts}

From now on, we fix an infinite set $I$ and $\cU$ be a free ultrafilter on $I$. 

\subsubsection{Ultraproduct of sets}
\label{subsub:ultraproduct_sets}

Let $E=(E_{i})_{i\in I}$ be a family of sets indexed by $I$. The \emph{ultraproduct} of $E$ is defined as
\begin{align*}
E_{\cU} := \prod_{\cU}E_{i} := \left(\prod_{i\in I}E_{i}\right)/\sim_{\cU},
\end{align*}
where $\sim_{\cU}$ denotes the equivalence relation on $\prod_{i\in I}E_{i}$ defined by
\begin{align*}
(x_{i})_{i\in I} \sim_{\cU} (y_{i})_{i\in I} \Leftrightarrow \{i\in I : x_{i}=y_{i}\}\in \cU.
\end{align*}
 
\subsubsection{Ultraproduct of fields}
\label{subsub:ultraproduct_fields}

Let $K=(K_{i})_{i\in I}$ be a family of fields indexed by $I$. The \emph{ultraproduct} of the family $K$ is defined, as a set, as
\begin{align*}
K_{\cU} := \displaystyle\prod_{\cU} K_{i}.
\end{align*}
It is a standard fact that $K_{\cU}$ is a field. Note that even if the $K_{i}$'s are countable, the ultraproduct $K_{\cU}$ is uncountable.

\subsubsection{Ultraproduct of topological spaces}
\label{subsub:ultraproduct_topological_space}

Let $(\Omega_{i})_{i\in I}$ be a family of topological spaces indexed by $I$. Their \emph{ultraproduct} is defined, as a set, as 
\begin{align*}
\Omega_{\cU} := \displaystyle\prod_{\cU} \Omega_{i}. 
\end{align*}
We equip $\Omega_{\cU}$ with the topology generated by \emph{ultraboxes} of the form 
\begin{align*}
U_{\Omega} := \prod_{\cU} U_{i},
\end{align*}
where the $U_{i}$'s run over the open subsets of the $\Omega_{i}$'s. Note that this so-called \emph{ultraproduct topology} identifies with the quotient topology of the box topology on $\prod_{i\in I}\Omega_{i}$. 

\begin{proposition}[\cite{Bankston77}, Appendix 1]
Let $(\Omega_{i})_{i\in I}$ be a family of topological spaces indexed by $I$. Assume that the $\Omega_{i}$'s are all discrete, resp. Hausdorff. Then the ultraproduct $\Omega_{\cU}$ is discrete, resp. Hausdorff. 
\end{proposition}

\subsubsection{Ultraproduct of (Borel) measures}
\label{subsub:ultraproduct_measures}

Let $(\Omega_{i})_{i\in I}$ be a family of topological spaces indexed by $I$. For any $i\in I$, denote by $\cB_{i}$ the Borel $\sigma$-algebra of $\Omega_{i}$ and fix a Borel measure $\nu_{i}:\cB_{i}\to[0,+\infty]$ on $\Omega_{i}$. For any family $E=(E_{i})_{i\in I}\in \prod_{i\in I}\cB_{i}$, let $E_{\cU}:=\prod_{\cU}E_{i}$ denote the corresponding ultraproduct. Then the collection $\cB_{\cU}:=(E_{\cU})_{E=(E_{i})_{i\in I}}$, where the families $E$ are as above, is a $\sigma$-algebra on $\Omega_{\cU}$ that identifies with the Borel $\sigma$-algebra. 

Consider the map $\nu_{\cU} : \cB_{\cU}\to\prod_{\cU}[0,+\infty]$ defined by
\begin{align*}
\forall E_{\cU}=\prod_{\cU}E_{i}\in \cB_{\cU}, \quad  \nu_{\cU}(E_{\cU}) := [\nu_{i}(E_{i})].
\end{align*}
Then $\nu_{\cU}$ satisfies the properties: 
\begin{itemize}
	\item[(i)] $\nu_{\cU}(\emptyset)=[(0)_{i\in I}]$;
	\item[(ii)] $\forall E^{(1)}_{\cU},E^{(2)}_{\cU}\in\cB_{\cU}$ such that $E^{(1)}_{\cU}\cap E^{(2)}_{\cU}=\emptyset$, $\nu_{\cU}(E^{(1)}_{\cU}\cup E^{(2)}_{\cU})=\nu_{\cU}(E^{(1)}_{\cU})+\nu_{\cU}(E^{(2)}_{\cU})$.
\end{itemize}
We say that $\nu_{\cU}=:\prod_{\cU}$ is the \emph{ultraproduct} of the family $(\nu_{i})_{i\in I}$ (w.r.t. the ultrafilter $\cU$). 

\bibliographystyle{alpha}
\bibliography{biblio}

\begin{thebibliography}{HWWY21}

\bibitem[Ahl66]{Ahlfors66}
Lars~Valerian Ahlfors.
\newblock {\em Complex Analysis}.
\newblock McGraw-Hill Book Company, 2nd edition, 1966.

\bibitem[Ara74]{Arakelov74}
Souren~Yu Arakelov.
\newblock Intersection theory of divisors on an arithmetic surface.
\newblock {\em Math. USSR Izvestija}, 8:1167--1180, 1974.

\bibitem[Ban77]{Bankston77}
Paul Bankston.
\newblock {Ultraproducts in topology}.
\newblock {\em General Topology and its Applications}, 7(3):283--308, 1977.

\bibitem[BC13]{BostChen13}
Jean-Beno{\^{i}}t Bost and Huayi Chen.
\newblock {Concerning the semistability of tensor products in Arakelov
  geometry}.
\newblock {\em Journal des Math\'ematiques Pures et Appliqu\'ees}, 99(4), 2013.

\bibitem[BE21]{BoucksomEriksson}
S{\'e}bastien Boucksom and Dennis Eriksson.
\newblock {Spaces of norms, determinant of cohomology and Fekete points in
  non-Archimedean geometry}.
\newblock {\em {Advances in Mathematics}}, 378:107501, February 2021.

\bibitem[Ber90]{Berko90}
Vladimir~G. Berkovich.
\newblock {\em Spectral Theory and Analytic Geometry over Non-Archimedean
  Fields}.
\newblock American Mathematical Soc., 1990.

\bibitem[BG06]{BombieriGubler}
Enrico Bombieri and Walter Gubler.
\newblock {\em Heights in Diophantine Geometry}.
\newblock New Mathematical Monographs 4, Cambridge University Press, Cambridge,
  2006.

\bibitem[BGPS16]{BPS16}
Jos\'e~Ignacio Burgos~Gil, Patrice Philippon, and Mart\'in Sombra.
\newblock Height of varieties over finitely generated fields.
\newblock {\em Kyoto Journal of Mathematics}, 56(1), April 2016.

\bibitem[Bog07]{Bogatchev07}
V.~I. Bogachev.
\newblock {\em Measure theory. {V}ol. {I}, {II}}.
\newblock Springer-Verlag, Berlin, 2007.

\bibitem[Bos96]{Bost96}
Jean-Beno{\^\i}t Bost.
\newblock P\'eriodes et isog\'enies des vari\'et\'es ab\'eliennes sur les corps
  de nombres.
\newblock In {\em S\'eminaire Bourbaki : volume 1994/95, expos\'es 790-804},
  number 237 in Ast\'erisque, pages 115--161. Soci\'et\'e math\'ematique de
  France, 1996.
\newblock talk:795.

\bibitem[Bos01]{Bost01}
Jean-Beno{\^\i}t Bost.
\newblock Algebraic leaves of algebraic foliations over number fields.
\newblock {\em Publications Math\'ematiques de l'IH\'ES}, 93:161--221, 2001.

\bibitem[Bou71]{BouTG}
Nicolas Bourbaki.
\newblock {\em Topologie G\'en\'erale}.
\newblock Springer-Verlag Berlin Heidelberg, 1971.

\bibitem[BS25]{BallaySombra24}
Fran\c{c}ois Balla\"{y} and Mart\'in Sombra.
\newblock Approximation of adelic divisors and equidistribution of small
  points, 2025.
\newblock \href{https://arxiv.org/abs/2407.14978}{arXiv:2407.14978}.

\bibitem[BYDHS24]{GVF24}
Ita\"{i} Ben~Yaacov, Pablo Destic, Ehud Hrushovski, and Micha\l{} Szachniewicz.
\newblock Globally valued fields: foundations, 2024.
\newblock \href{https://arxiv.org/abs/2409.04570}{arXiv:2409.04570}.

\bibitem[CG24]{CaiGubler24}
Yulin Cai and Walter Gubler.
\newblock Abstract divisorial spaces and arithmetic intersection numbers, 2024.
\newblock \href{https://arxiv.org/abs/2409.00611}{arXiv:2409.00611}.

\bibitem[CJ23]{ChenJeannin23}
Huayi Chen and Marion Jeannin.
\newblock {Harder-Narasimhan games}, 2023.
\newblock \href{https://arxiv.org/abs/2306.08283}{arXiv:2306.08283}.

\bibitem[CL06]{ChambertLoir06}
Antoine Chambert-Loir.
\newblock {Mesures et {\'{e}}quidistribution sur les espaces de Berkovich}.
\newblock {\em Journal f\"ur die reine und angewandte Mathematik},
  2006(595):215--235, 2006.

\bibitem[CLD12]{CLD12}
Antoine Chambert-Loir and Antoine Ducros.
\newblock {Formes diff\'erentielles r\'eelles et courants sur les espaces de
  Berkovich}, 2012.
\newblock \href{https://arxiv.org/abs/1204.6277}{arXiv:1204.6277}.

\bibitem[CM19]{ChenMori}
Huayi Chen and Atushi Moriwaki.
\newblock {\em Arakelov geometry over adelic curves}, volume 2258 of {\em
  Lectures Notes in Mathematics}.
\newblock Springer Singapore, 2019.

\bibitem[CM21]{ChenMori21}
Huayi Chen and Atushi Moriwaki.
\newblock Arithmetic intersection theory over adelic curves, 2021.
\newblock \href{https://arxiv.org/abs/2103.15646}{arXiv:2103.15646}.

\bibitem[CM24]{ChenMori24}
Huayi Chen and Atsushi Moriwaki.
\newblock {\em Positivity in Arakelov Geometry over Adelic Curves}.
\newblock Progress in Mathematics. Birkh\"{a}user Cham, 2024.

\bibitem[DH22]{DongHu22}
Xianjing Dong and Peichu Hu.
\newblock {On Griffiths conjecture}, 2022.
\newblock \href{https://arxiv.org/abs/2203.10985}{arXiv:2203.10985}.

\bibitem[DHS24]{DHS24}
Pablo Destic, Nuno Hultberg, and Micha\l{} Szachniewicz.
\newblock Continuity of heights in families and complete intersections in toric
  varieties, 2024.
\newblock \href{https://arxiv.org/abs/2412.15988}{arXiv:2412.15988}.

\bibitem[DZ25]{DolceZucconi25}
Paolo Dolce and Francesco Zucconi.
\newblock {On the generalization of Roth's theorem}.
\newblock {\em Kyoto Journal of Mathematics}, pages 1--33, 2025.

\bibitem[Eng89]{Engelking89}
Ryszard Engelking.
\newblock {\em General Topology}, volume~6 of {\em Sigma Series in Pure
  Mathematics}.
\newblock Heldermann, Berlin, 1989.

\bibitem[Fal91]{Faltings91}
Gerd Faltings.
\newblock {Diophantine Approximation on Abelian Varieties}.
\newblock {\em Annals of Mathematics}, 133(3):549--576, 1991.

\bibitem[FK18]{Fujiwara-Kato}
Kazuhiro Fujiwara and Fumiharu Kato.
\newblock {\em Foundations of rigid geometry. {I}}.
\newblock EMS Monographs in Mathematics. European Mathematical Society (EMS),
  Z\"urich, 2018.

\bibitem[For81]{Forster81}
Otto Forster.
\newblock {\em {Lectures on Riemann Surfaces}}, volume~81 of {\em Graduate
  Texts in Mathematics}.
\newblock Springer-Verlag New York, 1981.

\bibitem[Fre03]{Fremlin03}
D.H. Fremlin.
\newblock {\em Measure Theory: Topological Measure Spaces (Vol. 4)}.
\newblock 2003.

\bibitem[Gas25]{Gasbarri25}
Carlo Gasbarri.
\newblock Rational points of bounded height on entire curves, 2025.
\newblock \href{https://arxiv.org/abs/2504.06665}{arXiv:2504.06665}.

\bibitem[Gau20]{Gaudron20}
Eric Gaudron.
\newblock {MINIMA AND SLOPES OF RIGID ADELIC SPACES}.
\newblock In Ga{\"e}l R{\'e}mond and Emmanuel Peyre, editors, {\em {Arakelov
  Geometry and diophantine applications}}, volume Lecture Notes in Mathematics
  of {\em Arakelov Geometry and Diophantine Applications}. {Springer
  International Publishing}, 2020.
\newblock Available in video format on Hal (@book{book:{91572743}, title =
  {Measure Theory: Topological Measure Spaces (Vol. 4)}, author = {D.H.
  Fremlin}, isbn = {0953812944; 9780953812943}, year = {2003}, url =
  {libgen.li/file.php?md5=d0eca533bba2a6c314df429da1857c37}}medihal-01718485
  for Part I) and Youtube.

\bibitem[Gil11]{Gillam11}
William~D. Gillam.
\newblock {Localization of Ringed Spaces}.
\newblock {\em Advances in Pure Mathematics}, 1:250--263, 2011.

\bibitem[GK17]{GK17}
Walter Gubler and Klaus K\"unnemann.
\newblock {A tropical approach to nonarchimedean Arakelov geometry}.
\newblock {\em Algebra and Number Theory}, 11(1):77--180, Jan 2017.

\bibitem[GK19]{GK19}
Walter Gubler and Klaus K{\"{u}}nnemann.
\newblock {Positivity properties of metrics and delta-forms}.
\newblock {\em Journal fur die Reine und Angewandte Mathematik}, 2019(752),
  2019.

\bibitem[GR17]{GaudronRemond17}
Eric Gaudron and Ga{\"e}l R{\'e}mond.
\newblock {Corps de Siegel}.
\newblock {\em {Journal f{\"u}r die reine und angewandte Mathematik}},
  2017(726):187--247, January 2017.

\bibitem[Gub97]{Gubler97}
Walter Gubler.
\newblock {\em {Heights of subvarieties over M-fields. \emph{In} Arithmetic
  geometry (Cortona 1994)}}, pages p. 190--227.
\newblock Cambridge University Press, 1997.

\bibitem[GW10]{Gortz10}
Ulrich G\"ortz and Torsten Wedhorn.
\newblock {\em Algebraic geometry . I, Schemes with examples and exercises}.
\newblock Springer Studium Mathematik - Master. Springer Spektrum Wiesbaden,
  2nd ed. 2020 edition, 2010.

\bibitem[Hru16]{Hrushovski16}
Ehud Hrushovski.
\newblock A logic for global fields.
\newblock S\'eminaire d'Arithm\'etique et de G\'eom\'etrie Alg\'ebrique
  d'Orsay, 2016.

\bibitem[HWWY21]{Heittokangas2021}
J~Heittokangas, J~Wang, Z~T Wen, and H~Yu.
\newblock {Meromorphic functions of finite $\phi$-order and linear q-difference
  equations}.
\newblock {\em Journal of Difference Equations and Applications},
  27(9):1280--1309, 2021.

\bibitem[Iss66]{Isssa66}
Hej Iss'sa.
\newblock {On the Meromorphic Function Field of a Stein Variety}.
\newblock {\em Annals of Mathematics}, 83(1):34--46, 1966.

\bibitem[KS90]{KS90}
Masaki Kashiwara and Pierre Schapira.
\newblock {\em Sheaves on manifolds}, volume 292 of {\em Grundlehren der
  Mathematischen Wissenschaften [Fundamental Principles of Mathematical
  Sciences]}.
\newblock Springer-Verlag, Berlin, 1990.
\newblock With a chapter in French by Christian Houzel.

\bibitem[KST17]{Kerz17}
Moritz Kerz, Florian Strunk, and Georg Tamme.
\newblock Algebraic k-theory and descent for blow-ups.
\newblock {\em Inventiones mathematicae}, 211(2):523--577, August 2017.

\bibitem[Lan74]{Lang74}
Serge Lang.
\newblock {Higher dimensional diophantine problems}.
\newblock {\em Bulletin of the American Mathematical Society}, 80(5):779--787,
  1974.

\bibitem[Lan86]{Lang86}
Serge Lang.
\newblock {Hyperbolic and Diophantine analysis}.
\newblock {\em Bulletin (New Series) of the American Mathematical Society},
  14(2):159--205, 1986.

\bibitem[Lan91]{Lang91}
Serge Lang.
\newblock {\em Number Theory III: Diophantine geometry}.
\newblock (Encyclopedia of mathematical sciences 60). Springer,
  Berlin-Heidelberg, 1991.

\bibitem[LC90]{LangCherry90}
Serge Lang and William Cherry.
\newblock {\em Topics in Nevanlinna theory}.
\newblock Lecture Notes in Mathematics. Springer, Berlin, 1990.

\bibitem[LP24]{LemanissierPoineau24}
Thibaud Lemanissier and J\'er\^ome Poineau.
\newblock {\em {Espaces de Berkovich Globaux}}.
\newblock Progress in Mathematics. Birkh\"{a}user Cham, 2024.

\bibitem[Mor00]{Moriwakiheights}
Atushi Moriwaki.
\newblock Arithmetic height functions over finitely generated fields.
\newblock {\em Inventiones Mathematicae}, 2000.

\bibitem[Mor25]{Morrow25}
Jackson~S. Morrow.
\newblock Global pluripotential theory for adelic line bundles, 2025.
\newblock \href{https://arxiv.org/abs/2507.10410}{arXiv:2507.10410}.

\bibitem[Nev70]{Nevanlinna70}
Rolf Nevanlinna.
\newblock {\em Analytic functions}.
\newblock Die Grundlehren der mathematischen Wissenschaften. Springer, 1970.

\bibitem[Ngu24]{Nguyen24}
Dong Quan~Ngoc Nguyen.
\newblock {Ultra-Galois theory and an analogue of the Kronecker--Weber theorem
  for rational function fields over ultra-finite fields}, 2024.
\newblock \href{https://arxiv.org/abs/2408.02158}{arXiv:2408.02158}.

\bibitem[NW14]{Noguchi14}
Junjiro Noguchi and Jörg Winkelmann.
\newblock {\em Nevanlinna Theory in Several Complex Variables and Diophantine
  Approximation}.
\newblock Springer Japan, 2014.

\bibitem[Osg81]{Osgood81}
C.F. Osgood.
\newblock {A number-theoretic-differential equations approach to generalizing
  Nevanlinna theory}.
\newblock {\em Indian Journal of Mathematics}, 23(1-3):1--15, 1981.

\bibitem[Pau09]{Paugam09}
Fr\'ed\'eric Paugam.
\newblock Global analytic geometry.
\newblock {\em Journal of Number Theory}, 129(10):2295 -- 2327, 2009.

\bibitem[Poi13]{Poineau13a}
J\'er\^ome Poineau.
\newblock Espaces de berkovich sur $\mathbb{Z}$: \'etude locale.
\newblock {\em Inventiones mathematicae}, 194(3):535--590, Jan 2013.

\bibitem[Poi25]{Poineau25}
J\'er\^ome Poineau.
\newblock Valuative compactifications of analytic varieties, 2025.
\newblock \href{https://arxiv.org/abs/2503.18643}{arXiv:2503.18643}.

\bibitem[Rot55]{Roth55}
Klaus~Friedrich Roth.
\newblock Rational approximations to algebraic numbers.
\newblock {\em Mathematika}, 1955.

\bibitem[RT96]{RoyTHunder96}
Damien Roy and Jeffrey~L Thunder.
\newblock {An absolute Siegel's Lemma}.
\newblock {\em Journal f{\"{u}}r die reine und angewandte Mathematik},
  1996(476):1--26, 1996.

\bibitem[Sch23]{Schapira23}
Pierre Schapira.
\newblock An introduction to categories and sheaves, 2023.
\newblock
  \href{https://webusers.imj-prg.fr/pierre.schapira/LectNotes/CatShv.pdf}{https://webusers.imj-prg.fr/\textasciitilde
  pierre.schapira/LectNotes/CatShv.pdf}.

\bibitem[S{\'e}d23]{Sedillot23_differentiability}
Antoine S{\'e}dillot.
\newblock Differentiability of the $\chi$-volume function over an adelic curve,
  2023.
\newblock \href{https://arxiv.org/abs/2303.03377}{arXiv:2303.03377}.

\bibitem[S{\'e}d24]{Sedillotthese}
Antoine S{\'e}dillot.
\newblock {\em {Study of projective varieties over adelic curves}}.
\newblock Theses, {Universit{\'e} Paris Cit{\'e}}, July 2024.
\newblock
  \href{https://theses.hal.science/tel-04990131v1/file/va_Sedillot_Antoine.pdf}{https://theses.hal.science/tel-04990131v1/file/va\textunderscore
  Sedillot\textunderscore Antoine.pdf}.

\bibitem[S{\'e}d25]{Sedillot_pav}
Antoine S{\'e}dillot.
\newblock Pseudo-absolute values: foundations.
\newblock {\em Israel Journal of Mathematics}, (to appear), 2025.
\newblock
  \href{https://webusers.imj-prg.fr/~antoine.sedillot/pseudo-absolute_values_final.pdf}{https://webusers.imj-prg.fr/\textasciitilde
  antoine.sedillot/pseudo-absolute\textunderscore values\textunderscore
  final.pdf}.

\bibitem[Son25]{Song25}
Yinchong Song.
\newblock {Parametrization of geometric Beilinson--Bloch heights via adelic
  line bundles}, 2025.
\newblock \href{https://arxiv.org/abs/2406.19912}{arXiv:2406.19912}.

\bibitem[{Sta}23]{stacks}
The {Stacks project authors}.
\newblock The stacks project.
\newblock \url{https://stacks.math.columbia.edu}, 2023.

\bibitem[Sza23]{Michal23}
Micha\l{} Szachniewicz.
\newblock {Existential closedness of $\overline{\mathbb{Q}}$ as a globally
  valued field via Arakelov geometry}, 2023.
\newblock \href{https://arxiv.org/abs/2306.06275}{arXiv:2306.06275}.

\bibitem[Tem11]{Temkin11}
Michael Temkin.
\newblock {Relative Riemann-Zariski spaces}.
\newblock {\em Israel Journal of Mathematics}, 185(1):1--42, September 2011.

\bibitem[Ull98]{Ullmo98}
Emmanuel Ullmo.
\newblock {Positivit\'e et Discr\'etion des Points Alg\'ebriques des Courbes}.
\newblock {\em Annals of Mathematics}, 147(1):167--179, 1998.

\bibitem[V{\"a}t07]{Vath07}
Martin V{\"a}th.
\newblock {\em Nonstandard Analysis}.
\newblock Birkh\"{a}user Basel, 2007.

\bibitem[Voj87]{Vojta87}
Paul Vojta.
\newblock {\em Diophantine Approximations and Value Distribution Theory}.
\newblock Lecture Notes in Mathematics 1239. Springer-Verlag Berlin Heidelberg,
  1987.

\bibitem[Voj97]{Vojta97}
Paul Vojta.
\newblock {On Cartan's Theorem and Cartan's Conjecture}.
\newblock {\em American Journal of Mathematics}, 119(1):1--17, 1997.

\bibitem[Voj10]{Vojta10}
Paul Vojta.
\newblock {\em {Diophantine Approximation and Nevanlinna Theory}}, pages
  111--224.
\newblock Springer Berlin Heidelberg, Berlin, Heidelberg, 2010.

\bibitem[Voj21]{Vojta21}
Paul Vojta.
\newblock {Roth's Theorem over arithmetic function fields}.
\newblock {\em Algebra and Number Theory}, 15(8):1943--2017, November 2021.

\bibitem[Wei51]{Weil51}
Andr\'e Weil.
\newblock Arithmetic on algebraic varieties.
\newblock {\em Annals of Mathematics}, 53(3):412--444, 1951.

\bibitem[YZ21]{YuanZhang21}
Xinyi Yuan and Shou-Wu Zhang.
\newblock Adelic line bundles over quasi-projective varieties, 2021.
\newblock \href{https://arxiv.org/abs/2105.13587}{arXiv:2105.13587}.

\bibitem[Zha95]{Zhang95}
Shou-Wu Zhang.
\newblock Positive line bundles on arithmetic vaireties.
\newblock {\em Journal of the American Mathematical Society}, 1995.

\bibitem[Zha98]{Zhang98}
Shou-Wu Zhang.
\newblock {Equidistribution of Small Points on Abelian Varieties}.
\newblock {\em Annals of Mathematics}, 147(1):159--165, 1998.

\end{thebibliography}
\end{document}